\documentclass[12pt,openany]{amsbook}
\allowdisplaybreaks[1]

\usepackage{geometry}
\usepackage{chngcntr}
\usepackage{hyperref}
\usepackage{amssymb,amsmath,appendix,amsfonts,amsthm,nomencl,mathrsfs}
\usepackage[arrow, matrix, curve]{xy}
\usepackage[latin1]{inputenc}
\newcommand{\IC}{\mathbb{C}}
\newcommand{\IR}{\mathbb{R}}

\newcommand{\supp}{\mathrm{supp}}

\renewcommand{\H}{ H}

\newcommand{\kat}{\mathcal{K}}
\newcommand{\dyn}{\mathcal{D}}

\newcommand{\IT}{T}

\newcommand{\IX}{\mathbb{X}}
\newcommand{\IP}{\mathbb{P}}

\newcommand{\IMM}{\mathscr{M}}

\newcommand{\IBB}{\mathscr{B}}
\newcommand{\ILL}{\mathscr{L}}
\newcommand{\IJJ}{\mathscr{J}}

\newcommand{\IHH}{\mathscr{H}}

\newcommand{\IFF}{\mathscr{F}}
\newcommand{\IAA}{\mathscr{A}}

\newcommand{\rank}{\mathrm{rank}}

\renewcommand{\c}{ c }

\renewcommand{\symbol}{\mathrm{Symb}}

\newcommand{\IL}{L}
\newcommand{\IB}{B}

\newcommand{\IDD}{\mathscr{D}}

\newcommand{\ICC}{C^{\infty}}

\newcommand{\dom}{\mathrm{Dom}}

\newcommand{\loc}{\mathrm{loc}}
\newcommand{\IN}{\mathbb{N}}
\newcommand{\IZ}{\mathbb{Z}}
\newcommand{\IK}{\mathbb{K}}

\newcommand{\Id}{ d}

\newcommand{\f}{\frac}
\newcommand{\nn}{\nonumber}

\newcommand{\Q}{Q}

\newtheorem{theorem}{Theorem}[chapter]
\newtheorem{theoreme}{Theorem}[chapter]

\newtheorem{Lemma}[theorem]{Lemma}
\renewcommand{\theLemma}

\renewcommand{\theLemmae}

\newtheorem{Corollary}[theorem]{Corollary}
\renewcommand{\theCorollary} 

\newtheorem{Corollarye}[theoreme]{Corollary}
\renewcommand{\theCorollarye}

\renewcommand{\theConjecture}

\renewcommand{\theConjecturee}

\newtheorem{Theorem}[theorem]{Theorem}
\renewcommand{\theTheorem}

\newtheorem{Theoreme}[theoreme]{Theorem}
\renewcommand{\theTheoreme}

\newtheorem{Proposition}[theorem]{Proposition}
\renewcommand{\theProposition}

\newtheorem{Propositione}[theoreme]{Proposition}
\renewcommand{\thePropositione}

\newtheorem{Propandef}[theorem]{Proposition and definition}\renewcommand{\thePropandef}

\renewcommand{\thePropandefe}

\theoremstyle{definition} 

\newtheorem{Definition}[theorem]{Definition}
\renewcommand{\theDefinition}

\newtheorem{Definitione}[theoreme]{Definition}
\renewcommand{\theDefinitione}

\newtheorem{Example}[theorem]{Example}
\renewcommand{\theExample}

\newtheorem{Examplee}[theoreme]{Example}
\renewcommand{\theExamplee}

\newtheorem{Remark}[theorem]{Remark}
\renewcommand{\theRemark}

\newtheorem{Remarke}[theoreme]{Remark}
\renewcommand{\theRemarke}

 \newtheorem{Notation}[theorem]{Notation}
\renewcommand{\theNotation} 

\newtheorem{Notatione}[theoreme]{Notation}
\renewcommand{\theNotatione}

\newcommand{\chapternumbering}[1]{
  \setcounter{chapter}{0}
   \renewcommand{\thechapter}{\csname #1\endcsname{chapter}}}

\newcounter{myenumi}

\begin{document}

\title[Covariant Schrödinger semigroups]{Covariant Schrödinger semigroups on Riemannian manifolds\footnote{This is a shortened version (the Chapters VII - XIII have been removed). The full version has appeared in December 2017 as a monograph in the Birkhäuser series \emph{Operator Theory: Advances and Applications}, and only the latter version should be cited. \\
The full version can be obtained on:\\
\url{http://www.springer.com/de/book/9783319689029}}}

  \author{
  Batu Güneysu\vspace{6mm}

Humboldt-Universität zu Berlin
}

\maketitle 

\begin{center}
{\center\huge \emph{ For my daughter Elena}\vspace{3mm}

$\heartsuit$ }
\end{center}

%\author[B. G\"uneysu]{Batu G\"uneysu}

%\address{Batu G\"uneysu, Institut f\"ur Mathematik, Humboldt-Universit\"at zu Berlin, 12489 Berlin, Germany} \email{gueneysu@math.hu-berlin.de}

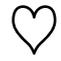
\tableofcontents

\chapternumbering{Roman}
\chapter*{Introduction}

Many problems in the analysis on Riemannian manifolds naturally lead to the study of operator semigroups having the form $(\mathrm{e}^{-t H^{\nabla}_V})_{t\geq 0}$, where the generator $H^{\nabla}_V$ is a covariant Schrödinger operator. In other words, $H^{\nabla}_V$ is a self-adjoint realization of $\nabla^{\dagger}\nabla+V$ in the underlying Hilbert space of square integrable sections, where $\nabla$ is a metric covariant derivative on a metric vector bundle over a Riemannian manifold, $\nabla^{\dagger}$ is its formal adjoint, and $V$ is a smooth self-adjoint endomorphism. For example, if we assume that $V$ is bounded from below by a constant, then a canonical choice for $H^{\nabla}_V$ is given by the Friedrichs realization of $\nabla^{\dagger}\nabla+V$.\\
Besides the semigroups that are induced by the scalar Laplace-Beltrami operator $-\Delta$ and the perturbation of $-\Delta$ by potentials, probably the most prominent examples of such covariant Schrödinger semigroups are provided by \lq\lq{}heat kernel proofs\rq\rq{} of the Atiyah-Singer index theorem on compact manifolds \cite{getzler}. In this situation, such a semigroup arises naturally through a Bochner-Weitzenböck formula $D^2=\nabla^{\dagger}\nabla+V$, with $D$ the underlying geometric Dirac operator. More precisely, the McKean-Singer formula states that the $\IZ_2$-graded index of $D$ is equal to the $\IZ_2$-graded trace of $\mathrm{e}^{-t H^{\nabla}_V}$ for all $t>0$, and so one can take $t\to 0$+ to obtain (in fact a stronger local version of) the Atiyah-Singer index theorem.\\
In this work, however, we are particularly interested in the analysis on noncompact Riemannian manifolds. In the context of geometric problems on such manifolds, the above covariant Schrödinger semigroups often appear in a disguise: Namely, one is often interested in estimating the gradient $\Id \mathrm{e}^{t \Delta} $ of the underlying \lq\lq{}heat semigroup\rq\rq{}. Noticing now that, at least formally, one has $\Id \mathrm{e}^{t \Delta} = \mathrm{e}^{t \Delta^{(1)}} \Id$ with $\Delta^{(1)}$ the Hodge-Laplacian acting on $1$-forms, one is back in the above covariant Schrödinger semigroup case, since by another Bocher-Weitzenböck formula  the operator $\Delta^{(1)}$ is again a covariant Schrödinger operator. A typical situation of this type appears in connection with heat kernel characterizations of the total variation: Namely, being motivated by E. de Giorgi\rq{}s Euclidean result \cite{giorgi}, one would like to establish the equality of the (total) variation $\mathrm{Var}(f)$ of a function $f$ on a noncompact Riemannian manifold to the limit $\lim_{t\to 0+}\left\|\Id \mathrm{e}^{t \Delta}\right\|_{\IL^1}$. \\
In the above geometric situations, the potential terms of the underlying covariant Schrödinger operators are smooth. One should notice that in the noncompact case many technical difficulties arise due to the behaviour of the potentials at $\infty$. On the other hand, it is well-known that in the standard representation of quantum mechanics one has to deal with Schrödinger operators in $\IR^3$ whose potential terms typically have local $1/|x|$-type Coulomb singuarities. \vspace{1mm}

The aim of this work is to establish the foundations of a general theory of \emph{covariant Schrödinger semigroups on noncompact Riemannian manifolds} which is flexible enough to deal with all above situations simultaneously. \vspace{1mm}

To this end, we introduce the concept of \emph{covariant Schrödinger bundles}: These are data of the form $(E,\nabla, V)\to M$, where
\begin{itemize}
\item $M$ is a smooth Riemannian $m$-manifold 
\item $E\to M$ is a smooth complex metric vector bundle with a finite rank
\item $\nabla$ is a smooth metric covariant derivative on $E\to M$
\item $V:M\to \mathrm{End}(E)$ is Borel measurable with $V(x):E_x\to E_x$ a linear self-adjoint map for every $x$. We will also refer to such an endomorphism $V$ as \emph{a potential on $E\to M$}. 
\end{itemize}

As usual, $E\to M$ and the corresponding metrics determine the complex Hilbert space $\Gamma_{\IL^2}(M,E)$ of square integrable sections. Using sesquilinear-form-techniques, one finds that under mild additional assumptions on $V$, such a Schrödinger bundle canonically induces a covariant Schrödinger operator in $\Gamma_{\IL^2}(M,E)$, which is semibounded from below. To explain this operator construction, let us denote by $p(t,x,y)$, $t>0$, $x,y\in M$, the minimal nonnegative heat kernel on $M$. In other words, for every fixed $y\in M$ the function $(t,x)\mapsto p(t,x,y)$ is the pointwise minimal nonnegative solution of the heat equation 
$$
(\partial/\partial t)p(t,x,y) =(1/2)\Delta_x p(t,x,y), \quad \lim_{t\to 0+}p(t,\bullet,y)=\delta_y.
$$
Then \emph{the contractive Dynkin class $\dyn(M)$} of the Riemannian manifold $M$ is given by all Borel functions $w$ on $M$ such that
\begin{align}\label{intro2}
\lim_{t\to 0+}\sup_{x\in M}\int^t_0\int_M p(s,x,y) |w(y)| \Id\mu(y)\Id s <1, 
\end{align}
with $\mu$ the Riemannian volume measure. Standard properties of the minimal nonnegative heat kernel imply $\IL^{\infty}(M)\subset \dyn(M)$. By definition, the \emph{Kato class} $\kat(M)$ is contained in $\dyn(M)$ as well, with $\kat(M)$ being defined as the space of all $w$\rq{}s that satisfy $\lim_{t\to 0+}\dots =0$ in (\ref{intro2}), instead of only $<1$. In view of $\IL^{\infty}(M)\subset \dyn(M)$, in the sequel a contractive Dynkin condition should always be understood as a very weak (and operator theoretic) form of boundedness.\vspace{1mm}

With these definitions, we call a potential $V$ as above \emph{contractively Dynkin decomposable}, if it can be decomposed $V=V_+-V_-$ into potentials with $V_{\pm}\geq 0$ \footnote{That is, for all $x\in M$ the eigenvalues of $V_{\pm}(x):E_x\to E_x$ are nonnegative.}, such that $V_+$ is $\IL^1_{\loc}$, and such that the fiberwise taken operator norm of $V_-$ satsfies $|V_-|\in \dyn(M)$. It is then possible with some efforts to prove that the symmetric sesquilinear form $Q^{\nabla}_V$ on $\Gamma_{\IL^2}(M,E)$ given by 
\begin{align}\label{aoid}
\dom(Q^{\nabla}_V)&:=\Gamma_{W^{1,2}_{\nabla,0}}(M,E)\cap \Big\{f:\int_M |(V f,  f)| \Id\mu <\infty \Big\},\\\nn
Q^{\nabla}_V(f_1,f_2)&:= (1/2)\int_M (\nabla f_1,\nabla f_2) \Id\mu + \int_M (V f_1, f_2) \Id\mu ,
\end{align}
is densely defined, closed and semibounded\footnote{In the sequel, \lq\lq{}semibounded\rq\rq{} is always understood to mean \lq\lq{}semibounded from below\rq\rq{}.}. Thus, by a classical functional analytic result this form canonically induces a self-adjoint semibounded operator $H^{\nabla}_V$ in $\Gamma_{\IL^2}(M,E)$, that we will refer to as the \emph{covariant Schrödinger operator induced by $(E,\nabla, V)\to M$}. Note that $H^{\nabla}_V$ can be formally interpreted as a self-adjoint realization of the operator \lq\lq{}$(1/2)\nabla^{\dagger}\nabla+V$\rq\rq{}. However, since in general $V$ is not assumed to be $\IL^2_{\loc}$, the potential $V$ does not even need to map smooth compactly supported sections into square integrable ones, so that the expression $(1/2)\nabla^{\dagger}\nabla+V$ is not meaningful, if interpreted in the usual sense. This is why we have to use the above \lq\lq{}weak\rq\rq{} formulation right away. On the other hand, it will turn out that if the contractively Dynkin decomposable potential $V$ \emph{is} $\IL^2_{\loc}$, then $H^{\nabla}_V$ is precisely the Friedrichs realization of $(1/2)\nabla^{\dagger}\nabla+V$. \\
In the case of $V\equiv 0$, we will use the natural notation
$$
Q^{\nabla}:=Q^{\nabla}_V|_{V\equiv 0},\quad H^{\nabla}:=H^{\nabla}_V|_{V\equiv 0}.
$$
Usual scalar Schrödinger operators of the form \lq\lq{}$-(1/2)\Delta+w$\rq\rq{}, with scalar potentials $w:M\to\IR$, are naturally included in the above setting as follows: One picks the trivial vector bundle $E=M\times \IC\to \IC$, in which case sections become nothing but complex-valued functions on $M$, and potentials become nothing but real-valued Borel functions on $M$. If furthermore $\nabla=\Id$ is taken to be the usual exterior derivative, then for a Dynkin decomposable potential $w:M\to\IR$ we can form the operator 
$$
H_w:= H^{\Id}_w\quad\text{ in  $\IL^2(M)$.}
$$
In view of the formula $-\Delta= \Id^{\dagger}\Id$ for the scalar Laplace-Beltrami operator on $M$, it now becomes clear that $H_w$ is a self-adjoint realization of \lq\lq{}$-(1/2)\Delta+w$\rq\rq{} in $\IL^2(M)$. In particular, we set
$$
H:=H_{w}|_{w\equiv 0},\quad \text{with $Q$ the corresponding sesquilinear form,} 
$$
and recall that in this case $H$ is just the Friedrichs realization of $-(1/2)\Delta$. In particular, it turns out that the minimal nonnegative heat kernel is precisely the integral kernel 
$$
p(t,x,y)=\mathrm{e}^{-t H}(x,y),\quad  t>0,\> x,y\in M.
$$

The first part of this work is completely devoted to a systematic examination of strongly continuous and self-adjoint semigroups of the form
$$
(\mathrm{e}^{-t H^{\nabla}_V})_{t \geq 0} \subset \ILL(\Gamma_{\IL^2}(M,E))
$$
which are defined by the spectral calculus. In the above situation, $(\mathrm{e}^{-t H^{\nabla}_V})_{t \geq 0}$ will be called the \emph{covariant Schrödinger semigroup induced by} $(E,\nabla, V)\to M$. An essential feature of our analysis will be that we will usually not require any kind of control on the geometry of $M$. Our general philosophy for the examination of these semigroups is to proceed in three steps: \vspace{1.2mm}

{\bf Step 1.} We analyze $(\mathrm{e}^{-t H})_{t\geq 0}$. Of course the study of this semigroup is essentially equivalent to the study of the scalar heat kernel $p(t,x,y)$, and it is well-known that there exist many special methods for the examination of $p(t,x,y)$, such as minimum principles, mean value inequalities, and so on. The starting point of our analysis is the following simple but nevertheless essential observation: For every relatively compact subset $U\subset M$ one has
$$
\sup _{x\in U}\sup_{y\in M}p(t,x,y)<\infty,
$$
without any further assumptions on the geometry of $M$. For example, this observation allows us to derive explicit $\IL^{q_1}\to \IL^{q_2}_{\loc}$ bounds for $(\mathrm{e}^{-t H})_{t>0}$ in this full generality, which extend to $\IL^{q_1}\to \IL^{q_2}$ bounds precisely if $M$ is ultracontractive, that is, if one has
$$
\sup_{x,y\in M}p(t,x,y)<\infty\quad\text{ for all  $t>0$}. 
$$
On the other hand, some applications require a good control of the quantities
$$
\sup_{y\in M}p(t,x,y),\quad x\in M,
$$
for both large and small $t$ (like compactness results for operators of the form $V(H^{\nabla}+1)^{-1})$), while some others only for small $t$ (like the $\dyn(M)$ assumption). Being motivated from results that are contained in A. Grigor\rq{}yan\rq{}s book on the analysis of $p(t,x,y)$, we introduce the concept of \emph{heat kernel control pairs $(\Xi,L)$ for the Riemannian manifold $M$} in order to measure the above effects simultaneously. These are pairs of functions 
$$
\Xi: M\to (0,\infty],\quad \tilde{\Xi}:(0,\infty)\longrightarrow (0,\infty),
$$
such that $\Xi$ is continuous with 
$$
\sup_{y\in M}p(t,x,y)\leq \Xi(x) \tilde{\Xi}(t)\quad\text{ for all $t>0$,  $x\in M$,} 
$$
and such that $\tilde{\Xi}$ has the $\IL^q$-properties of $t\mapsto t^{m/2}$ near $t=0$ and behaves no worse then $t\mapsto \exp(Ct)$, for some $C>0$, at $t=\infty$. Using a parabolic $\IL^1$-mean-value inequality we find, in particular, that every Riemannian manifold admits a canonical heat kernel control pair, a result that has been established by the author in \cite{guenkat}, based on results by A. Grigor\rq{}yan. Whenever one has some specific knowledge of the geometry, one can simply pick \lq{}better\rq{} control pairs.
\vspace{1.2mm}

{\bf Step 2.} We analyze scalar Schrödinger semigroups of the form $(\mathrm{e}^{-t H_w})_{t\geq 0}$ by reducing their analysis to that of the heat semigroup $(\mathrm{e}^{-t H})_{t\geq 0}$. It is here that, in our eyes, probabilistic methods are really efficient, and we will use them through the Feynman-Kac formula
\begin{align}\label{auebc}
\mathrm{e}^{-t \H_w}f(x)=\int_{\{t<\zeta\}}\mathrm{e}^{-\int^t_0 w(\IX_s)\Id s}f(\IX_t)\Id \mathbb{P}^x,
\end{align}
where $\mathbb{P}^x$ denotes integration with respect to Brownian motion starting from $x$, and where $\{t<\zeta\}$ denotes the set of Brownian paths which do not explode until the time $t\geq 0$. Note that on a general \emph{stochastically incomplete} Riemannian manifold it can happen that $\mathbb{P}^x(\{t<\zeta\})<1$. Since, in particular, one has
$$
\mathrm{e}^{-t \H}f(x)=\int_{\{t<\zeta\}}f(\IX_t)\Id \mathbb{P}^x,
$$
the Feynman-Kac formula and the inequality $-w\leq w_-$ show that the analysis of $\mathrm{e}^{-t \H_w}f(x)$ can be controlled by that of $\mathrm{e}^{-t \H}f(x)$, once one has a machinery to estimate exponentials of the form
\begin{align}\label{aepqg}
\int_{\{t<\zeta\}}\mathrm{e}^{\int^t_0 \alpha w_-(\IX_s)\Id s}f(\IX_t)\Id \mathbb{P}^x, \quad \alpha\geq 0,
\end{align}
with $w_-\in\dyn(M)$. Ultimately, based on \lq\lq{}Euclidean ideas\rq\rq{} by M. Aizenman and B. Simon from their seminal paper \cite{aizenman}, we will develop such a machinery in our general geometric context, too. As a remark, the idea behind these results is that by the very definition of Brownian motion, the Dynkin property (\ref{intro2}) of $w_-$ is equivalent to the probabilistic contraction property
\begin{align}\label{intro5}
\lim_{t\to 0+}\sup_{x\in M}\int_{\{t<\zeta\}}w_-(\IX_t)\Id \mathbb{P}^x<1,
\end{align}
and that, somewhat surprisingly, the Markov property of Brownian motion makes it indeed possible to control exponential expressions of the form (\ref{aepqg}) with finite and explicit constants that arise from the assumption (\ref{intro5}). \vspace{1.2mm}

{\bf Step 3.}  We reduce the study of a given covariant Schrödinger semigroup $(\mathrm{e}^{-t H^{\nabla}_V})_{t\geq 0}$ to that of a scalar Schrödinger semigroup of the form $(\mathrm{e}^{-t H_w})_{t\geq 0}$. The central machinery in this context is provided by what we call the \emph{Kato-Simon inequality}, which states the following semigroup domination: \emph{If in the above situation one has $V\geq w$ $\mu$-a.e., then}
\begin{align}\label{intro1}
\left|\mathrm{e}^{-t H^{\nabla}_V}f(x)\right |\leq\mathrm{e}^{-t H_w} \left|f\right|(x)\quad\text{ for all $f\in\Gamma_{\IL^2}(M,E)$, $\mu$-a.e. $x\in M$.}
\end{align}

Following the original ideas by E. Nelson and B. Simon \cite{simon1}, the author established (\ref{intro1}) in \cite{guen} using probabilistic methods, namely a covariant Feynman-Kac formula for $\mathrm{e}^{-t H^{\nabla}_V}f(x)$. On the other hand, it is certainly of interest to have an entirely analytic proof of the Kato-Simon inequality in the above generality. Based on a (local) covariant version of Kato\rq{}s distributional inequality \cite{kato} and functional analytic results by H. Hess, R. Schrader and D. Uhlenbrock \cite{hess2}, we provide such a proof here. As we do not require $M$ to be geodesically complete, several technical difficulties arise in this context already for $V\equiv 0$, $w \equiv 0$. Without entering into the details, let us only mention here that a key observation in this context is the following regularity of sections in the underlying form domain:
\begin{align}\label{inclusion}
f\in \dom(Q^{\nabla})\Rightarrow    |f|\in \dom(Q).
\end{align}
This property is also true without geodesic completeness. This is one of the incidences where it becomes important to work with a distinguished self-adjoint realization. In other words, one uses special properties of Sobolev spaces of the form $W^{1,2}_0$.\vspace{2mm}

The above 3-step-analysis allows us to establish, among other things, the following results, valid for every Schrödinger bundle $(E,\nabla, V)\to M$, where the only a priori assumption is that $V$ is contractively Dynkin decomposable (and, in particular, $M$ is completely arbitrary):\vspace{2mm}

{\bf (a)} a very general weighted $\IL^q$-criterion for the compactness of $V(H^{\nabla}+1)^{-1}$, which at least in the physically relevant case $m\leq 3$ does not require any control on the geometry of $M$; in particular, these results entail the stability of the essential spectrum
$$
\sigma_{\mathrm{ess}}(H^{\nabla}_V)=\sigma_{\mathrm{ess}}(H^{\nabla}).
$$

{\bf (b)} explicit $\IL^q\to \IL^q$ estimates for  $\mathrm{e}^{-t H^{\nabla}_V}$ without any assumptions on $M$, where $q\in [1,\infty]$\vspace{2mm}

{\bf (c)} explicit $\IL^q\to \IL^{\infty}_{\loc}$ and $\IL^1\to \IL^{q}_{\loc}$ estimates for  $\mathrm{e}^{-t H^{\nabla}_V}$ without any assumptions on $M$, where $q\in (1,\infty)$\vspace{2mm}

{\bf (d)} explicit $\IL^{q_1}\to \IL^{q_2}$ estimates for  $\mathrm{e}^{-t H^{\nabla}_V}$ and all $q_1,q_2\in  [1,\infty]$ with $q_1\leq q_2$, in case that $M$ is ultracontractive\vspace{2mm}

{\bf (e)} the joint continuity of $(t,x)\mapsto\mathrm{e}^{-t H^{\nabla}_V}f(x)$ on $(0,\infty)\times M$, for all fixed $f\in \Gamma_{\IL^2}(M,E)$ and without any assumptions on $M$, as long as $|V|\in \kat(M)$\vspace{2mm}

{\bf (f)} the existence of an integral kernel $\mathrm{e}^{-t H^{\nabla}_V}(x,y)$ for $\mathrm{e}^{-t H^{\nabla}_V}$, without any further assumptions on $M$, as well as estimates for $\mathrm{e}^{-t H^{\nabla}_V}(x,y)$ and for the trace $\mathrm{tr}(\mathrm{e}^{-t H^{\nabla}_V})$
\vspace{2mm}

{\bf (g)} the essential self-adjointness of $H^{\nabla}_V$ on smooth compactly supported sections, if $V$ is $\IL^2_{\loc}$ and $M$ is geodesically complete\vspace{2mm}

{\bf (h)} the fact that $\Gamma_{\ICC_{\c}}(M,E)$ is a core of the form $Q^{\nabla}_V$, without any further assumptions on $M$ or $V$.\vspace{2mm}

We remark here that in their ultimate form, the above results are new, except {\bf (a)} (which stems from the paper \cite{brugun} by J. Brüning and the author) and {\bf (g)} (which has been proved by O. Post and the author in \cite{post}). \vspace{1.2mm}

Altogether, we believe that the above results support the following statement: \vspace{1mm}

\emph{Essentially every Euclidean result for covariant Schrödinger operators with Dynkin or Kato potentials remains true for covariant Schrödinger operators on an arbitrary Riemannian manifold $M$, without any further assumptions on the underlying geometry.} \vspace{1mm}

The geometry, on the other hand, comes into play in an essential way only in a second \lq\lq{}layer\rq\rq{} of the problem, namely when one wants to establish convenient $\IL^q$-critera for potentials to be in $\dyn(M)$ and $\kat(M)$. As we will explain later on, an inclusion of the form $\IL^q(M)\subset \kat(M)$ (with $q$ chosen appropriately depending on the dimension $m$ of $M$) does require some control on the geometry of $M$; on the other hand, our previously mentioned machinery of heat kernel control pairs $( \Xi,\tilde{\Xi})$ will entail that, in fact, $\Xi$-weighted $\IL^q$-spaces are always included in $\kat(M)\subset \dyn(M)$, without any assumption on $M$. Again, our philsosophy here is that a good control on the geometry simply allows to pick \lq\lq{}better\rq\rq{} weight functions.\vspace{3mm}

The second part of this work is devoted to the application of the above results {\bf (a)}-{\bf (h)} to particular situations that arise in geometric analysis and physics on noncompact Riemannian manifolds. In this context, we establish the following results:\vspace{2mm}

{\bf (i)} We provide a precise geometric formulation as well as a proof of the statement \lq\lq{}the presence of a magnetic field leads to an increase of the bottom of the spectrum\rq\rq{}, which is valid for every Riemannian manifold. \vspace{2mm}

{\bf (j)} We generalize a Euclidean result by J. Fröhlich, E.H. Lieb and M. Loss \cite{stab1} to the setting of geodesically complete nonparabolic Riemannian $\mathrm{spin}^{\IC}$-$3$-manifolds by proving a geometric stability result for hydrogen-type atoms in the presence of magnetic fields. Here, we consider both situations: firstly, we neglect the electron's spin (which does not lead to any restrictions on the fields or the nucleus); secondly, taking the electron's spin $1/2$ appropriately into account, we derive that one has stability for magnetic fields with a finite \lq\lq{}self-energy\rq\rq{}, if in addition the nucleus does not have too many protons. This part of the work has been taken from the author's paper \cite{gunstab}.\vspace{2mm}

{\bf (k)} We formulate and prove a geometric version of the above mentioned de-Giorgi-type heat kernel chracterization of the total variation, valid for all geodesically complete Riemannian manifolds with a Ricci curvature admitting some negative part in the Kato class $\kat(M)$. In this context, we also examine the structure behind the space of vector measures on a Riemannian manifold, and the precise connection between this space and the space of functions that have a finite total variation.This part stems essentially from the paper \cite{pallara} by D. Pallara and the author.\vspace{2mm}

{\bf (l)} We prove that geodesically complete Riemannian manifolds with a Ricci curvature that is bounded from below by a constant are $\IL^q$-positivity preserving for all $q\in [1,\infty]$, this property being a \lq\lq{}supersolution\rq\rq{} refinement of an $\IL^q$-Liouville property due tu R. Strichartz (cf. Lemma 3.1 in \cite{strich}). In particular, $\IL^q$-positivity preserving $M$'s are automatically $\IL^q$-Liouville. This part is a generalization of the corresponding result by the author \cite{G54}, which originally required a nonnegative Ricci curvature. \vspace{3mm}

This work is organized as follows: \vspace{1mm}

In {\bf Chapter \ref{C2}}, we start by introducing basic concepts such as symbols and adjoints of linear partial differential operators acting on sections of vector bundles. The central objects of this chapter are globally defined Sobolev-type spaces of sections. As it does not cause much extra work, we do not restrict ourselves to Riemannian data here. Instead, we work with Sobolev spaces that are defined with respect to a finite family of linear partial differential operators acting between sections of vector bundles over a manifold, where the manifold is equipped with an arbitrary smooth measure. The main result of this chapter is a Meyers-Serrin-type result for these spaces, under a generalized ellipticity condition on the family of operators.\vspace{1mm}

In {\bf Chapter \ref{C3}}, we give a detailed proof of the fact that the heat semigroup of every semibounded self-adjoint realization of a linear partial differential operator acting between sections of vector bundles has a jointly smooth integral kernel.\vspace{1mm}

In {\bf Chapter \ref{C4}}, we first collect some facts about Riemannian geometry and then specialize and refine the previously established results on abstract Sobolev spaces to families of operators that naturally appear in Riemannian geometry. In particular, we prove a generalized version of (\ref{inclusion}) here, too.

\vspace{1mm}

{\bf Chapter \ref{scalar}} deals with specific results concerning the minimal heat kernel $p(t,x,y)$. We prove the aforementioned $\IL^q$-bounds for $\mathrm{e}^{-t H}$ and the heat kernel bounds in terms of control pairs for $p(t,x,y)$. In addition, we collect results corresponding to the \lq\lq{}stochastic completeness\rq\rq{} and the \lq\lq{}parabolicity\rq\rq{} of Riemannian manifolds, notions that can be defined using $p(t,x,y)$ and that will play an important role in the sequel. \vspace{1mm}

{\bf Chapter \ref{stop}} deals with the definition and some technical results concerning the Wiener measure and Brownian motion.\vspace{1mm}

In {\bf Chapter \ref{C6}}, we then introduce the contractive Dynkin class $\dyn(M)$ and the Kato class $\kat(M)$ of a Riemannian manifold $M$, and we collect many useful results concerning these classes. In particular, we explain the connection between Brownian motion and these spaces, and we prove the aforementioned exponential estimates for expressions of the form (\ref{aepqg}). This chapter also deals with the above mentioned weighted $\IL^q$-criteria for these classes.\vspace{1mm}

In {\bf Chapter \ref{C7}}, we start by establishing the well-definedness of the above forms $Q^{\nabla}_V$. Afterwards we prove the Feynman-Kac formula (\ref{auebc}) in detail, and finally the Kato-Simon inequality (\ref{intro1}).\vspace{1mm}

The {\bf Chapters \ref{aiosys} to \ref{C12}} treat the above mentioned results {\bf (a)} to{\bf (h)}. \vspace{1mm}

{\bf Chapter \ref{C13}} deals with the above applications {\bf (i)} to{\bf (l)}.\vspace{1mm}

Finally, we have included an appendix containing the basics of the following topics: smooth manifolds and vector bundles, unbounded linear operators and unbounded sesquilinear forms in Hilbert spaces, and some partially nonstandard results from some measure theory.

\vspace{3mm}

\textbf{Acknowledgements:} I have benefited from mathematical discussions with Helga Baum, Jochen Brüning, Gilles Carron, Evgeny Korotyaev, Kazuhiro Kuwae, Jörn Müller, Diego Pallara, Alberto Setti, Christoph Stadtmüller, Anton Thalmaier, and Eren Ucar. Tobias Schwaibold has done a fantastic job concerning the copy editing.\\
Finally, I would like to thank my friends Francesco Bei, Sergio Cacciatori, Ognjen Milatovic, Olaf Post, and Stefano Pigola, who have always shared their knowledge very generously with me.

\newpage

\textbf{Conventions:} In the sequel, all manifolds are understood to be without boundary, unless otherwise stated. The reader may find some basics of differential geometry (in particular, the differential geometric notation which is used throughout this book) in Section \ref{difftop} of the appendix. Given a smooth vector bundle $E\to X$, the space of smooth (respectively, smooth compactly supported) sections in $E\to X$ is denoted by $\Gamma_{C^{\infty}}(X,E)$ (respectively $\Gamma_{C^{\infty}_c}(X,E)$), with an analogous notation for $C^k$-sections. In particular, given a smooth $m$-manifold $X$, the symbol 
$$
\mathscr{X}_{C^{\infty}} (X)= \Gamma_{C^{\infty}}(X, T X)
$$
will stand for the $\IR$-linear space of smooth vector fields on $X$, and for every $k=0,\dots, m$, the symbol
$$
\Omega^k_{C^{\infty}}(X)= \Gamma_{C^{\infty}}(X, \wedge^k T^*_{\IC} X)
$$
will stand for the $\IC$-linear space of smooth \emph{complex-valued} $k$-forms on $X$, with
$$
\Omega^0_{C^{\infty}}(X):=C^{\infty}(X):=C^{\infty}(X,\IC),\quad \Omega {C^{\infty}}(X):=\oplus^m_{k=0}\Omega^k_{C^{\infty}}(X).
$$
In particular, $C^{\infty}(X)$ denotes the smooth \emph{complex-valued} functions on $X$. When we say that $\alpha\in \Omega^k_{C^{\infty}}(X)$ is real-valued, this simply means that it is a smooth section of $\wedge^k T^* X \to X$. Notations such as $ C^{\infty}_{\IR}(X)$ will be self-explanatory then.\vspace{1mm}

The symbol $\Re ( z)$ denotes the real part of a complex number $z$, and $\Im (z)$ its imaginary part.\vspace{1mm}

Whenever there is no danger of confusion, the complexification of a linear map between real-linear spaces will be denoted by the same symbol again.\vspace{1mm}

Finally, we use the abbreviation that every complex scalar product is antilinear in its first slot, and that every measure is understood to be nonnegative and not necessarily complete, unless otherwise stated.

\chapter{Sobolev spaces on vector bundles}\label{C2}

The aim of this chapter is to establish the basics corresponding to differential operators that act on sections of vector bundles, and to introduce some induced abstract Sobolev-type spaces. Clearly, this requires some understanding of distributional or weak derivatives. In principle, such weak derivatives can be defined with the help of an intrinsic theory of distributions, in which the space of test sections of a smooth vector bundle $E\to X$ is given by the locally convex space of smooth compactly supported sections of $E^*\otimes |X|\to X$. In this case, the space of distributions is taken to be the continuous dual of this space of test sections. Here, $|X|\to X$ denotes the bundle of densities on $X$ \cite{waldmann}.\\
On the other hand, since we are only interested in $W^{k,q}$-type Sobolev spaces (where $k\in \IN$, $q\in [1,\infty]$), we can follow the usual Euclidean strategy and give an ad-hoc definition of \lq\lq{}weak derivative with respect to a differential operator acting on a manifold\rq\rq{}. This definition turns out to be sufficient for our aims. Distributions that are more singular will only appear in a few proofs, and there only locally. Therefore, we can simply use the standard Euclidean machinery in these situations.

\section[Differential operators]{Differential operators with smooth coefficients and weak derivatives on vector bundles}

Let $X$ be a smooth $m$-manifold.

\begin{Definition}
A \emph{smooth Borel measure on $X$} is understood to be a Borel measure which in any chart of $X$ has a smooth and strictly positive density function with respect to the Lebesgue measure.
\end{Definition}

It follows from a partition of unity argument that $X$ always admits a smooth Borel measure. Let $\rho$ be such a smooth Borel measure for the moment. Then $\rho$ has a full topological support, meaning that $\rho(U)>0$ for all nonempty open $U\subset X$, and $\rho$ is locally finite, which means that $\rho(K)<\infty$ for all compact $K\subset X$. In addition, $\rho$ is outer and inner regular, meaning, respectively, that for every Borel set $N\subset X$ one has
\begin{align*}
&\rho(N)=\inf\{\rho(U):\>\text{ $U\supset N$, $U$ is open}\},\\
&\rho(N)=\sup\{\rho(K):\>\text{ $K\subset N$, $K$ is compact}\}.
\end{align*}
All these facts follow from applying the corresponding results for the Lebesgue measure in charts. In addition, any two smooth Borel measures on $X$ are equivalent, that is, they are absolutely continuous with respect to each other. Locally, this is clear, since we assume the density functions to be $>0$. Globally, this follows from a partition of unity argument. Whenever we say that \emph{a property holds almost everywhere (a.e.) on $X$}, this means that there exists a smooth Borel measure $\rho$ on $X$ such that the property under consideration holds $\rho$-a.e.. Note that in this case, the statement automatically holds $\rho$-a.e. for every such $\rho$. \vspace{2mm}

Let us assume for the moment that we are given smooth $\IK$-vector bundles $E\to X$, $F\to X$. Let $\ell_0:=\mathrm{rank}(E)$ and $\ell_1:=\mathrm{rank}(F)$. 

\begin{Notation} The $\IK$-linear space of equivalence 
classes of Borel sections of $E\to X$ is denoted by $\Gamma(X,E)$. Given $\psi\in  \Gamma(X,E)$, its \emph{essential support} is the set $\supp(\psi)$ which is defined by 
\[
X\setminus\supp(\psi):= \bigcup    \{ U:  \text{ $U\subset X$ is open and $\psi(x)=0\in E_x$ for a.e. $x\in  U$}\}.
\]
\end{Notation}
The above equivalence classes are understood with respect to some smooth Borel measure on $X$. By what we have said, this notion does not depend on a particular choice of such measure. Furthermore, if $\psi$ is continuous, its essential support is equal to its usual topological support. If $X$ is an open subset of $\IR^m$ and $\psi\in \IL^1_{\loc}(X,\IK^{\ell_0})$, then the distributional support of $f$ is also equal to $\supp(\psi)$ defined as above.\vspace{2mm}

Before we can give a precise definition of smooth partial differential operators acting on sections of vector bundles, we record a simple fact on the restriction properties of certain operators. To this end, let
$$
P: \Gamma_{C^{\infty}}(X,E)\longrightarrow \Gamma_{C^{\infty}}(X,F)
$$
be a $\IK$-linear map which is local in the sense that for every $\psi\in\Gamma_{\ICC}(X,E)$ one has $\mathrm{supp}(P\psi)\subset  \mathrm{supp}(\psi)$. This property implies that $P$ maps compactly supported smooth sections into compactly supported smooth sections, and there exists a unique morphism of $\IK$-linear sheaves
$$
P|_{\bullet}: \Gamma_{C^{\infty}}(\bullet,E)\longrightarrow \Gamma_{C^{\infty}}(\bullet,F)
$$
such that $P|_U \psi|_U= (P\psi)|_U$ for all open $U\subset X$ and all $\psi\in\Gamma_{C^{\infty}}(X,E)$. This is easily checked using a cut-off function argument. Now we can give:

\begin{Definition}\label{ops} A local $\IK$-linear map
$$
P: \Gamma_{C^{\infty}}(X,E)\longrightarrow \Gamma_{C^{\infty}}(X,F)
$$
is called a \emph{smooth $\IK$-linear partial differential operator of order at most $k\in\IN_{\geq 0}$}, if for any chart $((x^1,\dots,x^m),U)$ of $X$ which admits frames $e_1,\dots,e_{\ell_0}\in\Gamma_{C^{\infty}}(U,E)$, 
$f_1,\dots,f_{\ell_1}\in\Gamma_{C^{\infty}}(U,F)$, and any multi-index\footnote{$\IN^m_k$ denotes the set of multi-indices $\alpha=(\alpha_1,\dots,\alpha_m)\in(\IN_{\geq 0})^m$ such that $\alpha_1+\cdots+\alpha_m\leq k$.} $\alpha\in\IN^m_k$, there 
are (necessarily uniquely determined) smooth functions
$$
P_{\alpha}:U\longrightarrow \mathrm{Mat}(\IK;\ell_0\times \ell_1)
$$
such that for all $(\phi^{(1)},\dots,\phi^{(\ell_0)})\in C^{\infty}(U,\IK^{\ell_0})$ one has
$$
P|_U\sum^{\ell_0}_{i=1}\phi^{(i)} e_i=\sum^{\ell_1}_{j=1}\sum^{\ell_0}_{i=1}\sum_{\alpha\in\IN^m_k}
P_{\alpha ij}\frac{\partial^{|\alpha|}\phi^{(i)}}{\partial x^{\alpha}}f_j\>\>\text{ in $U$}.
$$
\end{Definition}

The space of smooth $\IK$-linear partial differential operators (PDO\rq{}s) of at most $k$-th order is denoted by
$\mathscr{D}^{(k)}_{C^{\infty}}(X;E,F)$, where
$$
\mathscr{D}^{(k)}_{C^{\infty}}(X;E):=\mathscr{D}^{(k)}_{C^{\infty}}(X;E,E).
$$
These are $\IK$-linear spaces in the obvious way. In the scalar case, we will write
\begin{align*}
\mathscr{D}^{(k)}_{C^{\infty}}(X):=\mathscr{D}^{(k)}_{C^{\infty}}(X;X\times \IC)=\mathscr{D}^{(k)}_{C^{\infty}}(X;X\times \IC,X\times \IC),\\
%&\mathscr{D}^{(k)}_{C^{\infty}_{\IR}}(X):=\mathscr{D}^{(k)}_{C^{\infty}}(X;X\times \IR)=\mathscr{D}^{(k)}_{C^{\infty}}(X;X\times \IR,X\times \IR),
 \end{align*}
with the canonical identification of sections and functions being understood. We record some further facts: \vspace{2mm}

$\bullet$ The composition of a $\IK$-linear PDO of order $\leq k$ with a $\IK$-linear PDO of order $\leq l$ is a $\IK$-linear PDO of order $\leq l+k$. \\
$\bullet$ Given a $P\in\mathscr{D}^{(k)}_{C^{\infty}}(X;E,F)$, it is easily checked that for every open $U\subset X$ one has $P|_U\in\mathscr{D}^{(k)}_{C^{\infty}}(U;E|_U,F|_U)$.\\
$\bullet$ Any morphism of smooth $\IK$-vector bundles $T:E\to F$ over $X$ defines a 0-th order operator $P_T\in \mathscr{D}^{(0)}_{C^{\infty}}(X;E,F)$ in view of $P_T\psi(x):=T(x)\psi(x)$, $\psi\in\Gamma_{C^{\infty}}(X,E)$, and the assignment $T\mapsto P_T$ is an isomorphism of $\IK$-linear spaces. There is no risk of confusion if we simply write $T$ instead of $P_T$.\\
$\bullet$ Considering smooth functions as smooth vector bundle morphisms acting as scalars, it is straightforward to check that for every $k\geq 1$, $P\in\mathscr{D}^{(k)}_{C^{\infty}}(X;E,F)$, $f\in\ICC(X,\IK)$, the commutator of $P$ and $f$ satisfies $[P,f]\in \mathscr{D}^{(k-1)}_{C^{\infty}}(X;E,F)$.\\
$\bullet$ We have defined smooth linear partial differential operators to be local $\IK$-linear maps between smooth sections. It is a classical fact known as \emph{Peetre\rq{}s theorem} \cite{peetre} that, conversely, every local $\IK$-linear map between smooth sections \emph{locally} is a smooth partial differential operator of some \emph{locally constant order} (the point being that the locally determined orders can go to infinity in some infinite open cover of $X$, which nevertheless cannot happen if $X$ is compact).

\begin{Definition} Let $k\in\IN_{\geq 0}$ and let $P\in \mathscr{D}^{(k)}_{C^{\infty}}(X;E,F)$.\\
a) The (linear principal) symbol of $P$
is the unique morphism 
\begin{align*}
\symbol_P: (\IT^*X)^{\odot k}\longrightarrow \mathrm{Hom}(E,F)
\end{align*}
of smooth $\IR$-vector bundles over $X$,
where $\odot$ stands for the symmetric tensor product, such that for all $((x^1,\dots,x^m),U)$, $e_1,\dots,e_{\ell_0}$, $f_1,\dots,f_{\ell_1}$ as in Definition \ref{ops}, and all real-valued $(\zeta_1,\dots, \zeta_m)\in \ICC(U)^m$ one has
$$
\symbol_{P}\Big(  \Big(\sum^m_{r=1}\zeta_r\Id x^r\Big)^{\otimes_m} \Big)e_i= \sum_{\alpha\in\IN^m_k:\alpha_1+\cdots+\alpha_m=k}\sum^{\ell_1}_{j=1}P_{\alpha i j}f_j\>\>\text{ in $U$}.
$$
b) $P$ is called \emph{elliptic}, if for 
all $x\in X$, $v\in \IT^*_x X\setminus \{0\}$, the linear map $\symbol_{P,x}(v^{\otimes k}): E_x\to F_x$ is invertible.
\end{Definition}

It is straightforward to check that the symbol of the composition of two PDO\rq{}s is (in the obvious way) equal to the composition of the two symbols. We warn the reader that some authors include appropriate powers of $\sqrt{-1}$ in the definition of the symbol, which is convenient in the context of the symbol calculus of pseudo-differential operators. As we are not going to use pseudo-differential operators in the sequel, we omit these factors (which only make sense on complex vector bundles anyway). It is instructive to take a look at the exterior differential at this point:

\begin{Example} Given $k\in \{0,\dots,m \}$ and a chart $((x^1,\dots,x^m),U)$ of $X$, the frame for $\wedge^k_{\IC} T^* X\to X$ is given by the collection of all $\Id x^{I_1}\wedge \cdots \wedge \Id x^{I_k}$, where  $I=(I_1,\dots, I_k)\in (\IN_{\geq 1})^k$ is such that $ 1\leq I_1<\cdots< I_k \leq m$. Then the (complexified) \emph{exterior differential} $\Id_k$ is the uniquely determined $\IC$-linear and local map
$$
\Id_k : \Omega^k_{C^{\infty}}(X)\longrightarrow \Omega^{k+1}_{C^{\infty}}(X),
$$ 
such that for every chart as above and every collection of smooth maps $\alpha_{I_1,\dots, I_k}:U\to \IC$, where $I\in (\IN_{\geq 1})^k$ with $1\leq I_1<\cdots< I_k\leq m$, one has
\begin{align*}
&\Id_k|_{U} \sum_{ 1\leq I_1<\cdots< I_k\leq m}\alpha_{I_1,\dots, I_k}\Id x^{I_1}\wedge \cdots \wedge \Id x^{I_k}\\
&=\sum_{ 1\leq J_1<\cdots< J_{k+1}\leq m}\sum_{s=1}^m(-1)^s(\partial_{J_s}\alpha_{J_1,\dots,\widehat{J_s},\dots, J_m})\Id x^{J_1}\wedge \cdots \wedge \Id x^{J_{k+1}},
\end{align*}
where $\widehat{J_s}$ means that $J_s$ is omitted in the sum. From this representation, one easily sees that 
$$
\Id_k\in \mathscr{D}^{(1)}_{C^{\infty}}\big(X;\wedge^k_{\IC} T^*  X,\wedge^{k+1}_{\IC} T^*  X\big)
$$
Using a local formula for wedge products, one can also deduce from the above formula for the exterior derivative that the symbol of the exterior differential is given by
$$
\symbol_{\Id_k }(\zeta)(\alpha)=\zeta\wedge \alpha,\quad\zeta\in \IT^*_x X,\>\alpha\in \wedge^k_{\IC} T^*_x X,\>x\in X.
$$
In particular, $\Id_k$ is elliptic, if and only if $m=1$ (and then of course $k=0$). It follows from a straightforward calculation that $\Id_{k+1}\circ \Id_k=0$, so that these data induce a cochain complex $(\Omega^{\bullet}_{\ICC}(X), \Id_{\bullet})$, the (complexified) \emph{de Rham complex} of $X$. Of course, being a complexification, $\Id_k$ preserves \lq\lq{}reality\rq\rq{}. The following notation will be used in the sequel:
\begin{align*}
&\Id:=\Id_0:\ICC(X)\longrightarrow \Omega^{1}_{\ICC}(X),\\
&\underline{\Id}:=\Id_1\oplus \cdots \oplus \Id_{m}:\Omega _{\ICC}(X)\longrightarrow \Omega _{\ICC}(X),
\end{align*}
which are again smooth differential operators of order $1$. We refer the reader to Theorem 14.24 in \cite{lee} for an axiomatic definition of the exterior derivative.
\end{Example}

Another important class of smooth first order differential operators is provided by vector fields, when we read them as derivations on smooth real-valued functions:

\begin{Example} If $((x^1,\dots,x^m),U)$ is a smooth chart for $X$, then every smooth vector field $A\in \mathscr{X}_{C^{\infty}}(X)$ can be locally written as
$$
A|_U=\sum^{m}_{j=1}A_j\frac{\partial }{\partial x^{j}}
$$
for some uniquely determined real-valued functions $A_1,\dots, A_m\in C^{\infty}(M)$.
Thus, $A \in \mathscr{D}^{(1)}_{C^{\infty}}(X; TX)$.
\end{Example}

For $k\in\IN_{\geq 0}$, $q\in [1,\infty]$, the $\IK$-linear space $\Gamma_{W^{k,q}_{\mathrm{loc}}}(X,E)$ of 
\emph{local $ L^{q}$-Sobolev sections of $E\to X$ with differential order $k$} is defined 
to be the space of $f\in\Gamma(X,E)$ such that for all charts $U\subset X$ which admit a local 
frame $e_1,\dots,e_{\ell_0}\in\Gamma_{C^{\infty}}(U,E)$, one has $f^{(j)} \in  W^{k,q}_{\mathrm{loc}}(U)$ for all $j\in\{1,\dots,\ell_0\}$, if $f=\sum^{\ell_0}_{j=1}f^{(j)}e_j$ in $U$. In particular, we have the space of locally $q$-integrable sections
$$
\Gamma_{ L^q_{\mathrm{loc}}}(X,E):=\Gamma_{ W^{0,q}_{\mathrm{loc}}}(X,E).
$$ 
The Sobolev embedding theorem (Theorem 6, p. 284 in \cite{evans}) states in a simplified form  that 
\begin{align}
&\Gamma_{ W^{k,q}_{\mathrm{loc}}}(X,E)\subset \Gamma_{C^{k-[m/q]-1}}(X,E) \text{ for all $k>m/q$,}\\
&\text{in particular,}\>  \bigcap^{\infty}_{k=1}\Gamma_{ W^{k,q}_{\mathrm{loc}}}(X,E)=\Gamma_{\ICC}(X,E).
\end{align}
Hereby, $[m/q]$ denotes the largest integer $\leq m/q$.\\
Next, we recall the definition of the formal adjoint of a differential operator. To this end, we will denote by $(T,\alpha)$ the canonical pairing of a linear form $T$ on some finite-dimensional linear space with a vector $\alpha$ from that space.

\begin{Propandef}\label{dual} Let $\rho$ by a smooth Borel measure on $X$. Then for every $P\in \mathscr{D}^{(k)}_{C^{\infty}}(X;E,F)$ there exists a unique differential operator $P^{\rho}\in\mathscr{D}^{(k)}_{C^{\infty}}(X;F^*,E^*)$ which satisfies
\begin{align}\label{adk}
&\int_X \left( P^{\rho} \psi,\phi\right)\Id  \rho = \int_X \left(\psi,  P\phi \right) \Id \rho 
\end{align}
for all $\psi \in\Gamma_{C^{\infty}}(X,F^*)$, $\phi\in \Gamma_{C^{\infty}}(X,E)$ with 
either $\phi$ or $\psi$ compactly supported. The operator $P^{\rho}$ is called the \emph{formal adjoint of $P$ with respect to $\rho$.} An explicit local formula for $P^{\rho}$ can be found in the proof (cf. formula (\ref{formel})).
\end{Propandef}

\begin{proof} Applying the fundamental lemma of distribution theory locally, it is clear that there can be at most one operator satisfying (\ref{adk}). In order to prove the existence, it is sufficient to define $P^{\rho}$ locally (using a standard partition of unity argument and the fact that differential operators are local). Now, in the situation of Definition \ref{ops}, let $e_i^*$ and $f_j^*$ 
be the dual smooth frames over $U$ for $E\to X$, and $F\to X$, respectively. Then for all 
$(\psi^{(1)},\dots,\psi^{(\ell_1)})\in C^{\infty}(U,\IK^{\ell_1})$ we define
\begin{align} \label{formel}
P^{\rho}\sum^{\ell_1}_{j=1}\psi^{(j)} f_j^*:=\frac{1}{\rho}\sum^{\ell_0}_{i=1}\sum^{\ell_1}_{j=1}\sum_{\alpha\in\IN^m_k}
(-1)^{|\alpha|}\frac{\partial^{|\alpha|}\left(P_{\alpha ij}\rho\psi^{(j)}\right)}{\partial x^{\alpha}}e_i^*\>\>\text{ in $U$}.
\end{align}
Let $\psi:=\sum_j\psi^{(j)} f_j^*$ and $\phi=\sum_i\phi^{(i)} e_i$ be smooth sections of $ F^*\to X$ and $E\to X$ over $U$, 
respectively, one of which having a compact support. Integrating by parts, we can calculate
\begin{align*}
&\int_U \left( P^{\rho} \psi,\phi \right)\Id  \rho=
\sum^{\ell_0}_{i=1}\sum^{\ell_1}_{j=1}\sum_{\alpha\in\IN^m_k}
\int_U\frac{1}{\rho} (-1)^{|\alpha|}\frac{\partial^{|\alpha|}(P_{\alpha ij}\rho \psi^{(j)})}{\partial x^{\alpha}}\phi^{(i)} \rho \; \Id x \\
& =\sum^{\ell_0}_{i=1}\sum^{\ell_1}_{j=1}\sum_{\alpha\in\IN^m_k}
\int_U  \psi^{(j)}  P_{\alpha ij} \frac{\partial^{|\alpha|}\phi^{(i)}}{\partial x^{\alpha}} \rho \;  \Id x= 
\int_U \left(\psi , P\phi \right) \Id \rho ,
\end{align*}
which proves (\ref{adk}). 
\end{proof}

We continue with:

 \begin{Propandef}\label{glig} Let $P\in \mathscr{D}^{(k)}_{C^{\infty}}(X;E,F)$. Then given $f\in \Gamma_{ L^{1}_{\mathrm{loc}}}(X,E)$ we say that \emph{$Pf$ exists weakly}, if there exists $h\in \Gamma_{ L^{1}_{\mathrm{loc}}}(X,F)$, such that for all smooth measures $\rho$ on $X$ it holds that 
\begin{align}
&\int_X \left( P^{\rho}\psi,f\right)\Id  \rho = \int_X \left( \psi,h \right) \Id \rho\>
\text{ for all $\psi \in\Gamma_{C^{\infty}_{\c}}(X,F^*)$}.\label{sdf0}
\end{align}
This property is equivalent to (\ref{sdf0}) being true for \emph{some} smooth measure $\rho$. In this situation, the equivalence class $h$ is uniquely determined, and we set $Pf:=h$.
\end{Propandef}

\begin{proof} Again it is clear that there can only be at most one such $h$. Assume now that there is a smooth measure $\rho$ with (\ref{sdf0}), and let $\rho'$ be an arbitrary smooth measure. In order to see that one also has (\ref{sdf0}) with respect to $\rho'$, let 
$0<\frac{\Id\rho'}{\Id\rho}\in C^{\infty}(X)$ be the Radon-Nikodym derivative of $\rho'$ with respect to $\rho$. Note that, since the measures are equivalent, Radon-Nikodym\rq{}s theorem entails that $\frac{\Id\rho'}{\Id\rho}$ exists globally as a Borel function. The smoothness of the measures then entails that $\frac{\Id\rho'}{\Id\rho}$ can be represented smoothly in each chart. We have, for all $h_1 \in \Gamma_{C_c^\infty}(X,E)$ and all $h_2 \in \Gamma_{C_c^\infty}(X, F^*)$: 
\begin{align*}
 \int_X \left( h_2,Ph_1\right) \Id\rho'  &= 
 \int_X    \frac{\Id\rho'}{\Id\rho}  \left( h_2, Ph_1 \right)  \Id\rho = 
 \int_X  \left( P^{\rho}\big( \frac{\Id\rho'}{\Id\rho} h_2\big),h_1 \right) \Id\rho  
\\
& = \int_X   \frac{\Id\rho}{\Id\rho'} \left( P^{\rho}\big( \frac{\Id\rho'}{\Id\rho} h_2\big),h_1\right)\Id\rho',
 \end{align*}
 so that 
\begin{align}\label{aioiwww}
P^{\rho'} h_2  = \frac{\Id\rho}{\Id\rho'}P^{\rho}( \frac{\Id\rho'}{\Id\rho} h_2)
\end{align}
for all $h_2\in \Gamma_{C_c^\infty}(X, F^*)$. Thus if we have (\ref{sdf0}) with respect to $\rho$, it follows that  
\begin{align*}
 &\int_X \left( P^{\rho'} \psi,f\right) \Id\rho'  =  
 \int_X   \left( P^{\rho}\big(\frac{\Id\rho'}{\Id\rho} \psi\big),f \right) \Id\rho   = 
 \int_X  \left(\psi,h \right)  \frac{\Id\rho'}{\Id\rho}  \Id\rho  \\
& =  \int_X  \left(\psi,h \right)  \Id\rho', 
 \end{align*}       
as claimed.
\end{proof}

Accordingly, given a subset $\mathscr{M}\subset\Gamma_{ L^{1}_{\mathrm{loc}}}(X,F)$, the assumption $Pf\in \mathscr{M}$ means by definition that some (necessarily unique) $h\in \mathscr{M}$ exists such that $Pf=h$ in the sense of Proposition \ref{glig}. This definition is consistent with the case that $f$ is smooth. In the scalar real-valued case, we will need the following simple observation:

\begin{Propandef}\label{weakk} Given $P\in \mathscr{D}^{(k)}_{C^{\infty}_{\IR}}(X)$ and two real-valued and locally integrable functions $f,h \in \IL^1_{\loc,\IR}(M)$, we say that $Pf\geq h$ \emph{holds weakly}, if and only if for all smooth Borel measures $\rho$ on $X$ one has
\begin{align}\label{lala}
\int_X (P^{\rho}\psi) f \Id\rho \geq \int_X \psi h \Id\rho\quad\text{ for all $0\leq \psi\in\ICC_{\c}(X)$.}
\end{align} 
It is sufficient to check this weak inequality for \emph{some} smooth Borel measure on $X$ (cf. formula (\ref{aioiwww})).
\end{Propandef}

We recall:

\begin{Definition} A \emph{smooth metric $h_E$ on $E\to X$} is defined to be a smooth section $h_E$ in $E^*\otimes E^*\to X$, such that $h_E$ is fiberwise a $\IK$-scalar product. Then the pair $(E,h_E)\to M$ is referred to as a smooth \emph{metric $\IK$-vector bundle}. 
\end{Definition}

If nothing else is said, the trivial smooth vector bundle $X\times \IK^{\ell_0}\to X$ is equipped with its canonic smooth metric, which is induced by the Euclidean metric on $\IK^{\ell_0}$.\\
In applications, formal adjoints are typically used as follows:

\begin{Propandef} \label{aoc} Assume that $(E,h_E)\to X$ and $(F,h_F)\to X$ are smooth metric $\IK$-vector bundles, and let $\rho$ be a smooth Borel measure on $X$. Then for any $P\in \mathscr{D}^{(k)}_{C^{\infty}}(X;E,F)$ there is a uniquely determined operator $P^{\rho, h_E,h_F}\in\mathscr{D}^{(k)}_{C^{\infty}}(X;F,E)$ which satisfies
\begin{align*}
&\int_X h_E\left(P^{\rho, h_E,h_F} \psi,\phi\right)\Id  \rho = \int_X h_F\left(\psi , P\phi\right) \Id \rho 
\end{align*}
for all $\psi \in\Gamma_{C^{\infty}}(X,F)$, $\phi\in \Gamma_{C^{\infty}}(X,E)$ with 
either $\phi$ or $\psi$ compactly supported. The operator $P^{\rho, h_E,h_F}$ is called the formal \emph{adjoint of $P$ with respect to $(\rho, h_E,h_F)$}. An explicit local formula for $P^{\rho, h_E,h_F}$ can be found in the proof.
\end{Propandef}

\begin{proof} Again, it is sufficient to prove the local existence. To this end, in the situation of Definition \ref{ops}, we assume that $e_i$ and $f_j$ are orthonormal with respect to $h_E$ and $h_F$, respectively. Then in complete analogy to the proof of Proposition \ref{dual}, one finds that 
\begin{align}  
P^{\rho,h_E,h_F}\sum^{\ell_1}_{j=1}\psi^{(i)} f_j:=\frac{1}{\rho}\sum^{\ell_0}_{i=1}\sum^{\ell_1}_{j=1}\sum_{\alpha\in\IN^m_k}
(-1)^{|\alpha|}\frac{\partial^{|\alpha|}\left(\overline{P_{\alpha ji}}\rho\psi^{(j)}\right)}{\partial x^{\alpha}}e_i\>\>\text{ in $U$}
\end{align}
does the job.
\end{proof}

In the above situation, $P^{\rho,h_E,h_F}$ can be constructed from $P^{\rho}$ by means of the commutative diagram
\[
\begin{xy}
\xymatrix{
\Gamma_{C^{\infty}}(X,F^*)\>\>\ar[rr]^{P^{\rho}} & & 
\>\>\Gamma_{C^{\infty}}(X,E^*)\ar[dd]^{\tilde{h}_E^{-1}} \\	 \\
\Gamma_{C^{\infty}}(X,F)\ar[uu]^{\tilde{h}_F}\ar@{.>}_{P^{\rho, h_E,h_F}}[rr] 
& & \Gamma_{C^{\infty}}(X,E)
}
\end{xy}
\]
where $\tilde{h_E}$ and $\tilde{h_F}$ stand for the $\IK$-linear isomorphisms which are induced by $h_E$ and $h_F$, 
respectively (that is $\tilde{h_E}(\phi):=h_E(\bullet,\phi)$ and likewise for $\tilde{h_F}$). Furthermore, the assignment $P\mapsto P^{\rho}$ is a linear map with 
$$
(PQ)^{\rho}=Q^{\rho}P^{\rho}, \>(P^{\rho})^{\rho}=P,
$$
whereas $P\mapsto P^{\rho,h_E,h_F}$ is an antilinear map with
$$
(PQ)^{\rho, h_{E_1},h_{E_3}}=Q^{\rho, h_{E_2},h_{E_1}}P^{\rho, h_{E_3},h_{E_2}}, \>(P^{\rho, h_E,h_F})^{\rho, h_F,h_E}=P,
$$
whenever these expressions make sense. 

\begin{Lemma}\label{adjoo}
Given $f\in \Gamma_{ L^{1}_{\mathrm{loc}}}(X,E)$ and $h\in \Gamma_{ L^{1}_{\mathrm{loc}}}(X,F)$, 
one has $Pf=h$, if and only if for all triples $(\rho,h_E,h_F)$ as in Proposition \ref{aoc} it holds that 
\begin{align}
&\int_X h_E\left(P^{\rho, h_E,h_F} \psi,f\right)\Id  \rho = \int_X h_F\left(\psi , h\right) \Id \rho\>
\text{ for all $\psi \in\Gamma_{C^{\infty}_{\c}}(X,F)$ }.\label{sdf}
\end{align}
Furthermore, this property is equivalent to (\ref{sdf}) being true for \emph{some} triple $(\rho,h_E,h_F)$ of this kind. 
\end{Lemma}

\begin{proof} In view of Proposition \ref{glig}, it is sufficient to prove that if there exists a triple $(\rho,h_E,h_F)$ with 
(\ref{sdf}) and if $h'_E$ and $h'_F$ are other smooth metrics on $E\to X$ and on $F\to X$, respectively, then 
one also has (\ref{sdf}) with respect to $(\rho,h'_E,h'_F)$. To this end, define the isomorphisms of smooth $\IK$-vector 
bundles over $X$ given by
\begin{align*}
&S_E:E\longrightarrow E,\>h'_E(S_E\phi_1,\phi_2):=h_E(\phi_1,\phi_2),\\
&S_F:F\longrightarrow F,\>\> h'_F(S_F\psi_1,\psi_2):=h_F(\psi_1,\psi_2).  
\end{align*}
Note that $h_{E}(S_E^{-1}\phi_1,\phi_2)=h'_E(\phi_1,\phi_2)$, and likewise for $h_F$. As in the proof of Lemma \ref{glig}, one finds 
$$
P^{\rho,h'_E,h'_F}=S_E^{-1}P^{\rho,h_E,h_F}S_F,
$$
and using this formula one easily proves the claim. 
\end{proof}

Next, we record some results concerning local elliptic regularity. Namely, let $k\in\IN_{\geq 0}$, $q\in [1,\infty)$, let $E\to X$, $F\to X$ be smooth $\IK$-vector bundles, and let $P\in\IDD^{(k)}_{C^{\infty}}(U;E,F)$ be elliptic. Then for every $f\in \Gamma_{ L^{q}_{\mathrm{loc}}}(X,E)$ with $Pf\in \Gamma_{ L^{q}_{\mathrm{loc}}}(X,F)$ one has
\begin{align}\label{localb}
f\in 
\begin{cases}\Gamma_{ W^{k,q}_{\mathrm{loc}}}(X,E), \text{ if $q>1$}\\
\Gamma_{ W^{k-1,1}_{\mathrm{loc}}}(X,E), \text{ if $q=1$}.
\end{cases}
\end{align}

Note that for $q>1$ the above facts are classical, while the limit case $q=1$ requires some extra work. It follows, for example, from a result on local Besov regularity of solutions to systems of elliptic equations, due to D. Guidetti, D. Pallara and the author \cite{guidetti}. It is well-known that results of this type fail for $q=\infty$. We close this section by recording the following convenient abuse of notation that should cause no danger of confusion in the sequel:

\begin{Remark}\label{adjo}1. If $(E,h_E)\to X$ is a smooth metric $\IK$-vector bundle, then whenever there is no danger of confusion, we will simply denote the underlying metric structure with $(\bullet,\bullet)$. The corresponding norm and operator norm on $\IK$-linear operators on the fibers will then be denoted by $\left|\bullet\right|$. \\
2. Given a smooth Borel measure $\rho$ on $X$, smooth metric $\IK$-vector bundles $E\to X$, $F\to X$, and a differential operator $P\in \IDD^{(k)}_{\ICC}(X,E,F)$, we will simply write $P^{\rho}\in \IDD^{(k)}_{\ICC}(X,F,E)$ for the adjoint of $P$ with respect to $\rho$ and the corresponding metric structures, in accordance with the first part of this remark.
\end{Remark}

\section{Some remarks on covariant derivatives}\label{covo}

Let $F\to X$ be a smooth $\IK$-vector bundle. 

\begin{Definition}

A \emph{smooth covariant derivative} $\nabla^F$ on $F\to X$ is a $\IK$-linear map
$$
\nabla^F:\Gamma_{\ICC}(X,F)\longrightarrow \Gamma_{\ICC}(X,\IT^* X\otimes F)
$$
which satisfies the Leibniz rule
\begin{align}
\nabla^F (f\psi)= f\nabla^F \psi+  \Id f\otimes \psi\label{leib}
\end{align}
for all $\psi\in\Gamma_{\ICC}(X,F)$, and all real-valued $f\in\ICC(X)$.

\end{Definition}

One sets $\nabla^F_A\psi:=\nabla\psi(A)\in\Gamma_{\ICC}(X,F)$ for every smooth vector field $A\in \mathscr{X}_{C^{\infty}}(X)$. In particular, under the canonical isomorphisms (of $\ICC(X,\IK)$-modules)
\begin{align*}
\Gamma_{\ICC}(X,X\times \IK^{\ell})&\cong \ICC(X,\IK^{\ell}),\\
 \Gamma_{\ICC}(X,\IT^* X \otimes (X\times \IK^{\ell}))&\cong\Omega^1_{\ICC_{\IK}}(X)^\ell,
\end{align*}
the componentwise operating exterior derivative
$$
d:\ICC(X,\IK^\ell)\longrightarrow \Omega^1_{\ICC_{\IK}}(X)^\ell
$$
 becomes a covariant derivative.\\
Using the Leibniz rule together with a simple cut-off function argument entails that smooth covariant derivatives are always local (so that they can be restricted). Another important consequence of the Leibniz rule is that any two smooth covariant derivatives $\nabla_{F} $, $\nabla_{F}\rq{}$ on $F\to X$ differ by a smooth $1$-form which takes values in the endomorphisms of $F\to X$:
\begin{align}\label{affine}
\nabla_{F}-\nabla_{F}\rq{}\in
\Gamma_{\ICC}(X,\IT^* X\otimes\mathrm{End}(F)).
\end{align}
Conversely, the sum of an element of $\Gamma_{\ICC}(X,\IT^* X\otimes\mathrm{End}(F))$ and $\nabla_F$ is again a covariant derivative. In particular, one has the following local description of covariant derivatives: If $\ell:=\rank(F)$ and if $f_1,\dots,f_\ell\in\Gamma_{C^{\infty}}(U,F)$ is a smooth frame for $F\to X$, then there is a unique matrix
$$
\alpha\in \mathrm{Mat}\big(\Omega^1_{\ICC_{\IK}}(U);\ell\times \ell\big)
$$
such that $\nabla^F|_U=\Id|_U +\alpha$ in $U$ with respect to $(f_j)$, in the sense that for all 
$(\psi^{(1)},\dots,\psi^{(\ell)})\in C^{\infty}(U,\IK^\ell)$ one has
\begin{align*}
\nabla^F |_U\sum_{j}\psi^{(j)} f_j =\sum_{j}(\Id |_U\psi^{(j)} )\otimes f_j+ \sum_{j}\sum_{i}\psi^{(j)}\alpha_{ij} \otimes f_j.
\end{align*}   
In particular, it now becomes obvious that 
$$
\nabla^F\in \mathscr{D}_{C^{\infty}}^{(1)}\left(X;F,\IT^* X \otimes F\right),
$$
and that for every open $U\subset X$ it holds that $\nabla^F |_U$ is a smooth covariant derivative on $F|_U\to U$. From this local description, it is readily seen that the symbol of $\nabla$ is given by
$$
\symbol_{\nabla }(\zeta)\psi=\zeta\otimes \psi,\quad \zeta\in\IT^*_xX,\>\psi\in F_x,\>x\in X.  
$$

\begin{Notation} The curvature $R_{\nabla^F}$ of $\nabla^F$ is the tensor
$$
R_{\nabla^F}\in \Gamma_{\ICC}(X,(\wedge^2\IT^*X)  \otimes\mathrm{End}(F) ),
$$
which for vector fields $A,B\in \mathscr{X}_{C^{\infty}}(X)$ and smooth sections $\psi$ in $F\to X$ is given by 
$$
R_{\nabla^F}(A, B)\psi:= \nabla^F_{A} \nabla^F_{B} \psi-\nabla^F_{B}\nabla^F_{A} \psi-\nabla^F_{[A,B]}\psi\in \Gamma_{\ICC}(X,  F).
$$
\end{Notation}

The dual vector bundle $F^{*}\to X$ carries the smooth covariant derivative $\nabla^{F^*}$ given by
\begin{align*}
&(\nabla^{F^*}_A\alpha ,\psi):=(\alpha,\nabla^F_A \psi)+A(\alpha,\psi),\\
&\text{ for all $\alpha\in \Gamma_{\ICC}(X,F^*)$, $\psi\in \Gamma_{\ICC}(X,F)$, $A\in\mathscr{X}_{C^{\infty}} (X)$},
\end{align*}
where $A$ acts as a derivation on the smooth function $x\mapsto (\alpha(x),\psi(x))$.\\
If $F\to X$ is a metric bundle, then $\nabla^F$ is called \emph{metric} (or to be precise: metric with respect to the given metric structure on $F\to X$), if for all $\psi_1, \psi_2\in\Gamma_{\ICC}(X,F), A\in\mathscr{X}_{C^{\infty}}(X)$ it holds that
$$
A(\psi_1,\psi_2)=(\nabla^F_A\psi_1,\psi_2)+(\psi_1,\nabla^F_A\psi_2),
$$
where again $A$ acts as a derivation.\\
If $F\to X$ is metric, then so is $F^{*}\to X$ in a  canonical way, and if then $\nabla^F$ is metric, so is $\nabla^{F^*}$. By a partition of unity argument, one finds that every vector bundle admits a metric, and that every metric vector bundle admits a metric covariant derivative. In the case of the trivial smooth metric bundle $X\times \IK^{\ell}\to X$, the smooth covariant derivative $ d+\alpha$ is metric, if and only if 
$$
\alpha\in \mathrm{Mat}\big(\Omega^1_{\ICC_{\IK}}(X);\ell\times \ell\big)
$$
satisfies $\overline{\alpha_{ji}}=-\alpha_{ij}$. \vspace{1mm}

Given another smooth $\IK$-vector bundle $E\to X$ and a smooth covariant derivative $\nabla^E$ thereon, the smooth $\IK$-vector bundle $E\oplus F\to X$ carries the smooth covariant derivative $\nabla^E\oplus\nabla^F$. If $F\to X$, $E\to X$ are metric, then so is $F\oplus E\to X$ canonically, and if then $\nabla^F$, $\nabla^E$ are metric, the same is true for $\nabla^E\oplus \nabla^F$.\vspace{1mm}

The smooth vector bundle $F\otimes E\to X$ carries the smooth covariant derivative $\nabla^E{\tilde\otimes} \nabla^F$, which is uniquely determined by
\begin{align*} 
\nabla^E\tilde{\otimes}\nabla^F(\psi_1\otimes \psi_2)=(\nabla^E\psi_1)\otimes \psi_2+
\psi_1\otimes(\nabla^F\psi_2) 
\end{align*}
for all $\psi_1\in\Gamma_{\ICC}(X,E)$, $\psi_2\in\Gamma_{\ICC}(X,F)$. \vspace{2mm}
As above, if $F\to X$, $E\to X$ are metric, then so is $F\otimes E\to X$ canonically, if then $\nabla^F$, $\nabla^E$ are metric, the same is true for $\nabla^E{\tilde\otimes} \nabla^F$. In particular, the tensor product construction can be used to complexify real (metric) vector bundles and (metric) covariant derivatives, namely by tensoring with $X\times\IC\to \IC$ and the exterior derivative.\vspace{1mm}

Next, let $\diamond$ denote either the symmetric tensor product $\odot$, or the antisymmetric tensor product $\wedge$. Then the smooth vector bundle $E\diamond E\to X$ carries the smooth covariant derivative $\nabla^E{\tilde\diamond} \nabla^E$, which is uniquely determined by
\begin{align*} 
\nabla^E\tilde{\diamond}\nabla^E(\psi_1\diamond \psi_2)=(\nabla^E\psi_1)\diamond \psi_2+
\psi_1\diamond(\nabla^E\psi_2)  
\end{align*}
for all $\psi_1,\psi_2\in\Gamma_{\ICC}(X,E)$. \vspace{2mm}
As above, if $E\to X$ is metric, then so is $E\diamond E\to X$ canonically,\footnote{To be precise, $E\otimes E\to X$ canonically becomes metric, and $E\diamond E\to X$ inherits this structure.} and if then $\nabla^E$ is metric, the same is true for $\nabla^E{\tilde{\diamond}} \nabla^E$.

%
%
%If $\diamond\in\{\odot,\wedge\}$, then considering $F\diamond E\to X$ as a (metric) subbundle of $F\otimes E\to X$, one has the commutative diagram
%\[
%\begin{xy}
% \xymatrix{
%  \Gamma_{C^{\infty}}(X,E\diamond F)\ar[drr]_{\nabla^E{\tilde\diamond} \nabla^F} \ar[rr]^{\nabla^E{\tilde\otimes} \nabla^F}  &&
% \Gamma_{C^{\infty}}(X,\IT^*X \otimes E\otimes F)\ar[d]^{T_{\diamond}  } \\
% &&   \Gamma_{C^{\infty}}(X,\IT^*X\otimes E \diamond F)
%}
%\end{xy}, 
%\]
%where $T_{\diamond}$ stands for the symmetrization, respectively, the antisymmetrization.

\section{Generalized Sobolev spaces and Meyers-Serrin theorems }

In this section, \emph{let $\rho$ be a smooth Borel measure on $X$.}\vspace{1mm}

\begin{Notation}
Given a smooth metric $\IK$-vector bundle $E\to X$, for any $q\in [1, \infty]$, we get the corresponding $\IK$-Banach spaces $\Gamma_{\IL^q_{\rho}}(X,E)$ given by all $f\in\Gamma(X,E)$ such that $\left\| f\right\|_{L^q_{\rho}}<\infty$, where
$$
\left\| f\right\|_{L^q_{\rho}}:=\begin{cases}&\Big(\int_X  \big|f(x)\big|^q  \Id \rho(x)\Big)^{1/q} 
,\text{ if $q<\infty$} \\
&\inf\{C\geq 0: |f|\leq C\text{ $\rho$-a.e.}\},\text{ if $q=\infty$.}\end{cases}
$$
The symbol $\left\langle\bullet,\bullet \right\rangle_{\rho}$ will stand for the canonical inner product on the Hilbert space $\Gamma_{\IL^2_{\rho}}(X,E)$, which is given by
$$
\left\langle f_1 ,f_2\right\rangle_{L^2_{\rho}} =\int_X  \left(f_1(x),f_2(x)\right) \Id\rho( x).
$$
\end{Notation}

The following simple observation will be helpful in the sequel:

\begin{Remark}\label{global} Given a smooth metric $\IK$-vector bundle $E\to X$ with $\ell:=\mathrm{rank}(E)$, it is always possible to find a global orthonormal Borel measurable (of course not necessarily continuous) frame $e_1\dots, e_\ell:X\to E$. To see this, cover $X=\cup_{n\in\IN_{\geq 1}} U_n$ with open subsets $U_n\subset X$, such that on each $U_n$ there is a local smooth orthonormal frame $e_1^{(n)}\dots, e_{\ell}^{(n)}\in\Gamma_{\ICC}(U_n,E)$. By setting
$$
W_1:=U_1, \>W_n:=U_n\setminus  \Big(\bigcup_{l\in\IN_{\geq 0}:\> l\ne n } U_l\Big)\>\text{ for $n\geq 2$, so that\footnote{In the following expression, the symbol $\bigsqcup$ denotes a disjoint union.} $X=\bigsqcup_{n=1}^\infty W_n$,}
$$
we can define $e_j\mid_{W_n}:=e^{(n)}_j\mid_{W_n}$. \\
Any fixed global orthonormal Borel frame $e_1\dots, e_\ell:X\to E$ induces an isometric isomorphism of $\IK$-linear spaces 
$$
\Gamma_{\IL^q_{\rho}}(X,E)\cong \IL^{q}_{\rho}(X,\IK^\ell)\quad\text{for all $q\in [1,\infty]$},
$$ 
in particular, for every $q\in [1,\infty)$ the Banach space $\Gamma_{\IL^{q}_{\rho}}(X,E)$ is separable and in addition reflexive if $q\in (1,\infty)$. The existence of such a global frame also straightforwardly implies that for every pair of numbers $q_1,q_2\in [1,\infty]$ satisfying $1/q_1+1/q_2=1$ (with $1/0:=\infty$, $1/\infty:=0$), one has
$$
\left\|f\right\|_{L^{q_1}_{\rho}}=\sup_{\phi\in \Gamma_{\IL^{q_2}_{\rho}}(X,E),\left\|\phi\right\|_{L^{q_2}_{\rho}}\leq 1}\left|\int_X (f,\phi)\Id\rho\right|,
$$
and that for every $1<q_1<\infty$ the map
\begin{align*}
\Gamma_{\IL^{q_2}_{\rho}}(X,E)\ni f\longmapsto \int_X (f,\bullet)\Id\rho\in \Gamma_{\IL^{q_1}_{\rho}}(X,E)^*
 \end{align*}
is antilinear, isometric, and bijective. In the sequel, we will thus identify these two Banach spaces via the above map.
\end{Remark}

In order to be able to deal with many natural geometric situations simultaneously, we continue with the following definition:

\begin{Definition}\label{asa} 
Let $q\in [1,\infty]$, $s\in\IN_{\geq 1}$, $k_1\dots,k_s\in\IN_{\geq 0}$, and for each $i\in\{1,\dots,s\}$ 
let $E\to X$, $F_i\to X$ be smooth metric $\IK$-vector bundles and let $\mathfrak{P}:=\{P_1,\dots,P_s\}$ 
with $P_{i}\in\allowbreak \mathscr{D}_{C^{\infty}}^{(k_i)}(X;E,F_i)$. Then the $\IK$-Banach space 
\begin{align*}
&\Gamma_{ W^{\mathfrak{P},q}_{\rho}}(X,E)\\
&:= \big\{f\in\Gamma_{ L^{q}_{\rho}}(X,E): 
P_{i} f\in\Gamma_{  L^{q}_{\rho}}(X,F_i)\text{ for all $i\in\{1,\dots,s\}$}\big\}\\
&\subset \Gamma_{ L^{q}_{\rho}}(X,E),\>\text{ with its norm $\left\| f\right\|_{\mathfrak{P},L^q_{\rho}}
:=\Big(\left\|f\right\|^q_{L^q_{\rho}}+\sum^s_{i=1}\left\|P_{i} f\right\|^q_{L^q_{\rho}}\Big)^{1/q}$},
\end{align*}
is called the \emph{$\mathfrak{P}$-Sobolev space of $ L^{q}_{\rho}$-sections} in $E\to X$. Furthermore, we define the $\IK$-Banach space
$$
\Gamma_{ W^{\mathfrak{P},q}_{\rho,0}}(X,E)\subset \Gamma_{ W^{\mathfrak{P},q}_{\rho}}(X,E)
$$
to be the closure of $\Gamma_{\ICC_\c}(X,E)$ with respect to $\left\| \bullet\right\|_{\mathfrak{P},L^q_{\rho}}$.
\end{Definition}

Since closed subspaces as well as products of reflexive (separable) Banach spaces are reflexive (separable), it follows precisely as for the usual Euclidean Sobolev spaces that the spaces $\Gamma_{ W^{\mathfrak{P},q}_{\rho}}(X,E)$ and $\Gamma_{ W^{\mathfrak{P},q}_{\rho,0}}(X,E)$ are separable for all $q\in [1,\infty)$ and reflexive for all $q\in (1,\infty)$. The following result provides a generalization of the classical Meyers-Serrin theorem \cite{mse} to our abstract Sobolev spaces:

\begin{Theorem}\label{meyers} In the situation of Definition \ref{asa}, let $q\in [1,\infty)$ and assume that in case
$k:=\max\{k_1,\dots,k_s\}\geq 2$ one has 
$\Gamma_{ W^{\mathfrak{P},q}_{\rho}}(X,E)\subset \Gamma_{ W^{k-1,q}_{\mathrm{loc}}}(X,E)$ (with no further assumption if $k\in \{0,1\}$). Then for any $f\in\Gamma_{ W^{\mathfrak{P},q}_{\rho}}(X,E)$ there exists a sequence
$$
(f_n)\subset \Gamma_{C^{\infty}}(X,E)\cap \Gamma_{ W^{\mathfrak{P},q}_{\rho}}(X,E),
$$
which can be chosen in $\Gamma_{C^{\infty}_{\c}}(X,E)$ if $f$ is compactly supported, such that
\begin{align*}
&\left|f_n(x)\right|\leq\left\|f\right\|_{L^{\infty}_{\rho}}\in [0,\infty]\>\>\text{ for all $x\in X$, $n\in\IN_{\geq 0}$},\\
&\left\|f_n-f\right\|_{\mathfrak{P},L^q_{\rho}}\to 0 \text{ as $n\to\infty$.}
\end{align*}
\end{Theorem}

This result has been proved by D.Pallara, D. Guidetti and the author in \cite{guidetti}. Its proof relies on a local \lq\lq{}higher order\rq\rq{} approximation result that is build on Friedrichs mollifiers. To formulate that approximation result, we recall that given a distribution $T$ acting on $\ICC_{\c}(\IR^m,\IK^{\ell})$, the convolution of $T$ with $\varphi\in \ICC_{\c}(\IR^m)$ is the $\IK^{\ell}$-valued function $T*\varphi$ on $\IR^m$, defined by 
$$
T*\varphi(x):=\left\langle T,\varphi(x-\bullet) \right\rangle. 
$$
For example, given $U\subset \IR^m$ open, $f\in  L^{1}_{\loc}(U,\IK^{\ell})$, and if we define $\underline{f}\in\IL^1_{\loc}(\IR^m, \IK^{\ell})$ to be the trivial extension of $f$ to $\IR^m$ by zero, then one readily sees that 
$$
\underline{f}*\varphi(x)=\int_U f(z) \varphi(x-z) \Id z.
$$
There will be no danger of confusion in simply writing $f*\varphi$ instead of $\underline{f}*\varphi$. 

\begin{Definition} 1. Every $0\leq h\in C^{\infty}_{\c}(\IR^m)$ such that $h(x)=0$ for all $x$ with 
$|x|\geq 1$ and $\int_{\IR^m} h(x)\Id x =1$ will be called a \emph{mollifier} on $\IR^m$ in the sequel. For every such $h$ and every $\epsilon>0$ we  define 
$0\leq h_{\epsilon}\in C^{\infty}_{\c}(\IR^m)$ by 
$h_{\epsilon}(x):=\epsilon^{-m}h(\epsilon^{-1}x)$.\\
2. Given a mollifier $h$ on $\IR^m$ and a distribution $T$ acting on $\ICC_{\c}(\IR^m,\IK^\ell)$, the \emph{Friedrichs mollification of $T$ with respect to $h$} is defined to be the family of functions $(T*h_{\epsilon})_{\epsilon>0}$. 
\end{Definition}

Note that\footnote{In the sequel, $\IB^{\IR^m}(y,r):=\{z\in\IR^m: |z-y|<r\}\subset \IR^m$ will denote the open balls with respect to the usual Euclidean metric on $\IR^m$.} $\supp(h_{\epsilon})\subset \IB^{\IR^m}(0,\epsilon)$. We list some standard properties of the Friedrichs mollification (cf. Section 7.2 in \cite{gilbarg} and Lemma 2.9 in \cite{gri}) in the following remark.

\begin{Remark}\label{controlle} The following statements hold for every mollifier $h$ on $\IR^m$:\\
i) If $T$ is a distribution acting on $\ICC_{\c}(\IR^m,\IK^{\ell})$, then for every $\epsilon>0$ one has $T*h_{\epsilon}\in \ICC_{\c}(\IR^m,\IK^{\ell})$
with
$$
\partial^{\alpha} (T*h_{\epsilon})=(\partial^{\alpha} T)*h_{\epsilon}= T*(\partial^{\alpha}h_{\epsilon})\quad\text{ for every multi-index $\alpha\in(\IN_{\geq 0})^{m}$, }
$$
and moreover
$$
\mathrm{supp}(T*h_{\epsilon})\subset \Big\{x\in \IR^m:\>\inf_{a\in \mathrm{supp}(T)}|x-a|<\epsilon\Big\},
$$
where $\partial^{\alpha} T$ and $\mathrm{supp}(T)$ are understood in the sense of distributions.\\
ii) If $q\in [1,\infty)$, and $f\in L^{q}_{\loc}(U,\IK^{\ell})$ is compactly supported in $U$, then one has $f*h_{\epsilon}\in C^{\infty}_{\c}(U,\IK^{\ell})$ for all sufficiently small $\epsilon>0$ by i), and
$$
\left\|f*h_{\epsilon}-f\right\|_{ L^{q}(U,\IK^{\ell})}\to 0\quad\text{ as $\epsilon\to 0+$.}
$$
In particular, by picking a subsequence of $f*h_{1/n}$, $n\in\IN_{\geq 1}$, this entails that for every open $U\subset \IR^m$ and every $f\in L^{\infty}_{\loc}(U,\IK^{\ell})$ with a compact support in $U$, there exists a sequence $(f_n)\subset C^{\infty}_{\c}(U,\IK^{\ell})$ such that $|f_n|\leq \|f\|_{\infty}$ and $f_n\to f$ almost everywhere, as $n\to\infty$.
\end{Remark}

The following higher order result on Friedrichs mollifiers has been noted in \cite{guidetti}. In fact, it is proved straightforwardly by using a classical \lq\lq{}first order\rq\rq{} result by K. Friedrichs from 1944 (!).

\begin{Proposition}\label{local} Let $h$ be a mollifier on $\IR^m$, let $U\subset \IR^m$ be open, and pick $k\in\IN_{\geq 0}$, $\ell_0,\ell_1\in\IN_{\geq 1}$, $q\in [1,\infty)$. Assume furthermore that 
$$
P=\sum_{\alpha\in \IN^m_k}P_{\alpha}\partial^{\alpha}\in\IDD^{(k)}_{C^{\infty}}(U;\IK^{\ell_0},\IK^{\ell_1})
$$
is a linear partial differential operator with matrix coefficients, and let $f\in  L^{q}_{\mathrm{loc}}(U,\IK^{\ell_0})$ have a compact support in $U$ with 
$Pf\in  L^{q}_{\mathrm{loc}}(U,\IK^{\ell_1})$. Assume furthermore that either 
$k<2$ or $f\in W^{k-1,q}_{\mathrm{loc}}(U,\IK^{\ell_0})$. Then one 
has 
$$
\left\|P(f*h_{\epsilon})-Pf\right\|_{ L^{q}(U,\IK^{\ell_1})}\to 0\quad\text{ as $\epsilon\to 0+$.}
$$
\end{Proposition}

\begin{proof}  We start by quoting the following classical result by Friedrichs: \emph{Given a $C^1$-function\footnote{Using Rademacher's theorem, one finds that it is in fact sufficient to assume that $Q$ is locally Lipschitz continuous (cf. Appendix A in \cite{Br}).}
$$
Q:U\longrightarrow \mathrm{Mat}(\IK, \ell_0\times \ell_1 )
$$
and $j\in\{1,\dots, m\}$, it follows that for every $F\in   L^{q}_{\mathrm{loc}}(U,\IK^{\ell_0})$ with a compact support in $U$ one has }
\begin{align}\label{fried}
\left\| (Q\partial_jF)*h_{\epsilon}-Q\partial_j(F*h_{\epsilon})\right\|_{ L^{q}(U,\IK^{\ell_1})}\to 0\quad\text{ as $\epsilon\to 0+$.}
\end{align}
This result follows from equation (3.8) in \cite{friedrichs}.\\
Returning to our situation, we first note that as one has $(Pf)*h_{\epsilon} \to Pf$ in $ L^q(U, \IK^{\ell_1})$ by Remark \ref{controlle} ii) (where we use that $f\in  L^{q}_{\mathrm{loc}}(U,\IK^{\ell_0})$ is compactly supported in $U$ and that $Pf\in  L^{q}_{\mathrm{loc}}(U,\IK^{\ell_1})$). Therefore, it suffices to prove that 
$$
\left\|(Pf)*h_\epsilon - P(f*h_\epsilon)\right\|_{ L^q(U, \IK^{\ell_1})} \to 0.
$$
To this end, let us show that 
\begin{align}\label{aoap}
\left\|(P_\alpha \partial^\alpha f)*h_\epsilon - P_\alpha \partial^\alpha (f*h_\epsilon)\right\|_{ L^q(U, \IK^{\ell_1})} \to 0
\end{align}
for every $\alpha\in\IN_k^m$. In fact, since one has either $k<2$ or $f\in W^{k-1,q}_{\mathrm{loc}}(U,\IK^{\ell_0})$, we can pick a $j \in \{1, \dots, m\}$ and an $\alpha' \in \IN^m_{k-1}$, such that $\partial^{\alpha}f=\partial_j\partial^{\alpha\rq{}}f$ and $\partial^{\alpha'}f \in  L^q_{\loc}(U, \IK^{\ell_0})$ (with a compact support in $U$). Moreover by Remark \ref{controlle} i) we have
$$
(P_\alpha \partial^\alpha f) *h_\epsilon - P_\alpha \partial^\alpha (f*h_\epsilon) = 
(P_\alpha \partial_j \partial^{\alpha'}f)*h_\epsilon - P_\alpha \partial_j ((\partial^{\alpha'}f)*h_\epsilon) .
$$
Thus (\ref{aoap}) follows by applying (\ref{fried}) with $Q=P_{\alpha}$, $F=\partial^{\alpha\rq{}}f$. 
\end{proof}

\begin{proof}[Proof of Theorem \ref{meyers}] Let 
$$
\ell_0:=\mathrm{rank}(E), \>\ell_j:=\mathrm{rank}(F_j),\>\> \text{ for any
$j\in\{1,\dots,s\}$.} 
$$
We take a \emph{relatively compact} atlas\footnote{This means that each $U_n$ is relatively compact.} $\bigcup_{n\in\IN_{\geq 0}  }U_n=X$ such that 
each $U_n$ admits smooth orthonormal frames for 
\[
E\longrightarrow X, F_1\longrightarrow X,\dots,F_s\longrightarrow X.
\]
Let $(\varphi_n)$ be a smooth partition of unity which is subordinate to $(U_n)$. Now let $f\in \Gamma_{ W^{\mathfrak{P},q}_{\rho}}(X,E)$, 
and $f_n:=\varphi_n f$. \\
Let us first show that 
$f_n\in \Gamma_{ W^{\mathfrak{P},q}_{\rho,\c}}(U_n,E)$. To see this, let $j\in\{1,\dots,s\}$. Clearly we have $\varphi_n P_{j}f\in \Gamma_{ L^{q}_{\rho}}(U_n,E)$. Furthermore, as we have $f\in \Gamma_{ W^{k-1,q}_{\mathrm{loc}}}(X,E)$, and so
$$
\left(\partial^{\alpha}f_1,\dots,\partial^{\alpha}f_{\ell_0}\right)
\in L^p_{\mathrm{loc}}(U_n,\IK^{\ell_0})
\>\text{ for all $\alpha\in\IN^m_{k-1}$,}
$$
it follows from 
$$
[P_j,\varphi_n]\in \mathscr{D}^{(k_j-1)}_{C^{\infty}}(U_n;E,F_j)
$$
that $[P_j,\varphi_n] f\in \Gamma_{ L^{q}_{\rho}}(U_n,E)$, since the coefficients of $[P_j,\varphi_n]$ are bounded in $U_n$ (being smooth and compactly supported). Thus, the following formula for weak derivatives,
$$
P_{j} f_n= \varphi_n P_{j}f+[P_j,\varphi_n] f,
$$
holds and completes the proof of $f_n\in \Gamma_{ W^{\mathfrak{P},q}_{\rho,\c}}(U_n,E)$. \\
But now, given $\epsilon >0$, we may appeal to Proposition \ref{local} and Remark \ref{controlle} ii) to pick an 
$f_{n,\epsilon}\in\Gamma_{C^{\infty}_{\c}}(X,E)$ with a compact support in $U_n$ such that 
$$
\left\|f_n- f_{n,\epsilon}\right\|_{\mathfrak{P},L^q_{\rho}}<\epsilon/2^{n+1}.
$$
Finally, $f_{\epsilon}(x):=\sum_n f_{n,\epsilon}(x)$, $x\in X$, is a locally finite sum and 
thus defines an element in $\Gamma_{C^{\infty}}(X,E)$ which satisfies
$$
\left\|f_{\epsilon}-f\right\|_{\mathfrak{P},L^q_{\rho}}\leq 
\sum^{\infty}_{n=1}\left\|f_{n,\epsilon}-f_n\right\|_{\mathfrak{P},L^q_{\rho}}<\epsilon,
$$ 
which proves the first assertion of the theorem.\\
If $f$ is compactly supported, then by picking a \emph{finite} cover of the support of $f$ with $U_n\rq{}s$ as above, the above proof shows that the approximating family $(f_{\epsilon})$ can be chosen such that each $f_{\epsilon}$ has a compact support. This completes the proof.
\end{proof}

As a first simple application of the generalized Meyers-Serrin theorem, we record the following formula for integration by parts:

\begin{Lemma}\label{int} Let $q,q^*\in (1,\infty)$ with $1/q+1/q^*=1$, $k\in\IN_{\geq 0}$, and let $E\to X$, $F\to X$ be smooth metric $\IK$-vector bundles. Assume furthermore that $P\in \IDD^{(k)}_{\ICC}(X;E,F)$. Then for all $f_1\in \Gamma_{ W^{P,q}_{\rho,0}}(X,E)$, $f_2\in \Gamma_{ W^{P^{\dagger},q^*}_{\rho}}(X,F)$ one has
$$
\int_X (Pf_1,f_2)\Id\rho= \int_X (f_1,P^{\dagger}f_2)\Id\rho.
$$ 
\end{Lemma}

\begin{proof} Assume first that $f_1$ is smooth and compactly supported, and pick a sequence of smooth sections $f_{2,n}\in \Gamma_{ W^{P^{\dagger},q^*}_{\rho}}(X,F)$ with 
$$
\left\|f_2-f_{2,n}\right\|_{P^{\rho},L^{q^*}_{\rho}}\to 0.
$$
Then the asserted formula holds with $f_2$ replaced with $f_{2,n}$, and it extends to $f_2$ by Hölder's inequality. Having established this case, the general case now follows by picking $f_{1,n}\in \Gamma_{\ICC_{\c}}(X,E)$ with $\left\|f_1-f_{1,n}\right\|_{P,L^q_{\rho}}\to 0$ and using Hölder once more.
\end{proof}

Given $q,q^*\in (1,\infty)$ with $1/q+1/q^*=1$, metric $\IK$-vector bundles $E\to X$, $F\to X$ and a densely defined operator $T$ from $\Gamma_{\IL^{q}_{\rho}}(X,E)$ to $\Gamma_{\IL^{q}_{\rho}}(X,F)$, its Banach adjoint can be identified canonically with the densely defined operator $T^{*}$ from $\Gamma_{\IL^{q^*}_{\rho}}(X,F)$ to $\Gamma_{\IL^{q^*}_{\rho}}(X,E)$, given as follows: $\dom(T^{*})$ is given by all $f\in \Gamma_{\IL^{q^*}_{\rho}}(X,F)$ which satisfy the property that there exists $\psi\in \Gamma_{\IL^{q^*}_{\rho}}(X,E)$ such that for all $h\in\dom(T)$ one has
$$
\int_X (Th,f)\Id\rho=\int_X (h,\psi)\Id\rho,
$$
and then $T^*f:=\psi$. It is an abstract functional fact that adjoints are automatically closed. Note also that $\dom(T^{*})$ is precisely the space of all $f\in \Gamma_{\IL^{q^*}_{\rho}}(X,F)$ which satisfy the property that there exists a constant $C>0$ such that for all $h\in\dom(T)$ one has
$$
\left|\int_X (Th,f)\Id\rho\right|\leq C\left\|h\right\|_{L^q_{\rho}}.
$$
Given another densely defined operator $S$ with\footnote{As usual, $S\subset T$ for operators $S$, $T$ in a common Banach space means that $\dom(S)\subset \dom(T)$ and $S=T$ on $\dom(S)$. In this case, $T$ is called an \emph{extension of $S$}.} $S\subset T$, it follows that $T^*\subset S^*$. In case $T$ as above is closable, then, as for Hilbert spaces, the closure of $T$ can be identified with the operator $\overline{T}$ given as follows: $f\in \Gamma_{\IL^{q}_{\rho}}(X,E)$ is in $\dom(\overline{T})$, if and only if there exists a sequence $(f_n)\subset \dom(T)$ with $\left\|f-f_n\right\|_{L^q_{\rho}}$ such that $\left\|Tf_n-h\right\|_{L^q_{\rho}}\to 0$ for some $h\in \Gamma_{\IL^{q}_{\rho}}(X,E)$, as $n\to\infty$, and then $\overline{T}f:=h$. If $T$ is densely defined and closable, then $T^*$ is densely defined with $T^*=\overline{T}^*$ and $(T^*)^*=\overline{T}$. (The latter result uses the reflexivity of the underlying Banach spaces.) We refer the reader to \cite{kato1} for the proofs of these abstract Banach space facts.\\

The following well-known constructions will turn out to be a useful tool in the sequel:

\begin{Definition}\label{adw} Let $q\in (1,\infty)$, $k\in\IN_{\geq 0}$, let $E\to X$, $F\to X$ be smooth metric $\IK$-vector bundles and let $P\in \IDD^{(k)}_{\ICC}(X;E,F)$. We denote by $P^{(q)}_{\rho,\min}$ the closure of $P$ considered as acting from $\Gamma_{\IL^{q}_{\rho}}(X,E)$ to $\Gamma_{\IL^{q}_{\rho}}(X,F)$, defined initially on $\Gamma_{\ICC_{\c}}(X,E)$. $P^{(q)}_{\rho,\min}$ is called the \emph{minimal extension of $P$ with respect to $\IL^{q}_{\rho}$}. Likewise, one defines the corresponding \emph{maximal extension} $P^{(q)}_{\rho,\max}$ as follows: 
$$
\dom(P^{(q)}_{\rho,\max}):=\Gamma_{ W^{P,q}_{\rho}}(X,E),\quad  P^{(q)}_{\rho,\max}f:=Pf,\quad f\in \dom(P^{(q)}_{\rho,\max}).
$$
In the case of $q=2$, we will simply write
$$
P_{\rho,\min}:=P^{(2)}_{\rho,\min},\quad P_{\rho,\max}:=P^{(2)}_{\rho,\max}.
$$
\end{Definition}

An integration by parts shows that $P|_{\Gamma_{\ICC_{\c}}(X,E)}$ indeed is closable in $\Gamma_{\IL^{q}_{\rho}}(X,E)$, so that the minimal operator is well-defined. In fact, one has
$$
\dom(P^{(q)}_{\rho,\min})=\Gamma_{ W^{P,q}_{\rho,0}}(X,E).
$$

We record the following fact, which follows easily from the above considerations: 

\begin{Lemma}\label{minmax} Let $q\in (1,\infty)$, $k\in\IN_{\geq 0}$, let $E\to X$, $F\to X$ be smooth metric $\IK$-vector bundles and let $P\in \IDD^{(k)}_{\ICC}(X;E,F)$. If $q^*\in (1,\infty)$ is such that $1/q^*+1/q=1$, then one has
$$
\big( (P^\dagger)^{(q^*)}_{\rho,\max} \big)^*=P^{(q)}_{\rho,\min}.
$$
\end{Lemma}

%\begin{proof} The inclusion
%$$
%P^{(q)}_{\min}\subset \big( (P^\dagger)^{(q)}_{\max} \big)^*
%$$
%follows from the definition. The inclusion
%$$
%P^{(q)}_{\min}\subset \big( (P^\dagger)^{(q)}_{\max} \big)^*
%$$
%is equivalent to 
%$$
%P^{(q)}_{\min}^* \supset \big( (P^\dagger)^{(q)}_{\max} \big)
%$$
%which is easily seen.
%\end{proof}

\chapter{Smooth heat kernels on vector bundles}\label{C3}

In this chapter, \emph{let $\rho$ be a smooth Borel measure on $X$, and let $E\to X$ be a smooth metric $\IC$-vector bundle with $\rank(E)=\ell$.}

\vspace{2mm}

%where we will denote by
%$$
%F^*\boxtimes F=\bigsqcup_{(x,y)\in M\times M}\hspace{-2ex} F^*_y \otimes F_x \longrightarrow M\times M
%$$
%the corresponding (smooth complex) product vector bundle. 

Assume that we are given some $k\in\IN_{\geq 0}$ and $P\in\IDD^{(k)}(X;E)$ which is elliptic, formally self-adjoint (that is $P^{\rho}=P$ in the sense of Remark \ref{adjo}.2), and semibounded (that is, there exists $C\geq 0$ with 
$$
\left\langle P\psi,\psi\right\rangle_{L^2_{\rho}}\geq -C \left\|\psi\right\|^2_{L^2_{\rho}}\>\text{ for all $\psi\in\Gamma_{\ICC_\c}(X,E)$ ).} 
$$
Then the operator $P$ with $\dom(P)=\Gamma_{\ICC_\c}(X,E)$ is a genuine symmetric operator in $\Gamma_{\IL^2_{\rho}}(X,E)$ which is $\geq -C$. In particular, such an operator admits semibounded self-adjoint extensions (for example its Friedrichs realization; cf. appendix, Example \ref{friedrichs}). Given such a semibounded self-adjoint extension $\widetilde{P}$, note that Lemma \ref{minmax} implies 
$$
(P^n)_{\rho,\min}\subset (\widetilde{P})^n=((\widetilde{P})^n)^*\subset ((P^n)_{\rho,\min})^*=(P^n)_{\rho,\max},
$$
in particular,
$$
\dom((\widetilde{P})^n)\subset  \Gamma_{ W^{P^n,2}_{\rho}}(X,E).
$$

The \lq\lq{}heat semigroup\rq\rq{} 
$$
(\mathrm{e}^{-t \widetilde{P}})_{t\geq 0}\subset \ILL(\Gamma_{\IL^2_{\rho}}(X,E) )
$$
is defined by the spectral calculus. It is a strongly continuous and self-adjoint semigroup of bounded operators (cf. appendix, Remark \ref{beispiele}). It follows that for every $f\in\Gamma_{\IL^2_{\rho}}(X,E)$ the path 
$$
[0,\infty)\ni t\longmapsto \mathrm{e}^{-t \widetilde{P}}f\in \Gamma_{\IL^2_{\rho}}(X,E)
$$
is the uniquely determined continuous path 
$$
[0,\infty)\longrightarrow \Gamma_{\IL^2_{\rho}}(X,E)
$$
which is $C^1$ in $(0,\infty)$ (in the norm topology) with values in $\dom(\widetilde{P})$ thereon, and which satisfies the abstract \lq\lq{}heat equation\rq\rq{}
$$
(\Id/\Id t )\mathrm{e}^{-t \widetilde{P}}f=-\widetilde{P} \mathrm{e}^{-t \widetilde{P}}f,\quad t>0,
$$
subject to the initial condition $\mathrm{e}^{-t \widetilde{P}}f|_{t=0} = f$. The heat semigroups corresponding to operators of the form $\widetilde{P}$ are always induced by jointly smooth heat kernels in the following sense:

\begin{Theorem}\label{heat} Let $k\in\IN_{\geq 0}$, let $P\in\IDD^{(k)}_{\ICC}(X;E)$ be elliptic, formally self-adjoint and semibounded, and let $\widetilde{P}$ be a semibounded self-adjoint extension of $P$ in $\Gamma_{\IL^2_{\rho}}(X,E)$. Then:\\
a) There is a unique smooth map\footnote{The reader may find the precise definition of the smooth vector bundle $E^*\boxtimes E\to X\times X$ in the appendix, Section \ref{difftop}.}
$$
(0,\infty)\times X\times X\ni (t,x,y)\longmapsto \mathrm{e}^{-t \widetilde{P}}(x,y)\in \mathrm{Hom}(E_y,E_x)\subset E^*\boxtimes E,
$$
the \emph{heat kernel of $\widetilde{P}$}, such that for all $t>0$, $f\in \Gamma_{\IL^2_{\rho}}(X,E)$, and $\rho$-a.e. $x\in X$ one has 
\begin{align}
\mathrm{e}^{-t \widetilde{P}}f(x)=\int_X\mathrm{e}^{-t \widetilde{P}}(x,y)f(y)\Id\rho( y).
\end{align}
b) For all $s,t>0$, $x,y\in X$ one has
\begin{align}
\int_X\left|\mathrm{e}^{-t \widetilde{P}}(x,z)\right|^2\Id\rho(z)<\infty,
\end{align}
\begin{align}
\mathrm{e}^{-t \widetilde{P}}(y,x)=\mathrm{e}^{-t \widetilde{P}}(x,y)^*\quad\text{ (adjoints of finite-dimensional operators)},
\end{align}
\begin{align}
\mathrm{e}^{-(t+s) \widetilde{P}}(x,y)=\int_X\mathrm{e}^{-t \widetilde{P}}(x,z)\mathrm{e}^{-s \widetilde{P}}(z,y)\Id\rho( z).
\label{chap}
\end{align}
c) For any $f\in \Gamma_{\IL^2_{\rho}}(X,E)$, the section
$$
(0,\infty)\times X\ni (t,x)\longmapsto f(t,x):=\int_X\mathrm{e}^{-t \widetilde{P}}(x,y)f(y)\Id\rho( y)\in E_x\subset E
$$
is smooth, and one has
$$
\frac{\partial}{\partial t} f(t,x)=-Pf(t,x)\quad\text{ for all $(t,x)\in(0,\infty)\times X$}.
$$
%and if $f\in \Gamma_{\ICC_{\c}}(X,E)$ then one furthermore has
%$$
%\lim_{t\to 0+}\left\| f(t,\bullet)-f\right\|_{P^k,2,\rho}=0\>\text{ for all $k\in \IN_{\geq 0}$} . 
%$$
\end{Theorem}

\begin{proof} The proof is based on the scalar case, a well-known result a special case of which can be found in \cite{davies}. However, since there is no need for a globally defined smooth frame to exist, we have to do some extra work.\\
Before we come to the proof of the actual statements of Theorem \ref{heat}, let us first establish some auxiliary results.\\
\emph{Step 1: For fixed $t>0$, there exists a smooth version of $x\mapsto\mathrm{e}^{-t \widetilde{P}}f(x)$.} \\
Proof: To see this, note that for any $n\in\IN_{\geq 1}$ one has 
$$
\dom((\widetilde{P})^n)\subset \Gamma_{ W^{P^n,2}_{\rho}}(X,E)\subset \Gamma_{ W^{k+n,2}_{\loc}}(X,E),
$$
where the second inclusion follows from local elliptic regularity. By the spectral calculus and the local Sobolev embedding, this implies
$$
\mathrm{Ran}(\mathrm{e}^{-t \widetilde{P}})\subset \bigcap_{n\in\IN_{\geq 1}}\dom((\widetilde{P})^n)\subset \Gamma_{\ICC}(X,E)\>\text{ for any $t>0$.}
$$
\emph{Step 2: For any $t>0$, $U\subset X$ open and relatively compact, the map
\begin{align}
\mathrm{e}^{-t \widetilde{P}}:\Gamma_{\IL^2_{\rho}}(X,E)\longrightarrow  \Gamma_{C_{ b  } }(U,E)\label{buu}
\end{align}
is a bounded linear operator between Banach spaces, where the space of bounded continuous sections $\Gamma_{C_{b}}(U,E)$ is equipped with its usual uniform norm.} \\
Proof: A priory, this map is algebraically well-defined by step 1. The asserted boundedness follows from the closed graph theorem, noting that the $\Gamma_{\IL^2_{\rho}}(X,E)$-convergence of a sequence implies the existence of a subsequence which converges $\rho$-a.e.\\ 
\emph{Step 3: For fixed $s>0$, the map 
$$
\Gamma_{\IL^2_{\rho}}(X,E)\times X\ni (f,x)\longmapsto\mathrm{e}^{-s \widetilde{P}}f(x)\in E_x\subset E
$$
is jointly continuous.} \\
Proof: Let $U\subset X$ be an arbitrary open and relatively compact subset. Given a sequence 
$$
((f_n,x_n))_{n\in\IN_{\geq 0}}\subset \Gamma_{\IL^2_{\rho}}(X,E)\times U
$$
 which converges to 
$$
(f,x)\in \Gamma_{\IL^2_{\rho}}(X,E)\times U,
$$
we have 
\begin{align*}
&\left| \mathrm{e}^{- s \widetilde{P}}f_n(  x_n)-\mathrm{e}^{- s \widetilde{P}}f(  x)\right| \\
&\leq \left|\mathrm{e}^{- s \widetilde{P}}[f_n-f](  x_n)\right| +\left|\mathrm{e}^{- s \widetilde{P}}f(  x)-\mathrm{e}^{- s \widetilde{P}}f(  x_n)\right| \\
&\leq \left\|\mathrm{e}^{- s \widetilde{P}}\right\|_{\ILL (\Gamma_{\IL^2_{\rho}}(X,E),  \Gamma_{C_{b}}(U,E) )}\left\|f_n-f\right\|_{L^2_{\rho}} +\left|\mathrm{e}^{- s \widetilde{P}}f(  x)-\mathrm{e}^{- s \widetilde{P}}f(  x_n)\right| \\
&\>\>\to 0, \text{ as $n\to\infty$},
\end{align*}
by step 2 and step 1.\\
\emph{Step 4: For fixed $\epsilon>0$ and $f\in \Gamma_{\IL^2_{\rho}}(X,E)$, the map
$$
\{\Re>\epsilon\}\times X\ni (z,x)\longmapsto \mathrm{e}^{-z \widetilde{P}}f(x)
$$
is jointly continuous.}\\
Proof: Indeed, this map is equal to the composition of the maps
$$
\{\Re> \epsilon\}\times X \xrightarrow[]{(z, x)\mapsto (\mathrm{e}^{-(z-\epsilon) \widetilde{P}}f, x) } \Gamma_{\IL^2_{\rho}}(X,E)\times X\xrightarrow[]{(f, x)\mapsto \mathrm{e}^{-\epsilon \widetilde{P}}f( x) } E,
$$
where the second map is continuous by Step 3. The first map is continuous, since the map
\begin{align} 
\{\Re> 0\} \ni z\longmapsto \mathrm{e}^{-z \widetilde{P}}f\in\Gamma_{\IL^2_{\rho}}(X,E)\label{anal}
\end{align}
is holomorphic. Note that, a priory, (\ref{anal}) is a weakly holomorphic semigroup by the spectral calculus, which is then indeed (norm-) holomorphic by the weak-to-strong differentiability theorem (cf. appendix, Theorem \ref{weaki}).\\
\emph{Step 5: For any $f\in \Gamma_{\IL^2_{\rho}}(X,E)$, there exists a jointly smooth version of $(t,x)\mapsto \mathrm{e}^{-t \widetilde{P}}f(x)$, which satisfies}
\begin{align}\label{eia}
\frac{\partial}{\partial t} \mathrm{e}^{-t \widetilde{P}}f(x)=-P\mathrm{e}^{-t \widetilde{P}}f(x).
\end{align}
Proof: By Step 4, for arbitrary $f\in \Gamma_{\IL^2_{\rho}}(X,E)$, the map
$$
\{\Re>0\}\times X\ni (z,x)\longmapsto \mathrm{e}^{-z \widetilde{P}}f(x)\in E_x\subset E
$$
is jointly continuous. It then follows from the holomorphy of (\ref{anal}) that for any open ball $B$ in the open right complex plane which has a nonempty intersection with $(0,\infty)$, for any $t\in B\cap (0,\infty)$, and for \emph{any} $  x\in X$, we have Cauchy's integral formula
$$
\mathrm{e}^{-t \widetilde{P}}f(x) = \oint_{\partial B} \frac{\mathrm{e}^{-z \widetilde{P}}f(  x)}{t-z}\Id z,
$$
noting that the holomorphy of (\ref{anal}) a priori only implies Cauchy's integral formula for \emph{almost every} $x$. Now the claim follows from differentiating under the line integral, observing that for fixed $z\in \{\Re>0\}$, the map
$$
X\ni   x\longmapsto \mathrm{e}^{-z \widetilde{P}}f(x )=\mathrm{e}^{-\Re (z) \widetilde{P}} \left[\mathrm{e}^{-\sqrt{-1}\Im (z) \widetilde{P}}f\right](  x)\in E_x\subset E
$$
is smooth by Step 1. Finally, the asserted formula (\ref{eia}) follows from the by now proved smoothness of $(t,x)\mapsto \mathrm{e}^{-t \widetilde{P}}f(x)$ and the fact that
$$
(\Id/\Id t) \mathrm{e}^{-t \widetilde{P}}f=-P\mathrm{e}^{-t \widetilde{P}}f,\>\>t>0,
$$
in the sense of norm differentiable maps $(0,\infty)\to \Gamma_{\IL^2_{\rho}}(X,E)$.\vspace{2mm}

Let us now come to the actual proof of Theorem \ref{heat}. We will prove a), b) and c) simultaneously. \\
First of all, it is clear that any such heat kernel is uniquely determined (by testing any two such kernels against arbitrary compactly supported smooth sections). To see its existence, we start by remarking that for every global Borel section $\phi$ in $X\to E$ and every $x\in X$, $t>0$, the complex linear functional given by
$$
\Gamma_{\IL^2_{\rho}}(X,E)\ni f\longmapsto \widetilde{\Psi(x,t,\phi)}[f]:= \big(\phi(x),\mathrm{e}^{-t \widetilde{P}}f(x)\big)\in \IC
$$
is bounded by Step 2. Thus by Riesz-Fischer\rq{}s representation theorem, there exists a unique section $\Psi(x,t,\phi)\in \Gamma_{\IL^2_{\rho}}(X,E)$ such that for all $f\in\Gamma_{\IL^2_{\rho}}(X,E)$ one has
\begin{align}\label{ries}
\big(\phi(x),\mathrm{e}^{-t \widetilde{P}}f(x)\big)=\widetilde{\Psi(x,t,\phi)}[f] = \left\langle  \Psi(x,t,\phi),f\right\rangle_{L^2_{\rho}},
\end{align}
and it follows immediately from step 5 that $(t,x)\mapsto \Psi(x,t,\phi)$ is weakly smooth if $\phi$ is smooth. In this case, this map is in fact norm smooth as a map $(0,\infty)\times X \to \Gamma_{\IL^2_{\rho}}(X,E)$ by the weak-to-strong differentiability theorem. We claim that the integral kernel which is well-defined by the \lq\lq{}regularization\rq\rq{}
\begin{align}\label{fmq}
( \phi_1(x)\mathrm{e}^{-t \widetilde{P}}(x,y)\phi_2(y)):= \left\langle \Psi(x,t/2,\phi_1),\Psi(y,t/2,\phi_2)\right\rangle_{L^2_{\rho}},
\end{align}
where $\phi_j$ are arbitrary smooth sections, has the desired properties. Indeed, firstly, the smoothness of $(t,x,y)\mapsto \mathrm{e}^{-t \widetilde{P}}(x,y)$ follows immediately from the norm smoothness of $(t,x)\mapsto \Psi(x,t,\phi)$ and the smoothness of the Hilbertian pairing $(f,g)\mapsto \left\langle f,g\right\rangle_{L^2_{\rho}}$. Then, picking a global orthonormal Borel frame $e_1,\dots, e_{\ell}$ for $E\to X$ (cf. Remark \ref{global}), we see that
$$
\left|\mathrm{e}^{-t \widetilde{P}}(x,y)e_j(y)\right|^2=\left|\mathrm{e}^{-\f{t}{2} \widetilde{P}}\Psi(y,t/2,e_j)(x)\right|^2,
$$ 
which follows from the definition of $\Psi(\dots)$, so that
$$
\int_X \left|\mathrm{e}^{-t \widetilde{P}}(x,y)\right|^2\Id\rho(x)\leq \int_X\left|\mathrm{e}^{-\f{t}{2} \widetilde{P}}\Psi(y,t/2,e_j)(x)\right|^2\Id\rho(x)<\infty,
$$
which implies the asserted square integrability of the integral kernel. Note that, by construction, one has the symmetry
$$
\mathrm{e}^{-t \widetilde{P}}(x,y)=\mathrm{e}^{-t \widetilde{P}}(y,x)^{*}.
$$
Finally, by a straightforward calculation which only uses the symmetry of $\mathrm{e}^{-t \widetilde{P}}$ and the definition of $\Psi(\dots)$, we get
$$
\big(\phi(x),\mathrm{e}^{-t \widetilde{P}}f(x)\big)=\int_X \big(\phi(x),\mathrm{e}^{-t \widetilde{P}}(x,y)f(y)\big)\Id\rho(y)
$$
for all smooth compactly supported sections $\phi$, which also implies the asserted semigroup property of the integral kernel, using the semigroup property of $\mathrm{e}^{-t \widetilde{P}}$ (a priori for all fixed $x$, and $\rho$-a.e. $y$, a posteriori for all $(x,y)$ by a standard continuity argument). This completes the proof.
\end{proof}

%\begin{Remark} Note that all statements of Theorem \ref{heat} remain true for \emph{arbitrary} semibounded self-adjoint extensions $\tilde{\H}_P$ of $P$, with $P$ as in Theorem \ref{heat}. Indeed, any such $\tilde{\H}_P$ necessarily satisfies 
%$$
%\dom(\tilde{\H}_P)\subset\Gamma_{ W^{P,2}_{\rho}}(X,E)\subset \Gamma_{ W^{k,2}_{\mathrm{loc}}}(X,E),
%$$
%and the latter inclusion was all we used in the above proof. We have only restricted ourselves to the Friedrichs case as this is the only situation that will be considered in the sequel, so that we could fix some notation.\vspace{1mm}
%
%\end{Remark}

\chapter{Basic differential operators in Riemannian manifolds}\label{C4}

\section{Preleminaries from Riemannian geometry}

A smooth \emph{Riemannian metric $g$} on a smooth $m$-manifold $M$ is by definition a smooth metric on the vector bundle $\IT M\to M$. The pair $(M,g)$ is then referred to as a \emph{smooth Riemannian manifold}.\vspace{2mm}

\emph{We fix once for all a smooth, connected, possibly noncompact Riemannian $m$-manifold $M\equiv (M,g)$. }\vspace{2mm}

This notation indicates that all Riemannian data on $M$ will be understood with respect to the fixed smooth Riemannian metric $g$ on $M$, unless otherwise stated. In accordance with our previous conventions, whenever there is no danger of confusion, $g$ itself and all canonically induced smooth metrics (for example those on $\IT^*M$ or $\IT^* M\wedge \IT^* M$ or duals therof) will simply be denoted by $(\bullet,\bullet)$, where $|\bullet|$ stands for the induced fiberwise norms and operator norms.\\

Let us collect some well-known facts and formulae from Riemannian geometry: The \emph{Levi-Civita connection} on $M$ is the uniquely determined smooth metric covariant derivative
$$
\nabla^{\IT M}\in \IDD^{(1)}_{\ICC}(M;\IT M, \IT^* M\otimes \IT M)
$$
which is torsion free, in the sense that
$$
\nabla^{\IT M}_A B-\nabla^{\IT M}_B A=[A,B]\text{ for all vector fields $A,B\in \mathscr{X}_{C^{\infty}}(M)$.}
$$
The \emph{Riemannian curvature tensor} $\mathrm{Riem}$ is defined to be the curvature of $\nabla^{\IT M}$, 
$$
\mathrm{Riem}:=R_{\nabla^{\IT M}}\in \Gamma_{\ICC}(M,  (\wedge^2 \IT^* M)\otimes \mathrm{End}(\IT M) ).
$$
We recall that for smooth vector fields $A,B,C\in \mathscr{X}_{C^{\infty}}(M)$, this tensor is explictly given by 
$$
\mathrm{Riem}(A, B)C:= \nabla^{\IT M}_A \nabla^{\IT M}_B C-\nabla^{\IT M}_B\nabla^{\IT M}_A C-\nabla^{\IT M}_{[A,B]}C\in\mathscr{X}_{C^{\infty}}(M).
$$
Then the \emph{Ricci curvature} 
$$
\mathrm{Ric}\in \Gamma_{\ICC}(M,  \IT^*M\odot \IT^*M )
$$
is the field of symmetric bilinear forms on $\IT M$ given by the fiberwise ($g$-)trace 
$$
\mathrm{Ric}(A,B)\mid_U\>=\sum^m_{j=1}( \mathrm{Riem}(e_j, B)A,e_j),
$$
where $e_1,\dots,e_m\in \mathscr{X}_{C^{\infty}} (U)$ is a local orthonormal frame, and $A,B\in \mathscr{X}_{C^{\infty}} (M)$. The \emph{scalar curvature} is the smooth real-valued function $\mathrm{scal}\in\ICC(M)$ given by
$$
\mathrm{scal}\mid_U\>:=\sum^m_{j=1}\mathrm{Ric}(e_j,e_j).
$$
Another concept that we will need from time to time are the \emph{sectional curvatures} of $M$: If $m\geq 2$, then for every $x\in M$ the sectional curvature $\mathrm{Sec}(v_x)$ of a two-dimensional subspace 
$$
v_x=\mathrm{span}(A(x),B(x))\subset\IT_x M,\quad A,B\in\mathscr{X}_{C^{\infty}}(M), 
$$
is well-defined by
$$
\mathrm{Sec}(v_x):=\frac{(\mathrm{Riem}(A,B)B,B)}{|A\wedge B|^2}|_x\in\IR.
$$

The \emph{Riemannian volume measure} is the uniquely determined smooth Borel measure $\mu$ on $M$, such that for every smooth chart $((x^1,\dots,x^m),U)$ and any Borel set $N\subset U$, one has
\begin{align}\label{mas1}
&\mu(N) = \int_N \sqrt{\det(g(x))}\Id x, 
\end{align}
where $\det(g(x))$ is the determinant of the matrix $g_{ij}(x):=g(\partial_{i},\partial_{j})(x)$ and where $\Id x=\Id x^1\cdots \Id x^m$ stands for the Lebesgue integration (cf. Theorem 3.11 in \cite{gri}). The metric $g$ canonically induces an isomorphism of smooth $\IR$-vector bundles
\begin{align}
\IT^* M\longrightarrow \IT M,\>\>\alpha\longrightarrow \alpha^{\sharp},
\end{align}
with its inverse
\begin{align}
\IT M\longrightarrow \IT^* M,\>\>A\longrightarrow A^{\flat}.
\end{align}
We have the \emph{scalar Laplacian} (also called the \emph{Laplace-Beltrami operator})  
$$
-\Delta:=\Id^{\dagger}\Id\in \IDD^{(2)}_{\ICC}(M).
$$
Note that our sign convention for $\Delta$ is the one from the mathematical physics literature, and not the one which is typically used in the geometry literature. As we will see in a moment, $\Delta$ is elliptic. 

\begin{Remark}
The reader may find it helpful to know that in a smooth chart $((x^1,\dots, x^m),U)$, the scalar Laplacian is given for every smooth function $f:U\to\IC$ and $x\in U$ by
$$
-\Delta f (x)= -\sum^m_{i,j=1}\f{1}{\sqrt{\det (g)}}\f{\partial }{\partial x^i } \left(\sqrt{\det (g)} g^{ij}\f{\partial f}{\partial x^j } \right) (x)  ,
$$
where $g^{ij}:=g^*(d x^i, dx^j)$. Moreover, the following \emph{strong elliptic minimum principle} holds (cf. Corollary 8.14 in \cite{gri}): If $\lambda\in\IR$ and if $0\leq f\in C^2(M)$ vanishes at some point of $M$ and is $\alpha$-superharmonic in the sense that 
$$
-\Delta f+\lambda f\geq 0,
$$
then $f$ vanishes everywhere in $M$. In particular, every real-valued superharmonic ($=$ $0$ - superharmonic) function $f\in C^2(M)$ with $f(x_0)=\inf f$ for some $x_0\in M$ satisfies $f(x)=\inf f$ for all $x\in M$: This follows from applying the strong elliptic minimum principle to $f-\inf f$. As a consequence, we get the following \emph{elliptic minimum principle:} For every relatively compact open subset $U\subset M$ with a nonempty boundary, and every real-valued superharmonic function $f\in C^2(U)\cap C(\overline{U})$,  one has
$$
\inf_{\overline{U}} f =\inf_{\partial U} f.
$$
Indeed, using the strong elliptic minimum principle, one finds straightforwardly that the set
$$
\Big\{x\in \overline{U}: f(x)=\inf_{\overline{U}}f \Big\}\subset M
$$
intersects $\partial U$ (cf. \cite{gri}, p. 230). Like all minimum principles, these results rely on our assumption that $M$ is connected.\end{Remark}

In addition to the scalar Laplacian, we have for any $j\in\{0,\dots,m\}$ the \emph{Hodge-Laplacian on $j$-forms}, which is defined by 
$$
-\Delta^{(j)}:=\Id^{\dagger}_{j}\Id_{j}+\Id_{j-1}\Id^{\dagger}_{j-1}\in\IDD^{(2)}_{\ICC}(M;\wedge^j_{\IC} T^*  M),
$$
where of course $\Id^{\dagger}_{-1}:=0$. Then \emph{Weitzenböck's formula} states that  
$$
V^{(j)}:=-\Delta^{(j)}-(\nabla^{\wedge^j T^*  M})^{\dagger}\nabla^{\wedge^j T^* M}\in\IDD^{(0)}_{\ICC}(M;\wedge^j_{\IC} T^* M)
$$
is an explicitly (in terms of the Riemannian curvature tensor) given zeroth order operator with $(V^{(j)})^{\dagger}=V^{(j)}$. Weitzenböck's formula can be derived from the following formula for the adjoint of the exterior differential: For every $\alpha \in \Omega^k_{\ICC}(M)$, one has 
\begin{align}\label{sabine}
\Id^{\dagger}_k\alpha(A_1,\dots, A_{k-1})=-\sum_{j=1} \nabla^{\wedge^k T^*  M}_{e_j}\alpha(e_j, A_1,\dots,A_{k-1}), 
\end{align}
valid for all smooth vector fields $A_1,\dots,A_{k-1}\in \mathscr{X}_{C^{\infty}} (M)$, and every smooth local orthonormal frame $e_1,\dots,e_m\in \mathscr{X}_{C^{\infty}} (U)$. These facts are standard and can be found, for example, in \cite{getzler}.

\begin{Notation} $\mathrm{Ric}_{\sharp}\in\Gamma_{\ICC}(M,  T^* M\otimes T^*M)$ denotes the field of symmetric sesquilinear forms given by $\mathrm{Ric}(\sharp,\sharp)$, the composition of the Ricci curvature with $\sharp$. 
\end{Notation}

For $j=1$ the endomorphism $V^{(1)}$ has a very simple form, namely
$$
V^{(1)}=\mathrm{Ric}_{\sharp}, \quad\text{and therefore}\quad -\Delta^{(1)}=(\nabla^{T^* M})^{\dagger}\nabla^{ T^* M}+\mathrm{Ric}_{\sharp},
$$
when $\mathrm{Ric}_{\sharp}$ is considered an element of $\IDD^{(0)}_{\ICC}(M;T^*_{\IC} M)$. The general formula for $V^{(j)}$ with $j>1$ is usually not easy to control analytically (although the so-called \emph{Gallot-Meyer estimate} \cite{gallot} states that it can be controlled in a certain sense from below by $\mathrm{Riem}$, or more precisely, by the so-called curvature endomorphism of $\nabla^{\IT M}$). \vspace{1.2mm}

We continue our list of formulae: Given $f_1,f_2\in\ICC(M)$,  $\alpha\in\Omega^1_{\ICC}(M)$ and $f\in\ICC(\IR)$, one has the following product rule, and chain rule, respectively:
\begin{align}\label{gaga2}
&\Delta(f_1f_2)=f_1\Delta f_2+f_2\Delta f_1+2\Re(\Id f_1,\Id f_2),\\ \label{gaga3}
&\Delta (f\circ f_1)= (f\rq{}\rq{} \circ  f_1) |\Id f|^2+(f\rq{}\circ f_1) \Delta  f_1.
\end{align}

Let $\nabla$ be a smooth metric covariant derivative on the smooth metric $\IK$-vector bundle $E\to M$. Then one has 
\begin{align}\label{jkg}
\nabla^{\dagger}( \alpha\otimes\psi)=(\Id^{\dagger}\alpha)\psi -\nabla_{\alpha^{\sharp}}\psi
\end{align}
for all real-valued $\alpha\in\Omega^1_{\ICC }(M)$, $\psi\in\Gamma_{\ICC}(M,E)$. This can be seen as follows: Let $\psi_1\in\Gamma_{\ICC_{\c}}(M,E)$. Since $\nabla$ is metric, we have 
$$
(\alpha,\Id(\psi_1,\psi))=\alpha^{\sharp}(\psi_1,\psi)=(\nabla_{\alpha^{\sharp}}\psi_1,\psi)+(\psi_1,\nabla_{\alpha^{\sharp}}\psi),
$$
thus
$$
\int \Big(\psi_1,\big((\Id^{\dagger}\alpha)-\nabla_{\alpha^{\sharp}}\big)\psi\Big)\Id\mu=\int (\nabla_{\alpha^{\sharp}}\psi_1,\psi) \Id\mu,
$$
which shows (\ref{jkg}), in view of the simple identity 
$$
(\alpha\otimes(\bullet))^{\dagger}\circ \nabla=\nabla_{\alpha^{\sharp}}\in\IDD^{(1)}_{\ICC}(M;E).
$$
Formula (\ref{jkg}) can be used to deduce the following local formula: Given a smooth local orthonormal frame $e_1,\dots,e_m\in \mathscr{X}_{C^{\infty}}(U)$, one has 
\begin{align}\label{zuh}
\nabla^{\dagger}\nabla  =-\sum_{i=1}^m\left(\nabla_{e_i}\nabla_{e_i} -\nabla_{\nabla^{\IT M}_{e_i}e_i}\right) \>\>\text{ in $U$.}
\end{align}
Indeed, let $\psi\in\Gamma_{\ICC}(M,E)$, and note that the right-hand-side of (\ref{zuh}) does not depend on a particular choice of $e_1,\dots,e_m$. Thus it is sufficient to prove the formula in a particular frame. We pick a frame with $\nabla^{\IT M}e_j(x)=0$ at a fixed $x\in U$, so that  (at $x$)
\begin{align}\label{aopqay}
-\sum_{i=1}^m\left(\nabla_{e_i}\nabla_{e_i}-\nabla_{\nabla^{\IT M}_{e_i}e_i}\right)\psi(x)=-\sum_{i=1}^m\nabla_{e_i}\nabla_{e_i}\psi(x),
\end{align}
and (\ref{zuh}) follows from
$$
\nabla \psi= \sum_i  e_i^*\otimes \nabla_{e_i}\psi
$$
and (\ref{jkg}). \\
It is also possible to deduce from formula (\ref{aopqay}) that
\begin{align}\label{sym}
\symbol_{\nabla^{\dagger}\nabla}(\zeta\otimes\zeta)\psi=-|\zeta|^2\psi,\quad\zeta\in\IT^*_x M,\>\psi\in E_x,\>x\in M
\end{align}
In particular, $\nabla^{\dagger}\nabla$ is elliptic. Being a lower order perturbation of operators having the latter form, it follows that each $\Delta_j$ is elliptic, too.\\
If one unpacks the definition of the dual covariant derivative and the tensor product of covariant derivatives, one finds the following useful formula\footnote{This formula is also valid in case $\nabla$ is not metric.} for $\nabla^{(2)}:=(\nabla\tilde{\otimes} \nabla^{\IT^*M}) \circ \nabla$,
\begin{align}\label{getzi}
 \nabla^{(2)} \psi(A,B)=\nabla_A\nabla_B\psi-\nabla_{\nabla^{\IT M}_A B}\psi,\quad A,B\in\mathscr{X}_{C^{\infty}} (M), \>\>\psi\in\Gamma_{\ICC}(M,E),
\end{align}
so that (\ref{zuh}) shows that 
\begin{align}\label{zuh2}
\nabla^{\dagger}\nabla \psi=-\sum^m_{j=1} \nabla^{(2)} \psi(e_j,e_j),
\end{align}
where $e_j$ is again an arbitrary smooth orthonormal frame.\\
Another important consequence of (\ref{jkg}) is the following product rule,
\begin{align}\label{gaga4}
\Id^{\dagger}( f\alpha)=f\Id^{\dagger} \alpha-(\alpha,\Id f)\quad \alpha \in\Omega^1_{\ICC}(M),\>f\in\ICC(M),
\end{align}
which will also be useful in the sequel.\\
An important result that links the underlying geometry with geometric analysis is the following \emph{Bochner identity}: For all $f\in\ICC(M)$ one has
\begin{align}\label{bochner}
\left|\nabla^{T^* M}\Id f\right|^2=\f{1}{2}\Delta |\Id f|^2-(\Id f,\Id \Delta f) - \mathrm{Ric}_{\sharp}(\Id f,\Id f).
\end{align}
In fact, $\nabla^{T^*M}d f$ is a natural Riemannian generalization of the usual Hessian. Of course it is sufficient to prove the Bochner identity for real-valued $f$'s, which is a standard result.\vspace{2mm}

We continue by recalling that the \emph{geodesic distance} on $M$ is given by
\begin{align}\label{abst}
&\varrho(x,y):=\inf\Big\{\int^1_0\left|\dot{\gamma}(s)\right|\Id s: \gamma\in\ICC (\left[0,1\right],M),\gamma(0)=x,\gamma(1)=y\Big\}.
\end{align}
The finiteness of this quantity for all $x$, $y$ requires that $M$ is connected, in which case $(M,\varrho)$ becomes a metric space. Furthermore, one can equivalently take piecewise smooth curves instead of smooth ones in the defining set of $\varrho(x,y)$, as long as $\int^1_0\left|\dot{\gamma}(s)\right|\Id s$ is interpreted in the obvious sense. The symbol 
$$
\IB(x,r):=\{y\in M: \varrho(x,y)<r\} 
$$
will denote the corresponding open geodesic balls, where $x\in M$, $r>0$. We recall that the manifold topology of $M$ is equal \cite{gri} to the topology which is induced by the metric $\varrho(\bullet,\bullet)$. The Riemannian manifold $M$ is called \emph{geodesically complete}, if the metric space $(M,\varrho)$ is so in the usual sense. By Hopf-Rinow\rq{}s Theorem, this completeness is equivalent to all bounded subsets being relatively compact, and also equivalent to all open geodesic balls being relatively compact. As we will see, geodesically complete $M$\rq{}s (such as the Euclidean $\IR^m$ or compact $M$'s) behave quite well analytically from points of view such as essential self-adjointness results for Schrödinger-type operators. On the other hand, even very simple Riemannian manifolds such as genuine open subsets of the Euclidean $\IR^m$ are not geodesically complete, so that we will try to avoid this assumption whenever possible. Other (possibly very complicated) incomplete Riemannian manifolds appear naturally in the topology of singular spaces: For example, if $X$ is a so-called smoothly Thom-Mather stratified compact $m$-pseudomanifold, then nontrivial topological information about $X$ is encoded in its so-called intersection cohomology (which depends on the choice of an additional datum, a so-called perversity function). It is a highly nontrivial fact that for every $(m-1)$-tuple $(c_2,\dots,c_m)$ of real numbers $\geq 1$, the so-called regular part $\mathrm{reg}(X)\subset X$ (a smooth $m$-manifold) carries a so-called \emph{iterated edge metric $g_{\mathrm{reg}}$ of type $(c_2,\dots,c_m)$} \cite{brass,JCh}, and that moreover one can use $g_{\mathrm{reg}}$ to calculate the intersection cohomology of the whole space $X$ \cite{JCh,hunz,FB1,FB2,FB3} for a perversity function that only depends on $(c_2,\dots,c_m)$. \vspace{2mm}

A place where geodesic completeness comes into play in a very crucial way is the question of whether or not families of cut-off functions exist, an important analytic tool:

%\begin{Remark}
%Given a second smooth Riemannian metric $g'$ on $M$, the metrics $g$ and $g'$ are called \emph{quasi-isometric}, if there is a constant $C>1$ with
%$$
%C^{-1}g'(x)\leq g(x)\leq C g'(x)\>\text{as symmetric bilinear-forms, for all $x\in M$.}
%$$ 
%Then one has 
%$$
%\sqrt{C^{-1}}\Id'(x,y)\leq \Id (x,y)\leq \sqrt{C} \Id'(x,y)\>\text{ for all $x,y\in M$}
%$$
%for the induced geodesic distances, and 
%$$
%C^{-\f{m}{2}}\Id\mu' \leq \Id \mu \leq  C^{\f{m}{2}} \Id\mu' 
%$$
%for the induced Riemannian volume measures. Quasi-isomtry defines an equivalence relation.
%\end{Remark}

\begin{Theorem}\label{first} a) $M$ is geodesically complete, if and only if $M$ admits a sequence $(\chi_n)\subset C^{\infty}_\c(M)$ of \emph{first order cut-off functions}, that is, $(\chi_n)$ has the following properties\emph{:}
\begin{itemize}
\item[(C1)] $0 \le \chi_n(x) \le 1$ for all $n\in\IN_{\geq 1}$, $x \in M$,

\item[(C2)] for all compact $K\subset M$, there is an $n_0(K)\in\IN$ such that for all $n\geq n_0(K)$ one has $\chi_n\mid_{K}= 1$,

\item[(C3)] $ \left\|\Id\chi_n \right\|_{\infty} \to 0$ as $n\to \infty$.
\end{itemize}
b) Assume that $M$ is geodesically complete with $\mathrm{Ric}\geq -C$ for some constant $C\geq 0$, that is, one has
$$
\mathrm{Ric}(A(x),A(x))\geq - C |A(x)|^2\quad\text{ for all $A\in \mathscr{X}_{C^{\infty}} (M)$, $x\in M$.}
$$

Then $M$ admits a sequence $(\chi_n)\subset C^{\infty}_\c(M)$ of \emph{Laplacian cut-off functions}, that is, $(\chi_n)$ has the above properties (C1), (C2) , (C3), and in addition 
\begin{itemize}
\item[(C4)] $  \left\|\Delta \chi_n \right\|_{\infty} \to 0$ as $n\to\infty$.
\end{itemize}
c) Assume that $M$ is geodesically complete with\footnote{The assumption as well as the conclusion of part b) are stronger than the ones in part c). } $|\mathrm{Sec}|\leq C$ for some constant $C\geq 0$. Then $M$ admits a sequence $(\chi_n)\subset C^{\infty}_\c(M)$ of \emph{Hessian cut-off functions}, that is, $(\chi_n)$ has the above properties (C1), (C2) , (C3), and in addition 
\begin{itemize}
\item[(C4\rq{})] $  \left\|\nabla^{T^*M}d \chi_n \right\|_{\infty} \to 0$ as $n\to\infty$.
\end{itemize}
\end{Theorem}

\begin{proof} a) This proof is borrowed from \cite{glob}: If $M\equiv (M,g)$ is geodesically complete, then by (a small generalization of)\footnote{As it stands, Nash\rq{}s embedding theorem does not require geodesic completeness, but it also does not give an isometric embedding with a \emph{closed} image. The way out of this is to pick an isometric smooth embedding $\iota\rq{}:M\hookrightarrow \IR^{s}$ by Nash and to use the geodesic completeness of $M$ in order to construct an isometric smooth embedding $\iota:M\hookrightarrow \IR^{s+1}$ from $\iota\rq{}$ into a larger space, which however really has a closed image, see for example \cite{mueller} for details. Note here that a smooth embedding of a smooth Riemannian manifold to another one is called \emph{isometric}, if it preserves the underlying Riemannian metrics.} Nash\rq{}s embedding theorem we can pick a smooth embedding $\iota: M\hookrightarrow \IR^l$ such that $g$ is the pull-back of the Euclidean metric on $\IR^l$, where $l\geq m$ is large enough, and such that $\iota(M)$ is a closed subset of $\IR^l$. Then clearly $\iota$ is proper, and therefore the composition
$$
f:M\longrightarrow \IR,\>\> f(x):=\log(1+|\iota(x)|^2)
$$
is a smooth proper function with $|\Id f|\leq 1$, since 
$$
\tilde{f}:\IR^l\longrightarrow \IR,\>\>\tilde{f}(v):= \log(1+|v|^2)
$$
is a smooth proper function whose gradient is absolutely bounded by $1$. Pick now a sequence $(\varphi_n)\subset\ICC_{\c}(\IR)$ of first order cut-off functions on the Eudlidean space $\IR$. (For example, let $\varphi:\IR\to [0,1]$ be smooth and compactly supported  with $\varphi=1$ near $0$, and set $\varphi_n(r):=\varphi(r/n)$, $r\in\IR$.) Then $\chi_n(x):=\varphi_n(f(x))$ obviously has the desired properties. \\
Conversely, suppose that $M$ admits a sequence $(\chi_n)\subset C^{\infty}_\c(M)$ of first order cut-off functions. Then given $\mathscr{O}\in M$, $r>0$, we are going to show that there is a compact set $A_{\mathscr{O},r}\subset M$ such that 
$$
\varrho(x,\mathscr{O})>r\>\text{ for all $x\in M\setminus A_{\mathscr{O},r}$,}
$$ 
which implies that any open geodesic ball is relatively compact. To see this, we pick a compact $A_{\mathscr{O}}\subset M$ such that $\mathscr{O}\in A_{\mathscr{O}}$, and a number $n_{\mathscr{O},r}\in \IN$ large enough such that $\chi_{n_{\mathscr{O},r}} =1$ on $A_{\mathscr{O}}$ and 
\begin{align}\label{boudn}
\sup_{x \in M} \left|\Id\chi_{n_{\mathscr{O},r}}(x)\right|\leq 1/(r+1).
\end{align}
Now let $A_{\mathscr{O},r}:=\mathrm{supp}(\chi_{n_{\mathscr{O},r}})$, let $x\in M\setminus A_{\mathscr{O},r}$, and let 
$$
\gamma:[0,1]\longrightarrow  M
$$
be a smooth curve with $\gamma(0)=x$, $\gamma(1)=\mathscr{O}$. Then we have
$$
1=\chi_n(\mathscr{O})-\chi_n(x)=\chi_n(\gamma(1))-\chi_n(\gamma(0))=\int^1_0 \big(\Id\chi_n(\gamma(s)),\dot{\gamma(s)} \big) \Id s
$$
By using (\ref{boudn}) and taking $\inf_{\gamma}\cdots$, we arrive at
$$
\varrho(x,\mathscr{O})\geq r+1 \> \text{ for all $x\in M\setminus A_{\mathscr{O},r}$,}
$$
as claimed.\\
b) This follows immediately from a highly nontrivial result by R. Schoen and S.-T. Yau on the existence of well-behaved exhaustion functions: Namely, Theorem 4.2 from \cite{schoen} states that under the given geometric assumptions on $M$ there exists a smooth proper function $f:M\to \IR$ and a constant $A>0$ such that 
$$
\sup_M\max (|\Delta f|, |\Id f|)\leq A.
$$
Picking a sequence $(\varphi_n)\subset\ICC_{\c}(\IR)$ of second order cut-off functions on the Euclidean space $\IR$ (for example the same  sequence as in the proof of part a)), one finds again that $\chi_n(x):=\varphi_n(f(x))$ has the desired properties in view of the Laplacian chain rule (\ref{gaga3}). Without entering any details of the construction of $f$, we only point out that the assumptions on the geometry of $M$ enter the game through (a consequence of) the Laplacian comparison theorem, which states that for every fixed reference point $\mathscr{O}\in M$, one has
$$
\Delta \varrho_{\mathscr{O}}\leq (m-1)/\varrho_{\mathscr{O}} +(m-1)\sqrt{C}.
$$
The last inequality is valid at each point of the set $M\setminus\{\mathscr{O}\}$ in which the distance function $x\mapsto \varrho_{\mathscr{O}}(x):=\varrho(x,\mathscr{O})$ is smooth. The function $\varrho_{\mathscr{O}}$ is then used together with the elliptic minimum principle to construct the function $f$.\\
c) Again, this follows from a subtle result concerning the existence of well-behaved exhaustion functions (by J. Cheeger and M. Gromov): Lemma 5.3 from \cite{cheegergromov} implies that under the given geometric assumptions on $M$ there exists a smooth proper function $f:M\to \IR$ and a constant $A>0$ such that 
$$
\sup_M\max (|\nabla^{T^*M} d f|, |\Id f|)\leq A.
$$
Picking a sequence $(\varphi_n)\subset\ICC_{\c}(\IR)$ of second order cut-off functions on the Euclidean space $\IR$ (for example the same  sequence as in the proof of part a)), one finds again that $\chi_n(x):=\varphi_n(f(x))$ has the desired properties.
\end{proof}

Part b) of the last theorem is a considerable generalization of a previously established result by M. Braverman, O. Milatovic and M. Shubin \cite{Br}, where the authors require that $M$ has a $\ICC$-bounded geometry (meaning that the Levi-Civita derivatives of the curvature tensor $\mathrm{Riem}$ are bounded up to all orders, and $M$ has a strictly positive injectivity radius). We will see later on that these sequences of cut-off functions also play an important role in the context of density problems in Riemannian Sobolev spaces. Also, we refer the interested reader to the recent paper \cite{pigola} by S. Pigola and the author for the interplay between sequences of cut-off functions, geometric Calderon-Zygmund inequalities (which control the Hessian in terms of the Laplacian in an $\IL^q$-sense) and Sobolev spaces.

\begin{Remark} 1. In \cite{G54}, the author has proved the existence of Laplacian cut-off functions under the much more restrictive assumption of a nonnegative Ricci curvature, and using a different proof which relies on the following result from Riemannian rigidity theory that has been established by J. Cheeger / T. Colding and F. Wang / X. Zhu (cf. Lemma 1.4 in \cite{wang}, the proof of which is based on arguments from \cite{cheeger}), namely: There is a constant $C(m)>0$, which only depends on $m$, such that for any fixed reference point $\mathscr{O}\in M$ and any smooth geodesically complete Riemannian metric $\tilde{g}$ on $M$ with $\mathrm{Ric}_{\tilde{g}}\geq 0$, there is a function $\chi_{\tilde{g}}=\chi_{\tilde{g},\mathscr{O}}\in C^{\infty}(M)$ which satisfies (with an obvious notation)
\begin{align*}
&0\leq \chi_{\tilde{g}}\leq 1,\>\mathrm{supp}(\chi_{\tilde{g}})\subset \IB_{\tilde{g}}(\mathscr{O},2), \>\chi_{\tilde{g}}=1\text{ on }\IB_{\tilde{g}}(\mathscr{O},1), \\
&|\Id \chi_{\tilde{g}}|_{\tilde{g}}\leq C(m),\> |\Delta_{\tilde{g}}\chi_{\tilde{g}}|\leq C(m). 
\end{align*}
From this, the existence of Laplacian cut-off functions can be deduced using a careful scaling argument. \\
2. Although it is not evident at all, it is possible to allow lower Ricci bounds which are not necessarily constant in Theorem \ref{first} b). Results of this type have been established recently by D. Bianchi and A. Setti \cite{alberto}.
\end{Remark}

\section{Riemannian Sobolev spaces and Meyers-Serrin theorems}

%\begin{Remark} Given a smooth $\IK$-vector bundle, the measure $\mu$ induces an isomorphism of topological $\IK$-linear spaces
%$$
%\Gamma_{\ICC_\c}(M,E)\longrightarrow \Gamma_{\mathsf{D}}(M,E),\>\psi\mapsto \psi\otimes \mu
%$$
%which, in the sequel, can and will be safely ommited in the notation.
%\end{Remark}

Let us begin with some convenient notations:

\begin{Notation}
Let $E\to M$, $F\to M$ be smooth metric $\IK$-vector bundles. We will use the following natural conventions in the Riemannian case, which are in the spirit of the conventions from Remark \ref{adjo}:\\
1. We define
\begin{align*}
&\Gamma_{\IL^q}(M,E):=\Gamma_{\IL^q_{\mu}}(M,E),\\
&\>\left\|\bullet\right\|_{q}:=\left\|\bullet\right\|_{L^q_{\mu}},\>\left\langle\bullet,\bullet\right\rangle:=\left\langle\bullet,\bullet\right\rangle_{L^2_{\mu}},
\end{align*}
where $\left\|\bullet\right\|_{q_1,q_2}$ will stand for the operator norm on the $\IK$-Banach space $\ILL\big(\Gamma_{\IL^{q_1}}(M,E),\Gamma_{\IL^{q_2}}(M,E)\big)$.
In the particular case of $E=\wedge^r_{\IC} T^* M\to M$ with its canonically given metric, we will simply write
$$
\Omega^r_{\IL^q}(M):=\Gamma_{\IL^q}(M,\wedge^r_{\IC} T^* M)
$$
for the space of complex $\IL^q$-differential forms.\\
2. Given $P\in\IDD^{(k)}_{\ICC}(M;E,F)$, we will simply write $P^{\dagger}$ for the formal adjoint of $P$ with respect to $\mu$ and the underlying metric structures. In addition, we will simply write $P^{(q)}_{\min}$ instead of $P^{(q)}_{\mu,\min}$, with $P_{\min}:=P^{(2)}_{\min}$, and likewise for the maximal extensions of $P$.
\end{Notation}

If $\nabla$ is a smooth covariant derivative on the smooth $\IK$-vector bundle $E\to M$, for any $j\in\IN$ the operator $\nabla^j$ is defined as follows: Firstly, we have the operator
$$
\nabla^{(j)}\in \mathscr{D}_{C^{\infty}}^{(1)}
\big(M; \left( \IT^*M\right)^{\otimes j-1} \otimes E, 
\left( \IT^*M\right)^{\otimes j} \otimes E\big)
$$ 
which is defined recursively by $\nabla^{(1)}:=\nabla$, 
$\nabla^{(j+1)}:=\nabla^{(j)}\widetilde{\otimes} \nabla^{\IT^* M}$. Then one sets
$$
\nabla^{j}:=\nabla^{(j)}\cdots \nabla^{(1)}\in 
\mathscr{D}_{C^{\infty}}^{(j)}\big(M;E , \left(\IT^*M\right)^{\otimes j} \otimes E\big).
$$

\begin{Definition} Let $\nabla$ be a smooth covariant derivative on the smooth metric $\IK$-vector bundle $E\to M$. For any $s\in \IN$ and $q\in [1,\infty]$, we define the $\IK$-Banach space
$\Gamma_{ W^{s,q}_{\nabla}}(M,E)$ to be 
\begin{align*}
&\Gamma_{ W^{s,q}_{\nabla}}(M,E)\\
&:=\big\{f\in\Gamma_{\IL^q}(M,E):\nabla^jf \in \Gamma_{\IL^q}(M, (\IT^*M)^{\otimes_j}\otimes E)\>\text{ for all $j=1,\dots,s$}\big\},
\end{align*}
where the underlying norm is canonically given by
\begin{align*}
\left\|f\right\|_{ W^{s,q}_{\nabla}}^q:=\sum^s_{j=0} \left\|\nabla^j f\right\|^q_{q}.
\end{align*}
The space $\Gamma_{ W^{s,q}_{\nabla}}(M,E)$ is called the \emph{Riemannian $\IL^q$-Sobolev space of differential order $s$ with respect to  $(\nabla,E)\to M$.} Furthermore, we define the $\IK$-Banach space $\Gamma_{ W^{s,q}_{\nabla,0}}(M,E)$ to be the closure of $\Gamma_{\ICC_{\c}}(M,E)$ in $\Gamma_{ W^{s,q}_{\nabla}}(M,E)$.
\end{Definition}

Note that $\Gamma_{ W^{s,2}_{\nabla}}(M,E)$ (and so also $\Gamma_{ W^{s,2}_{\nabla,0}}(M,E)$) becomes a Hilbert space in an obvious way. In order to make contact with our previous notation, we add:

\begin{Remark} Upon taking $\mathfrak{P}:= \{\nabla^1,\dots,\nabla^s\}$, one has in fact
\begin{align*}
\Gamma_{ W^{s,q}_{\nabla}}(M,E)= \Gamma_{ W^{\mathfrak{P},q}_{\mu}}(M,E),\text{ and } \left\|\bullet\right\|_{ W^{s,q}_{\nabla}}=\left\| \bullet\right\|_{\mathfrak{P},L^q_{\mu}},
\end{align*}
and furthermore 
$$
\Gamma_{ W^{s,q}_{\nabla,0}}(M,E)=\Gamma_{ W^{\mathfrak{P},q}_{\mu,0}}(M,E).
$$
\end{Remark}

\begin{Definition} Let $\nabla$ be a smooth covariant derivative on the smooth metric $\IK$-vector bundle $E\to M$. For any $s\in \IN$ and $q\in [1,\infty]$, we define the $\IK$-Banach space $\Gamma_{\widetilde{ W}^{2s,q}_{\nabla}}(M,E)$ to be 
\begin{align*}
&\Gamma_{\widetilde{ W}^{2s,q}_{\nabla}}(M,E)\\
&:=\big\{f\in\Gamma_{\IL^q}(M,E):(\nabla^{\dagger}\nabla)^j f\in  \Gamma_{\IL^q}(M,E)\text{ for all $j=1,\dots,s$} \big\},
\end{align*}
with its canonically given norm
$$
\left\|f\right\|^q_{\widetilde{ W}^{2s,q}_{\nabla}}:=\sum^s_{j=0} \left\|(\nabla^{\dagger}\nabla)^j f\right\|^q_{q}.
$$
The space $\Gamma_{\widetilde{ W}^{2s,q}_{\nabla}}(M,E)$ is called the \emph{elliptic Riemannian $\IL^q$-Sobolev space of differential order $s$ with respect to  $(\nabla,E)\to M$.} Furthermore, the $\IK$-Banach space $\Gamma_{\widetilde{ W}^{2s,q}_{\nabla,0}}(M,E)$ is defined to be the closure of $\Gamma_{\ICC_{\c}}(M,E)$ in $\Gamma_{\widetilde{ W}^{2s,q}_{\nabla}}(M,E)$.
\end{Definition}

Again, $\Gamma_{\widetilde{ W}^{2s,2}_{\nabla}}(M,E)$ (and so also $\Gamma_{\widetilde{ W}^{2s,2}_{\nabla,0}}(M,E)$) becomes a Hilbert space in an obvious way. Note that $\Gamma_{\widetilde{ W}^{k,q}_{\nabla}}(M,E)$ is only defined for even natural numbers $k$.

\begin{Remark} Upon taking 
$$
\widetilde{\mathfrak{P}}:=\{(\nabla^{\dagger}\nabla)^1,\dots,(\nabla^{\dagger}\nabla)^s\},
$$
we have
\begin{align*}
&\Gamma_{\widetilde{ W}^{2s,q}_{\nabla}}(M,E)=\Gamma_{ W^{\widetilde{\mathfrak{P}},q}_{\mu}}(M,E),\text{ and } \left\|\bullet\right\|_{\widetilde{ W}^{2s,q}_{\nabla}}=\left\| \bullet\right\|_{\widetilde{\mathfrak{P}},L^q_{\mu}},
\end{align*}
and
$$
\Gamma_{\widetilde{ W}^{2s,q}_{\nabla,0}}(M,E)=\Gamma_{ W^{\widetilde{\mathfrak{P}},q}_{\mu,0}}(M,E).
$$
\end{Remark}

In the scalar case, we will use the following standard notation:

\begin{Notation}\label{soos} In the simplest case of scalar functions with $\nabla=\Id$ the exterior derivative, we will write 
$$
 W^{s,q}(M),\> W^{s,q}_0(M),\>\widetilde{ W}^{2s,q}(M),\>\widetilde{ W}^{2s,q}_0(M)
$$
for the corresponding complex (!) Sobolev spaces of functions. Note that, by definition, for any smooth $f:M\to\IC$ one has
$$
\nabla f=\Id f, \>\nabla^2 f=\nabla^{T^* M}\Id f.
$$
Furthermore, (\ref{zuh2}) implies
$$
|\Delta f|\leq \sqrt{m} \left|\nabla^2f\right|,
$$
since for every $x\in M$ the quantity $\left|\nabla^2f(x)\right|$ is nothing but the Hilbert-Schmidt norm of the symmetric sesquilinear form $\nabla^2f(x)$.
\end{Notation}

Next, we record a Riemannian variant of Meyers-Serrin\rq{}s theorem:

\begin{Proposition}\label{meyers2} Let $\nabla$ be a smooth (not necessarily metric) covariant derivative on the smooth metric $\IK$-vector bundle $E\to M$, and let $s\in \IN$, $q\in [1,\infty)$. Then for any $f\in\Gamma_{ W^{s,q}_{\nabla}}(M,E)$ there is a sequence
$$
(f_n)\subset \Gamma_{C^{\infty}}(M,E)\cap \Gamma_{ W^{s,q}_{\nabla}}(M,E),
$$
which can be chosen in $\Gamma_{C^{\infty}_{\c}}(M,E)$ if $f$ is compactly supported, such that
\begin{align*}
&\left|f_n(x)\right|\leq\left\|f\right\|_{\infty}\in [0,\infty]\>\>\text{ for all $x\in M$, $n\in\IN_{\geq 0}$},\\
&\left\|f_n-f\right\|_{ W^{s,q}_{\nabla}}\to 0 \text{ as $n\to\infty$.}
\end{align*}
The same statements hold, if we make the replacement
$$
(\Gamma_{ W^{s,q}_{\nabla}}(M,E),\left\|\bullet\right\|_{ W^{s,q}_{\nabla}})\>\leadsto\> (\Gamma_{\widetilde{ W}^{2s,q}_{\nabla}}(M,E),\left\|\bullet\right\|_{\widetilde{ W}^{2s,q}_{\nabla}}).
$$
\end{Proposition}

\begin{proof} Concerning the case $(\Gamma_{ W^{s,q}_{\nabla}}(M,E),\left\|\bullet\right\|_{\nabla,s,q})$, once one has established the inclusion
$$
\Gamma_{ W^{s,q}_{\nabla}}(M,E)\subset \Gamma_{ W^{s,q}_{\loc}}(M,E),
$$
the statement follows directly from Theorem \ref{meyers}. To see the above inclusion, one can apply the following fact inductively: Given a smooth metric $\IK$-vector bundle $F\to M$ of rank $l$ with a smooth covariant derivative $\nabla_F$, one has the implication
\begin{align*}
&\left[f\in \Gamma_{ L^{q}_{\loc}}(M,E), \nabla_F f\in  \Gamma_{ L^{q}_{\loc}}(M,E\otimes \IT^* M)\right]\\
\Rightarrow &\>f\in \Gamma_{ W^{1,q}_{\loc}}(M,E),
\end{align*}
which can be easily seen by writing $\nabla_F|_U=\Id|_U +\alpha$ for some 
$$
\alpha\in \mathrm{Mat}\big(\Omega^1_{\ICC_{\IK}}(U); l \times l\big),
$$
in each chart $U\subset M$ in which $F\to M$ admits a smooth frame. \\
For the case of $\Gamma_{\widetilde{ W}^{2s,q}_{\nabla}}(M,E)$, since $\nabla^{\dagger}\nabla$ is elliptic, it follows from local elliptic regularity (\ref{localb}) that
$$
\Gamma_{\widetilde{ W}^{2s,q}_{\nabla}}(M,E)\subset \Gamma_{ W^{2s-1,q}_{\loc}}(M,E),
$$
and again the claim follows from Theorem \ref{meyers}.
\end{proof}

We immediately get the following important result:

\begin{Corollary}\label{compp} In the context of Proposition \ref{meyers2}, one has
$$
\Gamma_{ W^{s,q}_{\nabla,\c}}(M,E)\subset\Gamma_{ W^{s,q}_{\nabla,0}}(M,E),\>
 \Gamma_{\widetilde{ W}^{2s,q}_{\nabla,\c}}(M,E)\subset \Gamma_{\widetilde{ W}^{2s,q}_{\nabla,0}}(M,E).
$$
\end{Corollary}

We continue with some special features of first order Sobolev spaces. By definition, every element of $ W^{s,q}_0(M)$ can be approximated in the $\left\|\bullet\right\|_{ W^{s,q}}$-norm by smooth compactly supported functions. A well-known important refinement of this fact is that for $s=1$ this approximation property is positivity preserving in the following sense:

\begin{Lemma}\label{bett1} Let $l\in [1,\infty)$. For every $0\leq f\in  W^{1,l}_0(M)$ there exists a sequence $0\leq f_n\in \ICC_{\c}(M)$, $n\in\IN$, with $\left\|f_n-f\right\|_{ W^{1,l}}\to 0$ as $n\to\infty$.
\end{Lemma}

\begin{proof} For $l=2$ this is precisely the statement of Lemma 5.4 in \cite{gri}. The same proof applies to all $l\in [1,\infty)$. 
\end{proof}
 
The following notation will be useful in the sequel:
\begin{Notation} For any section $f$ of $E\to M$, the section
$\mathrm{sign}(f)\in\Gamma_{ L^{\infty}}(M,E)$ is defined by 
$$
\mathrm{sign}(f)(x):=\begin{cases}&\f{f(x)}{|f(x)|}, \>\text{ if } \>f(x)\ne 0\\
&0, \>\text{ else.}\end{cases}
$$
\end{Notation}

An important regularity result, which also makes sense and holds in the vector bundle case, is that each $ W^{1,l}_*$ class is stable under taking fiberwise norms:

\begin{Lemma}\label{bett2} Let $\nabla$ be a smooth metric covariant derivative on the smooth metric $\IK$-vector bundle $E\to M$, and let $l\in (1,\infty)$. Then the following statements hold:\\
a) For every $f\in \Gamma_{ W^{1,l}_{\nabla}}(M,E)$ one has $|f|\in W^{1,l}(M)$ with\footnote{Note that by definition one has $\left\||f|\right\|_{l}=\left\|f\right\|_{l}$.}
\begin{align}\label{absch}
\left\||f|\right\|_{ W^{1,l}}\leq \left\|f\right\|_{ W^{1,l}_{\nabla}} .
\end{align}
b) For every $f\in \Gamma_{ W^{1,l}_{\nabla,0}}(M,E)$ one has $|f|\in W^{1,l}_0(M)$ with (\ref{absch}).
\end{Lemma}

\begin{proof} The following is an $\IL^l$-variant of the corresponding $\IL^2$-proof from \cite{bei}: Without loss of generality we will consider the complex case $\IK=\IC$. Before we come to the proof of the actual statements, we first prove an auxiliary result for smooth sections that will also be useful in a different context later on:\\
\emph{Claim 1: For all $\psi\in \Gamma_{ W^{1,l}_{\nabla}\cap\ICC}(M,E)$ one has\footnote{This claim also holds for $l=1,\infty$.} $|\psi|\in W^{1,l}(M)$ with}
\begin{align}\label{absch3}
\left\||\psi|\right\|_{ W^{1,l}}\leq \left\|\psi\right\|_{ W^{1,l}_{\nabla}} .
\end{align}
Proof of claim 1: For every $s>0$ we define a smooth function on $M$ by $|\psi|_{s}:=\sqrt{s^2+|\psi|^2}>0$. We pick an arbitrary open subset $U\subset M$ that admits a local orthonormal frame $e_1,\dots,e_m\in  \mathscr{X}_{C^{\infty}} (U)$. Then, on $U$, using the chain rule and the fact that $\nabla$ is metric, we find
\begin{align}\label{prev}
\Id |\psi|_{s}= (2|\psi|_s)^{-1}\Id |\psi|^2=|\psi|_s^{-1}\Re(\nabla_{\bullet}\psi,\psi),
\end{align}
where $\nabla_{\bullet}\psi$ denotes the smooth $E$-valued $1$-form on $M$ which to a smooth vector field $A$ on $M$ assigns the section $\nabla_{A}\psi$. Now let $\alpha\in\Omega^1_{\ICC_{\c}}(M)$ be an arbitrary smooth compactly supported $1$-form. We can calculate as follows,
\begin{align*}
&\int |\psi| \Id^{\dagger}\alpha  \ \Id\mu=\lim_{s\to 0+}\int (\Id |\psi|_s,\alpha ) \Id\mu=\lim_{s\to 0+}\int \big(\Re(\nabla\psi,|\psi|_s^{-1}\psi),\alpha \big) \Id\mu\\
&=\int \big(\Re(\nabla_{\bullet}\psi,\mathrm{sign}(\psi)),\alpha \big) \Id\mu,
\end{align*}
where we have used dominated convergence and integration by parts for the first equality, (\ref{prev}) for the second, and dominated convergence once more for the last equality. Thus we have
$$
\Id |\psi|=\Re(\nabla_{\bullet}\psi,\mathrm{sign}(\psi)), 
$$
and therefore a use of
$$
|\Re(\nabla_{\bullet}\psi,\mathrm{sign}(\psi))|\leq| \nabla\psi|\in \IL^l(M)
$$
implies 
$$
|\Id |\psi||=|\Re(\nabla_{\bullet}\psi,\mathrm{sign}(\psi))|\leq | \nabla\psi|\in \IL^l(M),
$$
which proves claim 1.\\
Let us now come to the actual statements:\\
a) By the Riemannian Meyers-Serrin theorem (cf. Proposition \ref{meyers2}), we can pick a sequence of smooth sections $(f_n)\subset\Gamma_{ W^{1,l}_{\nabla}\cap\ICC}(M,E)$ such that 
$$
\lim_n\left\|f-f_n\right\|_{ W^{1,l}_{\nabla}}=\lim_n\big(\left\|f-f_n\right\|_{l}+\left\|\nabla f-\nabla f_n\right\|_{l}\big)= 0.
$$
In particular, $\left\|\nabla f-\nabla f_n\right\|_{l}\to 0$, and so 
$$
C:=\sup_n\left\|\nabla f_n\right\|_{l}<\infty.
$$
Then claim 1 implies 
$$
\sup_n\left\|\Id |f_n|\right\|_{l}\leq C.
$$
Let $l^*\in  (1,\infty)$ be defined by $1/l^*+1/l=1$. It follows from the boundedness of the sequence $\left\|\Id |f_n|\right\|_{l}$ and Banach-Alaoglu\rq{}s theorem that there exists a subsequence $F_n$ of $f_n$ and a form $\beta\in \Omega^1_{\IL^{l}}(M)$, such that  
\begin{align}\label{aspo}
\int (\beta,\theta)\Id\mu=\lim_n \int (\Id |F_n|,\theta)\Id\mu\quad\text{ for all $\theta\in \Omega^1_{\IL^{l^*}}(M)$.}
\end{align}
It follows that for all $\theta\in \Omega^1_{\ICC_{\c}}(M)$ one has 
$$
\int |f| \Id^{\dagger} \theta \Id\mu=\lim_n \int (\Id |F_n| , \theta )\Id\mu=\int (\beta,\theta)\Id\mu,
$$
where we have used $\lim_n\left\|f-F_n\right\|_{l}=0$ and thus 
$$
\int |f| \Id^{\dagger} \theta \Id\mu=\lim_n \int  |F_n| \Id^{\dagger} \theta \Id\mu
$$
by Hölder\rq{}s inequality, and where we have integrated by parts (Lemma \ref{int}), and finally (\ref{aspo}). This entails $\Id|f|=\beta\in \Omega^1_{\IL^{l}}(M)$, and so $|f|\in W^{1,l}(M)$. In order to prove the estimate (\ref{absch}), we can do the following estimate: 
\begin{align*}
&\left\|\Id |f|\right\|_l=\left\|\beta\right\|_l=\sup_{\theta\in \Omega^1_{\IL^{l^*}}(M), \left\|\theta\right\|_{l^*}\leq 1}\left|\int (\beta,\theta)\Id\mu\right|\\
&= \sup_{\theta\in \Omega^1_{\IL^{l^*}}(M), \left\|\theta\right\|_{l^*}\leq 1}\lim_n\left|\int (\Id |F_n|,\theta)\Id\mu\right|\leq  \lim_n\left\|\Id |F_n|\right\|_l\leq \lim_n\left\|\nabla F_n\right\|_l\\
&=\left\|\nabla f\right\|_l,
\end{align*}
where we have used Hölder's inequality and the above claim 1. This completes the proof of a).\\
b) Once we have established $|f|\in W^{1,l}_0(M)$, the asserted estimate (\ref{absch}) follows immediately from a). To see the former, we first remark that by definition we can pick a sequence $(f_n)\subset\Gamma_{\ICC_{\c}}(M,E)$ such that 
$$
\lim_n\left\|f-f_n\right\|_{ W^{1,l}_{\nabla}}=\lim_n\big(\left\|f-f_n\right\|_{l}+\left\|\nabla f-\nabla f_n\right\|_{l}\big)= 0.
$$
Then precisely as in the proof of part a), we can deduce the existence of a subsequence $F_n$ of $f_n$ and a form $\beta\in \Omega^1_{\IL^{l}}(M)$, such that  
\begin{align}\label{aspo2}
\int (\beta,\theta)\Id\mu=\lim_n \int (\Id |F_n|,\theta)\Id\mu\quad\text{ for all $\theta\in \Omega^1_{\IL^{l^*}}(M)$.}
\end{align} 
By Corollary \ref{compp} and Claim 1, we have 
$$
|F_n|\in W^{1,l}_{\c}(M)\subset  W^{1,l}_{0}(M)=\dom(\Id^{(l)}_{\min}).
$$
Given an arbitrary 
$$
\theta\in \dom\left(   (\Id^{\dagger})^{(l^*)}_{\max}   \right)\subset\Omega^1_{\IL^{l^*}}(M),
$$
we may now integrate by parts to deduce
\begin{align*}
&\int (|f|,\Id^{\dagger}\theta)\Id\mu=\lim_{n\to\infty}\int (F_n,\Id^{\dagger}\theta)\Id\mu=\lim_{n\to\infty}\int (\Id |F_n|,\theta)\Id\mu=\int (\beta,\theta)\Id\mu,\\
&\text{ so  $\Id |f|=\beta$ and}\\
&\left|\int (\beta,\theta)\Id\mu\right|\leq \left\|\beta\right\|_l\left\|\theta\right\|_{l^*},
\end{align*}
which implies 
$$
|f|\in \dom\left( \big( (\Id^{\dagger})^{(l^*)}_{\max}\big)^*\right)=\dom\left(\Id^{(l)}_{\min}\right)= W^{1,l}_{0}(M),
$$
where we have used Lemma \ref{minmax}. This completes the proof.
\end{proof}

Next, we record (one of many) Sobolev-Leibniz rules:

\begin{Lemma}\label{proff} Let $\nabla$ be a smooth (not necessarily metric) covariant derivative on the smooth metric $\IK$-vector bundle $E\to M$, and let $l\in [1,\infty)$. Assume further that $\psi\in \Gamma_{ W^{1,l}_{\nabla,0}}(M,E)$, and that $h:M\to\IC$ is bounded and Lipschitz continuous (with respect to the Riemannian distance), with a Lipschitz constant $\leq C$. Then one has $\Id h\in\Omega^1_{\IL^{\infty}}(M)$,  with $\>|\Id h|\leq C$ $\mu$-a.e., and $h\psi\in \Gamma_{ W^{1,l}_{\nabla,0}}(M,E)$ with
\begin{align}\label{pro}
\nabla (h\psi)=\Id h\otimes \psi+h\nabla\psi\>\text{ $\mu$-a.e.}
\end{align}
\end{Lemma}

\begin{proof} Clearly, for any Lipschitz function $h$ on $M$, the weak derivative $\Id h$ is a ($\mu$-essentially) bounded $1$-form. This follows readily from Rademacher\rq{}s Theorem (cf. Theorem 11.3 in \cite{gri}).\\
We prove the asserted regularity together with the Leibniz rule in three steps:\\
\emph{Step 1:} Let us first assume that $\psi\in \Gamma_{\ICC_{\c}}(M,E)$ and that $h$ is Lipschitz with a compact support. Then $\Id h$ is bounded with a compact support and thus in $ W^{1,l}(M)$. Furthermore, Corollary \ref{compp} even entails that $h\in  W^{1,l}_0(M)$. Pick a sequence $(h_n)\in\ICC_{\c}(M)$ with $h_n\to h$ in $ W^{1,l}(M)$. Then we have
\begin{align}
\label{cvb1}\left\|h_n\psi- h\psi\right\|_l\leq \left\|\psi\right\|_{\infty}\left\|h_n- h\right\|_l\to 0 ,
\end{align}
and
\begin{align}
\label{cvb2} &\left\|\nabla(h_n\psi)-(\Id h\otimes \psi+h\nabla\psi) \right\|_l\\\nn
&=\left\|\Id h_n\otimes \psi+h_n\nabla \psi-\Id h\otimes \psi-h\nabla\psi \right\|_l\\
\nn&\leq\left\|\psi\right\|_{\infty} \left\|\Id h_n-\Id h \right\|_l+\left\|\nabla\psi\right\|_{\infty}\left\|h_n- h \right\|_l
\to 0. 
\end{align} 
In particular, $h_n\psi$ is a Cauchy sequence in $\Gamma_{ W^{1,l}_{\nabla,0}}(M,E)$ which by (\ref{cvb1}) necessarily converges in $\Gamma_{ W^{1,l}_{\nabla,0}}(M,E)$ to $h\psi$. This proves 
$$
h\psi\in\Gamma_{ W^{1,l}_{\nabla,0}}(M,E),
$$
where (\ref{pro}) is implied by (\ref{cvb2}).\\
\emph{Step 2:} Let $h$ be bounded and Lipschitz and $\psi\in \Gamma_{\ICC_{\c}}(M,E)$. Then we obtain (\ref{pro}) on every open relatively compact subset $U\subset M$ by applying the previous case to $\widetilde{h}=\phi h$ and $\psi$, where $\phi\in \ICC_{\c}(M)$ is such that $\phi=1$ on $U$. (\ref{pro}) then implies $h\psi\in\Gamma_{ W^{1,l}_{\nabla}}(M,E)$, where it is used that $h$ and its differential are bounded. By Corollary \ref{compp}, it follows that $h\psi\in\Gamma_{ W^{1,l}_{\nabla,0}}(M,E)$. \\
\emph{Step 3:} In the general case, that is $h$ is bounded and Lipschitz and $\psi\in\Gamma_{ W^{1,l}_{\nabla,0}}(M,E)$, we pick a sequence $(\psi_n)\subset \Gamma_{\ICC_{\c}}(M,E)$ such that $\psi_n\to \psi$ in $\Gamma_{ W^{1,l}_{\nabla}}(M,E)$. Then by the previous case we have $h\psi_n \in\Gamma_{ W^{1,l}_{\nabla,0}}(M,E)$ and
\begin{align*}
&\left\|h\psi_n- h\psi\right\|_l\to 0,\\
& \left\|\nabla(h\psi_n)-(\Id h\otimes \psi+h\nabla\psi) \right\|_l  \\
&=\left\|\Id h_n\otimes \psi+h_n\nabla \psi-\Id h\otimes \psi-h\nabla\psi \right\|_l\to 0
\end{align*} 
analogously to step 1, since now $h$ and $\Id h$ are bounded, and (again as in step 1) we arrive at $h\psi\in\Gamma_{ W^{1,l}_{\nabla,0}}(M,E)$ with (\ref{pro}).
\end{proof}

A question which is much more subtle than the generally valid Proposition \ref{meyers2} is the denseness of $\Gamma_{C^{\infty}_{\c}}(M,E)$ in the corresponding Sobolev spaces, in other words, whether or not one has 
$$
\Gamma_{ W^{s,q}_{\nabla,0}}(M,E)=\Gamma_{ W^{s,q}_{\nabla}}(M,E)\>\text{ and/or }\>
\Gamma_{\widetilde{ W}^{2s,q}_{\nabla,0}}(M,E)=\Gamma_{\widetilde{ W}^{2s,q}_{\nabla}}(M,E).
$$
This question seems to depend heavily on $\nabla^{\IT M}$ and on $\nabla$ in general. The following results, however, only require geodesic completeness:

\begin{Proposition}\label{dgs} Let $M$ be geodesically complete, and let $\nabla$ be a smooth (not necessarily metric) covariant derivative on the smooth metric $\IK$-vector bundle $E\to M$.\\
a) For all $q\in [1,\infty)$ one has 
$$
\Gamma_{ W^{1,q}_{\nabla,0}}(M,E)=\Gamma_{ W^{1,q}_{\nabla}}(M,E).
$$
b) One has
$$
\Gamma_{\widetilde{ W}^{2,2}_{\nabla,0}}(M,E)=\Gamma_{\widetilde{ W}^{2,2}_{\nabla}}(M,E).
$$
\end{Proposition}

\begin{proof} a) In view of Theorem \ref{meyers2}, it is sufficient to prove that for any \emph{smooth} $f\in \Gamma_{ W^{s,q}_{\nabla}}(M,E)$ there is a sequence $(f_n)\subset \Gamma_{\ICC_{\c}}(M,E)$ with $\left\|f_n-f\right\|_{\nabla,s,q}\to 0$. To this end, we define $f_n:=\chi_nf$, with $(\chi_n)$ a sequence of first order cut-off functions as in Theorem \ref{first}. Then the claim follows easily from dominated convergence, using the Leibniz rule (\ref{leib})
$$
\nabla f_n=  \chi_n \nabla f+ \Id \chi_n\otimes f. 
$$
b) In this case, one can use a well-known Hilbert space argument \cite{Br}: The claim is proved, once we can show that the operator $T$ given by $\nabla^{\dagger}\nabla$ with domain of definition $\Gamma_{C^{\infty}_{\c}}(M,E)$ is essentially self-adjoint in $\Gamma_{\IL^2}(M,E)$. For this, it is sufficient to prove (cf. appendix, Theorem \ref{defect}) that $\mathrm{Ker}((T+1)^*)=\{0\}$. Since $\overline{T+1}$ is precisely $(\nabla^{\dagger}\nabla+1)_{\min}$, it follows from $(T+1)^*=\overline{T+1}^*$ and Lemma \ref{minmax} that
$$
(T+1)^*=(\nabla^{\dagger}\nabla+1)_{\max}.
$$
Now let  
$$
f\in \mathrm{Ker}((T+1)^*).
$$
By the above, this is equivalent to $f\in \Gamma_{\IL^2}(M,E)$ and $\nabla^{\dagger}\nabla f=-f$, in particular, $f$ is smooth by elliptic regularity. We pick again a sequence $(\chi_n)$ of first order cut-off functions. Then by the Leibniz rule we have
\begin{align*}
&(\nabla (\chi_n f),\nabla (\chi_n f))\\
&=(\nabla  f,\chi_n\Id\chi_n\otimes f)+(\nabla f,\chi_n^2\nabla f)+|\Id\chi_n \otimes f|^2+(\Id\chi_n\otimes f,\chi_n\nabla f),
\end{align*}
which, using
$$
(\nabla f,\nabla (\chi_n^2f))=(\nabla f,\chi_n^2\nabla f)+2(\nabla f, \chi_n\Id \chi_n\otimes f),
$$
implies
\begin{align*}
&|\nabla (\chi_n f)|^2=(\nabla (\chi_n f),\nabla (\chi_n f))\\
&=(\nabla f,\nabla (\chi_n^2f))+|\Id\chi_n \otimes f|^2-(\nabla f, \chi_n\Id \chi_n\otimes f)+(\Id\chi_n\otimes f,\chi_n\nabla f)\\
&=(\nabla f,\nabla (\chi_n^2f))+|\Id\chi_n \otimes f|^2-(\nabla f, \chi_n\Id \chi_n\otimes f)+(\chi_n\Id\chi_n\otimes f,\nabla f).
\end{align*}
This in turn implies (after adding the complex conjugate of the formula to itself) 
$$
2|\nabla (\chi_n f)|^2=2\Re(\nabla f,\nabla (\chi_n^2f) )+2|\Id\chi_n \otimes f|^2.
$$
Integrating and then integrating by parts in the last equality, we get 
$$
\int|\nabla (\chi_n f)|^2\Id\mu=\Re\int(\chi_n\nabla^{\dagger}\nabla f,\chi_nf)\Id\mu +\int|\Id\chi_n \otimes f|^2\Id\mu.
$$
Using $\nabla^{\dagger}\nabla f=-f$, we see
$$
\int|\chi_n|^2|f|^2\Id\mu \leq\int|\Id\chi_n \otimes f|^2\Id\mu,
$$
which implies $\int|f|^2\Id\mu=0$ and thus $f=0$ by dominated convergence, using the properties of $(\chi_n)$.
\end{proof}

In the scalar case, one can use curvature bounds to get the following results, whose part b) stems from\footnote{The classical reference for this type of density results is E. Hebey\rq{}s book \cite{hebe}, which however does not contain the results from Proposition \ref{aofyy}.} \cite{bashra}:

\begin{Proposition}\label{aofyy} a) If $M$ is geodesically complete with $|\mathrm{Sec}|\leq C$ for some constant $C\geq 0$, then for all $q\in [1,\infty)$ one has
$$
W^{2,q}_0(M)=W^{2,q}(M).
$$
b) Assume that $M$ is geodesically complete with $\mathrm{Ric}\geq -C$ for some constant $C\geq 0$. Then one has
$$
 W^{2,2}_0(M)= W^{2,2}(M).
$$
\end{Proposition}

\begin{proof} a) This follows straightforwardly from Theorem \ref{meyers2}, by taking a sequence $(\chi_n)$ of Hessian cut-off functions as in Theorem \ref{first}, and using the product rules
\begin{align*}
&d (\chi_nf)= fd\chi_n+ \chi_n df,\\
&\nabla^{2}(\chi_n f)= \chi_n\nabla^{2} f+f\nabla^{2} \chi_n +\Id f\otimes \Id\chi_n+\Id\chi_n\otimes \Id f ,\> f\in\ICC(M).
\end{align*}
b) Let $f$ be a smooth compactly supported function on $M$. Integrating Bochner\rq{}s identity (\ref{bochner}) and integrating by parts several times, we have 
\begin{align*}
&\int | f|^2\Id\mu+\int |\nabla f|^2\Id\mu+\int |\nabla^2 f|^2\Id\mu\\
&=\int | f|^2\Id\mu+\int |\nabla f|^2\Id\mu+(1/2)\int \Id^{\dagger} (\Id |\Id f|^2)\cdot 1\Id\mu+\int|\Delta f|^2\Id\mu\\
&\quad-\int\mathrm{Ric}_{\sharp}(\Id f,\Id f)\Id\mu\\
&=\int | f|^2\Id\mu-\int \overline{\Delta f} f\Id\mu+\int|\Delta f|^2\Id\mu-\int\mathrm{Ric}_{\sharp,\IC}(\Id f,\Id f)\Id\mu,
\end{align*}
which, using $\mathrm{Ric}\geq -C$ (and therefore $\mathrm{Ric}_{\sharp}\geq -C$) and applying the elementary inequality $ab\leq a^2+b^2$ to $a=|\Delta f|$, $b= |f|$, is 
$$
\leq C\rq{}\int | f|^2\Id\mu+  C\rq{}\int|\Delta f|^2\Id\mu.
$$

On the other hand, we have the trivial inequality 
\begin{align}\label{eigen}
|\Delta h|\leq \sqrt{m}|\nabla^2 h|\>\text{ for all $h\in\ICC(M)$,}
\end{align}
so that 
\begin{align*}
\int |f|^2\Id\mu+\int |\Delta f|^2\Id\mu&\leq \int |f|^2\Id\mu+m\int  |\nabla^2 f|^2\Id\mu\\
&\leq  \int |f|^2\Id\mu+m\int  |\nabla^2 f|^2\Id\mu + \int |\nabla f|^2\Id\mu.
\end{align*}
We have thus shown the equivalence of norms $\|\bullet\|_{ W^{2,2}}\sim \|\bullet\|_{\tilde{ W}^{2,2}}$ on $\ICC_{\c}(M)$. Thus, we have
$$
 W^{2,2}_0(M)=\widetilde{ W}^{2,2}_0(M),
$$
so that the geodesic completeness together with Proposition \ref{dgs} b) give the last equality in
\begin{align}
 W^{2,2}(M)\supset W^{2,2}_0(M)=\widetilde{ W}^{2,2}_0(M)=\widetilde{ W}^{2,2}(M).
\end{align}
Finally, Theorem \ref{meyers2} in combination with (\ref{eigen}) also implies 
$$
 W^{2,2}(M)\subset\widetilde{ W}^{2,2}(M),
$$
and this completes the proof.
\end{proof}

The crucial part in the proof of Proposition \ref{aofyy} b) was to establish the existence of constants $C_j>0$ such that 
$$
\left\|\nabla^{2}f\right\|_2\leq C_1\left\|\Delta f\right\|_2+C_2\left\| f\right\|_2\quad\text{ for all $f\in C^{\infty}_c(M)$,}
$$
which relied on a lower bound of the Ricci curvature. The $L^q$-version, $q\in (1,\infty)$, of the last inequality is called the \emph{$L^q$-Calderon-Zygmund inequality} in \cite{G54,pigola}, and its importance in the context of density problems on Sobolev spaces has first been realized in \cite{G54}. The validity of the $L^q$-Calderon-Zygmund inequality depends very sensitively on the geometry on $M$. For example, even the $L^2$-Calderon-Zygmund inequality is in general false without a lower bound of the Ricci curvature \cite{pigola}. On the other hand, the $L^q$-Calderon-Zygmund inequality on $M$ holds for all $q\in (1,\infty)$, if $M$ has a bounded Ricci curvature and a positive injectivity radius. Surprisingly, using covariant Riesz-transform techniques, it can also be shown that the $L^q$-Calderon-Zygmund inequality on $M$ holds for all $q\in (1,2]$ with a $C^1$-control on the curvature and a generalized volume doubling condition. (In particular, one does not have to impose any control on the injectivity radius for small $p$.) We refer the interested reader to \cite{pigola} for these and other facts on the $L^q$-Calderon-Zygmund inequality.

\section{The Friedrichs realization of $\nabla^{\dagger}\nabla/2$}

We begin this section with some notation as well:

\begin{Notation} Given a smooth (not necessarily metric) covariant derivative $\nabla$ on the smooth metric $\IC$-vector bundle $E\to M$, we will denote by $\H^{\nabla} \geq 0$ the Friedrichs realization (cf. appendix, Example \ref{friedrichs}) of $\nabla^{\dagger}\nabla/2$, and with $\Q^{\nabla}\geq 0$ the closed densely defined symmetric sesquilinear form corresponding to $\H^{\nabla}$ (cf. appendix, Theorem \ref{kaq1}). 
\end{Notation}

The normalization $\nabla^{\dagger}\nabla/2$ is common in probability theory\footnote{In probability theory, this convention has the advantage that every Brownian motion $(W_t)_{t\geq 0}$ in the Euclidean $\IR^1$ (which by definition is a $(-1/2)\Delta_{\IR^1}$-diffusion process) has  the quadratic covariation $[W_t,W_t]=t$, almost surely for all $t\geq 0$.}. Note that the operator $\nabla^{\dagger}\nabla$ is elliptic, regardless of the fact whether $\nabla$ is metric or not\footnote{In case $\nabla$ is metric, wesaw this in (\ref{sym}); in the general case, one can pick an endomorphism-valued $1$-form $\alpha$ such that $\nabla_1:=\nabla+\alpha$ is metric, but clearly $\nabla_1^{\dagger}\nabla_1$ and $\nabla^{\dagger}\nabla$ have the same symbol.}. The operators $\H^{\nabla}$ are of a fundamental importance for us: They serve as our underlying \emph{free} operators. The semigroups corresponding to perturbations of $\H^{\nabla}$ by singular potentials are the central objects of this work. We record the following explicit respresentation of $\H^{\nabla}$:

\begin{Proposition}\label{fll} Let $\nabla$ be a smooth covariant derivative on the smooth metric $\IC$-bector bundle $E\to M$. Then one has
\begin{align*}
&\dom(\H^{\nabla})=\Gamma_{\tilde{ W}^{2,2}_{\nabla}}(M,E)\cap \Gamma_{ W^{1,2}_{\nabla,0}}(M,E),\\
&\H^{\nabla}f=(1/2)\nabla^{\dagger}\nabla f,\text{ and}\\
&\dom(\Q^{\nabla})=\dom(\sqrt{\H^{\nabla}})= \Gamma_{ W^{1,2}_{\nabla,0}}(M,E),\\
&\Q^{\nabla}(f_1,f_2)=\langle \sqrt{\H^{\nabla}}f_1,\sqrt{\H^{\nabla}}f_2\rangle=(1/2)\int _M(\nabla f_1 ,\nabla f_2)  \Id \mu.
\end{align*}
\end{Proposition}

\begin{proof} Using that for all smooth compactly supported sections $f$ one has
$\left\langle \nabla^{\dagger}\nabla  f,f\right\rangle= \left\langle \nabla  f,\nabla f\right\rangle$, and furthermore
$$
\dom((\nabla^{\dagger}\nabla|_{\Gamma_{\ICC_{\c}}(M,E)})^*)=\dom((\nabla^{\dagger}\nabla)_{\max})=\Gamma_{\tilde{ W}^{2,2}_{\nabla}}(M,E),
$$
all assertions follow easily from abstract functional analytic facts (cf. appendix, Theorem \ref{kaq1} and Theorem \ref{kaq2}). 
\end{proof}

More specifically, we will use the following conventions concerning the scalar Laplace-Beltrami operator:

\begin{Notation} We will write 
$$
\H:=\H^{\Id}\geq 0
$$
for the Friedrichs realization of $1/2$-times the scalar Laplacian $-\Delta=\Id^{\dagger}\Id$, and $\Q:=\Q^{\Id}\geq 0$ for its sesquilinear form (cf. appendix, Theorem \ref{kaq1}), where again the exterior differential is considered a covariant derivative. 
\end{Notation}

In this case, Proposition \ref{fll} boils down to 
\begin{align*}
\dom(\H)=\tilde{ W}^{2,2}(M)\cap  W^{1,2}_0(M),\>\H f=-(1/2)\Delta f,\end{align*}
and
\begin{align}\label{fll2}
&\dom(\Q)=\dom(\sqrt{\H})=  W^{1,2}_0(M),\\
&\Q(f_1,f_2)= \langle \sqrt{\H} f_1, \sqrt{\H} f_2 \rangle=(1/2)\int _M(\Id f_1 ,\Id f_2)  \Id \mu.
\end{align}

\chapter[The minimal heat kernel on $M$ ]{Some specific results for the minimal heat kernel}\label{scalar}

Let us begin with an important definition.

\begin{Definition} The heat kernel
$$
p(t,x,y):=\mathrm{e}^{-t \H}(x,y)
$$ 
of $\H$, in the sense of Theorem \ref{heat}, will be called the \emph{the minimal nonnegative heat kernel} on the Riemannian manifold $M$.
\end{Definition}

The reason for this name will become clear in a moment (cf. Theorem \ref{ddfh} below). We remark that the semigroup identity (\ref{chap}) for $p(t,x,y)$ is usually referred to as \emph{Chapman-Kolmogorov identity}, which now means nothing but that for all $t,s>0$, $x,y\in M$, one has
$$
p(t+s,x,y)=\int_M p(t,x,z)p(s,z,y)\Id\mu(z).
$$
In the rest of this chapter, we are going to collect some facts concerning $p(t,x,y)$. Whenever possible, we will refer to A. Grigor\rq{}yan's excellent book\footnote{We warn the reader that Grigor\rq{}yan's heat kernel $p(t,x,y)$ is the one corresponding to the heat semigroup defined by $-\Delta$, and not $-(1/2)\Delta$.} for results concerning $p(t,x,y)$. Although we will actually be concerned with the semigroups corresponding to covariant Schrödinger operators later on (that is, perturbations by potentials of covariant operators of the form $H^{\nabla}$), the semigroup $\mathrm{e}^{-t \H}$ plays nevertheless a special role in the covariant context as well: First of all, covariant Schrödinger semigroups can be controlled in a certain sense by usual scalar Schrödinger semigroups of the form $\mathrm{e}^{-t(H+w)}$, where $w:M\to\IR$ is an appropriate potential\footnote{This will be reflected by the Kato-Simon inequality later on.}. The scalar semigroups $\mathrm{e}^{-t(H+w)}$, on the other hand, can be ultimately controlled (for example using probabilistic methods) by $\mathrm{e}^{-t \H}$. In addition to this machinery, which reduces many problems of interest to $\mathrm{e}^{-t \H}$, there exist special methods (minimum principles, mean value inequalities,...) that allow a control of $\mathrm{e}^{-t \H}$ in a very direct way.\vspace{2mm}

We start by recalling that in the Euclidean case (that is, $M=\IR^m$ with its standard Euclidean Riemannian metric), the correspondig heat kernel is given by
$$
p_{\IR^m}(t,x,y)=(2\pi t)^{-m/2}\mathrm{e}^{\f{-|x-y|^2}{ 2t}}.
$$
Only very few other heat kernels are explicitly known, which is one of the main reasons why one is interested in abstract heat kernel estimates. Nevertheless, another class of Riemannian manifolds with a (more or less) explicitly given heat kernel is provided by hyperbolic spaces:

\begin{Example}\label{hyp}  For each $m\geq 2$, the hyperbolic space $\mathbb{H}^m$ is the uniquely determined\footnote{Note that the space is determined up to an isometry, that is, a smooth diffeomorphism that preserves the Riemannian metrics.} geodesically complete and simply connected Riemannian manifold whose sectional curvatures all equal $-1$. In this case, the heat kernel is a function $p(t,x,y)\equiv p(t,\varrho(x,y))$ only of $t$ and the geodesic distance $r=\varrho(x,y)$, and one finds the following formulae for $r>0$: If $m=2n+1$, then 
$$
p_{\mathbb{H}^{2n+1}}(t,r)= \f{(-1)^n}{(2\pi)^n (2\pi t)^{1/2}} \left(\f{1}{\mathrm{sinh}(r)}\partial_r\right)^n\mathrm{e}^{-\f{n^2 t}{ 2} -\f{r^2}{2t}},
$$ 
while if $m=2n$, then
\begin{align*}
p_{\mathbb{H}^{2n}}(t,r)=&\f{(-1)^n\sqrt{2}}{(2\pi)^n (2\pi t)^{3/2}}\mathrm{e}^{-\f{(2n+1)^2t}{8}} \left(\f{1}{\mathrm{sinh}(r)}\partial_r\right)^n\\
&\times\int^{\infty}_{r}\f{s\mathrm{e}^{-\f{s^2}{2t}}}{\left(\mathrm{cosh}(s)-\mathrm{cosh}(r)\right)^{1/2}}\Id s.
\end{align*}
These identities can be found, for example, in the paper \cite{hyper} by A. Grigor'yan and M. Noguchi. There, these equations have been derived by transforming the wave operator to the heat operator, noting that on hyperbolic spaces the \lq\lq{}wave kernel\rq\rq{} can be calculated using group theoretic methods. An analogous wave-to-heat transformation will also play an important  role for us in the context of the essential self-adjointness of covariant Schrödinger operators (cf. Section \ref{estt}). An important property of the hyperbolic spaces is that their bottom of the spectrum is strictly positive, namely
$$
\min \sigma(H_{\mathbb{H}^{m}})= (m-1)^2/8>0.
$$
This is shown on p. 319 in \cite{gri}.
\end{Example}

Let us return to the general case again: It is possible to give an alternative definition of $p(t,x,y)$, which we record for the sake of completeness (cf. Corollary 8.12 and Theorem 9.5 in \cite{gri}):

\begin{Theorem}\label{ddfh} One has $p(t,x,y)\geq 0$ for all $t>0$, $x,y\in M$, and for every fixed $y\in M$, $p(\bullet,\bullet,y)$ is a solution of 
\begin{align}\label{ddedw}
\frac{\partial}{\partial t} u =(1/2)\Delta u,\>\>\lim_{	t\to 0+} u(t,\bullet)=\delta_y,
\end{align}
where the initial value means as usual that
$$
\lim_{t\to 0+}\int u(t,x) \phi(x)\Id\mu(x)=\phi(y)\quad\text{ for all $\phi\in\ICC_{\c}(M)$.}
$$
In fact, $p(\bullet,\bullet,y)$ is the pointwise minimal nonnegative smooth solution of (\ref{ddedw}), that is, any other nonnegative smooth\footnote{In fact, by local parabolic regularity \cite{gri} every (weak) solution of (\ref{ddedw}) is atomatically smooth.} solution 
$$
u:(0,\infty)\times M\longrightarrow [0,\infty) 
$$
of (\ref{ddedw}) automatically satisfies $u(t,x)\geq p(t,x,y)$ for all $t>0$ and all $x\in M$. 
\end{Theorem}

This result directly implies the following important domain monotonicity:

\begin{Corollary}\label{monon} Let $U\subset M$ be an arbitrary connected open subset, let $H_U$ denote\footnote{In other words, $H_U$ is the Dirichlet realization of $(-1/2)\Delta$ in $U$.} $H$ defined with $M=U$ and the induced Riemannian metric $g|_U$, and let $p_U(t,x,y):=\mathrm{e}^{-t H_U}(x,y)$ be the correponding heat kernel. Then one has
$$
p_U(t,x,y)\leq p(t,x,y)\quad\text{ for all $(t,x,y)\in (0,\infty)\times U\times U$. }
$$
\end{Corollary}

We remark that, at least morally, the nonnegativity and the minimality properties of $p(t,x,y)$ from Theorem \ref{ddfh} correspond to the fact that, by definition, $p(t,x,y)$ is the heat kernel corresponding to the \emph{Friedrichs realization} of $(1/2)\Id^{\dagger}\Id$, which is the nonnegative self-adjoint extension of $(1/2)\Id^{\dagger}\Id$ which has the largest \lq\lq{}energy\rq\rq{}, in a sense that can be made precise (cf. appendix, Example \ref{friedrichs}). The ultimate reason behind all these results is that $H$ has the form domain $ W^{1,2}_0(M)$, and this space is stable under the operation $f\to \max(f,0)$ for real-valued $f$'s (cf. p. 126 in \cite{gri}).\vspace{2mm}

The next result that we would like to address relies on the following \emph{strong parabolic maximum principle} which is satisfied by $\frac{\partial}{\partial t} -(1/2)\Delta$: Namely, if $I\subset   \IR$ is an open interval, if $0\leq u\in C^2(I\times M)$ satisfies 
$$
\frac{\partial}{\partial t}u -(1/2)\Delta u\geq 0\quad\text{ in $I\times M$, }
$$
and if there exists a point $(t_0,x_0)\in I\times M$ with $u(t_0,x_0)=0$, then one has $u(t,x)=0$ for all $(t,x)\in I\times M$ with $t\leq t_0$. Let us remark that our standing assumption of $M$ being connected is again crucial for this result. As a simple consequence of the fact that $p(\bullet,\bullet,y)$ satisfies the initial value problem (\ref{ddedw}) and the strong parabolic maximum principle, one gets (cf. Corollary 8.12 in \cite{gri}):

\begin{Proposition}\label{saio} There holds the strict positivity
$$
p(t,x,y)>0\quad\text{ for all $(t,x,y)\in (0,\infty)\times M\times M$.}
$$
\end{Proposition}

Compared to the property $p(t,x,y)\geq 0$ (which does not need connectedness), the proof of the strict positivity from Proposition \ref{saio} is rather complicated. \\
Proposition \ref{saio} has an important consequence: It implies that $\mathrm{e}^{-t \H}$ is positivity improving for all $t>0$, that is, one has the implication 
$$
f\in L^2(M)\setminus\{0\},f\geq 0\>\text{$\mu$-a.e.}\> \Rightarrow \>\mathrm{e}^{-t \H}f>0\>\text{$\mu$-a.e.}
$$

This automatically extends to appropriate powers of the resolvent:

\begin{Remark}\label{resss} Let $S$ be any self-adjoint and semibounded operator in a complex Hilbert space $\IHH$, and let $\lambda\in\IC$ with $\Re\lambda <\min\sigma(S)$. Then for every $b>0$ one has the \emph{Laplace transformation formula}
\begin{align}
\label{lpo}
(S-\lambda)^{-b}=\f{1}{\Gamma(b)}\int^{\infty}_0 s^{b-1}\mathrm{e}^{\lambda s}\mathrm{e}^{-s S}\Id s,
\end{align}
where the integral is defined weakly (cf. appendix, Remark \ref{beispiele}). We remark that the definition of the integral in (\ref{lpo}) can also be interpreted in some \lq\lq{}strong\rq\rq{} sense, for example as an improper strong Riemann integral. However, the weak definition in combination with norm estimates will be sufficient for us in the sequel. In particular, taking Laplace transforms, it follows that the resolvent powers $(\H-\lambda)^{-b}$, where $b>0$ and $\lambda<0$, are also positivity improving. Indeed, this property is equivalent to the validity of the implication
$$
f_1,f_2\in L^2(M)\setminus\{0\},f_j\geq 0\>\text{$\mu$-a.e.}\> \Rightarrow \>\left\langle (\H-\lambda)^{-b}f_1,f_2\right\rangle>0\>\text{$\mu$-a.e.},
$$  
which is cleary implied by the positivity improvement property of $\mathrm{e}^{-t \H}$ and the Laplace transformation formula. 
\end{Remark}

The positivity improving property of the semigroup has a well-known spectral consequence, which is important for applications in quantum mechanics:

\begin{Corollary}\label{ddff} If $\lambda:=\min \sigma(\H)$ is an eigenvalue of $\H$, then $\lambda$ is simple and there is a unique eigenfunction $\psi$ of $\H$ corresponding to $\lambda$ which is strictly positive $\mu$-a.e. and satisfies $\left\|\psi\right\|_2=1$. 
\end{Corollary}

\begin{proof} This follows from a well-known (Perron-Frobenius-type) functional analytic fact about the generators of positivity improving semigroups on $\IL^2$-spaces (cf. Theorem XIII.44 in \cite{reed4}).
\end{proof}

We will see later on that the positivity improving property of $(\mathrm{e}^{-t \H})_{t>0}$, and thus the analogue of Corollary \ref{ddff}, remains true for certain semigroups of the form $\mathrm{e}^{-t (\H+w)}$, where $w:M\to \IR$ is such that its negative part admits some mild control.\vspace{2mm}

We continue with the following well-known $\IL^q$-results:

\begin{Theorem}\label{mart} a) For any $t>0$, $x\in M$ one has
\begin{align}\label{mar}
\int_M p(t,x,z)\Id\mu(z)\leq 1.
\end{align}
b) For any $q\in [1,\infty]$, $f\in \IL^q(M)$, the function
$$
(0,\infty)\times M\ni (t,x)\longmapsto \mathrm{e}^{-t H}f(x):=\int_M\mathrm{e}^{-t \H}(x,y)f(y)\Id\mu(y)\in \IC
$$
is well-defined and smooth, and for all $t>0$, $x\in M$ one has  
\begin{align}\label{ugh}
&\left\|\mathrm{e}^{-t H}f \right\|_q\leq \left\|f\right\|_q,\\
&\label{hett}
\frac{\partial}{\partial t} \mathrm{e}^{-t H}f(x)=(1/2)\Delta \mathrm{e}^{-t H}f(x).
\end{align}
If $q<\infty$, then one also has $\left\|\mathrm{e}^{-t H}f-f\right\|_q\to 0$ as $t\to 0+$.
%\emph{(iii)}  For some/any $T\in (0,\infty]$, and every $f\inC_{\mathrm{b}}(M)$ there is at most one smooth solution 
%$$
%u:  (0,T)\times M\longrightarrow \IC
%$$
%of the following Cauchy problem:
%\begin{align}
%&\f{\partial}{\partial t} u(t,x)= \f{1}{2}\Delta u(t,x)\\
%&\lim_{t\to 0+}\left\|u(t,\bullet)\mid_{K}- f\mid_{K} \right\|_{\infty}\>\text{ for all compact $K\subset M$.}
%\end{align}
%\emph{d)} For any fixed $y\in M$, the function $u:=p(\bullet,\bullet,y)$ is the pointwise minimal nonnegative smooth function which satisfies
%\begin{align}
%&\f{\partial}{\partial t} u(t,x)= \f{1}{2}\Delta u(t,x)\\
%&\lim_{t\to 0+}u(t,\bullet)= \delta_y\>\>\text{ in the sense of $\mathsf{D}\rq{}(M)$}.
%\end{align}
\end{Theorem}

\begin{proof} a) The inequality (\ref{mar}) is contained in Theorem 7.13 from \cite{gri}. Ultimately, this follows again from the stability of the form domain $ W^{1,2}_0(M)$ under the operation $f\to \max(f,0)$, if $f$ is real-valued.\\
b) Let us first show (\ref{ugh}): The bound (\ref{mar}) trivially implies (\ref{ugh}) for $q=1,\infty$. For the case $1< q<\infty$, we define a Borel sub-probability measure 
$$
\Id\mu_{t,x}(  y):=p(t,x,y)\Id\mu(y)\>\text{ on $M$.}
$$
Then with $q^*$ the dual Hölder exponent of $q$, one has
\begin{align*}
& \left\|\mathrm{e}^{-t H}f\right\|_q^q=\int_M \left|\int_Mp(t,x,y)f(y)\Id\mu(y)\right|^q \Id\mu(x)\\
&\leq  \int_M \left(\int_M 1\cdot |f(y)|\mu_{t,x}(\Id y)\right)^q \Id\mu(x)\\
&\leq  \int_M \left[\mu_{t,x}(M)^{1/q^*}\left(\int_M  |f(y)|^q\mu_{t,x}(\Id y)\right)^{1/q}\right]^q \Id\mu(x)\\
&\leq \int_M\int_M|f(y)|^q p(t,x,y) \Id\mu(y)\Id\mu(x)\leq \left\|f\right\|_{q}^q,
\end{align*}
where we have used Hölder\rq{}s inequality for $\mu_{t,x}(\Id y)$, Fubini and (\ref{mar}).\\
In order to prove the asserted smoothness and (\ref{hett}), we can assume $f\geq 0$ (otherwise write 
$$
f=f_1-f_2+\sqrt{-1}(f_3-f_4)\>\text{ with $f_j\geq 0$}
$$ 
and apply the result to each $f_j$). Then Theorem 7.15 in \cite{gri} implies the asserted smoothness with (\ref{hett}), once we can show that $(t,x)\mapsto \mathrm{e}^{-t H}f(x)$ is in $\IL^1_{\loc}((0,\infty)\times M)$ (a consequence of local parabolic regularity). But in view of (\ref{ugh}), for any $T_2>T_1>0$ and any compact $K\subset M$, we clearly have
\begin{align*}
&\int_K\int^{T_2}_{T_1} \mathrm{e}^{-t H}f(x) \Id t   \Id\mu(x)\leq \int^{T_2}_{T_1}\int_K (\mathrm{e}^{-t H}f(x)+1)^q \Id \mu(x) \Id t\\
&\leq 2^{q-1}\left\|f\right\|_q^q (T_2-T_1) +2^{q-1}\mu(K)(T_2-T_1)<\infty\end{align*}
in the case of $q<\infty$, and
$$
\int_K\int^{T_2}_{T_1} \mathrm{e}^{-t H}f(x) \Id t  \  \Id\mu(x)\leq \left\|f\right\|_{\infty}\mu(K)^{T_2}_{T_1}<\infty,
$$
in the case of $q=\infty$.\\
Finally, assume $q<\infty$. In order to see $\left\|\mathrm{e}^{-t H}f-f\right\|_q\to 0$, we can pick a sequence $(f_n)\subset \ICC_{\c}(M)$  with $\left\|f_n-f \right\|_q\to 0$ as $n\to\infty$. We have 
\begin{align*}
& \left\| \mathrm{e}^{-t H}f -f\right\|_q=  \left\| \mathrm{e}^{-t H} (f    - f_n)   +  f_n-f   +  \mathrm{e}^{-t H}f_n-f_n\right\|_q\nn\\
&\leq   2  \left\| f-f_n\right\|_q+ \left\| \mathrm{e}^{-t H}f_n-f_n\right\|_q\text{ for any $n$}, 
\end{align*}
where we used (\ref{ugh}), and so it remains to prove  $\left\| \mathrm{e}^{-t H}f_n-f_n\right\|_q\to 0$ as $t\to 0+$, for all $n$. The case $q=1$ has been established in Theorem 7.19 from \cite{gri}, so let us assume $1<q<\infty$. Then we can estimate
\begin{align*}
&\left\| \mathrm{e}^{-t H}f_n-f_n\right\|^q_q=\int_M |\mathrm{e}^{-t H}f_n-f_n | |\mathrm{e}^{-t H}f_n-f_n |^{q-1}\Id\mu\\
&\leq \left\| \mathrm{e}^{-t H}f_n-f_n\right\|^{q-1}_{\infty}\left\| \mathrm{e}^{-t H}f_n-f_n\right\|_{1}\leq   (  2\left\| f_n\right\|_{\infty}  )^{q-1}     \left\| \mathrm{e}^{-t H}f_n-f_n\right\|_{1},
\end{align*}
where we used (\ref{ugh}) again. Therefore, the claim follows from the case $q=1$.\end{proof}

\begin{Remark} While it is not true that $\left\|\mathrm{e}^{-t H}f-f\right\|_{\infty}\to 0$ as $t\to 0+$ for all $f\in \IL^{\infty}(M)$, the following local result for bounded \emph{continuous} functions, which follows from the initial value in (\ref{ddedw}) and a straightforward approximation argument (cf. Theorem 7.16 in \cite{gri}), is often useful: For all $f\in C_{b}(M)$ and all compact $K\subset M$, one has
\begin{align}\label{capaa}
\left\|1_{K}(\mathrm{e}^{-t H}f-f)\right\|_{\infty}\to 0\quad\text{  as $t\to 0+$. }
\end{align}
\end{Remark}

%let us provide the reader with a simple proof of this fact: \\
%In view of the strong convergence
%$$
%\mathrm{e}^{-t H}=\lim_{n\to \infty} (n/t)^{n}(H+n/t)^{-n}
%$$
%it is sufficient to prove that $(H+a)^{-1}$ is positivity preserving for all $a>0$, so let $0\leq f\in\IL^2(M)$ be given. Note first that $(H+a)^{-1}$ preserves reality (as $H$ arises from a complexification). With 
%$$
%h:=(H+a)^{-1}f\in \dom(H)\subset W^{1,2}_0(M),
%$$
%we have $0\leq h_-:=-\min\{h,0\}\in W^{1,2}_0(M)$ by the Sobolev chain rule, with $\Id h_-= -\Id h$ on $\{h<0\}$ and $\Id h = 0$ elsewhere. Now we can calculate 
%\begin{align*}
%&-a\left\|h_-\right\|^2=a\left\langle h,h_-\right\rangle=\left\langle f,h_-\right\rangle-\left\langle Hh,h_-\right\rangle\\
%&=\left\langle f,h_-\right\rangle+\int_{\{h<0\}}|\Id h(x)|^2 \Id \mu \geq 0,
%\end{align*}
%where the first equality is trivial from the definition of the negative part of a function, the second identity follows from multiplying $(H+a)h=f$ with $h_-$ and integrating, and the last equality follows from integrating by parts using the above formula for $\Id h_-$. We have thus shown $h_-=0$, which shows that $(H+a)^{-1}$, and as a consequence, $\mathrm{e}^{-t H}$ is positivity preserving.  \vspace{2mm}
We continue with an important consequence of Theorem \ref{mart}, namely that the following (partially localized) $\IL^{q_1}(M)\to \IL^{q_2}(M)$ bounds are valid \emph{on any Riemannian manifold:}

\begin{Theorem}\label{aa} a) For all $t\geq 0$, $q\in [1,\infty]$, one has $\left\|\mathrm{e}^{-t \H}\right\|_{q,q}\leq 1$, where for any $f\in\IL^q(M)$ we define
$$
\mathrm{e}^{- t \H}f(x):=\int_M p(t,x,y)f(y)\Id\mu(y).
$$
b) For any $t>0$ and any relatively compact open subset $U\subset M$, one has
\begin{align}\label{buc}
C_U(t):=\sup_{x\in U,y\in M}p(t,x,y)<\infty.
\end{align}
Morover, for any $t>0$, any open $U\subset M$ with $C_U(t)<\infty$ and any $q_1,q_2\in [1,\infty]$ with $q_1\leq q_2$, it holds that
\begin{align}\label{daaa}
\left\|1_U\mathrm{e}^{-t \H}\right\|_{q_1,q_2}\leq C_U(t)^{\f{1}{q_1}-\f{1}{q_2}}.
\end{align} 
\end{Theorem}

\begin{proof} a) This statement is included in  Theorem \ref{mart} b).\\
%b) In order to see (\ref{buc}) assume first that $x=y$. Then by the continuity of $p(t,\bullet,\bullet)$ we clearly have
%$$
%p(t,x,x)\leq A_U(t):=\sup_{z\in \overline{U}}p(t,z,z)<\infty.
%$$
b) In order to see (\ref{buc}), note first that by the smoothing part of Theorem \ref{mart} b) we have the a priori algebraic mapping property
\begin{align}\label{fsl}
1_U\mathrm{e}^{-s \H}:\IL^1(M)\longrightarrow \IL^\infty(M),
\end{align}
which by the closed graph theorem (keeping in mind that the $\IL^1$-convergence of a sequence implies the existence of a subsequence which converges $\mu$-a.e.) self-improves in the sense that (\ref{fsl}) is in fact a \emph{bounded} operator. Let us denote the operator norm of (\ref{fsl}) by $B_U(s) <\infty$, for any $s>0$. Using the Chapman-Kolmogorov equation, an application of this boundedness to $p(t/2,\bullet,y)\in \IL^1(M)$ and using (\ref{mar}), we find that for all $x\in U$, $y\in M$ one has 
\begin{align*}
&p(t,x,y)=\left[\mathrm{e}^{-\frac{t}{2} \H}p(t/2,\bullet,y)\right] (x)\leq \sup_{x'\in U}\left[\mathrm{e}^{-\frac{t}{2} \H}p(t/2,\bullet,y)\right] (x')\\
&\leq B_U(t/2) \int_M p(t/2,z,y)\Id\mu(z)\leq B_U(t/2) ,
\end{align*}
thus we arrive at the bound
$$
C_U(t)\leq   B_U(t/2) <\infty.
$$
Let us now give a proof of (\ref{daaa}). In view of $C_U(t)<\infty$, this is certainly possible using Riesz-Thorin\rq{}s interpolation theorem. However, it is also possible to give a direct proof. For this, let $U$ be an arbitrary Borel set with $C_U(t)<\infty$, and let $f\in\IL^{q_1}(M)$. \\
Case $1<q_1<q_2<\infty$: Let $r$ be given as $1-1/r=1/q_1-1/q_2$. Applying Hölder's inequality with the exponents 
\[
p_1=q_2,\>\>p_2=\f{r}{1-\f{r}{q_2}},\>\>p_3=\f{q_1}{1-\f{q_1}{q_2}} 
\]
shows that $\left\|1_U\mathrm{e}^{-t \H}f\right\|^{q_2}_{q_2}$ is
\begin{align}
&\leq \int_U\left( \int_M \left(  p(t,x,y)^r |f(y)|^{q_1}\right)^{\f{1}{q_2}}p(t,x,y)^{1-\f{r}{q_2}}|f(y)|^{1-\f{q_1}{q_2}}    \Id\mu(y)\right)^{q_2} \Id\mu(x)\nn\\
&\leq \int_U \left(\int_M p(t,x,y)^r|f(y)|^{q_1} \Id\mu(y) \right) \left(\int_M p(t,x,y)^r\Id\mu(y)  \right)^{\f{q_2}{r}\left(1-\f{r}{q_2} \right)}\nn\\
&\>\>\>\>\times \left( \int_M |f(y)|^{q_1}\Id\mu(y) \right)^{\f{q_2}{q_1}\left(1-\f{q_1}{q_2} \right) } \Id\mu(x),\nn
\end{align}
so that by using (\ref{mar}) and (\ref{buc}) twice we get
\begin{align*}
&\left\|1_U\mathrm{e}^{-t \H}f\right\|^{q_2}_{q_2}\\
&\leq   C_U(t)^{(1-\f{1}{r}) q_2\left(1-\f{r}{q_2} \right)}\left\|f\right\|^{q_2\left(1-\f{q_1}{q_2} \right) }_{q_1} \int_M |f(y)|^{q_1}\int_U p(t,x,y)^r \Id\mu(x) \Id\mu(y)\\
&\leq    C_U(t)^{q_2\left( \f{1}{q_1}-\f{1}{q_2}\right) }\left\|f\right\|^{q_2 }_{q_1}.
\end{align*}
Case $1<q_1<q_2=\infty$: With $q_1^*$ the Hölder dual exponent of $q_1$, we get
\begin{align*}
\left\|1_U\mathrm{e}^{-t \H}f\right\|_{\infty}\leq \sup_{x\in U}  \left\|p(t,x,\bullet)\right\|_{q_1^*}\left\|f\right\|_{q_1}\leq  C_U(t)^{1/q_1}\left\|f\right\|_{q_1}.
\end{align*}
Case $1=q_1<q_2<\infty$: One immediately gets 
\begin{align*}
 \left\| 1_U\mathrm{e}^{-t \H}f\right\|^{q_2}_{q_2} \leq \int_U\left( \int_M \left(  p(t,x,y)^{q_2} |f(y)|\right)^{\f{1}{q_2}}|f(y)|^{1-\f{1}{q_2}}    \Id\mu(y)\right)^{q_2} \Id\mu(x).
\end{align*}
Applying the Hölder inequality with the exponents 
\[
p_1=q_2,\>\> p_2=\f{1}{1-\f{1}{q_2}}
\]
gives
\begin{align*}
 \left\| 1_U\mathrm{e}^{-t \H}f\right\|^{q_2}_{q_2} \leq \left\|f\right\|^{q_2-1 }_1\int_U\int_M  p(t,x,y)^{q_2} |f(y)|\Id\mu(y) \Id\mu(x),
\end{align*}
so that Fubini, (\ref{mar}) and (\ref{buc}) imply
\[
 \left\|1_U\mathrm{e}^{-t \H}f\right\|^{q_2}_{q_2}\leq C_U(t)^{q_2\left( 1-\f{1}{q_2}\right) }\left\|f\right\|^{q_2}_1.
\]
The cases $q_1=q_2$ follow from part a), and the case $q_1=1$, $q_2=\infty$ is trivial. This completes the proof.
\end{proof}

\begin{Remark}\label{noni}1. For any fixed $x\in M$, the function
$$
(0,\infty)\ni t\longmapsto  p(t,x,x)\in (0,\infty)
$$
is nonincreasing. Indeed, using the Chapman-Kolomogorov equation and $\left\|\mathrm{e}^{-u \H}\right\|_{2,2}\leq 1$ for all $u>0$, we get the following estimates for all $s<t$:
\begin{align*}
p(t,x,x)&=\left\|p(t/2,x,\bullet)\right\|_2^2=\left\|\mathrm{e}^{-(t/2-s/2) \H}p(s/2,x,\bullet)\right\|^2_2\\
&\leq \left\|p(s/2,x,\bullet)\right\|^2_2=p(s,x,x).
\end{align*}
2. By the Chapman-Kolomogorov identity and Cauchy-Schwarz, one has
\begin{align}\label{self}
p(t,x,y)\leq \sqrt{p(t,x,x)}\sqrt{p(t,y,y)}\>\text{ for all $x,y\in M$,}
\end{align}
on any Riemannian manifold.\\
3. By (\ref{self}) we get
\begin{align}\label{miniself}
C(t):=\sup_{x\in M}p(t,x,x)=\sup_{x,y\in M}p(t,x,y)\in [0,\infty]\>\text{ for all $t>0$}.
\end{align}
If for some $t>0$ one has $C(t)<\infty$, then by the first part of this remark one automatically has $C(T)<\infty$ for all $T\geq t$. Morever, under the condition $C(t)<\infty$ one can take $U=M$ in (\ref{daaa}). It should be noted, however, that the validity of the global \lq\lq{}ultracontractivity\rq\rq{} $C(t)<\infty$ depends very sensitively on the geometry (that is, the Riemannian metric).
\end{Remark}

A generalization of Theorem \ref{aa} to covariant Schrödinger semigroups will be derived later on.\\

We continue with the following new concept that will be convenient in a moment:

\begin{Definition}\label{deee} Given $x\in M$ and $b>1$, let $r_{\mathrm{Eucl}}(x,b)$ be the supremum of all $r>0$ such that $\IB(x,r)$ is relatively compact and admits a coordinate system 
$$
\phi:\IB(x,r)\longrightarrow U\subset \IR^m
$$
with $\phi(x)=0$, and with respect to which one has the following inequality for all $y\in \IB(x,r)$:
\begin{align}\label{equi}
&\f{1}{b}(\delta_{ij})\leq (g_{ij}(y)):= \left(g(\partial_i,\partial_j)(y)\right)\leq b (\delta_{ij})\>\text{ as symmetric bilinear forms.}\end{align}
We call $r_{\mathrm{Eucl}}(x,b)$ \emph{the Euclidean radius of $M$ at $x$ with accuracy $b$}. Every coordinate system on $\IB(x,r)$, where $r<r_{\mathrm{Eucl}}(x,b)$, which satisfies (\ref{equi}) will be called a \emph{Euclidean coordinate system with accuracy $b$}.
\end{Definition}

In the following lemma, we collect some elementary properties of the Euclidean radius:

\begin{Lemma}\label{ddvc} a) For any $x\in M$ and $b>1$, one has $r_{\mathrm{Eucl}}(x,b)\in (0,\infty]$, and for every fixed $\epsilon>0$, the function
$$
M\longrightarrow (0,\epsilon],\>\>x\longmapsto \min(r_{\mathrm{Eucl}}(x,b),\epsilon) 
$$
is $1$-Lipschitz with respect to the Riemannian distance. In particular, 
$$
\inf_{x\in K}r_{\mathrm{Eucl}}(x,b)>0\>\>\text{ for every compact $K\subset M$.}
$$
b)  Let $x\in M$, $b>1$, $0< r < r_{\mathrm{Eucl}}(x,b)$, and let
$$
\phi:\IB(x,r)\longrightarrow U\subset \IR^m
$$
be a Euclidean coordinate system with accuracy $b$. Then one has
 \begin{equation}\label{inclusion2}
 \IB^{\IR^m}(0,b^{-1/2}r) \subset \phi(\IB(x,r)) \subset \IB^{\IR^m}(0,b^{1/2}r),
 \end{equation}
 where $\IB^{\IR^m}\subset \IR^{m}$ denotes the Euclidean balls. Moreover, one has the estimates
\begin{align}\label{stef2}
&\varrho(x,z)\leq b^{1/2}|\phi(z)|\quad\text{ for all $z\in \phi^{-1}(\IB^{\IR^m}(0,b^{-1/2}r))$, and}\\\label{stef3}
&|\phi(z)|\leq b^{1/2}\varrho(x,z)\quad\text{ for all $z\in \IB(x,r)$.}
\end{align}
\end{Lemma}

\begin{proof} a) Clearly we have $r_{\mathrm{Eucl}}(x,b)\in (0,\infty]$, since around each point $x\in M$ we can pick a coordinate system whose domain is included in a compact subset of $\IR^m$. For such a coordinate system, we have (\ref{equi}) for some $b\rq{}>1$, and scaling induces a coordinate system with (\ref{equi}). \\
To see the asserted Lipschitz continuity, let $x\in M$ and set $r(x):=r_{\mathrm{Eucl}}(x,b)$, $\tilde{r}(x):=\min(r(x),\epsilon)$. \\
Let first $y\in \IB(x,\tilde{r}(x))$, so that $r(y)\geq \tilde{r}(x)-\varrho(x,y).$ Moreover 
$$
0<\tilde{r}(x)-\varrho(x,y)<1
$$ 
follows from $\tilde{r}(x)=\min(\epsilon,r(x))$ and $\varrho(x,y)< \tilde{r}(x)$. Therefore  
$$
\min(\epsilon,r(y))\geq \min(r(x),\epsilon)-\varrho(x,y),\ \text{ that is }\ \tilde{r}(y)\geq \tilde{r}(x)-\varrho(x,y).
$$ 
If $\tilde{r}(x)\geq \tilde{r}(y)$, we can conclude that 
$$
|\tilde{r}(x)-\tilde{r}(y)|\leq \varrho(x,y).
$$
 If $\tilde{r}(x)<\tilde{r}(y)$, then $x\in \IB(y,\tilde{r}(y))$. This implies $r(x)\geq \tilde{r}(y)-\varrho(x,y)$. This inequality, as before, leads to the conclusion that 
$$
\tilde{r}(x)\geq \tilde{r}(y)-\varrho(x,y),\>\text{ so that }\>|\tilde{r}(x)-\tilde{r}(y)|\leq \varrho(x,y).
$$
Suppose now that $y\notin \IB(x,\tilde{r}(x))$. If $x\notin \IB(y,\tilde{r}(y))$ as well, we immediately get 
$$
|\tilde{r}(x)-\tilde{r}(y)|\leq \varrho(x,y).
$$
If $x\in \IB(y,\tilde{r}(y))$, we have, as above,
$$
r(x)\geq \tilde{r}(y)-\varrho(x,y),\>\text{ that is }\>r(x)\geq \min(r(y),\epsilon)-\varrho(x,y),
$$ 
which in turn implies 
$$
\min(r(x),\epsilon) \geq \min(r(y),\epsilon)-\varrho(x,y),
$$
which shows
$$
\tilde{r}(x)\geq \tilde{r}(y)-\varrho(x,y).
$$ 
Finally, in this last case we have $\tilde{r}(y)>\tilde{r}(x)$, so we can conclude that 
$$
|\tilde{r}(x)-\tilde{r}(y)|\leq \varrho(x,y).
$$
This completes the proof of part a).\\
b) The following proof has been communicated to the author by S. Pigola: Let us first prove (\ref{stef3}), which also directly implies the second inclusion in (\ref{inclusion2}). For this, let $z \in \IB(x,r)$ and, having fixed $0< \epsilon \ll 1$ such that
\[
\varrho(x,z) + \epsilon < r,
\]
consider any piecewise smooth curve $\gamma_{\epsilon} : [0,1] \to M$ connecting $x$ with $z$ and satisfying
\[
(\varrho(x,z) \leq )\quad \ell(\gamma_{\epsilon}) \leq \varrho(x,z) + \epsilon \quad (<r),
\]
where 
$$
\ell(\gamma):= \int^1_0\left|\dot{\gamma}(t)\right|\Id t= \int^1_0\sqrt{g(\dot{\gamma}(t),\dot{\gamma}(t))}\Id t 
$$
denotes the length of a smooth curve $\gamma:[0,1]\to M$ (with an obvious modification, if $\gamma$ is only piecewise smooth). Likewise, $\ell^{\IR^m}$ will denote its analogue with respect to the Euclidean metric on $\IR^m$. Clearly,
\[
\gamma_{\epsilon}(t) \in \IB(x,r)\quad \text{ for all $t \in [0,1]$}.
\]
Indeed, this follows from
\[
\varrho(x, \gamma_{\epsilon}(t)) \leq \ell(\gamma_{\epsilon}|_{[0,t]}) \leq \ell(\gamma_{\epsilon}) <r.
\]
Then 
$$
\psi_{\epsilon}: = \phi\circ \gamma_{\epsilon} : [0,1] \longrightarrow \phi(\IB(x,r)) = U
$$
is a piecewise smooth curve connecting $\psi_{\epsilon}(0) = 0$ with $\psi_{\epsilon}(1) = \phi(z)$, and we have
\begin{align*}
|\phi(z)| &\leq \ell^{\IR^m}(\psi_{\epsilon}) = \int_{0}^{1} \sqrt{ \sum_{ij}\delta_{ij} \dot \psi_{\epsilon}^{i} \dot \psi_{\epsilon}^{j}} \Id t\\
&\leq b^{1/2} \int_{0}^{1} \sqrt{\sum_{ij}g_{ij}  \dot \psi_{\epsilon}^{i} \dot \psi_{\epsilon}^{j}} \Id t
= b^{1/2} \ell(\gamma_{\epsilon})\leq b^{1/2} ( \varrho(x,z) + \epsilon ).
\end{align*}
By letting $\epsilon \to 0$, we conclude $|\phi(z)| \leq b^{1/2}  \varrho(x,z)$, so we have (\ref{stef3}) and the second inclusion  in (\ref{inclusion2}).\\
Next, let us prove the first inclusion in (\ref{inclusion2}). For this, let $\xi \in \IB^{\IR^m}(0,b^{-1/2}r)$. By contradiction, suppose that $\xi \not\in \phi(\IB(x,r))$. Consider the segment $\psi : [0,|\xi|] \to \IR^{m}$ given by $\psi(t) = \frac{\xi}{|\xi|} t$ and let $t^{\ast} \in (0,|\xi|]$ be the first time exit of $\psi$ from the domain $\phi(\IB(x,r))$. Thus, $\psi(t) \in \phi(\IB(x,r))$ for every $0 \leq t < t^{\ast}$ and $\psi(t^{\ast}) \not\in \phi(\IB(x,r))$. Now consider the curve $\gamma : [0,t^{\ast}) \to \IB(x,r)$ such that $\gamma(t) = \phi^{-1}\circ \psi|_{[0,t^{\ast})}(t)$. Take any sequence $ t_{k}\nearrow t^{\ast}$ as $k\to\infty$. Then, for every $k\in\IN$, we have
\begin{align}\label{dist-estimate}
\varrho(x, \gamma(t_{k})) &\leq \ell(\gamma|_{[0,t_{k}]}) \leq b^{1/2} \ell^{\IR^m}(\psi|_{[0,t_{k}]}) \\ &= b^{1/2} t_{k}
\leq b^{1/2}t^{\ast} \leq b^{1/2} |\xi| <r\nn.
\end{align}
It follows that 
\[
\{ \gamma(t_{k}) : k\in\IN \} \subset  \bar \IB(x,b^{1/2} |\xi|) \subset \IB(x,r).
\]
Since the closed ball $\bar \IB(x,b^{1/2} |\xi|)$ is compact, we can extract a converging subsequence
\[
 \gamma(t_{k'})  \to \bar x \in \bar \IB(x,b^{1/2} |\xi|) \subset \IB(x,r)
\]
as $k\rq{}\to\infty$, and therefore
\[
\phi(\gamma(t_{k'})) \to \phi(\bar x) \in  \phi(\IB(x,r)).
\]
On the other hand,
\[
\phi(\gamma(t_{k'})) = \psi(t_{k'}) \to \psi(t^{\ast}) \not\in \phi(\IB(x,r)).
\]
Due to the uniqueness of the limit $\psi(t^{\ast}) = \phi(\bar x)$, we get a contradiction. This finishes the proof of the first inclusion in (\ref{inclusion2}). Finally, 
$$ 
\IB^{\IR^m}(0,b^{-1/2}r) \subset \phi ( \IB(x,r) )
$$
shows that we are allowed to apply \eqref{dist-estimate} with $t_{k} = |\xi|$, showing that for every $z \in \phi^{-1}( \IB^{\IR^m}(0,b^{-1/2}r))$ one has
\[
\varrho(x,z) \leq b^{1/2} |\phi(z)|.
\]
This completes the proof.
\end{proof}

The following generally valid heat kernel estimate is based on the Euclidean radius and will be of central importance in the sequel:

\begin{Theorem}\label{mean} For all $b>1$ there is a constant $C=C(m,b)>0$ which only depends on $m$ and $b$, such that for all $\epsilon_1>0$, $\epsilon_2>1$, and all $t>0$, $x,y,\in M$, one has 
\begin{align}\label{mean2}
p(t,x,y)&\leq  \f{C}{\min\big(t,R(x,b,\epsilon_1,\epsilon_2)^2\big)^{m/2}}\leq   \f{C}{t^{m/2}}+\f{C}{R(x,b,\epsilon_1,\epsilon_2)^{m}}\\\nn
&\leq  \f{C}{R(x,b,\epsilon_1,\epsilon_2)^m} \left(\f{\epsilon_1^m }{\epsilon_2^m t^{m/2}}+1\right),
\end{align}
where 
$$
R(x,b,\epsilon_1,\epsilon_2):=\min(r_{\mathrm{Eucl}}(x,b),\epsilon_1)/\epsilon_2.
$$
\end{Theorem}

Above, the second and the third inequality are elementary. Theorem \ref{mean} has been established by the author in \cite{guenkat}, and it improves an earlier on-diagonal heat kernel estimate from \cite{brugun}. The proof of the off-diagonal estimate heavily relies on the following parabolic $\IL^1$-mean-value inequality, while the on-diagonal result from \cite{brugun} uses a parabolic $\IL^2$-mean-value inequality.

\begin{Theorem}\label{rtt0} There exists a constant $C=C(m)>0$, which only depends on $m$, with the property that
\begin{itemize}
\item for all $x\in M$, $r>0$ such that $\IB(x,r)$ is relatively compact and such that there exists a constant $a>0$ with the property that for every open $U\subset \IB(x,r)$, one has the Faber-Krahn-type inequality\footnote{See Corollary \ref{monon} for the meaning of $H_U$; it is a Dirichlet-Laplacian.} 
$$
\min \sigma(H_U)\geq a \mu(U)^{-2/m},
$$ 
\item for all $\tau\in (0,r^2]$, $t\geq \tau$,
\item for all nonnegative solutions $u$ of the heat equation $\partial_t u=(1/2) \Delta u$ in $(t-\tau,t]\times \IB(x,\sqrt{\tau})$, 
\end{itemize}
one has the bound
\begin{align}
u(t,x)\leq \f{Ca^{-\f{m}{2}}}{\tau^{1+\f{m}{2}}}\int^{t}_{t-\tau}\int_{\IB(x,r)} u(s,y) \Id\mu(y)\Id s.
\end{align}
\end{Theorem}

\begin{proof} Applying Theorem 15.1 in \cite{gri} (a variant of a parabolic $\IL^2$-mean-value inequality) to the radius $\sqrt{\tau}$ and to the solution 
$$
(0,\tau]\times \IB(x,\sqrt{\tau})\ni (s,y)\longmapsto u(t-\tau+s,y)\in [0,\infty)
$$
of the heat equation in $(0,\tau]\times \IB(x,\sqrt{\tau})$ immediately implies the parabolic $\IL^2$-mean-value inequality
\begin{align}\label{hilf}
u(t,x)^2\leq \f{Ca^{-\f{m}{2}}}{\tau^{1+\f{m}{2}}}\int^{t}_{t-\tau}\int_{\IB(x,\sqrt{\tau})} u(s,y)^2 \Id\mu(y)\Id s.
\end{align}
From here, one we can follow Li/Wang's parabolic $\IL^2$-to-$\IL^1$ reduction machinery from pp. 1269/1270 in \cite{li2}, which however has to be carefully adjusted to our situation: Setting 
$$
D:=Ca^{-\f{m}{2}},
$$
and applying (\ref{hilf}) with $\tau$ replaced by $\tau/4$ implies
$$
u(t,x)^2\leq D \tau^{-(1+m/2)}4^{-(1+m/2)}\int^{t}_{t-\tau/4}\int_{\IB(x,\sqrt{\tau}/2)} u(s,y)^2 \Id\mu(y)\Id s,
$$
so that setting
$$
Q:=\tau^{-(1+m/2)}\int^{t}_{t-\tau}\int_{\IB(x,\sqrt{\tau})} u(s,y) \Id\mu(y)\Id s,
$$
and for every $k\in\IN$,
$$
S_k:=\sup_{\left[t-\tau\sum^k_{i=1}4^{-i},t\right]\times \IB\left(x,\sqrt{\tau} \sum^k_{i=1}2^{-i}\right)  } u,
$$
we immediately get
\begin{align}\label{hfj}
u(t,x)^2\leq D 4^{-(1+m/2)}Q S_1\leq D Q S_1.
\end{align}
Let us next prove that for all $k$ one has
\begin{align}\label{rej}
S_k\leq D^{1/2} Q^{1/2} S_{k+1}^{1/2}.
\end{align}
To see this, pick 
$$
(s,y)\in \left[t-\tau\sum^k_{i=1}4^{-i},t\right]\times \IB\left(x,\sqrt{\tau} \sum^k_{i=1}2^{-i}\right) 
$$
with $u(s,y)=S_k$. Applying now (\ref{hilf}) with $t$ replaced by $s$, and $\tau$ replaced by $\tau/4^{k+1}$, and using
$$
\left[t-\tau\sum^{k+1}_{i=1}4^{-i},t\right]\times \IB\left(x,\sqrt{\tau} \sum^{k+1}_{i=1}2^{-i}\right) \supset \left[s-\tau /4^{k+1},s\right]\times \IB\left(y,\sqrt{\tau}/2^{k+1}\right) 
$$
to estimate the resulting space-time integral, we get
$$
u(s,y)^2\leq 4^{-(k+1)(1+m/2)}DQ S_{k+1},
$$
which implies (\ref{rej}). We claim that for all $k$ one has
\begin{align}\label{hkkp}
u(t,x)^2\leq D^{\sum^{k}_{i=1}2^{-i+1}}Q^{\sum^{k}_{i=1}2^{-i+1}} S_k^{\f{1}{2^{k-1}}}.
\end{align}
The proof is by induction on $k$: The case $k=1$ has already been shown in (\ref{hfj}). Given the statement for $k$, (\ref{rej}) gives us
\begin{align*}
&u(t,x)^2\leq D^{\sum^{k}_{i=1}2^{-i+1}}Q^{\sum^{k}_{i=1}2^{-i+1}} S_k^{\f{1}{2^{k-1}}}\\
&\leq  D^{\sum^{k}_{i=1}2^{-i+1}}Q^{\sum^{k}_{i=1}2^{-i+1}} D^{1/2^k} Q^{1/2^k} S_{k+1}^{1/2^k}\\
&=D^{\sum^{k+1}_{i=1}2^{-i+1}}Q^{\sum^{k+1}_{i=1}2^{-i+1}} S_{k+1}^{\f{1}{2^{k}}},
\end{align*}
which completes the proof of (\ref{hkkp}). As $(S_k)_k$ is a bounded sequence\footnote{For example, we have (estimating the sums with geometric series)
$$
S_k=\sup_{\left[t-\tau\sum^k_{i=1}4^{-i},t\right]\times \IB\left(x,\sqrt{\tau} \sum^k_{i=1}2^{-i}\right)  } u\leq  \sup_{\left[t-\f{3}{4}\tau,t\right]\times \IB\left(x,\sqrt{\tau} \right)  } u<\infty.
$$
}, we now get from letting $k\to\infty$ in (\ref{hkkp}) the bound
$$
u(t,x)^2\leq  D^{\sum^{\infty}_{i=0}2^{-i}}Q^{\sum^{\infty}_{i=0}2^{-i}} \lim_{k\to\infty}S_k^{\f{1}{2^{k-1}}} = D^{2}Q^{2}.
$$
This completes the proof of the $\IL^1$-mean-value inequality, recalling that $\tau\leq r^2$.
\end{proof}

\begin{proof}[Proof of Theorem \ref{mean}] As we have already remarked, it is sufficient to prove the first inequality. To this end, we set $R(x):=R(x,b,\epsilon_1,\epsilon_2)$. Then one easily finds that the function $R:M\to (0,\infty)$ has the following properties (just use (\ref{equi}), the local formula (\ref{mas1}) for $\mu$, and the fact that the Euclidean $\IR^m$ satisfies a global Faber-Krahn inequality; cf. p. 367 in \cite{gri}): There exists a constant $a=a(m,b)$ which only depends on $m$ and on $a$, such that for all $x\in M$ the ball $\IB(x,R(x))$ is relatively compact, and for every open $U\subset \IB(x,R(x))$ one has the Faber-Krahn inequality
$$
\min \sigma(H_U)\geq a \mu(U)^{-2/m}.
$$
Now fix an arbitrary 
$$
(t,x,y)\in (0,\infty)\times M\times M.
$$
Since 
$$
(s,z)\longmapsto u(s,z):=p(s,z,y)
$$
is a nonnegative solution of the heat equation on $(0,\infty)\times M$, by the above considerations an application of Theorem \ref{rtt0} with $r:=R(x)$ immediately implies
\begin{align*}
p(t,x,y)\leq \f{C'a^{-\f{m}{2}}}{\tau^{1+\f{m}{2}}}\int^{t}_{t-\tau}\int_{M} p(s,z,y) \Id\mu(z)\Id s
\end{align*}
for all $\tau\in (0,R(x)^2]$, where $C'>0$ only depends on $m$. Since we have 
\begin{align*}
\int_M p(s,z,y\rq{}) \Id\mu(z)=\int_M p(s,y\rq{},z) \Id\mu(z)\leq 1\>\>\text{ for all $(s,y\rq{})\in (0,\infty)\times M$, }
\end{align*}
we arrive at $p(t,x,y)\leq C'a^{-\f{m}{2}}\tau^{-\f{m}{2}}$, which proves the result upon taking $\tau:=\min(R(x)^2,t)$.
\end{proof}

The following definition is motivated by Theorem \ref{mean}:

\begin{Definition}\label{control} An ordered pair $(\Xi,\tilde{\Xi})$ of functions
$$
\Xi:M\longrightarrow (0,\infty], \>\>\tilde{\Xi}: (0,\infty)\longrightarrow (0,\infty)
$$
is called a \emph{heat kernel control pair for the Riemannian manifold $M$}, if the following assumptions are satisfied:
\begin{enumerate}
\item[$\bullet$] $\Xi$ is continuous with $\inf \Xi>0$, $\tilde{\Xi}$ is Borel
\item[$\bullet$] for all $x\in M$, $t>0$ one has 
$$
\sup_{y\in M} p(t,x,y)\leq \Xi(x)\tilde{\Xi}(t)
$$
\item[$\bullet$] for all $q\rq{} \geq 1$ in the case of $m=1$, and for all $q\rq{}>m/2$ in the case of $m\geq 2$, one has
$$
\int^{\infty}_0 \tilde{\Xi}^{1/q\rq{}}(t) \mathrm{e}^{-At} \Id t<\infty\>\>\text{ for some $A>0$}.
$$
\end{enumerate}
\end{Definition}

The motivation for the above definition stems from the fact that the concept of heat kernel control pairs is very general and very flexible in the following sense: Firstly, every Riemannian manifold (canonically) admits such a pair, and secondly, if one has some control on the geometry, one can pick somewhat sharper and more explicit control pairs. Thus we can treat both situations on an equal footing without losing any information. This is the content of the following remark and the subsequent example.

\begin{Remark}\label{ddghq}1. Every Riemannian manifold admits a canonically given family of heat kernel control pairs: Indeed, it follows from Theorem \ref{mean} and Lemma \ref{ddvc} that there exists a constant $C=C(m)>0$ such that for every choice of $b>1$ and $\epsilon_1>0$, $\epsilon_2>1$, the functions
$$
\Xi(x)= \f{C\epsilon_2^{m}}{\min(r_{\mathrm{Eucl}}(x,b),\epsilon_1)^{m}} ,\>\tilde{\Xi}(t)= \f{\epsilon_1^{m} }{\epsilon_2^{m} t^{m/2}}+1
$$
define such a pair.\\
2. Assume that there exist constants $C>0$, $T\in [0,\infty]$ such that one has the ultracontractiveness
$$
\sup_{x\in M}p(t,x,x)\leq C t^{-m/2}\>\>\text{ for all $0<t< T$}.
$$
Then, since $p(t,x,x)$ is always monotonely decreasing in $t$ (cf. Remark \ref{noni}.1), we get the bound
$$
\sup_{x\in M}p(t,x,x)\leq C \min(t,T)^{-m/2}\>\>\text{ for all $t>0$.}
$$
Therefore by Remark \ref{noni}.3, the pair 
$$
(\Xi(x), \tilde{\Xi}(t)):=(1,C \min(t,T)^{-m/2})
$$
is a heat kernel control pair, which is constant in its first slot.
\end{Remark}

As a typical example for the fact that some knowledge on the geometry leads to more explicit heat kernel control pairs, we consider the important class of geodesically complete manifolds whose Ricci curvature is bounded from below by a constant: 

\begin{Example}\label{ric} Assume that $M$ is geodesically complete with $\mathrm{Ric}\geq - (m-1)K$ for some constant $K\geq 0$. Then the following facts hold true:\vspace{1.2mm}

\emph{(i) Li-Yau-type heat kernel bounds:} For every $\delta_1,\delta_2>0$ which satisfy
$$
\delta_1\delta_2> ((m-1)^2K)/8,
$$
there exists a constant $C_{\delta_1,\delta_2,K,m}>0$ which only depends on $\delta_j$, $K$ and $m$, such that for all $t>0$, $x,y\in M$ one has
\begin{align*}
p(t,x,y)\leq &\ C_{\delta_1,\delta_2,K,m} \ \mu(\IB(x,\sqrt{t}))^{-1}\\
&\times \exp\Big(-(1-\delta_1)\varrho(x,y)^2/(2 t)+(\delta_2-\min\sigma(H))t\Big).
\end{align*}
Note that a comparable lower bound also exists (that we will not need in the sequel). Based on results by P. Li and S.-T. Yau \cite{li3} as well as B. Davies \cite{davies}, the above heat kernel upper bound has been derived in its ultimate form in the paper \cite{sturm3} by K.-T. Sturm.\vspace{1.2mm} 

\emph{(ii) Cheeger-Gromov volume estimate:} For every $s>0$, $ x\in M$, one has
$$
\mu(\IB(x,s))\leq |\mathbb{S}^m|s^{m}\exp((m-1)\sqrt{K}s),
$$
where $\mathbb{S}^m$ denotes the standard m-sphere.

\emph{(iii) Volume doubling property:}\footnote{Traditionally, the term \lq\lq{}doubling\rq\rq{} refers to the case $s=2s\rq{}$.} For every $0<s\rq{}<s$, $ x\in M$, one has
$$
\mu(\IB(x,s))\leq \mu(\IB(x,s\rq{})) (s/s\rq{})^m \exp((m-1)\sqrt{K}s).
$$
The last two results can be found, for example, in \cite{saloff-buch} (p. 177) and the references therein. \\
In order to derive a heat kernel control pair from these observations, we proceed as follows: First, it follows from the doubling property that for, say, $t<1$, we have
$$
\mu(\IB(x,\sqrt{t}))^{-1}\leq \mu(\IB(x,1))^{-1}t^{-m/2} \exp((m-1)\sqrt{K}),
$$
while clearly for $t\geq 1$ we have 
$$
\mu(\IB(x,\sqrt{t}))^{-1}\leq \mu(\IB(x,1))^{-1} .
$$
Thus for all $t>0$, $x,y\in M$,
\begin{align*}
&p(t,x,y)\\
&\leq C_{\delta_1,\delta_2,K,m} \left(\mu(\IB(x,1))^{-1}t^{-m/2} \exp((m-1)\sqrt{K})+\mu(\IB(x,1))^{-1}\right) \\
&\quad\times \exp\left(-\f{(1-\delta_1)\varrho(x,y)^2}{2 t}+(\delta_2-\min\sigma(H))t \right)\\
&\leq C_{\delta_1,\delta_2,K,m}\mu(\IB(x,1))^{-1}\left(\mathrm{e}^{(m-1)\sqrt{K}}t^{-m/2} +1\right)\\
&\quad\times \exp\big((\delta_2-\min\sigma(H))t\big) ,
\end{align*}
and we have thus derived the heat kernel control pair given by
\begin{align*}
&\Xi(x):=C_{\delta_1,\delta_2,K,m}\mu(\IB(x,1))^{-1},\\
&\tilde{\Xi}(t):=\left(\mathrm{e}^{(m-1)\sqrt{K}}t^{-m/2} +1\right)\exp\Big((\delta_2-\min\sigma(H))t\Big).
\end{align*}
\end{Example}

The Li-Yau heat kernel estimate, the Cheeger-Gromov volume estimate and the volume doubling property can be localized in a very exact way under geodesic completeness (cf. Theorem 6.1 and the inequalities (1), (2) in \cite{saloff}). For example, these localized estimates have been used in \cite{abba} in a probabilistic context, where it is shown that on every geodesically complete Riemannian manifold the Brownian brigde is a semimartingale including its terminal time. For example, the localized Cheeger-Gromov volume estimate reads as follows: If $M$ is geodesically complete and if $x\in M$, $r>0$, $K\geq 0$ are such that $\mathrm{Ric}\geq - (m-1)K$ in $B(x,2r)$, then for every $0<s<2r$ one has
\begin{align}\label{loki}
\mu(\IB(x,s))\leq |\mathbb{S}^m|s^{m}\exp((m-1)\sqrt{K}s).
\end{align}

Returning to the general situation, we recall that one always has 
$$
\int_M p(t,x,y)\Id\mu(y)\leq 1
$$
for all $t>0$, $x\in M$. Keeping this in mind, we record the following definition that will become important for us later on (in the context of Brownian motion). 

\begin{Definition}\label{stochi} $M$ is called \emph{stochastically complete}, if one has
\begin{align}\label{st}
\int_M p(t,x,y)\Id\mu(y)=1\quad\text{for all $t>0$, $x\in M$.} 
\end{align}
\end{Definition}

An important and simple consequence of stochastic completeness and Theorem \ref{ddfh} is the uniqueness of solutions of the initial value problem (cf. Corollary 9.6 in \cite{gri})
 \begin{align}\label{ddedw2}
\frac{\partial}{\partial t} u =(1/2)\Delta u,\>\>\lim_{	t\to 0+} u(t,\bullet)=\delta_y
\end{align}
in the following class of functions:

\begin{Proposition} For every $y\in M$ and every (necessarily smooth) solution 
$$
u: (0,\infty)\times M\to [0,\infty)
$$
 of (\ref{ddedw2}) 
with
$$
\int_M u(t,x) \mu(x)\leq 1\quad\text{ for all $t>0$,}
$$
one has $u(t,x)=p(t,x,y)$ for all $(t,x)\in (0,\infty)\times M$.
\end{Proposition}

In general, stochastic completeness is completely independent from geodesic completeness: There exist stochastically complete $M$'s which are geodesically incomplete (for example $\IR^m\setminus \{0\}$), and vice versa (cf. Theorem 7.9 in \cite{azencott}). Under geodesic completeness, however, there is the following very general volume test for stochastic completeness by A. Grigor'yan: 

\begin{Theorem} If $M$ is geodesically complete with
\begin{align}\label{apgb}
\int^{\infty}_1 \f{s}{\log\mu(\IB(x_0,s))} \Id s=\infty\quad\text{ for some $x_0\in M$,}
\end{align}
then $M$ is stochastically complete.
\end{Theorem}

For example, the last result immediately implies that geodesically complete $M$'s with Ricci curvature bounded from below by a constant are stochastically complete: Namely, the Cheeger-Gromov volume estimate from Example \ref{ric} shows that  
$$
\mu(B(x,r))\leq C_1\exp(C_2 r)\quad \text{ for all $r>0$, $x\in M$.}
$$
More generally, the last exponential volume growth holds with $r$ replaced by $r^2$ for geodesically complete $M$\rq{}s with
$$
\mathrm{Ric}(x)\geq -\varrho(x,x_0)^2-C\>\>\text{ for some $C>0$, some fixed $x_0$, and all $x\in M$},
$$
therefore this larger class is also stochastically complete. The latter volume estimate follows easily from applying the localized volume estimate (\ref{loki}) with $x=x_0$, $s=r$, $K=(4r^2+C)/(m-1)$.
\vspace{2mm}

We close this chapter with some facts about parabolicity:

\begin{Definition}\label{adpqa} a) A \emph{Green\rq{}s function} $\tilde{G}$ on $M$ is a Borel function $\tilde{G}:M\times M\to [-\infty,\infty]$ satisfying $\tilde{G}(x,y)=\tilde{G}(y,x)$, $\tilde{G}(x,\bullet)\in\IL^1_{\mathrm{loc}}(M)$ and 
\begin{align}\label{edc}
-(1/2)\Delta \tilde{G}(x,\bullet) = \delta_x, \>\>\text{ for all $x,y\in M$.}
\end{align} 
b) $M$ is called \emph{parabolic}, if $M$ does not admit a nonnegative Green\rq{}s function. Otherwise, $M$ is called \emph{nonparabolic}.
\end{Definition}

The name \lq\lq{}parabolic\rq\rq{} should not be confused with the corresponding notion from the uniformization theorem for Riemann surfaces. Even worse, a simply connected Riemann surface is parabolic in the sense of Definition \ref{adpqa}, if and only if \cite{uber} it is \emph{hyperbolic} in the sense of the uniformization theorem (which means that the surface is conformally equivalent to $\mathbb{H}^2$).

\begin{Remark}
It is obvious that if $M$ admits a Green\rq{}s function, then $M$ is noncompact, so that in particular compact $M$'s are parabolic. On the other side, based on an observation by B. Malgrange, P. Li and L.-F. Tam have \cite{litam} proved that every noncompact geodesically complete $M$ admits a Green\rq{}s function, which of course may change its sign. To illustrate the last fact, we can consider the Euclidean $\IR^2$, which is parabolic and nevertheless admits the Green\rq{}s function $\tilde{G}(x,y):=K\log(|x-y|^{-1})$ with some constant $K$. It follows from Proposition \ref{ecoll} below that $\IR^m$ is nonparabolic, if and only if $m\geq 3$.
\end{Remark}

The following result is based on Theorem 13.17 and Exercise 13.30 in \cite{gri}:

\begin{Propandef}\label{ecoll} $M$ is nonparabolic, if and only if for all $x\ne y$ one has
$$
G(x,y):=\int^{\infty}_0p(t,x,y)\Id t< \infty,
$$
and then 
\[
G:M\times M\longrightarrow  (0,\infty],\>\>G(x,y):= \int^{\infty}_0 p (t,x,y) \Id t
\]
is the pointwise minimal nonnegative Green's function on $M$, called the \emph{Coulomb potential on $M$}.
\end{Propandef}

The reader should compare the above result with the generally valid Theorem \ref{ddfh}. The following theorem collects some further results concerning parabolicity:

\begin{Theorem}\label{ahka} a) If $M$ is parabolic, then $M$ is stochastically complete and $\min\sigma(H)=0$.\\
b) If $M$ is geodesically complete with
\begin{align*}
\int^{\infty}_1 \f{s}{\mu(\IB(x_0,s))} \Id s=\infty\quad\text{ for some $x_0\in M$,}
\end{align*}
then $M$ is parabolic.
\end{Theorem}

\begin{proof} a) The first claim follows, for example, from Proposition \ref{ecoll}, and the second claim follows from Theorem 13.4 in \cite{gri} (see also Exercise 13.30 in \cite{gri}).\\
b) This is Theorem 11.14 in \cite{gri}, again keeping Exercise 13.30 therein in mind.
\end{proof}

Since hyperbolic spaces have a strictly positive bottom of the spectrum, it follows from Theorem \ref{ahka} a) that $\mathbb{H}^m$ is nonparabolic for every $m\geq 2$.\vspace{2mm}

Note that Theorem \ref{ahka} b) provides a volume test for the parabolicity of geodesically complete $M$'s which is in the spirit of (\ref{apgb}). On the other hand, there are many geodesically incomplete Riemannian manifolds that are nevertheless parabolic. For example, it has been shown in \cite{LT} by P. Li and G. Tian that Bergman metrics on the regular parts of algebraic varieties are parabolic and therefore stochastically complete. In addition, F. Bei and the author have recently established in \cite{beib} that iterated edge metrics on the regular parts of compact stratified pseudomanifolds are parabolic and thus stochastically complete. \vspace{2mm}

Concerning the connection between a global Gaussian upper bound on $p(t,x,y)$ and nonparabolicity, one has:

\begin{Theorem}\label{nonp}
Assume that there is a $c_1>0$ such that for all $t>0$ one has
\begin{align}
\sup_{x\in M}p(t,x,x)\leq c_1t^{-\f{m}{2}}. \label{fd21}
\end{align}
Then the following assertions hold:\\
a) There are $c_2,c_3>0$ such that for all $t>0$, $x,y\in M$ one has
\begin{align}
p(t,x,y)\leq c_2t^{-\f{m}{2}}\mathrm{e}^{-\f{\varrho(x,y)^2}{c_3 t}}.  \label{fd22}
\end{align}
b) One has a lower Euclidean volume growth, in the sense that
\begin{align}
\inf_{x\in M, r>0} \f{\mu(\IB(x,r))}{r^m}>0.   \label{sjw}
\end{align}
c) If $m\geq 3$, then $M$ is nonparabolic and there is a $c_4>0$ such that
\begin{align}
G(x,y)\leq c_4\varrho(x,y)^{m-4} \text{ for all $x,y\in M$ with $x\neq y$.}\label{marki0}
\end{align}
d)  If $m\geq 3$, then there is a $c_5>0$ such that for any $f\in W^{1,2}_0(M)$ one has the Sobolev inequality
 \begin{align}
\left\|f \right\|_{\frac{2m}{m-2}} 
\leq c_5\left\|\Id f \right\|_2.\label{sobo}
\end{align}
\end{Theorem}

\begin{proof} a) This is the content of Corollary 15.17 in \cite{gri}.\\
b) Cf. Exercise 14.5 in \cite{gri}.\\
c) Using part a), a simple calculation using the Gamma function shows that under the condition $m\geq 3$, we find for $x\neq y$
$$
\int^{\infty}_0p(t,x,y) \Id t\leq C(c_2,c_3,m)\varrho(x,y)^{m-4}.
$$
%b) The given Gaussian upper bound implies the following Faber-Krahn inequality: There is a $c>0$ such that for all relatively compact open nonempty $N\subset M$ and all $f\in  W^{1,2}_0(N)$ one has 
%\begin{align}
%\int_M |\Id f|^2\Id \mu \geq c \mu(N)^{-2/m}.
%\end{align}
%Applying this to appropriate cut-off functions $f$\rq{}s implies (\ref{sjw}).\\
\end{proof}

In particular, it follows from (\ref{sjw}) that whenever the Gaussian upper bound (\ref{fd22}) holds, then $M$ has an infinite volume.\\
Furthermore, it follows from the Li-Yau-type heat kernel bound from Example \ref{ric} and Theorem \ref{nonp} that geodesically complete Riemannian manifolds of dimension $m\geq 3$ with $\mathrm{Ric}\geq -K$ and
$$
\inf_{x\in M, r>0} \f{\mu(\IB(x,r))}{r^m}>0
$$
and in addition
$$
\text{ either $K=0$ or $\min\sigma(H)>0$}
$$
are nonparabolic, since then one has the Gaussian upper bound (\ref{fd21}).

\chapter{Wiener measure and Brownian motion on Riemannian manifolds}\label{stop}

\section{Introduction}

Roughly speaking, one would like to construct   Brownian motion $X(x_0)$ on $M$, starting from $x_0\in M$, as follows: It should be an $M$-valued process\footnote{We recall that given two measurable spaces $\Omega_1$ and $\Omega_2$, a map
$$
X:[0,\infty)\times \Omega_1\longrightarrow \Omega_2,\quad (t,\omega)\longmapsto X_t(\omega)
$$
is called an \emph{$\Omega_2$-valued process}, if for all $t\geq 0$ the induced map $X_t: \Omega_1\to \Omega_2$ is measurable. The maps $t\mapsto X_t(\omega)$, with fixed $\omega\in\Omega_1$, are referred to as the \emph{paths of $X$}.} with continuous paths
\begin{align}\label{bb}
X(x_0): [0,\infty)\times \Omega\longrightarrow M,
\end{align}
which is defined on some probability space $(\Omega,\IP,\IFF)$, and which has the transition probability densities given by $p(t,x,y)$. In other words, given  $n\in\IN$, a finite sequence of times $0<t_1< \dots< t_n$ and Borel sets $A_1,\dots,A_n\subset M$, setting $\delta_j:=t_{j+1}-t_j$ with $t_0:=0$, we would like the probability of finding the Brownian particle simultaneously in $A_1$ at the time $t_1$, in $A_2$ at the time $t_2$, and so on, to be given by the quantity
\begin{align}\label{shgi}
&\IP\{X_{t_1}(x_0)\in A_1,\dots,X_{t_n}(x_0)\in A_n\} \\\nn
&=\int\cdots \int 1_{A_1}(x_1)p(\delta_0 ,x_0,x_1) \cdots \\\nn
&\quad\quad\quad\quad\times 1_{A_n}(x_n) p(\delta_{n-1} ,x_{n-1},x_n) \Id\mu(x_1)\cdots \Id\mu(x_n),
\end{align}
whenever the particle starts from $x_0$. Equivalently, one could say that a Brownian motion on $M$ with starting point $x_0$ is a process with continuous paths (\ref{bb}), such that the finite-dimensional distributions of its law are given by the right-hand side of (\ref{shgi})\footnote{The law of $X(x_0)$ is by definition the probability measure on the space of continuous paths on $M$, which is defined as the pushforward of $\IP$ under the induced map
$$
\Omega\longrightarrow C([0,\infty),M),\>\>\omega\longmapsto X_{\bullet}(x_0)(\omega).
$$}. In fact, such a path space measure is uniquely determined by its finite-dimensional distributions (cf. Remark \ref{azi} below). In particular, all Brownian motions should have the same law, which we will call the \emph{Wiener measure} later on.\vspace{2mm}

Ultimately, the above prescriptions indeed turn out to work perfectly well in terms of giving Brownian motion for the Euclidean $\IR^m$ or for compact Riemannian manifolds. On the other hand, we see from (\ref{shgi}) that, in particular, it is required that for all $t> 0$,
$$
\IP\{X_{t}(x_0)\in M\}=\int_M p(t,x_0,y)\Id\mu(y),
$$
and already if $M$ is any open bounded subset of $\IR^m$, it automatically happens that
\begin{align}\label{sh0}
\int_M p(t,x_0,y)\Id\mu(y)<1\text{ for some $(t,x_0)\in (0,\infty)\times M$},
\end{align}
This leads to the conceptual difficulty that the process can leave its space of states with a strictly positive probability. This problem arises, if and only if $M$ is stochastically incomplete, which ultimately justifies Definition \ref{stochi}.\vspace{2mm}

Since we aim to work on arbitrary Riemannian manifolds, we need to solve the above conceptual problem of stochastic incompleteness. This is done by using the Alexandrov compactification of $M$. Since it does not cause much extra work, we start by explaining the corresponding constructions in the setting of an arbitrary Polish space, recalling that a topological space is called Polish, if it is separable and if it admits a complete metric which induces the original topology.

\section{Path spaces as measurable spaces}

\begin{Notation} Given a locally compact Polish space $N$, we set
\begin{align*}
&\widetilde{N}:=\\
&\begin{cases}&N, \text{if $N$ is compact}\\
&\text{Alexandrov compactification $N\cup \{\infty_N\}$, if $N$ is noncompact.}
\end{cases}
\end{align*}
We recall here that $\infty_N$ is any point $\notin N$, and that the topology on $N\cup \{\infty_N\}$ is defined as follows: $U\subset N\cup \{\infty_N\}$ is declared to be open, if and only if either $U$ is an open subset of $N$ or if there exists a compact set $K\subset N$ such that $U=(N \setminus K)\cup\{\infty_N\}$. This construction depends trivially on the choice of $\infty_N$, in the sense that for any other choice $\infty_N'\notin N$, the canonical bijection $N\cup \{\infty_N\}\to N\cup \{\infty_N'\}$ is a homeomorphism.\\
We consider the path space $\Omega_N:=C([0,\infty),\widetilde{N})$, and thereon we denote (with a slight abuse of notation) the canonically given coordinate process by
$$
\mathbb{X}: [0,\infty)\times \Omega_N\longrightarrow \widetilde{N},\>\>\mathbb{X}_t(\gamma):=\gamma(t).
$$
We consider $\Omega_N$ a topological space with respect to the topology of uniform convergence on compact subsets, and we equip it with its Borel sigma-algebra $\IFF^N$. 
\end{Notation}

We fix such a locally compact Polish space $N$ (e.g., a  manifold) for the moment. It is well-known that $\Omega_N$ as defined above is Polish again. In fact, $\widetilde{N}$ is Polish, and if we pick a bounded metric $\varrho_{\widetilde{N}} : \widetilde{N}\times \widetilde{N}\to [0,1]$ which induces the original topology on $\widetilde{N}$, then
$$
\varrho_{\Omega_N}(\gamma_1,\gamma_2):=\sum_{j=1}^{\infty}\max_{0\leq t\leq j}\varrho_{\widetilde{N}}(\gamma_1(t),\gamma_2(t))
$$ 
is a complete separable metric\footnote{In fact, it is easy to see that this is a complete metric which induces the original topology. On the other hand, the proof that this topology is separable is a little tricky, cf. p. 170 in D. Stroock\rq{}s book \cite{stroock2}. Although it is not so easy to find a precise reference, we believe that these results can be traced back to Kolmogorov.} on $\Omega_N$ which induces the original topology (of local uniform convergence). Furthermore, since evaluation maps of the form
$$
X_1\times C(X_1,X_2)\longrightarrow X_2,\>\>(x,f)\longmapsto f(x)
$$
are always jointly continuous, if $X_1$ is locally compact and Hausdorff and if $C(X_1,X_2)$ is equipped with its topology of local uniform convergence, it follows that $\mathbb{X}$ is in fact jointly continuous. In particular, $\mathbb{X}$ is jointly (Borel) measurable.

\begin{Notation} Given a set $\Omega$ and a collection $\mathscr{C}$ of subsets of $\Omega$ or of maps with domain $\Omega$, the symbol $\left\langle\mathscr{C}\right\rangle$ stands for the smallest sigma-algebra on $\Omega$ which contains $\mathscr{C}$. Furthermore, whenever there is no danger of confusion, we will use notations such as
$$
\{f\in A\}:=\{y\in \Omega:f(y)\in A\}\subset \Omega,
$$
where $f:\Omega\to \Omega'$ and $A\subset\Omega'$.
\end{Notation}

\begin{Definition}1. A subset $C\subset \Omega_N$ is called a \emph{Borel cylinder}, if there exist $n\in\IN$, $0<t_1< \dots<t_n$ and Borel sets $A_1,\dots,A_n\subset \widetilde{N}$, such that
$$
C=\{\IX_{t_1}\in A_1,\dots,\IX_{t_n}\in A_n\}=\bigcap_{j=1}^n\IX_{t_j}^{-1}(A_j).
$$ 
The collection of all Borel cylinders in $\Omega_N$ will be denoted by $\mathscr{C}^N$. \\
2. Likewise, given $t\geq 0$, the collection $\mathscr{C}^N_t$ of \emph{Borel cylinders in $\Omega_N$ up to the time $t$} is defined to be the collection of subsets $C\subset\Omega_N$ of the form
$$
C=\{\IX_{t_1}\in A_1,\dots,\IX_{t_n}\in A_n\}=\bigcap_{j=1}^n\IX_{t_j}^{-1}(A_j),
$$
where $n\in\IN$, $0<t_1< \dots< t_n< t$, and where $A_1,\dots,A_n\subset \widetilde{N}$ are Borel sets.
\end{Definition}

It is easily checked inductively that both $\mathscr{C}^N$ and $\mathscr{C}^N_t$ are $\pi$-systems in $\Omega_N$, that is, both collections are (nonempty and) stable under taking finitely many intersections. The following fact makes $\IFF^N$ handy in applications:

\begin{Lemma}\label{cyl} One has 
\begin{align}\label{gjao}
\IFF^N=\left\langle\mathscr{C}^N\right\rangle=\left\langle(\mathbb{X}_s:\Omega_N\longrightarrow \widetilde{N})_{s\geq 0}\right\rangle. 
\end{align}
\end{Lemma}

\begin{proof} Since for every fixed $s\geq 0$ the map 
$$
\IX_s:\Omega_N\longrightarrow \tilde{N},\> \gamma\longmapsto \gamma(s)
$$
is $\IFF^N$-measurable, it is clear that $\mathscr{C}^N\subset \IFF^N$, and therefore
$$
 \left\langle\mathscr{C}^N\right\rangle\subset \IFF^N.
$$
In order to see 
$$
 \IFF^N\subset  \left\langle\mathscr{C}^N\right\rangle,
$$ 
pick a topology-defining metric $\varrho_{\widetilde{N}}$ on $\widetilde{N}$ and denote the corresponding closed balls by $\overline{B_{\widetilde{N}}}(x,r)$. Then, since the elements of $\Omega_N$ are continuous, for all $\gamma_0\in\Omega_N$, $n\in\IN$, $\epsilon>0$ one has
\begin{align*}
&\left\{\gamma: \max_{0\leq t\leq n}\varrho_{\widetilde{N}}(\gamma(t),\gamma_0(t))\leq \epsilon\right\}\\
&=\bigcap_{0\leq t\leq n, \text{ $t$ is rational}}\left\{\gamma:\gamma(t)\in \overline{B_{\widetilde{N}}}(\gamma_0(t),\epsilon)\right\},\\
&=\bigcap_{0< t\leq n, \text{ $t$ is rational}}\left\{\gamma:\gamma(t)\in \overline{B_{\widetilde{N}}}(\gamma_0(t),\epsilon)\right\}.
\end{align*}
Therefore, sets of the form
\begin{align}\label{au}
\left\{\gamma:\max_{0\leq t\leq n}\varrho_{\widetilde{N}}(\gamma(t),\gamma_0(t))\leq \epsilon\right\},\>\text{  $\gamma_0\in\Omega_N$, $n\in\IN$, $\epsilon>0$ }
\end{align}
are $\left\langle\mathscr{C}^N\right\rangle$-measurable. Since the collection of sets of the form (\ref{au}) generates the topology of local uniform convergence\footnote{To be precise, this collection forms a basis of neighbourhoods of this topology.}, it is clear that the induced Borel sigma-algebra $\IFF^N$ satisfies $\IFF^N\subset  \left\langle\mathscr{C}^N\right\rangle$.\\
The inclusion  
$$
\left\langle\mathscr{C}^N\right\rangle\subset \left\langle(\mathbb{X}_s:\Omega_N\longrightarrow \widetilde{N})_{s\geq 0}\right\rangle
$$
is clear, since each set in $\mathscr{C}^N$ is a finite intersection of sets of the form $\IX^{-1}_s(A)$, $s>0$, $A\subset \widetilde{N}$ Borel. To see
$$
\left\langle(\mathbb{X}_s:\Omega_N\longrightarrow \widetilde{N})_{s\geq 0}\right\rangle\subset \left\langle\mathscr{C}^N\right\rangle,
$$
note that for every metric $\varrho_{\widetilde{N}}$ that generates the topology on $\widetilde{N}$, one has
$$
\left\langle(\mathbb{X}_s:\Omega_N\longrightarrow \widetilde{N})_{s\geq 0}\right\rangle=\left\langle\big\{\IX_s^{-1}\big(\overline{B_{\widetilde{N}}}(x,r)\big): x\in \widetilde{N}, r>0, s\geq 0\big\}\right\rangle,
$$
with the corresponding closed balls $\overline{B_{\widetilde{N}}}(\dots)$, so that it only remains to prove 
$$
\IX_0^{-1}\big(\overline{B_{\widetilde{N}}}(x,r)\big)\in \left\langle\mathscr{C}^N\right\rangle
$$
for all $x\in \widetilde{N}$, $r>0$. This, however, follows from
$$
\IX_0^{-1}\big(\overline{B_{\varrho_{\widetilde{N}}}}(x,r)\big)=\big\{\gamma: \lim_{n\to\infty} \varrho_{\widetilde{N}}(\gamma(1/n),x)\leq r\big\},
$$ 
since clearly $\gamma\mapsto \varrho_{\widetilde{N}}(\gamma(1/n),x)$ is a $\left\langle\mathscr{C}^N\right\rangle$-measurable function on $\Omega_N$ (the pre-image of an interval of the form $[0,R]$ under this map is the cylinder set $\IX_{1/n}^{-1}\big(\overline{B_{\widetilde{N}}}(x,R)\big)$). This completes the proof.
\end{proof}

\begin{Remark}\label{azi} By the above lemma, $\mathscr{C}^N$ is a $\pi$-system that generates $\IFF^N$. It then follows from an abstract measure theoretic result (cf. appendix, Corollary \ref{cara}) that every finite measure on $\IFF^N$ is uniquely determined by its values on $\mathscr{C}^N$.
\end{Remark}

\begin{Definition} Setting 
$$ 
\IFF^N_{t}:=\left\langle(\mathbb{X}_s:\Omega_N\longrightarrow \widetilde{N})_{0\leq s\leq t}\right\rangle\quad\text{ for every $t\geq 0$,}
$$
it follows from Lemma \ref{cyl} that
$$
\IFF^N_*:=\bigcup_{t\geq 0} \IFF^N_{t}
$$
becomes a filtration of $\IFF^N$. It is called the \emph{filtration generated by the coordinate process on $\Omega_N$}.
\end{Definition}

Precisely as for the second equality in (\ref{gjao}), one proves
\begin{align}\label{aaip}
\IFF^N_t=\left\langle\mathscr{C}^N_t\right\rangle\quad\text{for all $t\geq 0$.} 
\end{align}

Particularly important $\IFF^N_t$-measurable sets are provided by exit times:

\begin{Definition}
Given an arbitrary subset $U\subset \widetilde{N}$, we define
\begin{align}
&\zeta_U: \Omega_N\longrightarrow [0,\infty],
&\zeta_U:=\inf\{t\geq 0:\>\mathbb{X}_t\in \widetilde{N}\setminus U\},
\end{align}
and call this map the \emph{the first exit time of $\IX$ from $U$,} with $\inf\{...\}:=\infty$ in case the set is empty.
\end{Definition}

There is the following result, which in a probabilistic language means that first exit times from open sets are $\IFF^N_*$-optional times:\footnote{Let $(\Omega,\IFF)$ be a measure space, and let $\IFF_*=(\IFF_t)_{t\geq 0}$ be a filtration of $\IFF$. Then a map $\tau:\Omega\to [0,\infty]$ is called a \emph{$\IFF_*$-optional time}, if for all $t\geq 0$ one has $\{t<\tau\}\in\IFF_t$, and it is called a \emph{$\IFF_*$-stopping time}, if for all $t\geq 0$ one has $\{t\leq \tau\}\in\IFF_t$. }

\begin{Lemma} Assume that $U\subset \widetilde{N}$ is open with $U\ne \widetilde{N}$. Then one has
$$
\{t<\zeta_U\}\in \IFF^N_t\quad\text{for all $t\geq 0$.}
$$
\end{Lemma}

\begin{proof} The proof actually only uses that $\IX$ has continuous paths and that $\widetilde{N}$ is metrizable: Pick a metric $\varrho_{\widetilde{N}}$ on $\widetilde{N}$ which induces the original topology. Then, since $\tilde{N}\setminus U$ is closed and $\IX$ has continuous paths, we have
$$
\{t<\zeta_U\}=\bigcup_{n\in\IN}\bigcup_{0\leq s\leq t, \text{  $s$ is rational}}\{\varrho_{\widetilde{N}}(\IX_s,\tilde{N}\setminus U)\geq 1/n\}.
$$
The set on the right-hand side clearly is $\in\IFF^N_t$, since the distance function to a nonempty set is continuous and thus Borel.
\end{proof}

\section{The Wiener measure on Riemannian manifolds}

We return to our Riemannian setting. In order to apply the above abstract machinery in this case, we have to extend some Riemannian data to the compactification of $M$ (in the noncompact case):

\begin{Notation} Let $\widetilde{\mu}$ denote the Borel measure on $\widetilde{M}$ given by $\mu$ if $M$ is compact, and which is extended to $\infty_M$ by setting $\mu(\infty_M)=1$ in the noncompact case. Then we define a Borel function
$$
\widetilde{p}: (0,\infty)\times \widetilde{M}\times \widetilde{M}\longrightarrow [0,\infty)
$$
as follows: $\widetilde{p}:=p$ if $M$ is compact, and in case $M$ is noncompact, then
 for $t>0$, $x,y\in M$ we set
\begin{align*}
&\widetilde{p}(t,x,y):=p(t,x,y),\>\widetilde{p}(t,x,\infty_M):=0,\>\widetilde{p}(t,\infty_M,\infty_M):=1, \\
&\widetilde{p}(t,\infty_M,y):= 1-\int_M p(t,y,z)\Id\mu(z).
\end{align*}
\end{Notation}

It is straightforward to check that the pair $(\widetilde{p}, \widetilde{\mu})$ satisfies the Chapman-Kolmogorov equations, that is, for all $s,t>0$, $x,y\in\widetilde{M}$ one has
\begin{align}\label{chaa}
\int_{\widetilde{M}} \widetilde{p}(t,x,z)\widetilde{p}(s,y,z)\Id\widetilde{\mu}(z)=\widetilde{p}(s+t,x,y). 
\end{align}
Furthermore, one has
\begin{align}\label{dddo}
\int_{\widetilde{M}} \widetilde{p}(t,x,y)\Id \widetilde{\mu}(y)=1\text{ for all $x\in \widetilde{M}$},
\end{align}
in contrast to the possibility of $\int_{M}p(t,x,y)\Id \mu(y)<1$ in case $M$ is stochastically incomplete. It is precisely the conservation of probability (\ref{dddo}) which motivates the above Alexandrov machinery. Since there is no danger of confusion, the following abuse of notation will be very convenient in the sequel: 

\begin{Notation}\label{wiener2} We write $\zeta:=\zeta_M$ for the first exist time of the coordinate process $\IX$ on $\Omega_M$ from $M\subset \widetilde{M}$. 
\end{Notation}

For obvious reasons, $\zeta$ is also called the \emph{explosion time} of $\IX$. Note also that one has $\zeta>0$, and that by our previous conventions we have $\zeta\equiv \infty$ if $M$ is compact. The last fact is consistent with the fact that compact Riemannian manifolds are stochastically complete. \vspace{2mm}

The following existence result will be central in the sequel:

\begin{Propandef} The \emph{Wiener measure} $\IP^{x_0}$ with initial point $x_0\in M$ is defined to be the unique probability measure on $(\Omega_M,\IFF^M)$ which satisfies
\begin{align*}
&\IP^{x_0}\{\IX_{t_1}\in A_1,\dots,\IX_{t_n}\in A_n\} \\
&=\int\cdots\int 1_{A_1}(x_1)\widetilde{p}(\delta_0 ,x_0,x_1) \cdots \\
&\quad\quad\quad\quad  \times 1_{A_n}(x_n)\widetilde{p}(\delta_{n-1} ,x_{n-1},x_n) \Id\widetilde{\mu}(x_1)\cdots \Id\widetilde{\mu}(x_n)
\end{align*}
for all $n\in\IN$, all finite sequences of times $0<t_1<\dots<t_n$ and all Borel sets $A_1,\dots,A_n\subset \widetilde{M}$, where $\delta_j:=t_{j+1}-t_j$ with $t_0:=0$. It has the additional property that
\begin{align}\label{expla}
\mathbb{P}^{x_0}\left(  \{\zeta=\infty\}\bigcup\Big\{\text{$\zeta<\infty$ and $\IX_t=\infty_{M}$ for all $t\in [\zeta,\infty)$}\Big\}   \right) = 1,
\end{align}
in other words, the point at infinity $\infty_{M}$ is a \lq\lq{}trap\rq\rq{} for $\mathbb{P}^{x_0}$-a.e. path.\footnote{It is a trap in the sense that once a path touches $\infty_{M}$, it remains there for all times.}
\end{Propandef}

\begin{proof} We refer the reader to Section 8 in \cite{gribrown} and the references therein. For readers who are familiar with Dirichlet forms, we only remark here that the essential abstract property of $p(t,x,y)$ that we actually use is that the corresponding semigroup stems from a regular Dirichlet form $Q$. For each such semigroup, the above construction can be carried through to yield a corresponding Wiener measure which is concentrated on the space of right-continuous $\tilde{M}$-valued paths that have left limits. Since our $Q$ is even a \emph{local} Dirichlet form, it ultimately follows that the Wiener measures are concentrated on \emph{continuous} $\tilde{M}$-valued paths. We also refer the interested reader to \cite{roeck} for the details of the approach that uses Dirichlet form theory.
\end{proof}

An obvious but nevertheless very important consequence of (\ref{expla}) is that for all $x_0\in M$ one has
\begin{align}\label{expl2}
\IP^{x_0}\{  1_{\{t<\zeta\}}=1_{\{X_t\in M\}}   \}=1.
\end{align}

In the sequel, integration with respect to the Wiener measure will often be written as an expectation value,
$$
\mathbb{E}^{x_0}\left[\Psi\right]:=\int \Psi\Id\IP^{x_0}:=\int \Psi(\gamma)\Id\IP^{x_0}(\gamma),
$$
where $\Psi:\Omega_M\to \IC$ is any appropriate (say, nonnegative or integrable) Borel function. We remark that using monotone convergence, the defining relation of the Wiener measure implies that for all $n\in\IN$, all finite sequences of times $0<t_1<\dots<t_n$ and all Borel functions 
$$
f_1,\dots, f_n:\widetilde{M}\longrightarrow [0,\infty),
$$
one has
\begin{align}\label{ajp}
&\mathbb{E}^{x_0}\left[f_1(\IX_{t_1})\cdots f_n(\IX_{t_n}) \right]\\\nn
&=\int\cdots\int   f_1(x_1)\widetilde{p}(\delta_0 ,x_0,x_1) \cdots \\
&\quad\quad\quad\times f_n(x_n)\widetilde{p}(\delta_{n-1} ,x_{n-1},x_n) \Id\widetilde{\mu}(x_1)\cdots \Id\widetilde{\mu}(x_n),
\end{align}
where $\delta_j:=t_{j+1}-t_j$ with $t_0:=0$. In particular, by the very construction of $\widetilde{M}$ and $\widetilde{\mu}$, the above formula in combination with (\ref{expl2}) implies 
\begin{align}\label{ajp2}
&\mathbb{E}^{x_0}\left[1_{\{t_1<\zeta\}}f_1(\IX_{t_1})\cdots 1_{\{t_n<\zeta\}}f_n(\IX_{t_n}) \right]\\\nn
&=\mathbb{E}^{x_0}\left[1_{\{\IX_{t_1}\in M\}}f_1(\IX_{t_1})\cdots 1_{\{\IX_{t_n}\in M\}}f_n(\IX_{t_n}) \right]\\\nn
&=\int\cdots\int   f_1(x_1)p(\delta_0 ,x_0,x_1) \cdots \\
&\quad\quad\quad\quad\times f_n(x_n)p(\delta_{n-1} ,x_{n-1},x_n) \Id\mu(x_1)\cdots \Id\mu(x_n),
\end{align}
therefore quantities that are given by averaging over paths that remain on $M$ until any fixed time can be calculated by genuine Riemannian data on $M$, as it should be. In the sequel, we will also freely use the following facts:

\begin{Remark}\label{ahop}1. Each of the measures $\IP^{x_0}$ is concentrated on the set of paths that start in $x_0$, meaning that
$$
\IP^{x_0}\{\IX_0=x\}=1\quad\text{ for all $x_0\in M$},
$$ 
as it should be. To see this, pick a metric $\widetilde{\varrho}$ on $\widetilde{M}$ which induces the topology on $\widetilde{M}$, and set 
$$
\widetilde{f}:=\widetilde{\varrho}(\bullet,x_0)-\widetilde{\varrho}(\infty_M,x_0)\in C(\tilde{M}).
$$
As $x_0\in M$, the very definition of $(\widetilde{p},\widetilde{\mu})$ implies that for all $t>0$ one has
$$
\int_{\widetilde{M}} \widetilde{p}(t,x_0,y)\widetilde{\varrho}(y,x_0)\Id \widetilde{\mu}(y)=\int_{M }p(t,x_0,y)\widetilde{f}|_{M}(y) \Id \mu(y)+ \widetilde{\varrho}(\infty_M,x_0),
$$
which, since $\widetilde{f}|_{M}$ is a continuous bounded function on $M$, implies through (\ref{ajp}) and (\ref{capaa}) the $\IL^1$-convergence
$$
\mathbb{E}^{x_0}  \left[\widetilde{\varrho}(\IX_{t},x_0)\right]=\int_{\widetilde{M}} \widetilde{p}(t,x_0,y)\widetilde{\varrho}(y,x_0) \Id \widetilde{\mu}(y)\to 0\>\text{ as $t\to 0+$}.
$$
Thus we can pick a sequence of strictly positive times $a_n$ with $a_n\to 0$ such that $\widetilde{\varrho}(\IX_{a_n},x)\to 0$ $\IP^{x_0}$-a.e., and the claim follows from
$$
\widetilde{\varrho}(\IX_0,x)\leq\widetilde{\varrho}(\IX_0,\IX_{a_n})+\widetilde{\varrho}(\IX_{a_n},x) \quad\text{for all $n\in\IN$}
$$
and the continuity of the paths of $\IX$.\vspace{1mm}

2. For every Borel set $N\subset M$ with $\mu(N)=0 $ and every $x\in M$, one has
\begin{align}\label{fio}
\int^{\infty}_0\int_{\Omega_M} 1_{\{(s\rq{},\gamma\rq{}):\>\gamma\rq{}(s\rq{})\in N \}} (s,\gamma) \Id\IP^{x}(\gamma)\Id s=\int^{\infty}_0\int_N p(s,x,y)\Id\mu(y)\Id s =0.
\end{align}
This fact follows immediately from the defining relation of the Wiener measure. For the first identity in (\ref{fio}), one also needs Fubini\rq{}s Theorem, which can be used due to $\mathbb{X}$ being jointly measurable.\vspace{1mm}

3. For each fixed $A\in \IFF^M$, the map 
\begin{align}\label{ahoue}
M\longrightarrow [0,1],\>\>x\longmapsto \IP^{x}(A)
\end{align}
is Borel measurable. In fact, this is obvious for $A\in\mathscr{C}^M$ by the defining relation of the Wiener measure, and it holds in general by the monotone class theorem (cf. appendix, Theorem \ref{monocl}), since $\mathscr{C}^M$ is a $\pi$-system which generates $\IFF^M$, and since the collection of sets
$$
\{A:A\in\IFF^M, \>\text{ (\ref{ahoue}) is Borel measurable}\}
$$
forms a monotone Dynkin-system. 
\end{Remark}

The following result is crucial:

\begin{Lemma}\label{mark1} The family of Wiener measures satisfies the following Markov property: For all $x_0\in M$, all times $t\geq 0$, all $\IFF^M_t$-measurable functions $\phi:\Omega_M\to [0,\infty)$, and all $\IFF^M$-measurable functions $\Psi:\Omega_M\to [0,\infty)$, one has
\begin{align}\label{markk}
\int \phi(\gamma)  \Psi(\gamma(t+\bullet))\Id\IP^{x_0}(\gamma)=\int \phi(\gamma) \int\Psi(\omega) \Id \mathbb{P}^{\gamma(t)}(\omega) \Id\IP^{x_0}(\gamma)\in [0,\infty].
\end{align}
\end{Lemma}

\begin{proof} By monotone convergence, it is sufficient to consider the case $\phi=1_{A}$, $\Psi=1_{B}$ with $A\in \IFF^M_t$, $B\in \IFF_M$. Furthermore, for fixed $A\in \IFF^M_t$, using a monotone class argument as in Remark \ref{ahop}.3, it follows that it is sufficient to prove the formula for $B\in\mathscr{C}^M$. Using yet another monotone class argument, it follows that ultimately we have to check the formula only for $\phi=1_{A}$, $\Psi=1_{B}$ with $A\in \mathscr{C}^M_t$, $B\in \mathscr{C}_M$. So we pick $k,l\in\IN$, finite sequences of times $0<r_1< \dots< r_k< t$, $0<s_1< \dots< s_l$, Borel sets 
$$
A_1,\dots,A_k, B_1,\dots,B_l\subset \widetilde{M} 
$$
with
$$
A=\bigcap^k_{i=1}\IX_{r_i}^{-1}(A_i),\quad B=\bigcap^l_{i=1}\IX_{s_i}^{-1}(B_i),
$$
and $s_0:=0$, $r_0:=0$. Then by the defining relation of the Wiener measure we have
\begin{align*}
&\int 1_A(\gamma)\cdot 1_B(\gamma(t+\bullet)) \Id \IP^{x_0}(\gamma)\\
&=\int   1_{\{\IX_{r_1}\in A_1\}} \cdots 1_{\{\IX_{r_k}\in A_k\}}1_{\{\IX_{s_1+t}\in B_1\}} \cdots1_{\{\IX_{s_l+t}\in B_l\}}  \Id\IP^{x_0}\\
&=\int\cdots\int1_{A_1}(x_1)\widetilde{p}(r_1-r_0,x_0,x_1)\cdots 1_{A_k}(x_k)\widetilde{p}(r_{k}-r_{k-1},x_{k-1},x_k)\\
&\>\>\>\times 1_{B_1}(x_{k+1})\widetilde{p}(s_1+t-r_k,x_{k},x_{k+1})\cdots\\
&\>\>\>\times 1_{B_l}(x_{k+l})\widetilde{p}(s_{l}-s_{l-1},x_{k+l-1},x_{k+l})\Id\widetilde{\mu}(x_1)\cdots \Id\widetilde{\mu}(x_{k+l}).
\end{align*}

On the other hand, if for every $y_0\in\widetilde{M}$ we set 
\begin{align*}
&\Psi(y_0):=\int\cdots\int 1_{B_1}(y_1)\widetilde{p}(s_1-s_0,y_0,y_1) \cdots \\
&\times 1_{B_l}(y_l)\widetilde{p}(s_l-s_{l-1},y_{l-1},y_l)\Id\widetilde{\mu}(y_1)\cdots \Id\widetilde{\mu}(y_l),
\end{align*}
then by using the defining relation of the Wiener measure for the $\Id \mathbb{P}^{\gamma(t)}(\omega)$ integration and then using (\ref{ajp}), we get
\begin{align*}
&\int 1_A(\gamma) \int 1_B(\omega) \Id \mathbb{P}^{\gamma(t)}(\omega) \Id\IP^{x_0}(\gamma)\\
&=\int  1_{\{\IX_{r_1}\in A_1\}} (\gamma)\cdots 1_{\{\IX_{r_k}\in A_k\}}(\gamma) \Psi(\gamma(t)) \Id\IP^{x_0}(\gamma)\\
&=\int \cdots\int 1_{A_1}(z_1)\widetilde{p}(r_1-r_0,x_0,z_1)\cdots 1_{A_k}(z_k)\widetilde{p}(r_{k-1}-r_k,z_{k-1},z_k)\\
&\times \widetilde{p}(t-r_k,z_k,z) 1_{B_1}(y_1)\widetilde{p}(s_1-s_0,z,y_1)\cdots 1_{B_l}(y_l)\widetilde{p}(s_{l}-s_{l-1},y_{l-1},y_l)\\
&\times \Id\widetilde{\mu}(y_1)\cdots \Id\widetilde{\mu}(y_l)\Id\widetilde{\mu}(z_1)\cdots \Id\widetilde{\mu}(z_k)\Id\widetilde{\mu}(z),
\end{align*}
which is equal to the above expression for
$$
\int 1_A(\gamma)\cdot 1_B(\gamma(t+\bullet)) \Id \IP^{x_0}(\gamma),
$$
since by the Chapman-Kolomogorov equation and recalling $s_0=0$, we have
\begin{align*}
&\int\int\widetilde{p}(t-r_k,z_k,z) 1_{B_1}(y_1)\widetilde{p}(s_1-s_0,z,y_1)\Id\widetilde{\mu}(z)\Id\widetilde{\mu}(y_1)\\
&=\int\widetilde{p}(t-r_k+s_1,z_k,y_1) 1_{B_1}(y_1)\Id\widetilde{\mu}(y_1).
\end{align*}
This completes the proof.
\end{proof}

Now we are in the position to define Brownian motion on an arbitrary Riemannian manifold:

\begin{Definition}\label{ahuo} 1. Let $(\Omega,\IFF,\IP)$ be a probability space, $x_0\in M$, and let 
$$
X(x_0): [0,\infty)\times\Omega \longrightarrow \widetilde{M},\quad (t,\omega)\longmapsto X_t(x_0)(\omega)
$$
be a continuous process. Then the tuple
$(\Omega,\IFF,\IP, X(x_0))$ is called a \emph{Brownian motion on $M$ with starting point $x_0$}, if the law of $X(x_0)$ with respect to $\IP$ is equal to the Wiener measure $\IP^{x_0}$. Recall that this means the following:  The pushforward of $\IP$ with respect to the $\IFF/\IFF^M$ measurable\footnote{Note that by assumption $X_t(x_0)$ is $\IFF^M_t$-measurable for all $t\geq 0$, so that indeed (\ref{incd}) is automatically $\IFF/\IFF^M$ measurable.} map 
\begin{align}\label{incd}
\Omega\longrightarrow \Omega_M,\>\>\omega\longmapsto \big(t\longmapsto X_t(x_0)(\omega)\big)
\end{align}
is $\IP^{x_0}$.\\
2. Assume that $(\Omega,\IFF,\IP, X(x_0))$ is a Brownian motion on $M$ with starting point $x_0$, and that $\IFF_*:=(\IFF_t)_{t\geq 0}$ is a filtration of $\IFF$. Then the tuple
$(\Omega,\IFF,\IFF_*,\IP, X(x_0))$ is called an \emph{adapted Brownian motion on $M$ with starting point $x_0$}, if $X(x_0)$ is adapted to $\IFF_*:=(\IFF_t)_{t\geq 0}$ (that is, $X_t(x_0):\Omega\to\widetilde{M} $ is $\IFF_t$-measurable for all $t\geq 0$) and if in addition the following Markov property holds: For all times $t\geq 0$, all $\IFF_t$ measurable functions $\phi:\Omega\to[0,\infty)$, and all Borel functions $\Psi:\Omega_M\to [0,\infty)$, one has
$$
\int\phi(\omega) \Psi(X_{t+\bullet}(x_0)(\omega))\Id\IP(\omega)=\int\phi(\omega) \int\Psi(\gamma) \Id \mathbb{P}^{X_t(x_0)(\omega)}(\gamma)\Id\IP(\omega).
$$ 
\end{Definition}

It follows from the above results that a canonical adapted Brownian motion with starting point $x_0$ is given in terms of the Wiener measure by the datum
\begin{align}\label{aiopaa}
(\Omega,\IFF,\IFF_*,\IP, X(x_0)):=(\Omega_M,\IFF^M,\IFF^M_*,\IP^{x_0},\IX).
\end{align}

Having recorded the existence of Brownian motion, we can immediately record the following characterization of the stochastic completeness property that was previously defined by the \lq\lq{}parabolic condition\rq\rq{}
$$
\int_M p(t,x_0,y)\Id\mu(y)=1\quad\text{ for all $(t,x_0)\in (0,\infty)\times M$:}
$$
\emph{Namely, $M$ is stochastically complete, if and only if for every $x_0\in M$ and every Brownian motion $(\Omega,\IFF,\IP, X(x_0))$ on $M$ with starting point $x_0$, one has
$$
\IP\{X_t(x_0)\in M\}=1\quad\text{ for all $t\geq 0$,}
$$ 
that is, if all Brownian motions remain on $M$ for all times.} This observation follows immediately from the defining relation of the Wiener measure. \\
The second part of Definition \ref{ahuo} is motivated by the fact that every Brownian motion has the required Markov property with respect to its own filtration:

\begin{Lemma} Every Brownian motion $(\Omega,\IFF,\IP, X(x_0))$ on $M$ with starting point $x_0$ is automatically an $(\IFF^{X(x_0)}_t)_{t\geq 0}$-Brownian motion, where
$$
\IFF^{X(x_0)}_t:=\left\langle(X_s(x_0))_{0\leq s\leq t}\right\rangle,\quad t\geq 0 
$$
denotes the filtration of $\IFF$ which is generated by $X(x_0)$.
\end{Lemma}

\begin{proof} We have to show that given $t\geq 0$, an $\IFF^{X(x_0)}_t$-measurable function $\phi:\Omega\to[0,\infty)$, and a Borel function $\Psi:\Omega_M\to [0,\infty)$, one has
$$
\int\phi(\omega) \Psi(X_{t+\bullet}(x_0)(\omega))\Id\IP(\omega)=\int\phi(\omega) \int\Psi(\gamma) \Id \mathbb{P}^{X_t(x_0)(\omega)}(\gamma)\Id\IP(\omega).
$$ 
Assume for the moment that we can pick an $\IFF^M_t$-measurable function $f:\Omega_M\to [0,\infty)$ such that $f(X\rq{}(x_0))=\phi$, where 
$$
X\rq{}(x_0): \Omega \longrightarrow \Omega_M
$$
denotes the induced $\IFF/\IFF^M$ measurable map (\ref{incd}). Then, since the law of $X(x_0)$ is $\IP^{x_0}$, we can use the Markov property from Lemma \ref{mark1} to calculate 
\begin{align*}
&\int\phi(\omega) \Psi(X_{t+\bullet}(x_0)(\omega))\Id\IP(\omega)\\
&=\int f(\omega\rq{}) \Psi(\omega\rq{}(t+\bullet))\Id\IP^{x_0}(\omega\rq{})\\
&=\int f(\omega\rq{}) \int\Psi(\gamma) \Id \mathbb{P}^{\omega\rq{}(t)}(\gamma) \Id\IP^{x_0}(\omega\rq{})\\
&=\int f\big(X(x_0)(\omega)\big) \int\Psi(\gamma) \Id \mathbb{P}^{X_t(x_0)(\omega)}(\gamma) \Id\IP(\omega)\\
&=\int \phi(\omega) \int\Psi(\gamma) \Id \mathbb{P}^{X_t(x_0)(\omega)}(\gamma) \Id\IP(\omega),
\end{align*}
proving the claim in this case. It remains to prove that one can always \lq\lq{}factor\rq\rq{} $\phi$ in the above form. Somewhat simpler variants of such a statement are usually called Doob-Dynkin lemma in the literature. An important point here is that the factoring procedure can be chosen to be positivity preserving. We give a quick proof: Set $X:=X(x_0)$,  $X\rq{}:=X\rq{}(x_0)$, and assume first that $\phi$ is a simple function, that is, $\phi$ is a finite sum $\phi=\sum_j c_j 1_{A_j}$ with constants $c_j\geq 0$ and disjoint sets $A_j\in \IFF^{X}_t$. Then by the definition of this sigma-algebra, there exist times $0\leq s_j\leq t$ and Borel sets $B_j\subset \widetilde{M}$ with $A_j=X_{s_j}^{-1}(B_j)$, such that with $C_j:=\IX_{s_j}^{-1}(B_j)\in \IFF^M_t $, the function $f:=\sum_j c_j 1_{C_j}$ on $\Omega_M$ is nonnegative, $\IFF^M_t$-measurable, and satisfies $f(X\rq{})=\phi$. In the general case, there exists an increasing sequence of nonnegative $\IFF^{X}_t$-measurable simple functions $\phi_n$ on $\Omega$ such that $\lim_n\phi_n =\phi$. By the above, we can pick for each $n$ an $\IFF^M_t$-measurable nonnegative function $f_n$ on $\Omega_M$ with $f_n(X\rq{})=\phi_n$. The set
$$
\Omega\rq{}:=\{f_n\> \text{ converges pointwise }\}\subset \Omega
$$
clearly contains the image of $X\rq{}$, and it is $\IFF^M_t$-measurable by Proposition \ref{el} in the appendix. Then $f:=\lim_{n} (f_n 1_{\Omega\rq{}})$ has the desired properties. Note that the above proof is entirely measure theoretic and does not use any particular (say, topological) properties of the involved quantities.
 \end{proof}

Without entering the details, we remark here that the importance of adapted Brownian motions stems from the fact that they are continuous $M$-valued semimartingales \cite{hsu} with respect to the given filtration. Being a continuous semimartingale, the paths of an adapted Brownian motion can be almost surely horizontally lifted (in a natural sense that relies on Stratonovic stochastic integrals) to smooth principal bundles that come equipped with a smooth connection \cite{hack}. This is a very remarkable fact, since Brownian paths are almost surely nowhere differentiable \cite{hack}. Such lifts are the main ingredient of probabilistic formulae for the semigroups associated with operators of the form $H^{\nabla}$ and perturbations thereof \cite{guen}.\\
Finally, we recall that parabolicity always implies stochastic completeness. The former property also has a probabilistic interpretation, for one can prove \cite{uber}:

\begin{Theorem} $M$ is parabolic, if and only if every Brownian motion $(\Omega,\IFF,\IP, X(x_0))$ on $M$ with starting point $x_0$ is transient, in the sense that for every precompact set $U\subset M$ one has
$$
\IP\{\text{there exists $s>0$ such that for all $t>s$ one has $X_t(x_0)\notin U$}\}=1,
$$
that is, if and only if all Brownian motions on $M$ eventually leave each precompact set almost surely.
\end{Theorem}

\chapter{Contractive Dynkin and Kato potentials}\label{C6}

\section{Generally valid results}

This chapter is devoted to the class of Kato (and more generally contractive Dynkin-) potentials, which will be the main class of perturbations under consideration. The main observations that make these classes so important are the following ones: If $w:M\to\IR$ is in the contractive Dynkin class, then 

\begin{itemize}

\item $H + w$ is a well-defined (in the sense of sesquilinear forms) self-adjoint and semibounded operator in $\IL^2(M)$

\item the corresponding semigroup has very natural $\IL^q(M)$-smoothing properties

\item the corresponding semigroup also has some pointwise $C(M)$-smoothing properties, if $w$ has some mild additional (Kato) regularity.

\end{itemize}

We will prove precise covariant generalizations of the above statements later on. Let us start with the actual definitions and some general facts about these classes. Since it does not cause any extra work to consider complex-valued $w\rq{}s$, we will record these facts for the complex case, although we will not make any particular use of the corresponding non-self-adjoint theory in the sequel.

\begin{Definition} Let $w:M\to \IC$ be a Borel function. Then $w$ is said to be in the \emph{contractive Dynkin class} $\dyn(M)$ of $M$, if there is a $t>0$ with
\begin{align}\label{dfdsd}
\sup_{x\in M}\int^{t }_0\int_M p(s,x,y) |w(y)|\Id\mu(y)\Id s<1,
\end{align}
and $w$ is said to be in the \emph{Kato class} $\kat(M)$ of $M$, if
\begin{align}
\lim_{t\to 0+}\sup_{x\in M}\int^{t}_0\int_M p(s,x,y)  |w(y)|\Id\mu(y)\Id s=0.
\end{align}
We also define the corresponding local classes
\begin{align*}
&\dyn_{\loc}(M):= \big\{w:1_Kw\in \dyn(M)\>\text{ for all compact $K\subset M$} \big\},\\
&\kat_{\loc}(M):= \big\{w:1_Kw\in \kat(M)\>\text{ for all compact $K\subset M$} \big\}
\end{align*}
\end{Definition}

In the literature, the class that we have called \emph{contractive Dynkin class} is sometimes also called \emph{generalized} or \emph{extended Kato class}. In any case, the name \emph{Dynkin class} is consistently reserved in the literature for the class of $w\rq{}s$ which satisfy the weaker assumption
\begin{align*}
\sup_{x\in M}\int^{t }_0\int_M p(s,x,y) |w(y)|\Id\mu(y)\Id s<\infty,
\end{align*}
which motivates the name \lq\lq{}contractive Dynkin class\rq\rq{}. The Dynkin class will not play any role in the sequel. Obviously, all these classes are (complex) linear spaces which depend on the geometry of $M$, and one has
\begin{align*}
\kat(M)\subset \dyn(M),\>\>\kat_{\loc}(M)\subset \dyn_{\loc}(M).
\end{align*}
It is not quite clear where the $\dyn(M)$ class really appeared for the first time. What can be said, however, is that this class has been systematically studied for the first time and in a very general context (replacing $Q$ with a general regular Dirichlet form and potentials by measures) by P. Stollmann\footnote{Note, however, that the authors do not reserve a particular symbol for this class.} and J. Voigt in \cite{peter}. We refer the reader also to \cite{sturm}, where K.-T. Sturm treats measure perturbations for the Laplace-Beltrami operator under lower Ricci bounds.

\begin{Remark}1. In typical applications, the contractive Dynkin class does not seem to play an important role, since one usually deals with Kato potentials. Nevertheless, many abstract results only require the Dynkin property. In addition, it has been shown by Z.-Q. Chen and K. Kuwae (cf. Example 4.3 in \cite{chen}) that on every geodesically complete $M$ with a Ricci curvature bounded from below by a constant and a strictly positive injectivity radius, one has $\dyn(M)\setminus \kat(M)\ne \emptyset$. \\
2. In order to illustrate what kind of singularities we are actually talking about, we remark right away that $\IL^{q}(\IR^m)\subset \kat(\IR^m)$ in the Euclidean $\IR^m$, if $m\geq 2$ and $q>m/2$. So for example the Coulomb potential $1/|x|$ is in $\kat(\IR^3)$. We will come to geometric generalizations of such inclusions later on.
\end{Remark}

The name \lq\lq{}Kato class\rq\rq{} stems from the paper \cite{aizenman} by B. Simon and M. Aizenman, referring to the original paper by T. Kato \cite{kato} where this class of potentials appeared for the first time in the context of essential self-adjointness of Schrödinger operators. In fact, Kato (essentially) introduced the $\kat(\IR^m)$ as follows: $w\in \kat(\IR^m)$, if and only if
\begin{align}\label{eins}
&w\in \IL^1_{\mathrm{unif, loc}}(\IR),\>\text{ if $m=1$},\\ \label{einse}
&\lim_{r\to 0+}\sup_{x\in \IR^m} \int_{|x-y|\leq r}|w(y)| h_m(|x-y|)\Id y=0,\>\text{ if $m\geq 2$,}
\end{align}
where $h_m:[0,\infty]\to [0,\infty]$ is given by
$$
h_2(r):=\log(1/r),\>\>h_m(r):=r^{2-m},\>\text{ if $m>2$,} 
$$
and where $\IL^1_{\mathrm{unif, loc}}$ stands for the uniformly locally integrable functions. The equivalence of the latter definition to our heat kernel definition is not obvious, and has been established in \cite{aizenman}. In principle, one can also define a class $\kat'(M)$ in the spirit of (\ref{eins}) and (\ref{einse}) with $\varrho(x,y)$ and the volume measure $\mu$ replacing their Euclidean analogues. However, the class $\kat'(M)$ does not seem to be useful for operator theory in general. The reason for this is that without having appropriate Gaussian heat kernel bounds of $p(t,x,y)$ at hand, there is no reason to expect that potentials from $\kat'(M)$ are (form-) bounded with respect to $H$ (which, however, will turn out to be the case for $\kat(M)$). As one might expect, in a sufficiently Euclidean situation, that is, if $p(t,x,y)$ admits appropriate global Gauss-type upper and lower bounds, it can be shown that $\kat'(M)=\kat(M)$. For example, the last equality is true if $M$ is geodesically complete with Ricci curvature bounded from below by a constant and a strictly positive injectivity radius. The state of the art concerning equalities of the type $\kat'(M)=\kat(M)$ is contained in the seminal paper \cite{kt} by K.Kuwae and M. Takahashi. \vspace{1.2mm}

Altogether, the heat kernel characterization of $\kat(\IR^m)$ from \cite{aizenman} should be considered the starting point for almost every result that we will establish in the sequel. Our general philosophy is as follows: The analogue of every Euclidean result for Schrödinger operators with $\kat(\IR^m)$-potentials holds true on arbitrary, possibly incomplete Riemannian manifolds, if one uses our definition of $\kat(M)$, and likewise for $\dyn(M)$. Assumptions on the geometry come into play in a second step only, namely when one wants to guarantee that $\kat(M)$ (which always contains $\IL^{\infty}(M)$) has large $\IL^q$-type subspaces, which allows an easy decision on whether a given singular potential is in $\kat(M)$ or not. In this context, as we have already remarked, it has been shown in \cite{aizenman} that for $m\geq 2$ one has $\IL^{q}(\IR^m)\subset\kat(\IR^m)$ for every $q> m/2$. Again, such a result cannot be expected in general. One of our central observations in this chapter is that nevertheless one always has a \emph{weighted} inclusion of the form $\IL^{q}(M, h\Id\mu)\subset\kat(M)$, with $h:M\to (0,\infty]$ a continuous density function which satisfies $\inf h >0$ and which is canonically given on every Riemannian manifold. In the case of $M=\IR^m$, the geometry is simply so mild that we can take $h\equiv \mathrm{const}$.\vspace{2mm}

We start by establishing the following well-known auxiliary results that are always true without any further assumptions on the geometry. Ultimately, these results justify the definitions of $\dyn(M)$ and $\kat(M)$. 

\begin{Lemma}\label{goof} a) It holds that $ L^{\infty}(M)\subset \kat(M)$.\\
b) One has $\dyn_{\loc}(M)\subset  L^1_{\loc}(M)$.\\
c) For any $w\in \dyn (M)$ one has
$$
\sup_{x\in M} \int_M\int^T_0 p(s,x,y)|w(y)| \Id s \ \Id\mu(y)<\infty\text{ for all $T>0$}.
$$
\end{Lemma}

\begin{proof} a) Let $w\in L^{\infty}(M)$. Using $\int p(s,x,y) \Id\mu(y) \leq 1$, for all $t>0$ we get
$$
\sup_{x\in M}\int^t_0 \int_M p(s,x,y) \left|w(y)\right| \Id\mu(y) \Id s\leq t \left\|w\right\|_{\infty}<\infty.
$$
b) Let $w\in\dyn_{\loc}(M)$, let $K\subset M$ be compact, and pick some $t >0$ such that
\begin{align}\label{endl}
\sup_{x\in M}\int^{t }_0\int_K p(s,x,y) |w(y)|\Id\mu(y)\Id s<\infty,
\end{align}
and pick some $C=C(K,t )>0$ such that for all $s\in [t /2,t ]$ and all
$x,y\in K$ one has $p(s,x,y)\geq C$. Then
\begin{align*}
\left(t -t /2\right)C\int_{K} \left|w(y)\right| \Id\mu(y)\leq \sup_{x\in M}\int_K\int^{t }_{0}p(s,x,y)\Id s \left|w(y)\right|
\Id\mu(y),
\end{align*}
which is finite. \\
c) We will follow \cite{kw0}: Take a $t>0$ with (\ref{dfdsd}), and pick $l\in\IN$ with $T<lt $. Then we can estimate 
\begin{align*}
& \sup_{x\in M} \int_M\int^T_0 p(s,x,y)|w(y)| \Id s \ \Id\mu(y)\nn\\
&\leq \sup_{x\in M} \int_M\int^{l t}_0 p(s,x,y)|w(y)|  \Id s \ \Id\mu(y) \nn\\
&\leq\sum^l_{k=1}\sup_{x\in M} \int_M\int^{ t}_0 p((k-1)t+s,x,y)|w(y)|  \Id s \ \Id\mu(y)\nn\\
&=\sum^l_{k=1}\sup_{x\in M} \int^{ t}_0 \int_M p((k-1)t,x,z) \int_M p(s,z,y)|w(y)|  \Id\mu(y)\Id\mu(z)\Id s\nn\\
&\leq  \left(\sum^l_{k=1} \sup_{x\in M}\int_M p((k-1)t,x,z)\Id\mu(z)\right)\times \nn\\
&\>\>\>\>\>\>\times \sup_{z\in M}\int^{ t}_0\int_M p(s,z,y)|w(y)| \Id\mu(y)\Id s \\ 
&\leq l \sup_{z\in M}\int^{ t}_0\int_M p(s,z,y)|w(y)| \Id\mu(y)\Id s <\infty,
\end{align*}
where we have used the Chapman-Kolomogorov identity and $$
\int p(s',x',y') \Id\mu(y') \leq 1.
$$
\end{proof}

Next, let us record the following simple inequalities (see also \cite{demuth}):

\begin{Lemma}\label{demuth} For any Borel function $w:M\to\IC$ and any $r,t >0$, one has 
\begin{align*}
&\left(1-\mathrm{e}^{-rt} \right)\sup_{x\in M} \int^{\infty}_0 \mathrm{e}^{-rs} \int_M p(s,x,y) |w(y)| \Id\mu(y) \Id s\nn\\
&\leq\sup_{x\in M}\int^t_0 \int_M p(s,x,y) \left|w(y)\right| \Id\mu(y) \Id s\nn\\
& \leq \mathrm{e}^{rt}\sup_{x\in M} \int^{\infty}_0 \mathrm{e}^{-rs} \int_M p(s,x,y) |w(y)| \Id\mu(y) \Id s.
\end{align*}
\end{Lemma}

\begin{proof} For any $x\in M$ we have
\begin{align*}
&\int^{\infty}_0 \mathrm{e}^{-rs} \int_M p(s,x,y) |w(y)| \Id\mu(y) \Id s\\
&=\sum^{\infty}_{k=0}\int^{t(k+1)}_{kt} \mathrm{e}^{-rs} \int_M p(s,x,y) |w(y)| \Id\mu(y) \Id s\\
&=\sum^{\infty}_{k=0}\mathrm{e}^{-rkt}\int_Mp(kt,x,z)\int^{t}_{0} \int_M\mathrm{e}^{-rs}  p(s,z,y) |w(y)| \Id\mu(y) \ \Id s \ \Id\mu(z)\\
&\leq \left(\sum^{\infty}_{k=0}\mathrm{e}^{-rkt}\right)\sup_{z\in M}\int^{t}_{0}\mathrm{e}^{-rs}   \int_Mp(s,z,y) |w(y)| \Id\mu(y) \Id s \\
&=\f{1}{1-\mathrm{e}^{-rt}}\sup_{z\in M}\int^{t}_{0}\mathrm{e}^{-rs}   \int_Mp(s,z,y) |w(y)| \Id\mu(y) \Id s, 
\end{align*}
from which the claims easily follow. Here, we have used the Chapman-Kolomogorov identity and $\int p(s',x',y') \Id\mu(y') \leq 1$. 
\end{proof}

Now we continue with a useful characterization of the contractive Dynkin and the Kato class, respectively. In view of
$$
(\H+r)^{-1}= \int^{\infty}_0 \mathrm{e}^{-rs}\mathrm{e}^{-s\H}\Id s,
$$
and recalling our notation for the Wiener measure (Notation \ref{wiener2}), the following lemma can be considered a resolvent/semigroup/Brownian motion equivalence-type result:

\begin{Lemma}\label{char} a) For a Borel function $w:M\to\IC$, the following statements are equivalent:
\begin{itemize}
\item[i)]   $w\in\dyn(M)$.
\item[ii)] One has
$$
\lim_{t\to 0+}\sup_{x\in M}\int^{t}_0\int_M p(s,x,y)  |w(y)|\Id\mu(y)\Id s<1.
$$
\item[iii)] There is an $r>0$ with
\begin{align*} 
\sup_{x\in M} \int^{\infty}_0 \mathrm{e}^{-rs} \int_M p(s,x,y) |w(y)| \Id\mu(y) \Id s <1.
\end{align*}
\item[iv)] One has
\begin{align*}
\lim_{r\to\infty}\sup_{x\in M} \int^{\infty}_0 \mathrm{e}^{-rs} \int_M p(s,x,y) |w(y)| \Id\mu(y) \Id s <1.
\end{align*}
\item[v)] For some $t>0$ one has 
\begin{align*}
\sup_{x\in M}\int^{t}_0\mathbb{E}^x\left[1_{\left\{s<\zeta\right\}}\left|w(\IX_s)\right|\right]\Id s<1.
\end{align*}
\item[vi)] One has
\begin{align*}
\lim_{t\to 0+}\sup_{x\in M}\int^{t}_0\mathbb{E}^x\left[1_{\left\{s<\zeta\right\}}\left|w(\IX_s)\right|\right]\Id s<1.
\end{align*}
\end{itemize}

b) For a Borel function $w:M\to\IC$, the following statements are equivalent: 
\begin{itemize}
\item[i)]  $w\in\kat(M)$.
\item[ii)] One has 
\begin{align*}
\lim_{r\to\infty}\sup_{x\in M} \int^{\infty}_0 \mathrm{e}^{-rs} \int_M p(s,x,y) |w(y)| \Id\mu(y) \Id s =0. 
\end{align*}
\item[iii)] One has
\begin{align*}
\lim_{t\to 0+}\sup_{x\in M}\int^{t}_0\mathbb{E}^x\left[1_{\left\{s<\zeta\right\}}\left|w(\IX_s)\right|\right]\Id s=0.
\end{align*}
\end{itemize}
\end{Lemma}

\begin{proof} a) The equivalence of i), ii), v), vi) and the equivalence of iii), iv) are clear, and the equivalence of ii) and iv) follows from Lemma \ref{demuth}.\\
b) The equivalence of i) and iii) is clear, and Lemma \ref{demuth} directly implies the equivalence of i) and ii).
\end{proof}

The following result is of fundamental importance, since it shows that expressions of the type $\int^t_0 w(\IX_s)\Id s$, which appear either directly or in estimates in the context of Feynman-Kac-type formulae, are well-defined for $\mu$-a.e. $x\in M$, if $w:M\to \IR$ is locally integrable, and even for \emph{all} $x$, if $w$ is locally contractively Dynkin. In the case of $M=\IR^m$, a result of this type goes back to W. Faris and B. Simon \cite{faris}, and the same proof applies to stochastically complete $M$'s. The general, possibly stochastically incomplete case has been treated by the author in \cite{guen}. It requires some additional technical adjustments, since one has to deal with explosive paths:

\begin{Lemma}\label{Katto0} a) Let $w\in L^1_{\mathrm{loc}}(M)$. Then for $\mu$-a.e. $x\in M$ one has 
\begin{align}\label{qtm0}
\mathbb{P}^x\left\{w(\IX_{\bullet})\in L^1_{\mathrm{loc}}[0,\zeta)\right\}=1. 
\end{align}
b) Let $w\in\dyn_{\loc}(M)$. Then for \emph{any} $x\in M$ one has (\ref{qtm0}).
\end{Lemma}

\begin{proof} We prepare the proof of the actual statements with some auxiliary results: Pick a continuous function $\rho:M\to [0,\infty)$ such that for all $c\in [0,\infty) $ the level sets $\{\rho\in [c,\infty)\}$ are compact. Then the collection of subsets  $(U_n)_{n\in\IN}$ of $M$ given by
$$
U_n:=\>\text{interior of}\>\{\rho\in [1/n,\infty)\}
$$ 
forms an exhaustion of $M$ with open relatively compact subsets. For every $n\in\IN$, define the first exit times
\begin{align*}
\zeta^{(1)}_n:=\zeta_{U_n}:   \Omega_M\longrightarrow [0,\infty].
\end{align*}
Then the sequence $\zeta^{(1)}_n$ \emph{announces} $\zeta$ with respect to $\IP^x$ for every $x\in M$ in the following sense: There exists a set $\Omega_x\subset\Omega_M$ with $\mathbb{P}^x(\Omega_x)=1$, such that for all paths $\gamma\in\Omega_x$ one has the following two properties: 
\begin{itemize}
	\item $\zeta^{(1)}_n(\gamma)\nearrow\zeta(\gamma)$ as $n\to\infty$, 
	\item the implication $\zeta(\gamma)<\infty \Rightarrow\zeta^{(1)}_n(\gamma)<\zeta(\gamma)$ holds true for all $n$. 
\end{itemize}

To see that $\zeta$ is indeed announced by $\zeta_{n}^{(1)}$ in the asserted form, one can simply set 
$$
\Omega_x:=\{\gamma\in\Omega_M:\>\gamma(0)=x\}.
$$
Then $\mathbb{P}^x(\Omega_x)=1$ by Remark \ref{ahop}.1, and the asserted properties follow easily from continuity arguments, since $\Omega_x$ is a set of continuous paths that start in $x$. It follows immediately that $\zeta^{(2)}_n:=\min(\zeta^{(1)}_n, n)$ also announces $\zeta$. As a consequence, for any $x\in M$, any Borel function $h:M\to \IC$ and $j=1,2$ we have 
\begin{align}
&\mathbb{P}^x\left\{h(\IX_{\bullet})\in L^1_{\mathrm{loc}}[0,\zeta)\right\}=\mathbb{P}^x \bigcap_{n\in\IN}\left\{\int^{\zeta^{(j)}_n}_0\left|h(\IX_s)\right|\Id s<\infty\right\}. \nn
\end{align}
a) Let us first assume that $w\in L^1(M)$. Then, using Fubini (recall that first exit times are $\IFF^M$-measurable and that $\IX$ is jointly measurable), for any $n$ we have
\begin{align} 
&\int_M \mathbb{E}^x\left[   \int^{\zeta^{(2)}_n}_0 \left|w(\IX_s) \right|\Id s    \right] \Id\mu(x) \nn\\
&\leq \int_M \mathbb{E}^x\left[   \int^{\min(\zeta, n)}_0 \left|w(\IX_s) \right|\Id s   \right] \Id\mu(x) \nn\\
& = \int_M \mathbb{E}\left[  \int^{n}_0 1_{ \{s<\zeta\}  } \left|w(\IX_s) \right| \Id s \right] \Id\mu(x)\nn\\
& = \int_M \int^{n}_0\mathbb{E}\left[   1_{ \{s<\zeta\}  } \left|w(\IX_s) \right|  \right] \Id s \Id\mu(x)\nn\\
&= \int^n_0 \int_M  \int_M p(s,x,y)  \Id\mu(x) \left|w(y)\right|  \Id\mu(y)\Id s <\infty,\label{gtz}
\end{align}
which implies (\ref{qtm0}) in this situation. If one only has $w\in L^1_{\mathrm{loc}}(M)$, then (since now $1_{U_n}w\in  L^1(M)$) for $\mu$-a.e. $x$ and all $n$ we have
\begin{align*}
&\mathbb{P}^x \left\{\int^{\zeta^{(1)}_n}_0\left|w(\IX_s)\right|\Id s=\infty\right\} \\
=\>\>\>&\mathbb{P}^x
\left\{\int^{\zeta^{(1)}_n}_0\left|\Big(1_{U_n}(\IX_s)+1_{M\setminus
U_n}(\IX_s)\Big)w(\IX_s)\right|\Id s=\infty\right\} \\
\leq \>\>\> &\mathbb{P}^x\left\{\int^{\zeta^{(1)}_n}_0\left|(1_{U_n}w)(\IX_s)\right|\Id
s=\infty\right\}=0, 
\end{align*}
which again implies (\ref{qtm0}).\\
b) Let $x\in M$, $w\in\dyn(M)$, $n\in\IN$. We have 
\begin{align*}
&\mathbb{E}^x\left[  \int^{\zeta^{(2)}_n}_0 \left|w(\IX_s) \right|\Id s  \right] \leq  \mathbb{E}^x\left[  \int^{\min(\zeta, n)}_0 \left|w(\IX_s) \right|\Id s   \right] \nn\\
& =  \mathbb{E}^x\left[  \int^{n}_0 \left|w(\IX_s) \right|1_{ \{s<\zeta\}  }  \Id s   \right] = \int^n_0   \int_M p(s,x,y)  \left|w(y)\right|  \Id\mu(y),
\end{align*}
and this number is finite for all $n$ (in view of Lemma \ref{goof}), which shows (\ref{qtm0}) in the global contractive Dynkin case. Now, one can use the same localization procedure as above to deduce (\ref{qtm0}) for arbitrary $w\in\dyn_{\loc}(M)$.
\end{proof}

The following result is again of fundamental importance, since it shows that given a contractive Dynkin function $w:M\to\IR$ one can make sense of $H+w$ as a self-adjoint semibounded operator in the sense of sesquilinear forms, using the KLMN theorem (cf. appendix, Theorem \ref{klmn}):

\begin{Lemma}\label{k2} For any $r>0$, any Borel function $w:M\to \IC$, and any $f\in  W^{1,2}_0(M)$
one has 
\begin{align}\label{wiwaldi}
\left\| \sqrt{|w|}f\right\|^2_2 \leq  \f{C_{r}(w)}{2}  \left\|\Id f\right\|^2_2 + C_{r}(w) r\left\|f\right\|^2_2,
\end{align}
where
$$
 C_{r}(w) :=\sup_{x\in M} \int^{\infty}_0 \mathrm{e}^{-r s} \int_M p(s,x,y) |w(y)| \Id\mu(y) \Id s\in [0,\infty].
$$
\end{Lemma}

\begin{proof} We can assume that $w$ is nonnegative. It suffices to show
\begin{align}\label{aiwq}
\left\|\widehat{w^{1/2}} (\H+r)^{-1/2}h\right\|^2_{2}\leq C_r(w)\left\|h\right\|^2_2\quad\text{ for all $h\in\IL^2(M)$},
\end{align}
where $\widehat{w^{1/2}}=\widehat{w}^{1/2}$ denotes the maximally defined multiplication operator induced by $w^{1/2}$, that is, $\dom(\widehat{w^{1/2}})$ is given by those $f\in \IL^{2}(M)$ which satisfy $w^{1/2}f\in \IL^{2}(M)$. Indeed, once we have established the above estimate, applying it to $h= (\H+r)^{1/2}f$ with $f\in  W^{1,2}_0(M)=\dom((\H+r)^{1/2})$ proves
$$
\left\|\widehat{w^{1/2}} f\right\|^2_{2}\leq C_r(w)\left\|(\H+r)^{1/2}f\right\|^2_2=C_r(w)\left\|\H^{1/2}f\right\|^2_2+rC_r(w)\left\|f\right\|^2_2,
$$
which is nothing but the asserted estimate. So it remains to prove (\ref{aiwq}). To this end, setting $w_n:=\min(w,n)\in\IL^{\infty}(M)$, $n\in\IN$,  and using monotone convergence and $C_r(w_n)\leq C_r(w)$, it is actually sufficient to prove that for all $n$ one has
\begin{align}\label{rgh}
\left\|\widehat{w^{1/2}_n} (\H+r)^{-1/2}\right\|^2_{2,2}\leq C_r(w_n).
\end{align}
Since $\widehat{w^{1/2}_n}$ and $(\H+r)^{-1/2}$ are self-adjoint and since $\ILL(\IL^{2}(M))$ is a $C^*$-algebra, one has 
\begin{align*}
&\left\|\widehat{w^{1/2}_n} (\H+r)^{-1}\widehat{w^{1/2}_n} \right\|_{2,2}=\left\|\widehat{w^{1/2}_n} (\H+r)^{-1/2}\left(\widehat{w^{1/2}_n}(\H+r)^{-1/2}\right)^* \right\|_{2,2}\\
&=\left\|\widehat{w^{1/2}_n} (\H+r)^{-1/2}\right\|^2_{2,2}.
\end{align*}
To estimate this expression, let $f_1,f_2\in\IL^{2}(M)$. Using the Laplace transform 
$$
(\H+r)^{-1}=\int^{\infty}_0\mathrm{e}^{-rs}\mathrm{e}^{-s \H}\Id s,
$$
we get
\begin{align*}
&\left|\left\langle \widehat{w^{1/2}_n} (\H+r)^{-1}\widehat{w^{1/2}_n} f_1,f_2 \right\rangle\right|\\
&\leq\int^{\infty}_0\int_M\int_Mw^{1/2}_n(x)|f_1(y)| w^{1/2}_n(y) |f_2(x)|   p(s,x,y)\mathrm{e}^{-rs}\Id\mu(y)\Id\mu(x)\Id s.
\end{align*}
Once we apply Cauchy-Schwarz to the Borel measure
$$
\Id\rho( y, x ,  s)= p(s,x,y)\mathrm{e}^{-rs}\Id\mu(y)\Id\mu(x)\Id s\>\text{ on $  M\times M\times(0,\infty)$},
$$
we therefore get
\begin{align*}
&\left|\left\langle \widehat{w^{1/2}_n} (\H+r)^{-1}\widehat{w^{1/2}_n} f_1,f_2 \right\rangle \right| \\
& \leq\left(\int^{\infty}_0\int_M\int_Mw_n(x)|f_1(y)|^2   p(s,x,y)\mathrm{e}^{-rs}\Id\mu(y)\Id\mu(x)\Id s\right)^{1/2} \\
&\times\left(\int^{\infty}_0\int_M\int_Mw_n(y) |f_2(x)|^2 p(s,x,y)\mathrm{e}^{-rs}\Id\mu(y)\Id\mu(x)\Id s\right)^{1/2}\\
& =\left(\int_M\int^{\infty}_0\mathrm{e}^{-rs}\int_M w_n(x)   p(s,y,x)\Id\mu(x)\Id s|f_1(y)|^2\Id\mu(y)\right)^{1/2} \\
&\times\left(\int_M\int^{\infty}_0\mathrm{e}^{-rs}\int_M w_n(y)   p(s,x,y)\Id\mu(y)\Id s|f_2(x)|^2\Id\mu(x)\right)^{1/2}\\
&\leq C_r(w_n)\|f_1\|_2\|f_2\|_2,
\end{align*}
which proves (\ref{rgh}).
\end{proof}

The above result is due to P. Stollmann and J. Voigt \cite{peter}, who even treat a more general context than Riemannian manifolds, namely regular Dirichlet forms. In fact, they also allow the perturbations to be Kato measures rather than Kato functions. Our proof is quite different from that of \cite{peter} (see also \cite{demuth}). \\
Finally, we record some exponential estimates. To this end, for every Borel function $w:M\to \IC$ and any $s\geq 0$, let
\begin{align}
D(w,s):=\sup_{x\in M}\mathbb{E}^x\left[\int^s_0 \left|w(\IX_r)\right|1_{\{r<\zeta\}}\Id
r\right]\in [0,\infty].\nn
\end{align}

Part a) of the following lemma is a classical result by M. Aizenman and B. Simon for $M=\IR^m$. Essentially the same proof (with some modifications taking the explosion of paths into account) works for manifolds as well:

\begin{Lemma}\label{xdd} a) For any $w\in\dyn(M)$, there are $c_j=c_j(w)>0$, $j=1,2$, such that
for all $t\geq 0$,
\begin{align}
\sup_{x\in M} \mathbb{E}^x\left[\mathrm{e}^{\int^t_0 \left|w(\IX_s)\right|\Id
s}1_{\{t<\zeta\}}\right]
\leq   c_1\mathrm{e}^{tc_2}<\infty. \label{way0}
\end{align}
In fact, for every $s>0$ with $D(w,s)<1$ one can pick the constants
\begin{align}
c_1=\f{1}{1-D(w,s)},\>c_2=\f{1}{s}\log\left(\f{1}{1-D(w,s)}\right)  .
\end{align}

b) For any $w\in\kat(M)$ and any $\delta>1$, there is a $c_{\delta}=c_{\delta}(w)>0$ such that for all $t\geq 0$,
\begin{align}
\sup_{x\in M} \mathbb{E}^x\left[\mathrm{e}^{\int^t_0 \left|w(\IX_s)\right|\Id
s}1_{\{t<\zeta\}}\right]
\leq \delta \mathrm{e}^{tc_{\delta}}<\infty. \label{way1}
\end{align}
In fact, for every $s_{\delta}>0$ with $D(w,s_{\delta})<1-1/\delta$ one can pick the constant
\begin{align}\label{ffeq}
c_{\delta}=\f{1}{s_{\delta}}\log\left(\f{1}{1-D(w,s_{\delta})}\right)  .
\end{align}
\end{Lemma}

\begin{proof} Let us first record some abstract facts:\\
1. With $\tilde{M}=M\cup\{\infty_M\}$ the Alexandrov compactification of $M$, we can
canonically extend every Borel function $v:M\to \IC$ to a Borel function $\widetilde{v}:\tilde{M}\to \IC$ by setting 
$\widetilde{v}(\infty_M)=0$. Then one trivially has 
\begin{align}
\mathbb{E}^x\left[\mathrm{e}^{\int^t_0 \left|v(\IX_s)\right|\Id
s}1_{\{t<\zeta\}}\right]
\leq  \mathbb{E}^x\left[\mathrm{e}^{\int^t_0 \left|\widetilde {v}(\IX_s)\right|\Id
s}\right].\label{way}
\end{align}
2. (Khas\rq{}minskii\rq{}s lemma) For any Borel function $v:M\to \IC$ and any $s\geq 0$,
let
\begin{align}
J(v,s):=
\sup_{x\in
M}\mathbb{E}^x\left[\mathrm{e}^{\int^s_0\left|\widetilde{v}(\IX_r)\right|\Id
r}\right]
\in [0,\infty].\nn
\end{align}
Then for every $s>0$ with $D(v,s)<1$ (of course, such an $s$ does not need to exist for an arbitrary $v$) it holds that
\begin{align}
J(v,s)\leq \f{1}{1-D(v,s)}.\label{way2}
\end{align}
Proof:  One has
$$
D(v,s)=\sup_{x\in M}\mathbb{E}^x\left[\int^s_0\left|\widetilde{v}(\IX_r)\right|\Id
r\right].
$$
For any $n\in\IN$, let
\[
s\sigma_n:=\Big\{q=(q_1,\dots,q_n): 0< q_1<\dots< q_n<
s\Big\}\subset \IR^n
\]
denote the open scaled simplex. In the chain of equalities
\begin{align*}
&\mathbb{E}^x\left[\mathrm{e}^{\int^s_0\left|\widetilde{v}(\IX_r)\right|\Id
r}\right]=1+\sum^{\infty}_{n=1}(1/n!)\int_{[0,s]^n}\mathbb{E}^x\left[\left|\widetilde{v}(\IX_{q_1})\right|\dots
\left|\widetilde{v}
(\IX_{q_n})\right|\right]\Id^n q\\
&=1+\sum^{\infty}_{n=1}\int_{s\sigma_n}\mathbb{E}^x\left[\left|\widetilde{v}(\IX_{q_1})\right|\dots
\left|\widetilde{v}
(\IX_{q_n})\right|\right]\Id^n q\\
&=1+\sum^{\infty}_{n=1}\int^s_0\int^s_{q_1}\cdots\int^s_{q_{n-1}}\mathbb{E}^x\left[\left|\widetilde{v}(\IX_{q_1})\right|\dots
\left|\widetilde{v}
(\IX_{q_n})\right|\right]\Id^n q,
\end{align*}
the first one follows from Fubini\rq{}s theorem, and the second one from combining the fact that the integrand is symmetric in the variables $q_j$ with the fact that the number of orderings of a real-valued tuple of length $n$ is $n!$. In particular, by comparison with a geometric series, it is sufficient to prove that for all natural $n\geq 2$, one has 
\begin{align}
J_n(v,s)&:=
\sup_{x\in
M}\int^s_0\int^s_{q_1}\cdots\int^s_{q_{n-1}}\mathbb{E}^x\left[\left|\widetilde{v}(\IX_{q_1})\right|\dots
\left|\widetilde{v}
(\IX_{q_n})\right|\right]\Id^n q\nn\\
&\leq D(v,s) J_{n-1}(v,s).
\end{align}
But the Markov property of the family of Wiener measures implies
\begin{align}
J_n(v,s)&=\sup_{x\in M}\int^s_0\int^s_{q_1}\cdots\int^s_{q_{n-2}}\int_{\Omega_M}\left|\widetilde{v}(\gamma(q_1))\right|\dots
\left|\widetilde{v}(\gamma(q_{n-1}))\right|\times\nn\\
&\quad \times
\int_{\Omega_M}\int^{s-q_{n-1}}_0\left|\widetilde{v}(\omega(u))\right|
\Id u \ \Id\IP^{\gamma(q_{n-1})}(\omega)\Id\IP^{x}(\gamma)\Id^{n-1}q\nn\\
&\leq D(v,s) J_{n-1}(v,s),
\end{align}
\vspace{1.2mm}
which proves Khas\rq{}minskii\rq{}s lemma.\\
3. Let $v:M\to \IC$ be a Borel function which admits an $s>0$ with $D(v,s)<1$. Then for any $t>0$ and any such $s$, one has 
$$
J(v,t)\leq  \f{1}{1-D(v,s)} \mathrm{e}^{ \f{t}{s } 
\mathrm{log}\left(\f{1}{1-D(v,s)}\right) }.
$$
Proof:  Pick a large $n\in\IN$ with $t< (n+1)s$. Then the Markov
property of the family of Wiener measures and Khas\rq{}minskii\rq{}s lemma imply 
\begin{align}
J(v,t) &\leq J(v,(n+1) s)\nn\\
&=\sup_{x\in M}
\int_{\Omega_M}\mathrm{e}^{\int^{ns}_0\left|\widetilde{v}(\gamma(r))\right|
\Id
r}\int_{\Omega_M} \mathrm{e}^{\int^{s}_0\left|\widetilde{v}(\omega(r))\right|\Id
r}\Id\IP^{\gamma(ns)}(\omega)
\Id\IP^x(\gamma)\nn\\
&\leq \f{1}{1-D(v,s)} J(v,n s)  \nn\\
&=\f{1}{1-D(v,s)}\times\nn\\
&\>\>\>\>\>\times \sup_{x\in M}
\int_{\Omega_M}\mathrm{e}^{\int^{(n-1)s}_0\left|\widetilde{v}(\gamma(r))\right|
\Id
r}\int_{\Omega_M} \mathrm{e}^{\int^{s}_0\left|\widetilde{v}(\omega(r))\right|\Id
r}\Id\IP^{\gamma((n-1)s)}(\omega)
\Id\IP^x(\gamma)\nn\\
&\leq \dots\>\text{($n$-times)}\nn\\
&\leq \f{1}{1-D(v,s)}\left(\f{1}{1-D(v,s)}\right)^{n}  \nn\\
&\leq \f{1}{1-D(v,s)} \mathrm{e}^{ \f{t}{s} 
\mathrm{log}\left(\f{1}{1-D(v,s)}\right) }\nn,
\end{align}
which proves (\ref{way0}) in view of (\ref{way}).\\
Using the above observations, the actual statement of Lemma \ref{xdd} can be proved as follows: In case $w\in\dyn(M)$, there exists an $s>0$ with $D(w,s)<1$, such that the claim follows from step 3. In case $w\in\kat(M)$, for any $\delta>1$ there exists a $s_{\delta}>0$ with $D(w,s_{\delta})<1-1/\delta$, and again the claim follows from step 3. This completes the proof.
\end{proof}

The exponential estimate (\ref{way0}) will turn out to be the actual reason for $\IL^q$-smoothing properties of semigroups generated by Schrödinger operators with locally integrable potentials that have some negative part in $\dyn(M)$. There seems to be some belief that these $\IL^q$-smoothing results require the stronger Kato assumption on the negative part of the potential. The stronger estimate (\ref{way1}), however, really requires a Kato condition. The importance of this better estimate has been noted by D. Pallara and the author in \cite{pallara} in the context of a de-Giorgi-type heat kernel characterization of the Riemannian total variation. We will come back to this later on.

\begin{Notation} The $\IL^q$-norm with respect to a Borel measure on $M$ of the form $\Xi(y)\Id\mu(y)$ will be denoted by $\left\|\bullet\right\|_{q;\Xi}$, using the additional convention $\left\|\bullet\right\|_{q}=\left\|\bullet\right\|_{q;\Xi|_{\Xi\equiv 1}}$. The corresponding $\IL^q$-spaces are to be denoted by
$$
\IL^q_{\Xi}(M):= \IL^q(M, \Xi\Id\mu).
$$
\end{Notation}

Now we can record the following useful Kato criterion:

\begin{Proposition}\label{ecl} Let $w=w_1+w_2:M\to \IC$ be a function which can be decomposed into Borel functions $w_j:M\to \IC$ satisfying the following two properties: 
\begin{enumerate}
\item[$\bullet$] $w_2\in \IL^{\infty}(M)$
\item[$\bullet$] there exists a real number $ q\rq{} <\infty$ such that $q\rq{}\geq 1$ if $m=1$, and $q\rq{} > m/2$ if $m\geq 2$, and a heat kernel control pair $(\Xi,\tilde{\Xi})$, such that\footnote{Note that one automatically has $\IL^{q\rq{}}_{\Xi}(M)\subset \IL^{q\rq{}}(M)$, which is implied by $\inf \Xi>0$.} $w_1\in  \IL^{q\rq{}}_{\Xi}(M)$.
\end{enumerate}
Then for all $u>0$ and all $x\in M$, one has the bound
\begin{align}\label{hgg}
\int_M p(u,x,y)|w(y)|\Id \mu(y)\leq  \tilde{\Xi}(u)^{1/q\rq{}}\left\|w_1\right\|_{q\rq{};\Xi}+ \left\|w_2\right\|_{\infty}.
\end{align}
In particular, for any choice of $q\rq{}$ and $(\Xi,\tilde{\Xi})$ as above one has
$$
 L^{q\rq{}}_{\Xi}(M)+ L^{\infty}(M)\subset \kat(M).
$$
\end{Proposition}

\begin{proof}  Once we have proved
\begin{align}\label{ggdp}
\int_M p(u,x,y)|w(y)|\Id \mu(y)\leq  \tilde{\Xi}(u)^{1/q\rq{}}\left\|w_1\right\|_{q\rq{};\Xi}+ \left\|w_2\right\|_{\infty},
\end{align}
the inclusion $w\in \kat(M)$ clearly follows from
\begin{align*}
&\lim_{t\to 0+}\sup_{x\in M}\int^t_0 \int_M p(u,x,y)|w(y)|\Id \mu(y) \Id u\\
&\leq C(w_1) \lim_{t\to 0+}\int^t_0\tilde{\Xi}(u)^{1/q\rq{}}\Id u+C(w_2)\lim_{t\to 0+} t=0.
\end{align*}
In order to derive (\ref{ggdp}), note first that the inequality 
\begin{align}\label{gp}
\int_M p(u,x,y) \Id\mu(y)\leq 1
\end{align}
shows that we can assume $w_2=0$. Furthermore, the case $q\rq{}=1$ (which is only allowed for $m=1$) is obvious, so let us assume $m\geq 2$ and $q\rq{}>m/2$. The essential trick to bound $\int_M  p(u,x,y)|w_1(y)|\Id \mu(y)$ is to factor the heat kernel appropriately: Indeed, with $1/q\rq{}+1/q:=1$, Hölder\rq{}s inequality and using (\ref{gp}) once more gives us the following estimate:
\begin{align*}
&\int_M p(u,x,y)|w_1(y)|\Id \mu(y)=\int_M p(u,x,y)^{\f{1}{q}} p(u,x,y)^{1-\f{1}{q}}|w_1(y)|\Id \mu(y)\\
&\leq \left(\int_M p(u,x,y)\Id\mu(y)\right)^{\f{1}{q}}\left(\int_M |w_1(y)|^{q\rq{}}p(u,x,y)\Id\mu(y)\right)^{\f{1}{q\rq{}}}\\
&\leq \left(\int_M |w_1(y)|^{q\rq{}}\big(\tilde{\Xi}(u)\Xi(y)\big)\Id\mu(y)\right)^{\f{1}{q\rq{}}}\\
&\leq \tilde{\Xi}(u)^{1/q\rq{}} \left\|w_1\right\|_{q\rq{};\Xi}.
\end{align*}
This completes the proof.
\end{proof}

We immediately get the following corollary to Lemma \ref{xdd} b):

\begin{Corollary}\label{ffrwc} Under the assumptions of Proposition \ref{ecl}, for any $\delta>1$ there exists a constant 
$$
c_{\delta}=c_{\delta}( q\rq{},\left\|w_1\right\|_{q\rq{};\Xi}, \left\|w_2\right\|_{\infty})>0,
$$
which only depends on the indicated parameters, such that for all $t\geq 0$ one has
\begin{align*}
\sup_{x\in M} \mathbb{E}^x\left[\mathrm{e}^{\int^t_0 \left|w(\IX_s)\right|\Id
s}1_{\{t<\zeta\}}\right]
\leq \delta \mathrm{e}^{tc_{\delta}}. 
\end{align*}
\end{Corollary}

The striking fact about Corollary \ref{ffrwc} is that by picking a heat kernel control pair as in Remark \ref{ddghq}.1, one can force the assumptions and constants to depend only on data which are entirely of zeroth order with respect to the metric $g$ on $M$. Looking at the definition of the Kato class, which a priori involves the heat kernel in a full global form, this a surprising fact.\\
The following highly nontrivial localization result now becomes an immediate consequence of Proposition \ref{ecl} and the existence of heat kernel control pairs (Remark \ref{ddghq}.1):

\begin{Corollary} For any $q\rq{}$ as in Proposition \ref{ecl}, one has $ L^{q\rq{}}_{\loc}(M)\subset \kat_{\loc}(M)$.
\end{Corollary}

\begin{proof} Indeed, pick a heat kernel control pair $(\Xi,\tilde{\Xi})$ for $M$. Given a compact set $K\subset M$ and $w\in L^{q\rq{}}_{\loc}(M)$, one has
$$
\int_K |w|^{q\rq{}}  \Xi\Id\mu \leq \left(\max_{K}\Xi\right)\int_K |w|^{q\rq{}}  \Id\mu<\infty,
$$
since $\Xi$ is continuous, and thus $1_Kw\in \kat(M)$.
\end{proof}

\section{Specific results under some control on the geometry}

While the previous results are true on every Riemannian manifold, it might not come as a surprise that one can deduce \lq\lq{}finer\rq\rq{} global results under some global control on the geometry, in particular for the Kato class. We record two useful results of this type now. First, the following fact is an immediate consequence of Proposition \ref{ecl} and Remark \ref{ddghq}.2:

\begin{Corollary}\label{dhj} Assume that there exists $C,T>0$ such that for all $0<t < T$ one has
\begin{align}
\sup_{x \in M} p(t,x,x)\leq C t^{ -\f{m}{2} }.\label{ab1}
\end{align}
Then for any $q\rq{}$ as in Proposition \ref{ecl}, one has
\[
 L^{q\rq{}}(M)+ L^{\infty}(M)\subset \kat(M).
\]
\end{Corollary}

%\begin{proof} 
%We can assume $q<\infty$ in the proof. Using (\ref{self}), one sees that the on-diagonal estimate (\ref{ab1}) is equivalent to the off-diagonal estimate
%\begin{align}
%\sup_{x,y\in M} p(t,x,y)\leq C t^{ -\f{m}{2} },\>0<t\leq t_0.\label{ab}
%\end{align}
%Next, we remark that in view of Lemma  \ref{goof} it is sufficient to prove $ L^q(M)\subset \kat(M)$, so let $w\in  L^q(M)$, $0<t\leq t_0$, $x\in M$ and $1/q+1/q\rq{}=1$. Then using Hölder's inequality we get 
%\begin{align}
%\int^t_0\int_M p(s,x,y)|w(y)|\Id\mu(y) \Id s\leq  \left\|w\right\|_{q} \int^t_0 \left\|p(s,x,\bullet)\right\|_{q\rq{}} \Id s.\label{ab2}
%\end{align}
%Since (\ref{ab}) and (\ref{mar}) give  
%\[
%\left(\int_M p(s,x,y)^{q\rq{}-1}p(s,x,y)  \Id\mu(y)\right)^{\f{1}{q\rq{}}} \leq   C^{\f{1}{q}}s^{ -\f{m }{2q} },
%\]
%one has that (\ref{ab2}) is
%\[
% \leq \left\|w\right\|_{q} C^{\f{1}{q}} \int^t_0 s^{ -\f{m }{2q}  }\Id s,
%\]
%which tends to $0$, as $t\to 0+$.
%\end{proof}

Second, we examine Coulomb potentials in the context of the Kato class: In dimensions $\geq 3$, we saw that global Gaussian upper bounds imply that $M$ is nonparabolic, therefore the question arises whether the corresponding Coulomb potential $G(\bullet,y)=\int^{\infty}_0 p(t,\bullet,y)\Id t $
is in any of the classes $\dyn(M)$ or $\kat(M)$ for fixed $y$. Using Corollary \ref{dhj}, it turns out that this is the case in the physically relevant situation $m=3$:

\begin{Proposition}\label{k1} Let $m=3$ and assume that there exists a $c_1>0$ such that for all $t>0$ one has $\sup_{x\in M}p(t,x,x)\leq c_1t^{-\f{3}{2}}$. Then for any fixed $y\in M$ one has 
\begin{align}\label{decom}
G(\bullet,y)\in  L^2(M)+ L^{\infty}(M)\subset \kat(M).
\end{align}
Moreover, there is a universal constant $A>0$ such that for any $y\in M$, $r>0$, one has 
\begin{align}\label{fflo}
\sup_{x\in M} \int^{\infty}_0 \mathrm{e}^{-rs} \int_M p(s,x,z) G(z,y) \Id\mu(z) \Id s\leq \f{c_1A}{\sqrt{r}}.
\end{align}
\end{Proposition}

\begin{proof} Pick an $r<r_{\mathrm{Eucl}}(y,2)$ and a Euclidean coordinate system 
$$
\phi: \IB(y,r)\longrightarrow U\subset \IR^3
$$
with accuracy $2$. We write
$$
G(x,y)=1_{\IB(y,r/2)}G(x,y)+1_{M\setminus \IB(y,r/2)}G(x,y),\>x\in M.
$$
By Theorem \ref{nonp}, we have $G(x,y)\leq C \varrho^{-1}(x,y)$ for all $x\in M\setminus \{y\}$. It follows that $x\mapsto 1_{M\setminus \IB(y,r/2)}G(x,y)$ is bounded on $M$, and using Lemma \ref{ddvc} b) we have
\begin{align*}
&\int_{\IB(y,r/2)}G(x,y)^2\Id \mu(x)\leq C^2\int_{\IB(y,r/2)}\varrho(x,y)^{-2}\Id \mu(x)\\
&=C^2\int_{\IB(y,r/2)}\varrho(x,y)^{-2}\sqrt{\det (g_{ij})(x)} \Id x\\
&\leq 2C^2 \cdot\sup_{ \overline{\IB(y,r/2)}}\sqrt{\det (g_{ij})}\cdot\int_{\IB^{\IR^m}(0,\sqrt{2}r)}|x|^{-2} \Id x<\infty.
\end{align*}
This finishes the proof of (\ref{decom}). In order to see (\ref{fflo}), we use the Chapman-Kolmogorov identity and (\ref{miniself}), which leads to
\begin{align*}
&\int^{\infty}_0 \mathrm{e}^{-rs}\int_M p(s,x,z) G(z,y) \Id\mu(z) \Id  s\\
&= \int^{\infty}_0 \int^{\infty}_0 \mathrm{e}^{-rs} p(s+t,x,y) \Id t \ \Id  s\nn\\
&\leq  c_1 \int^{\infty}_0 \mathrm{e}^{-rs} \int^{\infty}_0 (s+t)^{-\f{3}{2}} \Id t \  \Id  s= c_1\int^{\infty}_0 \mathrm{e}^{-rs} \int^{\infty}_s  t^{-\f{3}{2}} \Id t \ \Id  s\nn\\
&\leq c_1 A\rq{}\int^{\infty}_0 \mathrm{e}^{-rs}  s^{-\f{1}{2}}  \Id  s=c_1A\rq{}\rq{}\int^{\infty}_0 \mathrm{e}^{-rs^2}   \Id  s  =: \f{c_1A}{\sqrt{r}},
\end{align*}
which finishes the proof.
\end{proof}

Note that, in view of Lemma \ref{char}, the bound (\ref{fflo}) also provides a direct proof of $G(\bullet,y)\in \kat(M)$.

\chapter{Foundations of covariant Schrödinger semigroups}\label{C7}

\section{Notation and preleminaries}

The following definitions will be very convenient in the sequel:

\begin{Definition} a) Let $E\to M$ be a smooth metric $\IC$-vector bundle. Then a Borel section $V:M\to\mathrm{End}(E)$ in $\mathrm{End}(E)\to M$ is called a \emph{potential} on $E\to M$, if one has $V(x)=V(x)^*$ for all $x\in M$. Here, $V(x)^*$ denotes the adjoint of the finite-dimensional linear operator\footnote{In case $E=M\times\IC^\ell\to M$ is a trivial vector bundle, then under the usual identifications $V$ is nothing but a pointwise self-adjoint map $V:M\to\mathrm{Mat}(\IC;\ell\times \ell)$.} $V(x):E_x\to E_x$ with respect to the fixed metric on $E\to M$.\\
b) Any triple $(E,\nabla,V)$ with $E\to M$ a smooth metric $\IC$-vector bundle, $\nabla$ a smooth metric covariant derivative on $E\to M$, and $V$ a potential on $E\to M$ will be written as $(E,\nabla,V)\to M$ and called a \emph{covariant Schrödinger bundle} over $M$.
\end{Definition}

Every covariant Schrödinger bundle $(E,\nabla,V)\to M$ gives rise to the symmetric sesquilinear form $\Q_V$ in $\Gamma_{\IL^2}(M,E)$ given by
\begin{align*}
&\dom(\Q_V)=\Big\{f\in \Gamma_{\IL^2}(M,E): \int_M |(Vf,f)|\Id \mu<\infty\Big\}\\
&Q_V(f_1,f_2)=\int_M (Vf_1,f_2) \Id\mu.
\end{align*}
Note that potentials which agree $\mu$-a.e. give rise to the same sesqulinear form. If $|V|\in\IL^1_{\loc}(M)$, then $Q_V$ is densely defined, for then the domain contains $\Gamma_{\ICC_{\c}}(M,E)$. Furthermore, if $V\geq C$ for some $C\in\IR$, in the sense that for all $x\in M$ all eigenvalues of $V(x)$ are bounded from below by $C$, then $Q_V$ is closed: This follows from the lower-semicontinuity characterization of the closedness of a semibounded sesquilinear form (cf. Appendix \ref{selff}) and Fatou's lemma.

\begin{Definition} Let $(E,\nabla,V)\to M$ be a covariant Schrödinger bundle, and assume that $V$ admits a decomposition 
$V=V_+-V_-$ into potentials $V_{\pm}$ on $E\to M$ with $V_{\pm}\geq 0$, such that $|V_+|\in\IL^1_{\loc}(M)$ and $\Q_{V_-}$ is $\Q^{\nabla}$-bounded with bound\footnote{See the Definition \ref{bounded} in the appendix.} $<1$. In this case, let $\H^{\nabla}_V$ denote the semibounded self-adjoint operator in $\Gamma_{\IL^2}(M,E)$ which corresponds to the closed, symmetric, semibounded, densely defined sesquilinear form\footnote{See Theorem \ref{kaq2} in the appendix.} 
$$
\Q^{\nabla}_V:= \Q^{\nabla}+\Q_V=\Q^{\nabla}+\Q_{V_+}-\Q_{V_-}.
$$
The operator $\H^{\nabla}_V$ is called the \emph{covariant Schrödinger operator} induced by $(E,\nabla,V)\to M$, and 
$$
\left(\mathrm{e}^{-t \H^{\nabla}_V}\right)_{t\geq 0}\subset \ILL\left(\Gamma_{\IL^2}(M,E)\right)
$$
is called its \emph{covariant Schrödinger semigroup}.
\end{Definition}

Note that the (obviously symmetric) form $\Q^{\nabla}_V$ indeed has the asserted properties: Firstly, it is densely defined, for by definition we have
\begin{align*}
\dom(\Q^{\nabla}_V)&:=\dom(\Q^{\nabla})\cap \dom(\Q_{V})\\
&=
\dom(\Q^{\nabla})\cap \dom(\Q_{V_+})\cap\dom(\Q_{V_-}),
\end{align*}
and by the assumption on $V_-$, the above set is 
$$
=\dom(\Q^{\nabla})\cap \dom(\Q_{V_+}).
$$
Thus, since $V_+$ is assumed to be locally integrable, it follows that 
$$
\Gamma_{\ICC_{\c}}(M,E)\subset\dom(\Q^{\nabla}_V).
$$
Secondly, the form $\Q^{\nabla}-\Q_{V_-}$ is closed and semibounded by the KLMN theorem (cf. appendix, Theorem \ref{klmn}), and $Q^{\nabla}_V$ is the sum of that form and the closed nonnegative form $Q_{V_+}$. Therefore, $Q^{\nabla}_V$ is again semibounded and closed.\vspace{1mm}

Somewhat more explicitly, we have
\begin{align}\label{ddas}
\dom(\Q^{\nabla}_V)=\Gamma_{ W^{1,2}_{\nabla,0}}(M,E)\cap \dom(\Q_{V_+}),
\end{align}
with
\begin{align*}
\Q^{\nabla}_V(f_1,f_1)&=\Q^{\nabla}(f_1,f_2)+\Q_{V_+}(f_1,f_2)-\Q_{V_-}(f_1,f_2)\\
&=(1/2)\int_M (\nabla f_1,\nabla f_2)\Id \mu+\int_M (V_+ f_1, f_2)\Id \mu-\int_M (V_- f_1,  f_2)\Id \mu,
\end{align*}
and by an abstract functional analytic fact (cf. the second property in Theorem \ref{kaq2} in the appendix), it holds that $\dom(\H^{\nabla}_V)$ is precisely the space of $f\in\dom(\Q^{\nabla}_V)$ for which there exists $h\in\Gamma_{\IL^2}(M,E)$ such that for all $\psi\in\dom(\Q^{\nabla}_V)$ one has
\begin{align}\label{frf}
\left\langle h, \psi\right\rangle=\Q^{\nabla}_V(f,\psi),\>\text{ and then }\>\H^{\nabla}_Vf=h.
\end{align}
\vspace{1.2mm}

We recall that for every fixed $f\in\Gamma_{\IL^2}(M,E)$, the path 
$$
[0,\infty)\ni t\longmapsto \mathrm{e}^{-t \H^{\nabla}_V}f\in \Gamma_{\IL^2}(M,E)
$$
is the uniquely determined continuous map
$$
[0,\infty)\longrightarrow \Gamma_{\IL^2}(M,E)
$$
which is (norm) $C^1$ in $(0,\infty)$ with values in $\dom(H^{\nabla}_V)$, and which satisfies the abstract heat equation
$$
(\Id/\Id t )\mathrm{e}^{-t \H^{\nabla}_V}f=-H^{\nabla}_V \mathrm{e}^{-t \H^{\nabla}_V}f,\quad t>0,
$$
subject to the initial condition $\mathrm{e}^{-t \H^{\nabla}_V}f|_{t=0} = f$.

\begin{Definition} Let $E\to M$ be a smooth metric $\IC$-vector bundle. A potential $V$ on $E\to M$ is called \emph{contractively Dynkin decomposable} or in short \emph{$\dyn$-decomposable} (respectively \emph{Kato decomposable} or in short \emph{$\kat$-decomposable}), if there are potentials $V_{\pm}$ on $E\to M$ with $V_{\pm}\geq 0$ such that $V=V_+-V_-$, $|V_+|\in\IL^1_{\loc}(M)$, and $|V_-|\in\dyn(M)$ (respectively $|V_-|\in\kat(M)$). In this case, $V=V_+-V_-$ is called a \emph{contractive Dynkin decomposition} or in short a \emph{$\dyn$-decomposition} of $V$ (respectively a \emph{Kato decomposition} or in short a \emph{$\kat$-decomposition} of $V$).
\end{Definition}

We remark that any $\dyn$-decomposable potential is automatically locally integrable by Lemma \ref{goof}, and any $\kat$-decomposable potential is trivially $\dyn$-decomposable. Furthermore, using a previously established result on first order Sobolev spaces, we immediately get: 

\begin{Lemma}\label{well} Let $(E,\nabla,V)\to M$ be a covariant Schrödinger bundle with $V$ being $\dyn$-decomposable. Then for every $\dyn$-decomposition $V=V_+-V_-$, the form $\Q_{V_-}$ is $\Q^{\nabla}$-bounded with bound $<1$. In particular, $H^{\nabla}_V$ is well-defined. If $V$ is even $\kat$-decomposable with $\kat$-decomposition $V=V_+-V_-$, then $\Q_{V_-}$ is infinitesimally $\Q^{\nabla}$-bounded.
\end{Lemma}

\begin{proof} This follows immediately from combining the corresponding scalar result (\ref{wiwaldi}) with Lemma \ref{bett2} b): Indeed, translating the estimate (\ref{absch}) into the present situation implies that for every $f\in \dom(Q^{\nabla})$, one has $|f|\in \dom(Q)$ and $Q^{\nabla}(f,f)\geq Q(|f|,|f|)$. 
\end{proof}

Being equipped with the scalar result from Lemma \ref{goof}, the main ingredient of the above proof was to establish the inequality $Q^{\nabla}(f,f)\geq Q(|f|,|f|)$ for all $f\in \dom(Q^{\nabla})$. This inequality can be established with completely different methods, too: One can use a covariant Feynman-Kac formula for $\mathrm{e}^{-tH^{\nabla}}$ to establish the bound
$$
\left\langle \mathrm{e}^{-tH^{\nabla}}f,f \right\rangle\leq \left\langle \mathrm{e}^{-tH}|f|,|f| \right\rangle\quad f\in\Gamma_{\IL^2}(M,E),\>\quad t\geq 0, 
$$
which by an abstract functional analytic fact (cf. appendix, Theorem \ref{kaq2}) immediately implies $Q^{\nabla}(f,f)\geq Q(|f|,|f|)$ for all $f\in \dom(Q^{\nabla})$. This is the path that we followed in \cite{G1}, where Lemma \ref{well} stems from. Since there is no ad hoc way to establish a covariant Feynman-Kac formula for $\mathrm{e}^{-tH^{\nabla}}$, however, we believe that Lemma \ref{bett2} b) provides a more elementary approach.

\section[Scalar Schrödinger semigroups and the FK-formula]{Scalar Schrödinger semigroups and the Feynman-Kac (FK) formula}

\section{Kato-Simon inequality}\label{katos}

\chapter{Compactness of $V(H^{\nabla}+1)^{-1}$}\label{aiosys}

\chapter{$\IL^q$-properties of covariant Schrödinger semigroups}\label{C9}

\chapter[Continuity properties of covariant Schrödinger semigroups]{Continuity properties of covariant Schrödinger semigroups}\label{C10}

\chapter[Integral kernels]{Integral kernels for covariant Schrödinger semigroups}

\chapter[Essential self-adjointness]{Essential self-adjointness of covariant Schrödinger operators}\label{estt}

\chapter{Smooth compactly supported sections as form core}\label{C12}

\chapter{Applications}\label{C13}

\section{Hamilton operators corresponding to magnetic fields}\label{exit}

\section{Dirac operators}

\section[Hydrogen-type stability problems on $3$-folds]{Nonrelativistic hydrogen-type stability problems on Riemannian $3$-folds}

\section{Riemannian total variation}\label{tota}

\section{$\mathscr{C}$-positivity preservation}

\appendix

%\chapter{Vector bundles and principal bundles}

\chapter{Smooth manifolds and vector bundles}\label{difftop}

In this section, we recall some basic facts about smooth manifolds, closely following \cite{lee}. In the sequel, all manifolds are understood to be without boundary, unless otherwise stated. Let $m\in \IN_{\geq 1}$ and let $X$ be a (topological) $m$-manifold. By definition, this means that $X$ is a second countable Hausdorff space which locally looks like $\IR^m$, in the sense that for every $x\in X$ there exists on open neighbourhood $U$ of $x$, an open subset $V\subset \IR^m$ and a homeomorphism 
$$
\varphi=(x^1,\dots,x^m):U\longrightarrow V.
$$
Such a map $\varphi$ is called a \emph{chart} or a \emph{coordinate system} for $X$. Alternatively, in the above situation one calls $X$ a manifold and calls $m$ the \emph{dimension} of $X$ and writes $\dim(X)=m$, noting that this number is uniquely determined (cf. Theorem 1.2 in \cite{lee}). It then follows that $X$ is locally compact and metrizable (and therefore paracompact).\\
There is no danger in denoting a chart for $X$ as above by $(\varphi,U)$ or simply by its component functions $((x^1,\dots,x^m), U)$. Given another chart $(\varphi\rq{},U\rq{})$ for $X$, the charts $(\varphi,U)$ and $(\varphi\rq{},U\rq{})$ are called \emph{smoothly compatible}, if either $U\cap U\rq{} =\emptyset$ or if the transition map
$$
\varphi\rq{}\circ \varphi^{-1}: \varphi (U\cap U\rq{})\longrightarrow \varphi\rq{} (U\cap U\rq{}) 
$$
is smooth. \vspace{1mm}

A \emph{smoothly compatible atlas} for $X$ is a collection $\mathcal{A}=(\varphi_{\alpha},U_{\alpha})_{\alpha\in A}$ of charts for $X$, such that $(U_{\alpha})_{\alpha}$ is a cover of $X$ and such that for all $\alpha,\alpha\rq{}\in A$ the charts $(\varphi_{\alpha},U_{\alpha})$ and $(\varphi_{\alpha\rq{}},U_{\alpha\rq{}})$ are smoothly compatible.  \vspace{1mm}
A \emph{smooth structure} on $X$ is a smoothly compatible atlas $\mathcal{A}$ for $X$, such that for every chart $(\varphi,U)$ for $X$ which is smoothly compatible with every element of $\mathcal{A}$, one already has  $(\varphi,U)\in\mathcal{A}$. (Of course such a structure does not necessarily need to exist at all.) It follows easily that every smoothly compatible atlas $\mathcal{A}$ for $X$ is contained in a unique smooth structure on $X$ (cf. Proposition 1.17 in \cite{lee}), which will be called the \emph{smooth structure induced by $\mathcal{A}$}.  \vspace{1mm}

A \emph{smooth $m$-manifold}, or a \emph{smooth manifold of dimension $m$, $\dim(X)=m$}, is defined to be a pair $(X,\mathcal{A})$ given by an $m$-manifold $X$ and a smooth structure $\mathcal{A}$ on $X$. Whenever there is no danger of confusion, one omits the smooth structure in the notation and simply calls $X$ a smooth $m$-manifold in the above situation.  \\
When one says that $(\varphi,U)$ is a \emph{smooth chart} for the smooth $m$-manifold $X$, this means that $(\varphi,U)$ is contained in the underlying smooth structure. Likewise, when one says that $\mathcal{A}$ is a \emph{smooth atlas} for the smooth $m$-manifold $X$, this means that $\mathcal{A}$ is contained in the underlying smooth structure on $X$. Let us give some standard examples in order to show how these definitions work:

\begin{Examplee} 1. Consider $X:=\IR^m$ with its standard Euclidean topology. Then the single map $\mathcal{A}_{\IR^m}:=\{(\mathrm{id}_{\IR^m}, \IR^m)\}$ clearly is a smoothly compatible atlas for $X$. The induced smooth structure is called \emph{standard smooth structure on $\IR^m$}. If nothing else is said, $\IR^m$ will always be equipped with this smooth structure. The system of open unit balls $\{(\mathrm{id}_{B_1(x)},B_1(x)): x\in \IR^m\}$ also induces the standard smooth structure on $\IR^m$, while for example the single map $\{(\phi,\IR^1 )\}$ with $\phi(x)=x^3$ induces a smooth structure which is different from the standard smooth structure on $\IR^1$ (for $\mathrm{id}_{\IR^m}\circ \phi^{-1}(x)=x^{1/3}$ is not smooth at $x=0$).\\
2. Consider the standard $m$-sphere $\mathbb{S}^m\subset\IR^{m+1}$ with its subspace topology, and with $N,S\in \mathbb{S}^m$ the north pole and southpole. Let 
$$
\phi_{N}:\mathbb{S}^m\setminus \{N\}\longrightarrow \IR^m,\quad \phi_{S}:\mathbb{S}^m\setminus \{S\}\longrightarrow \IR^m
$$
denote the stereographic projection maps. Then the smoothly compatible atlas
$$
\mathcal{A}_{\mathbb{S}^m}:=\{(\phi_{N},\mathbb{S}^m\setminus \{N\}),(\phi_{S},\mathbb{S}^m\setminus \{S\})\}
$$
induces the \emph{standard smooth structure on $\mathbb{S}^m$}.\\
3. Every $m$-dimensional $\IR$-linear space canonically becomes a smooth $m$-manifold (cf. Example 1.24 in \cite{lee}): Picking a basis $f_1,\dots,f_m\in V$, the isomorphism 
$$
\phi: V\longrightarrow \IR^m, \quad \sum^m_{j=1} x^j f_j\longmapsto (x^1,\dots, x^m)
$$
induces a smooth structure.\\
4. By the implicit function theorem, level sets of regular smooth maps $\phi:U\to \IR$, where $U\subset \IR^m$ is open, canonically become smooth $(m-1)$-manifolds (cf. Example 1.32 in \cite{lee}).\\
5. Graphs of smooth maps $\phi:U\to\IR^k$, where $U\subset \IR^m$ is open, canonically become smooth $m$-manifolds (cf. Example 1.32 in \cite{lee}).
\end{Examplee}

It follows that every smooth $m$-manifold admits a smooth countable and locally finite atlas such that each chart domain is relatively compact. Using topological dimension theory, one finds that every smooth $m$-manifold admits a \emph{finite} smooth atlas (cf. p. 43 in \cite{greub}), where of course the underlying chart domains cannot be chosen to be relatively compact in general. \vspace{1mm}

The following result is very useful, for example, in order to obtain new smooth manifolds from old ones (cf. Lemma 1.3.5 in \cite{lee}):

\begin{Propositione}\label{aspoqqpo} Assume that $Y$ is a set and that, for some index set $B$, 
$$
\psi_{\beta}:W_{\beta} \longrightarrow \tilde{W}_{\beta}, \quad  \beta\in B,
$$
is a collection of bijective maps such that 
\begin{itemize} \item $W_{\beta}$ is a subset of $Y$ and $\tilde{W}_{\beta}$ is an open subset of $\IR^m$ for all $\beta\in B$
\item $\psi_{\beta}(W_{\beta}\cap W_{\alpha})$ and $\psi_{\alpha}(W_{\beta}\cap W_{\alpha})$ are open in $\IR^m$ for all $\beta,\alpha\in B$
\item for all $\beta,\alpha\in B$ with $W_{\beta}\cap W_{\alpha}\neq \emptyset$, the map
$$
\psi_{\alpha}\circ \psi_{\beta}^{-1}:\psi_{\beta}(W_{\beta}\cap W_{\alpha})\longrightarrow \psi_{\alpha}(W_{\beta}\cap W_{\alpha})
$$
is smooth
\item countably many $W_{\beta}$\rq{}s cover $X$
\item whenever $x$ and $y$ are two distinct points in $Y$, either there exists some $W_{\beta}$ containing $x$ and $y$, or there exist two disjoint sets $W_{\alpha}$, $W_{\beta}$ with $x\in W_{\alpha}$ and $y\in  W_{\beta}$. 
\end{itemize}
Then there exists a unique topology on $Y$ making it an $m$-manifold and a unique smooth structure on $Y$ such that $((\psi_{\beta},W_{\beta}))_{\beta\in B}$ becomes a smooth atlas for $Y$.
\end{Propositione}

From this result, it follows that an open subset $U$ of a smooth $m$-manifold $X$ canonically becomes a smooth $m$-manifold itself: Indeed, let $(\varphi_{\alpha},U_{\alpha})_{\alpha\in A}$ be a smooth atlas for $X$. Then we may apply Proposition \ref{aspoqqpo} with  
\begin{align*}
&B:=A,\quad
W_{\alpha}:=U\cap U_{\alpha},\quad \tilde{W}_{\beta}:=\varphi_{\alpha}(U\cap U_{\alpha}),\\
&\psi_{\alpha}(x):=\varphi_{\alpha},\quad x\in U\cap U_{\alpha}.
\end{align*}

Furthermore, if in addition $X\rq{}$ is a smooth
 $m\rq{}$-manifold, then $X\times X\rq{}$ canonically becomes a smooth $(m+m\rq{})$-manifold: Let $(\varphi_{\alpha\rq{}}\rq{},U_{\alpha\rq{}}\rq{})_{\alpha\rq{}\in A\rq{}}$ be a smooth atlas for $X\rq{}$. Then we may apply Proposition \ref{aspoqqpo} with 
\begin{align*}
&B:=A\times A\rq{},\quad
W_{(\alpha,\alpha\rq{})}:= U_{\alpha}\times U_{\alpha\rq{}}\rq{} ,\quad \tilde{W}_{(\alpha,\alpha\rq{})}:=\varphi_{\alpha}\times\varphi_{\alpha\rq{}}\rq{}(U_{\alpha}\times U_{\alpha\rq{}}\rq{} ),\\
&\psi_{(\alpha,\alpha\rq{})}(x,x\rq{}):=(\varphi_{\alpha}(x),\varphi_{\alpha\rq{}}\rq{}(x\rq{})), \quad (x,x\rq{})\in U_{\alpha}\times U_{\alpha\rq{}}\rq{}.
\end{align*}
 
For example, the torus 
$$
\mathbb{T}^m=\underbrace{\mathbb{S}^1\times\cdots\times \mathbb{S}^1}_{\text{$m$-times}}
$$
 canonically becomes a smooth $m$-manifold this way.\\
For the rest of this section, let $X$ be a smooth $m$-dimensional manifold. Given another smooth
 $m\rq{}$-manifold $X\rq{}$, a map $f:X\to X\rq{}$ is called \emph{smooth}, if for every $x\in X$ there exists a chart $(\varphi,U)$ for $X$ with $x\in U$ and a chart $(\varphi\rq{},U\rq{})$ for $X\rq{}$ with $f(U)\subset  U\rq{}$ such that the map
$$
 \varphi(U)\longrightarrow \varphi\rq{}(U\rq{}),\quad x\longmapsto \varphi\rq{}\circ f\circ \varphi^{-1}(x)
$$
is smooth. This is obviously a local property which implies continuity. One writes $C^{\infty}(X,X\rq{})$ for the set of smooth functions from $X$ to $X\rq{}$. A map $f\in C^{\infty}(X,X\rq{})$ is called a \emph{smooth diffeomorphism}, if it is bijective such that $f^{-1}$ is smooth, too. The space $C^{\infty}(X,\IK)$ becomes an (associative and commutative) $\IK$-algebra under the pointwise defined multiplication and addition for $\IK=\IR$ and $\IK=\IC$ (identified with the smooth $2$-manifold $\IR^2$). In particular, this structure canonically induces a ring structure on $C^{\infty}(X,\IK)$, too. Note also that smooth charts are, by definition, smooth maps in the sense of the above definition.\vspace{1mm}

Given any open cover $(U_{\alpha})_{\alpha\in A}$ of $X$, a \emph{smooth partition of unity subordinate to $(U_{\alpha})_{\alpha\in A}$} will be understood to be a collection of smooth maps $(\varphi_{\alpha})_{\alpha\in A}\subset \ICC(X,\IR)$ such that
\begin{itemize}
\item for each $\alpha\in A$ one has $0\leq\phi_{\alpha}\leq 1$ and $\mathrm{supp}(\phi_{\alpha})\subset U_{\alpha}$
\item $(\supp(\phi_{\alpha}))_{\alpha\in A}$ is a locally finite collection of sets
\item for each $x\in X$ one has $\sum_{\alpha\in A}\phi_{\alpha}(x)=1$ (note that this is automatically a finite sum due to the previous assumption).
\end{itemize}

The following important result (cf. Theorem 2.23 in \cite{lee}) allows to patch certain local results to global ones:

\begin{Propositione}\label{tde}
For every open cover of $X$, there exists a smooth partition of unity subordinate to it.
\end{Propositione}

Let us now turn to vector bundles, referring the reader, e.g., to \cite{koba} for a detailed study of the subject.\\
A smooth surjective map $\pi:E \rightarrow X$ from a smooth manifold $E$ to $X$ is called a \emph{smooth $\IK$-vector bundle over $X$ with rank $\ell$} (where $\IK=\IC, \IR$), if 
\begin{itemize}
 \item each fiber $E_x:=\pi^{-1}(\{x\})$ is an $\ell$-dimensional $\IK$-vector space
\item  for each $x_0\in X$ there is an open neighbourhood $U\subset X$ of $x_0$ which admits a smooth frame  
 $e_1, \dots, e_\ell:U\to E$.\footnote{The statement that the $e_j$'s form a frame for $E\to X$ means that each $e_j$ is a map such that $e_j(x)\in E_x$ and $e_1(x),\ldots,e_\ell(x)$ is a basis for $E_x$, for every $x\in U$.} 
\end{itemize}
Alternatively, in the above situation one calls $\pi:E\to X$ a smooth vector bundle, and the number $\mathrm{rank}(E):=\ell$ is called the \emph{rank of $\pi:E\to X$}. Note that $\dim(E)=m+\rank(E)$. We will usually ommit the map $\pi$ in the notation and simply denote the vector bundle by $E\to X$.\vspace{1mm}

 The simplest example of a smooth $\IK$-vector bundle over $X$ of rank $\ell$ is provided by the \emph{trivial vector bundle of rank $\ell$}: Here, one sets
$$
E:=X\times \IK^{\ell}\longrightarrow X, \quad \pi(x,v):=x.
$$
In this case, the fibers are given by the $\IK$-linear spaces $E_x=\{x\}\times \IK^{\ell}$, and a globally defined smooth frame is thus given by $e_j(x):=(x,e_j)$, $j=1,\dots,\ell$, where the $e_j$\rq{}s are the standard basis of $\IK^{\ell}$. As we will see later on, the tangent space of $X$ provides another example (which is not trivial in the sense that in general one does not have a globally defined smooth frame).\vspace{1mm}

Let $E\to X$ be an arbitrary smooth $\IK$-vector bundle of rank $k$. A \emph{section $f$ of $ E\to X$ over a subset $U\subset X$} is nothing but a map $f:U\to E$ such that $f(x)\in E_x$ for all $x$. Spaces of sections of $E\to X$ over $U$ are denoted by $\Gamma_{*}(U,E)$, where for example $*$ can stand for $\ICC$ (smooth), $C$ (continuous), $\ICC_{\c}$ (smooth and compactly supported\footnote{Since each fiber $E_x$ has its zero vector, there is an obvious way to say that a continuous section $f:X\to E$ has a compact support.}), and so on. Note that $\Gamma_{\ICC}(U,E)$ is a (left-) $\ICC(U,\IK )$-module via the pointwise defined operations 
\begin{align*}
&(\psi_1+\psi_2)(x):= \psi_1(x)+\psi_2(x),\\
&(f\psi_1)(x):=f(x)\psi_1(x),\quad \psi_1,\psi_2\in \Gamma_{\ICC}(U,E),\quad f\in \ICC(U,\IK ) .
\end{align*}

Let $F \rightarrow X$ be a smooth $\IK$-vector bundle with rank $l$. A \emph{morphism of smooth $\IK$-vector bundles} $f:E\to F $ \emph{over $X$} is a smooth map such that 
$$
f(x):E_x\longrightarrow F_x\quad\text{ $\IK$-linearly, for all $x \in X$.} 
$$
Such an $f$ is called an \emph{isomorphism of smooth $\IK$-vector bundles over $X$}, if it is bijective and $f^{-1}$ is a morphism of smooth $\IK$-vector bundles, too. \\
In the spirit of Proposition \ref{aspoqqpo}, there is a recipe for the construction of vector bundles:

\begin{Propositione}\label{vb} Assume that $\tilde{E}=\bigcup_{x\in X} \tilde{E}_x$ is a disjoint family of $\tilde{k}$-dimensional $\IK$-linear spaces and that $(W_{\alpha})_{\alpha\in B}$ is an open cover of $X$ such that 
\begin{itemize}
\item for every $\alpha \in B$ there exist maps
\begin{align}\label{rahmen}
\tilde{e}^{(\alpha)}_1, \dots, \tilde{e}_{\tilde{k}}^{(\alpha)}:W_{\alpha}\longrightarrow  \tilde{E}\quad\text{ with $\tilde{e}^{(\alpha)}_j(x)\in \tilde{E}_x$ for all $x\in W_{\alpha}$  }
\end{align}
which are pointwise linearly independent
\item for every $\alpha,\beta\in B$ with $W_{\alpha}\cap W_{\beta}\neq \emptyset$ the transition map
$$
W_{\alpha}\cap W_{\beta}\longrightarrow GL(\IK,\tilde{k}),\quad x\longmapsto T_{\alpha,\beta}(x)
$$
is smooth, where $T_{\alpha,\beta}(x)=(T_{\alpha,\beta}(x))_{ij}$ denotes the invertible $\tilde{k}\times \tilde{k}$ matrix which is defined by the change of base  
$$
\{\tilde{e}^{(\alpha)}_1(x), \dots, \tilde{e}_{\tilde{k}}^{(\alpha)}(x)\}\leadsto \{\tilde{e}^{(\beta)}_1(x), \dots, \tilde{e}_{\tilde{k}}^{(\beta)}(x)\}.
$$   
\end{itemize}
Then there is a unique topology on $\tilde{E}$ such that $\tilde{E}$ becomes an $(m+\tilde{k})$-manifold, and there is a unique smooth structure on $\tilde{E}$ such that the canonically given surjective map $\tilde{E}\to X$ becomes a smooth $\IK$-vector bundle over $X$ with rank $\tilde{k}$ in such a way that the maps (\ref{rahmen}) are smooth frames.
\end{Propositione}

Again, this result allows to construct new vector bundles from old ones: Let $(\varphi_{\alpha}, U_{\alpha})_{\alpha\in A}$ be an open cover of $X$ such that for every $\alpha\in A$ there exist frames
\begin{align*}
e^{(\alpha)}_1, \dots, e_k^{(\alpha)}:U_{\alpha}\longrightarrow  E,\quad f^{(\alpha)}_1, \dots, f_{l}^{(\alpha) }:U_{\alpha}\longrightarrow  F,
\end{align*}
a situation that can always be achieved. For example, the \emph{dual bundle}
$$
E^*= \bigcup_{x\in X}E_x^*\longrightarrow X
$$
becomes a smooth $\IK$-vector bundle of rank $k$ by setting 
\begin{align*}
\tilde{E_x}=E^*_x,\quad B:=A,\quad  W_{\alpha}=U_{\alpha},\quad \tilde{e}^{(\alpha)}_j:= (e^{(\alpha)}_j)^* \quad\text{(dual basis)}.
\end{align*}

In a complete analogy, the following smooth $\IK$-vector bundles are constructed\footnote{Of course, all algebraic operations are understood over $\IK$ here.}:
\begin{align*}
&E\oplus F=\bigcup_{x\in X}E_x\oplus F_x\longrightarrow X\quad\text{(Whitney sum),}\\
& E\otimes F=\bigcup_{x\in X}E_x\otimes F_x\longrightarrow X\quad\text{(tensor product bundle),}\\
&E\odot E=\bigcup_{x\in X}E_x\odot E_x\longrightarrow X\quad\text{(symmetric tensor product bundle),}\\
& E\wedge E=\bigcup_{x\in X}E_x\wedge E_x\longrightarrow X\quad\text{(antisymmetric tensor product bundle),}\\
&E^* \boxtimes E=\bigcup_{(x,y)\in X\times X} E_x^* \boxtimes E_y=\bigcup_{x\in X}\mathrm{Hom}(E_y,E_x) \longrightarrow X\times X.
\end{align*}

We will use the standard notation
\begin{align*}
&\wedge^j E=\underbrace{E\wedge \dots \wedge E}_{\text{$j$-times}}\longrightarrow X\quad\text{($j$-fold exterior product bundle),}\\
&\wedge E =\bigoplus^{\ell}_{j=1}\wedge^k E\longrightarrow X\quad\text{(exterior algebra bundle)}.
\end{align*}
Note that the smooth $\IK$-vector bundle
$$
\mathrm{End}(E)=\bigcup_{x\in X}\mathrm{End}(E_x)\longrightarrow X
$$
of endomorphisms on $E\to X$ is well-defined by the previous constructions, since 
$$
E_x\otimes E^*_x=\mathrm{End}(E_x)\quad\text{ for all $x\in X$.}
$$
Note also that Proposition \ref{vb} can be used to define the complexification 
$$
E_{\IC}=\bigcup_{x\in X} E_x\otimes_{\IR} \IC\longrightarrow X\text{ if $\IK=\IR$.}
$$
Serre-Swan's theorem \cite{serre} states that for every smooth $\IK$-vector bundle over $X$ there exists another smooth $\IK$-vector bundle over $X$ such that their Whitney sum is isomorphic to some trivial smooth vector bundle of the form $X\times \IK^r\to X$.\vspace{1mm} 

Another application of Proposition \ref{vb} is to provide a simple construction of the \emph{tangent bundle} $TX\to X$ of $X$. This prototype of a smooth vector bundle is defined as follows: For every $x\in X$ define $T_x X$ to be the $\IR$-linear space of $\IR$-derivations of $C^{\infty}(X,\IR)$ at $x$. In other words, $T_x X$ is given by all $\IR$-linear maps $A:C^{\infty}(X,\IR)\to \IR$ which satisfy the Leibniz rule
$$
A(f_1f_2)=f_2(x)Af_1+f_1(x)Af_2,\quad\text{ for all $f_1,f_2\in C^{\infty}(X,\IR)$.}
$$

Then $T_x X$ becomes an $m$-dimensional $\IR$-linear space. In fact, if 
$$
\varphi=(x^1,\dots,x^m): U\longrightarrow \tilde{U}
$$ 
is a smooth chart for $X$ with $x\in U$, then a basis of a $T_xX$ is given by the derivations at $x$ defined by 
\begin{align}\label{andddd}
C^{\infty}(X,\IR)\ni f\longmapsto \frac{\partial f}{\partial x^{j}}(x):=\partial_j(f\circ \varphi^{-1})(x)\in\IR\quad j=1,\dots, m.
\end{align}

For example, a smooth curve $\gamma: I\to M$ (with $I\subset \IR$ an interval) defines for each $t\in I$ an element $\dot{\gamma}(t)\in T_{\gamma(t)}X$ by means of
$$
C^{\infty}(X,\IR)\ni f\longmapsto \dot{\gamma}(t)f:= (d/dt )f(\gamma(t))\in \IR.
$$

Then the tangent bundle of $X$ is defined as follows:
$$
TX:=\bigcup_{x\in X} T_x X\longrightarrow X.
$$
In view of (\ref{andddd}), if one picks a smooth atlas for $X$, one can apply Proposition \ref{vb} in the obvious way to conclude that $TX\to X$ is a smooth $\IR$-vector bundle over $X$ with rank $m$, so that the maps 
$$
\frac{\partial  }{\partial x^{j}}:U\to TX
$$
induced by (\ref{andddd}) become frames. Given another smooth $m'$-manifold $X'$ and a smooth map $\Psi:X\to X'$, the tangent map
$$
T\Psi:TX\longrightarrow TX'
$$
is the uniquely determined smooth map such that for all $x\in X$ one has  
$$
T\Psi(x):TX\longrightarrow TX'\quad\text{$\IR$-linearly}
$$
and for $A\in T_xM$, $T\Psi(x)A$ is the derivation at $\Psi(x)$ given by $f\mapsto A(f\circ \Psi)$. Such an $\Psi$ is called a \emph{smooth embedding}, if $\Psi$ is an injective homeomorphism onto its image such that $T\Psi(x)$ has a full rank for all $x\in M$. By Whitney\rq{}s embedding theorem (cf. Theorem 6.15 in \cite{lee} for a detailed proof), $X$ can be smoothly embedded as a closed subset into $\IR^{2m+1}$.\vspace{1mm}

Using the previous construction of dual bundles, we can immediately define the smooth $\IR$-vector bundle 
$$
\IT^* X:= (TX)^*\longrightarrow X
$$
of rank $m$, the \emph{cotangent bundle} of $X$. The sections of these (and their induced) bundles are of distinguished importance. We set
\begin{align*}
%&\wedge^{k} X:= \wedge^{k}\IT^*_{\IC} X,\quad \wedge X:= \wedge\IT^*_{\IC} X ,\\
%
&\Omega^{k}_{\ICC}(X):=\Gamma_{\ICC}\left(X,\wedge^{k}_{\IC} T^* X\right),\\
&\Omega_{\ICC}(X):=\Omega^1_{\ICC}(X)\oplus\cdots\oplus \Omega^{m}_{\ICC}(X),\\
&\Omega^0_{\ICC}(X):=\ICC(X):=\ICC(X,\IC),
\end{align*}
where 
$$
\wedge^{k}_{\IC} T^* X=\wedge^{k} T^*_{\IC} X\longrightarrow X
$$
denotes the complexification of $\wedge^{k} T^* X\to X$. (Note that this operation commutes with $\wedge^{k}$.) For example, $\Omega^{k}_{\ICC}(X)$ is the complex linear space of \emph{complex-valued} smooth differential forms on $X$. An analogous notation will be used for other local or global regularity classes of sections of these bundles, for example 
% where the underlying real-valued spaces of differential forms will be denoted by
%\begin{align*}
%&\Omega^{k}_{\ICC_{\IR}}(X):=\Gamma_{\ICC}(X,\wedge^{k}\IT^*  X),\\
%&\Omega_{\ICC_{\IR}}(X):=\Omega^0_{\ICC_{\IR}}(X)\oplus\cdots\oplus \Omega^{\dim(X)}_{\ICC_{\IR}}(X),\\
%&\Omega^0_{\ICC_{\IR}}(X):=\ICC_{\IR}(X):=\{\text{smooth functions $f:X\to\IR$}\}.
%\end{align*}
$$
\Omega^k_{\ICC_{\c}}( X)= \Gamma_{\ICC_{\c}}(X, \wedge^{k}_{\IC} T^* X)
$$
will denote the smooth \emph{compactly supported} complex-valued differential forms on $X$.\vspace{1mm}

Finally, the \emph{real-linear} space of \emph{smooth} vector fields on $X$ will be denoted by 
$$
\mathscr{X}_{C^{\infty}} (X):=\Gamma_{C^{\infty}}(X,TX).
 $$
Let us denote by $\mathrm{Der}(C^{\infty}(X,\IR))$ the $C^{\infty}(X,\IR)$-module\footnote{This is meant under the pointwise defined operations.} of derivations on the $\IR$-algebra $C^{\infty}(X,\IR)$, that is, a $\IR$-linear map 
$$
D:C^{\infty}(X,\IR)\longrightarrow C^{\infty}(X,\IR)
$$
is in $\mathrm{Der}(C^{\infty}(X,\IR))$, if and only if
$$
D(f_1f_2)=f_1D(f_2)+f_2D(f_2)\quad\text{ for all $f_1,f_2\in C^{\infty}(X,\IR)$.}
$$
Every vector field $A\in \mathscr{X}_{C^{\infty}} (X)$ induces the derivation 
$$
D_A\in  \mathrm{Der}(C^{\infty}(X,\IR)), \quad D_Af(x):=A(x)f(x). 
$$
In fact, the assignment
$$
\mathscr{X}_{C^{\infty}} (X)\longrightarrow \mathrm{Der}(C^{\infty}(X,\IR)), \quad A\longmapsto D_A
$$
is an isomorphism of $C^{\infty}(X,\IR)$-modules (cf. Theorem 5.6.3 in \cite{indian}).

\chapter{Facts about self-adjoint operators}\label{selff}

\section{Self-adjoint operators and the spectral calculus}

For the convenience of the reader, we collect some facts about unbounded linear operators in Hilbert spaces here. For a detailed discussion of the below results, we refer the reader to \cite{weidmann1,weidmann2,kato1}.\vspace{1.2mm}

Let $\IHH$ be a complex separable Hilbert space. The underlying scalar product, which is assumed to be antilinear in its first slot, will be simply denoted by $\left\langle \bullet,\bullet \right\rangle$, and the induced norm (as well as the induced operator norm) is denoted by $\left\|\bullet\right\|$. The linear space of bounded  linear operators $\IHH\to \IHH$ is denoted by $\ILL(\IHH)$. If nothing else is said, convergence in $\IHH$ is understood to be norm convergence. Given a linear operator $S$ in $\IHH$, we denote by $\dom(S)\subset \IHH$ its domain, by $\mathrm{Ran}(S)\subset \IHH$ its range, and by $\mathrm{Ker}(S)\subset\IHH$ its kernel. \vspace{2mm}

In the sequel, let $S$ and $T$ be arbitrary linear operators in $\IHH$. 

\subsection{Basic definitions}

Firstly, $T$ is called \emph{an extension of $S$} (symbolically $S\subset T$), if $\dom(S)\subset \dom(T)$ and $Sf=Tf$ for all $f \in \dom(S)$.\vspace{1mm}

In case $S$ is densely defined, the \emph{adjoint} $S^*$ of $S$ is defined as follows: $\dom(S^*)$ is given by all $f\in \IHH$ for which there exists  $f^*\in\IHH$ such that
$$
\left\langle f^*,h\right\rangle=\left\langle f,Sh\right\rangle\quad\text{ for all $h\in\dom(S)$,}
$$
and then $S^*f:=f^*$. A densely defined $S$ is called \emph{symmetric}, if $S\subset S^*$; \emph{self-adjoint}, if $S=S^*$; and \emph{normal}, if $\dom(S)=\dom(S^*)$ and $\left\|Sf\right\|=\left\|S^*f\right\|$ for all $f\in\dom(S)$. Clearly, self-adjoint operators are symmetric and normal.\vspace{1mm}

The operator $S$ is called \emph{semibounded} (from below), if there exists a constant $C\geq 0$ such that for all $f\in\dom(S)$ one has 
\begin{align}\label{semib}
\left\langle Sf,f \right\rangle\geq -C\left\|f\right\|^2,
\end{align}
or in short: $S\geq -C$. Since $\IHH$ is assumed to be complex, semibounded operators are automatically symmetric (by complex polarization).\vspace{1mm}

$S$ is called \emph{closed}, if whenever $(f_n)\subset\dom(S)$ is a sequence such that $f_n\to f$ for some $f\in \IHH$ and $Sf_n\to h$ for some $h\in \IHH$, then one has $f\in\dom(S)$ and $Sf=h$. \vspace{1mm}

$S$ is called \emph{closable}, if it has a closed extension. In this case, $S$ has a smallest closed extension $\overline{S}$, which is \emph{called the closure of $S$.} The closure $\overline{S}$ is determined as follows: $\dom(\overline{S})$ is given by all $f\in\IHH$ for which there exists a sequence $(f_n)\subset \dom(S)$ such that $f_n\to f$ and such that $(Sf_n)$ converges, and then $\overline{S}f:=\lim_n S f_n$. \\
Since adjoints of densely defined operators are closed, it follows that symmetric operators are closable and that self-adjoint operators are closed. \\
If $S$ is densely defined and closable, then $S^*$ is densely defined and $S^{**}=\overline{S}$. We also record the following result (cf. Proposition 3.11 in \cite{bei2} for a complete proof), which seems to be well-known in the context of Hilbert complexes \cite{brulesch}:

\begin{Propositione}\label{comu2} Let $S$ be a densely defined and closed operator with $\mathrm{Ran}(S)\subset \dom(S)$ and $S^2=0$. Then $S+S^*$ is self-adjoint on its natural domain $\dom(S^*+S)=\dom(S)\cap \dom(S^*)$.
\end{Propositione}

In case $S$ is closed, a linear subspace $D\subset \dom(S)$ is called a \emph{core} of $S$, if $\overline{S|_D}=S$, and this core property is equivalent to the following property: For every $f\in \dom(S)$ there exists a sequence $(f_n)\subset \dom(S)$ such that $(f_n)$ converges to $f$ in the graph norm of $S$, that is,
$$
\left\|f_n-f\right\|+\left\|Sf_n-Sf\right\|\to 0.
$$
If $T$ is symmetric, then $T$ is called \emph{essentially self-adjoint}, if $\overline{T}$ is self-adjoint. It follows immediately from the definitions that if $T$ is symmetric and $ S$ is self-adjoint with $T\subset S$, then $T$ is essentially self-adjoint if and only if $\dom(S)$ is a core of $T$.\vspace{1mm}

We record: 

\begin{Theoreme}\label{defect} Assume that $S$ is semibounded with $S\geq -C$ for some constant $C\geq 0$. Then $S$ is essentially self-adjoint, if and only if there exists $z\in \IC\setminus [-C,\infty)$ such that $\mathrm{Ker}((S-z)^*)=\{0\}$.
\end{Theoreme} 

\begin{proof} This is a combination of Satz 5.14, Satz 10.2.a), Satz 10.3.a), Satz 10.11.b) in \cite{weidmann2}.
\end{proof}

\subsection{Spectrum and resolvent set}

The \emph{resolvent set} $\rho(S)$ is defined to be the set of all $z\in \IC$ such that $S-z$ is invertible as a linear map $\dom(S)\to \IHH$ and is in addition bounded as a linear operator from $\IHH$ to $\IHH$. If $S$ is closed and $(S-z)^{-1}$ invertible, then $(S-z)^{-1}$ is automatically bounded by the closed graph theorem. The \emph{spectrum} $\sigma(S)$ of $S$ is defined as the complement $\sigma(S):=\IC\setminus\rho(S)$. Resolvent sets of closed operators are open, therefore spectra of closed operators are always closed. \vspace{1.2mm}

A number $z\in \IC$ is called an \emph{eigenvalue of $S$}, if $\mathrm{Ker}(S-z)\ne \{0\}$. In this case, $\dim \mathrm{Ker}(S-z)$ is called the \emph{multiplicity} of $z$, and each $f\in \mathrm{Ker}(S-z)\setminus \{0\}$ is called an \emph{eigenvector of $S$ corresponding to $z$}. Of course each eigenvalue is in the spectrum. The eigenvalues of a symmetric operator are real, and the eigenvectors corresponding to different eigenvalues of a symmetric operator are orthogonal. A simple result that reflects the subtlety of the notion of a \lq\lq{}self-adjoint operator\rq\rq{} when compared to that of a\lq\lq{}symmetric operator\rq\rq{} is the following: A symmetric operator in $\IHH$ is self-adjoint, if and only if its spectrum is real. If $S$ is self-adjoint, then $S\geq -C$ for a constant $C\geq 0$ is equivalent to $\sigma(S)\subset [-C,\infty)$ (cf. Satz 8.26 in \cite{weidmann2}).

\vspace{1.2mm}

The \emph{essential spectrum} $\sigma_{\mathrm{ess}}(S)\subset\sigma(S)$ of $S$ is defined to be the set of all eigenvalues $\lambda$ of $S$ such that either $\lambda$ has an infinite multiplicity, or $\lambda$ is an accumulation point of $\sigma(S)$. Then the \emph{discrete spectrum} $\sigma_{\mathrm{dis}}(S)\subset\sigma(S)$ is defined as the complement 
$$
\sigma_{\mathrm{dis}}(S):=\sigma(S)\setminus\sigma_{\mathrm{ess}}(S).
$$
As every isolated point in the spectrum of a self-adjoint operator is an eigenvalue (cf. Folgerung 3, p. 191 in \cite{weidmann1}), it follows that in case of $S$ being self-adjoint, the set $\sigma_{\mathrm{dis}}(S)$ is precisely the set of all isolated eigenvalues of $S$ that have a finite multiplicity. 

\subsection{A result from perturbation theory}\label{swpp}

Let $\IHH\rq{}$ be another complex separable Hilbert space. We recall that given $q\in [1,\infty)$, some $K\in \ILL(\IHH, \IHH\rq{})$ is called
\begin{itemize}
\item \emph{compact}, if for every orthonormal sequence $(e_n)$ in $\IHH$ and every orthonormal sequence $(f_n)$ in $\IHH\rq{}$ one has $\left\langle Ke_n, f_n\right\rangle\to 0$ as $n\to \infty$ 
\item \emph{$q$-summable} (or an element of the \emph{$q$-th Schatten class of operators $\IHH\to \IHH\rq{}$}), if for every $(e_n)$, $(f_n)$ as above one has 
$$
\sum_n \left|\left\langle Ke_n, f_n\right\rangle\right|^q<\infty.
$$
\end{itemize}

(We refer the reader to Section 3 in \cite{weidmann2} for a detailed study of compact operators and the Schatten classes.) Let us denote the class of compact operators with $\IJJ^{\infty}(\IHH,\IHH\rq{})$ and the $q$-th Schatten class with $\IJJ^{q}(\IHH,\IHH\rq{})$, with the convention $\IJJ^{\bullet}(\IHH):=\IJJ^{\bullet}(\IHH,\IHH)$. These are linear spaces with
$$
\IJJ^{q_1}(\IHH,\IHH\rq{})\subset \IJJ^{q_2}(\IHH,\IHH\rq{})\quad\text{ for all $q_2\in [1,\infty]$, with $q_1\leq q_2$,}
$$
and one has inclusions of the type $\IJJ^{q}\circ\ILL\subset \IJJ^{q}$, $\ILL\circ\IJJ^{q}\subset \IJJ^{q}$ for all $q\in [1,\infty]$, and $\IJJ^{q_1}\circ\IJJ^{q_2}\subset \IJJ^{q_3}$ if $1/q_1+1/q_2=1/q_3$ with $q_j\in [1,\infty)$.\\
For obvious reasons, $\IJJ^1$ is called the \emph{trace class}, and moreover $\IJJ^2$ is called the \emph{Hilbert-Schmidt class}. A bounded operator on an $\IL^2$-space is Hilbert-Schmidt, if and only if it is an integral operator with a square integrable integral kernel. We record the following basic result from perturbation theory:

\begin{Theoreme}\label{comp} Assume that $S$ is self-adjoint and semibounded, that $T$ is symmetric with $\dom(S)\subset \dom(T)$ and that $T(S-z)^{-1}\in \IJJ^{\infty}(\IHH)$ for some (or equivalently, for all) $z\in \rho(S)$. Then $S+T$ is self-adjoint and semibounded on its natural domain $\dom(S)\cap\dom(T)=\dom(S)$, and moreover
$$
\sigma_{\mathrm{ess}}(S)=\sigma_{\mathrm{ess}}(S+T).
$$
\end{Theoreme}

\begin{proof} Combine Satz 9.7, Satz 9.13, and Satz 9.14 from \cite{weidmann2}.
\end{proof}

%Based on the latter, criterion, an important \lq\lq{}invariant-domain\rq\rq{} criterion (one of many) for essential self-adjointness has been proved by Chernoff:
%
%\begin{Theorem}\label{chernoff1} Assume that $T$ is a symmetric (and so by definition densely defined) operator in $\IHH$ which preserves $\dom(T)$, and that $(V(t))_{t\in \IR}\subset \ILL(\IHH)$ is a family of unitary operators such that each $V(t)$ preserves $\dom(T)$ and commmutes with $T$ on $\dom(T)$. Assume in addition that for all $f\in \dom(T)$ one has
%$$
%(\Id/ \Id t) V(t)f=i TV(t) f\quad\text{weakly  }
%$$
%\end{Theorem}

\subsection{Spectral calculus and the spectral theorem}

A \emph{spectral resolution} $P$ on $\IHH$ is a map $P:\IR\to \ILL(\IHH)$ such that 
\begin{itemize}
\item for every $\lambda\in\IR$ one has $P(\lambda)=P(\lambda)^*$, $P(\lambda)^2=P(\lambda)$ (that is, each $P(\lambda)$ is an orthogonal projection onto its image)
\item $P$ is monotone in the sense that $\lambda_1\leq \lambda_2$ implies $\mathrm{Ran}(P(\lambda_1))\subset \mathrm{Ran}(P(\lambda_2))$
\item $P$ is right-continuous in the strong topology of $\ILL(\IHH)$
\item $\lim_{\lambda\to-\infty}P(\lambda)=0$ and $\lim_{\lambda\to \infty}P(\lambda)=\mathrm{id}_{\IHH}$, both in the strong sense.
\end{itemize}

It follows that for every $f\in\IHH$, the function 
$$
\lambda\mapsto \left\langle P(\lambda)f,f\right\rangle=\left\| P(\lambda)f\right\|^2
$$
is right-continuous and increasing. Thus by the usual Stieltjes construction it induces a Borel measure on $\IR$, which will be denoted by $ \left\langle P(\Id \lambda)f,f\right\rangle$. This measure has the total mass 
$$
\left\langle P(\IR)f,f\right\rangle=\left\|f\right\|^2.
$$

Given such $P$ and a Borel function $\phi:\IR\to\IC$, the set
\begin{align*}
D_{P,\phi}:=\left\{ f\in \IHH: \int_{\IR} |\phi(\lambda)|^2  \left\langle P(\Id \lambda)f,f\right\rangle<\infty\right\}
\end{align*}
is a dense linear subspace of $\IHH$ (cf. Satz 8.8 in \cite{weidmann2}), and accordingly one can define a linear operator $\phi(P)$ with $\dom(\phi(P)):=D_{P,\phi}$ in $\IHH$ by mimicking the complex polarization identity,
\begin{align*}
\left\langle \phi(P)f_1,f_2\right\rangle:=\>(1/4)&\int_{\IR} \phi(\lambda) \  \left\langle P( \Id\lambda)(f_1+f_2),f_1+f_2\right\rangle\\
-(1/4)&\int_{\IR} \phi(\lambda) \   \left\langle P( \Id\lambda)(f_1-f_2),f_1-f_2\right\rangle\\
+(\sqrt{-1}/4)&\int_{\IR} \phi(\lambda) \  \left\langle P(\Id \lambda)(f_1-\sqrt{-1}f_2),f_1-\sqrt{-1}f_2\right\rangle\\
-(\sqrt{-1}/4)&\int_{\IR} \phi(\lambda) \   \left\langle P(\Id \lambda)(f_1+\sqrt{-1}f_2),f_1+\sqrt{-1}f_2\right\rangle,
\end{align*}
where $f_1,f_2\in \dom(\phi(P))$. Every spectral measure induces the following \lq\lq{}calculus\rq\rq{}:

\begin{Theoreme}\label{also} Let $P$ be a spectral resolution on $\IHH$, and let $\phi:\IR \to \IC$ be a Borel function. Then:\\
\emph{(i)} $\phi(P)$ is a normal operator with $\phi(P)^*=\overline{\phi}(P)$; in particular, $\phi(P)$ is self-adjoint, if and only if $\phi$ is real-valued.\\
\emph{(ii)} One has $\left\| \phi(P) \right\|\leq \sup_{\IR}|\phi|\in [0,\infty] $.\\
\emph{(iii)} If $\phi\geq -C$ for some constant $C\geq 0$, then one has $\phi(P)\geq -C$.\\
\emph{(iv)} If $\phi\rq{}:\IR\to \IC$ is another Borel function, then 
$$
\phi(P)+\phi\rq{}(P)\subset (\phi+\phi\rq{})(P),\>\>\dom(\phi(P)+\phi\rq{}(P))=\dom((|\phi|+|\phi\rq{}|)(P))
$$ 
and 
$$
\phi(P)\phi\rq{}(P)\subset (\phi\phi\rq{})(P),\quad\dom(\phi(P)\phi\rq{}(P))=\dom((\phi\phi\rq{})(P))\cap \dom(\phi\rq{});
$$ 
in particular, if $\phi\rq{}$ is bounded, then  
\begin{align*}
&\phi(P)+\phi\rq{}(P)= (\phi+\phi\rq{})(P),\\
&\phi(P)\phi\rq{}(P)= (\phi\phi\rq{})(P).
\end{align*}
\emph{(v)} For every $f\in \dom(\phi(P))$ one has
$$
\left\| \phi(P) f\right\|^2=\int_{\IR} |\phi(\lambda)|^2\left\langle P(\Id\lambda)f,f\right\rangle.
$$
 \end{Theoreme}

\begin{proof} The statements (i), (ii), (iii), (v), and the first two claims of (iv) are included in Satz 8.8 in \cite{weidmann2}. The first two claims of (iv) easily imply the last two claims of (iv). 
\end{proof}

One variant of the spectral theorem is:

\begin{Theoreme} For every self-adjoint operator $S$ in $\IHH$ there exists precisely one spectral resolution $P_S$ on $\IHH$ such that $S=\mathrm{id}_{\IR}(P_S)$. The operator $P_S$ is called \emph{the spectral resolution of $S$}, and it has the following additional properties:
\begin{itemize}
\item $P_S$ is concentrated on the spectrum of $S$ in the sense that for every Borel function $\phi:\IR\to \IC$ one has 
$$
\phi(P_S)=(1_{\sigma(S)}\cdot \phi)(P_S)
$$
\item if $\phi:\IR\to\IR$ is continuous, then $\sigma(\phi(P_S))=\overline{\phi(\sigma(S))}$
\item if $\phi,\phi\rq{}:\IR\to \IR$ are Borel functions, then one has the transformation rule $(\phi\circ \phi\rq{})(P_S)=\phi(P_{\phi\rq{}(P_S)})$.
\end{itemize}
\end{Theoreme}

\begin{proof} The existence and uniqueness of a spectral resolution $P_S$ with $S=\mathrm{id}_{\IR}(P_S)$ is the content of Satz 8.11 in \cite{weidmann2}. The other statements follow straightforwardly by combining Satz 8.17 in \cite{weidmann2} with Satz 8.21 in \cite{weidmann2}.
\end{proof}

In view of these results, given a self-adjoint operator $S$ in $\IHH$, the calculus of Theorem \ref{also} applied to $P=P_S$ is usually referred to as the \emph{spectral calculus of $S$}. Likewise, given a Borel function $\phi:\IR \to \IC$ one sets
$$
\phi(S):= \phi(P_S).
$$

\begin{Remarke}\label{beispiele} Let $S$ be a self-adjoint operator in $\IHH$.\\
1. The spectral calculus of $S$ is compatible with all functions of $S$ that can be defined \lq\lq{}by hand\rq\rq{}. For example, for every $z\in \IC\setminus \IK$ one has $\phi(S)=(S-z)^{-1}$ with $\phi(\lambda):= 1/(\lambda-z)$, or $S^n=\phi(S)$ with  $\phi(\lambda):=\lambda^n$.\\
2. If $S$ is a semibounded operator and $z\in\IC$ is such that $\Re z <\min\sigma(S)$, then the spectral calculus (together with a well-known Laplace transformation formula for functions) shows that for every $b>0$ one has the following formula for $f_1,f_2\in\IHH$:
\begin{align}
\label{lpo2}
\left\langle (S-z)^{-b}f_1,f_2 \right\rangle=\f{1}{\Gamma(b)}\int^{\infty}_0 s^{b-1}\left\langle \mathrm{e}^{z s}\mathrm{e}^{-s S}f_1,f_2\right\rangle \Id s.
\end{align}
3. The collection $(\mathrm{e}^{-it S})_{t\in \IR}$ forms a strongly continuous unitary group of bounded operators, and for every $\psi\in \dom(S)$, the path 
$$
\IR \ni t\longmapsto \psi(t):=\mathrm{e}^{-it S},\quad \psi\in \IHH
$$
is the unique (norm-)differentiable path with $\psi(0)=\psi$ which solves the \emph{abstract Schrödinger equation}
$$
(\Id/\Id t) \psi(t)=-\sqrt{-1}S\psi(t),\quad t\in\IR.
$$
In particular, $\psi(t)\in \dom(S)$ for all $t\in\IR$ (cf. Satz 8.20 in \cite{weidmann2}).\\
4. If $S\geq -C $ for some constant $C\geq 0$, then the collection $(\mathrm{e}^{-t S})_{t\geq 0}$ forms a strongly continuous self-adjoint semigroup of bounded operators (contractive, if one can pick $C=0$), and one has the abstract smoothing effect
$$
\mathrm{Ran}(\mathrm{e}^{-t S})\subset \bigcap_{n\in\IN_{\geq 1}}\dom(S^n)\quad\text{ for all $t>0$.}
$$
Moreover, for every $\psi\in\IHH$ the path 
$$
[0,\infty) \ni t\longmapsto \psi(t):=\mathrm{e}^{-t S},\quad \psi\in \IHH
$$
is the uniquely determined continuous path with $\psi(0)=\psi$ which is differentiable in $(0,\infty)$ and satisfies there the \emph{abstract heat equation}
$$
(\Id/\Id t) \psi(t)=-S\psi(t)
$$
(cf. Corollary 4.11 in \cite{gri}).\\
5. If $S\geq -C $ for some constant $C\geq 0$ and if $\mathrm{e}^{-t S}\in \IJJ^1(\IHH)$, then $S$ has a purely discrete spectrum  (cf. Lemma 10.7 in \cite{gri}).
\end{Remarke}

Finally, we record the following general commutation result (in fact, we will use it only in the situation $T=T^*$, where  the proof follows almost immediately from the spectral calculus).

\begin{Propositione}\label{comu} Let $T$ be a closed densely defined operator in $\IHH$. Then the operators $TT^*$ and $T^*T$ are self-adjoint and $\geq 0$, and for every Borel function $\phi:\IR\to\IR$ with 
\begin{align*}
\sup_{\lambda\in \IR}|\phi(\lambda^2)|+\sup_{\lambda\in[0,\infty)}|\sqrt{\lambda} \phi(\lambda)|<\infty
\end{align*}
one has $\phi(T^*T)T\subset T\phi(T^*T)$ and $ \phi(TT^*)T^*\subset T^*\phi(TT^*)$. All these operators are bounded on their dense domains.
\end{Propositione}
 
\begin{proof} Following the appendix of \cite{anton}, we define the operator
$$
\underline{T}:=\begin{pmatrix} 0  &T^*\\ T & 0\end{pmatrix}
$$
in $\IHH\oplus\IHH$, where the above is a symbolic notation for  $\dom(\underline{T})=\dom(T)\oplus \dom(T^*)$, $\underline{T}(f_1\oplus f_2)=T^*f_2\oplus Tf_1 $. Then $\underline{T}$ is self-adjoint\footnote{This self-adjointness is a standard result in \lq\lq{}supersymmetric\rq\rq{} quantum mechanics, where it means that \lq\lq{}supercharges are self-adjoint\rq\rq{} (cf. Lemma 5.3 in \cite{thaller} for a detailed proof).}, so that
$$
\underline{T}^2=\begin{pmatrix} T^*T  &0\\ 0 & TT^*\end{pmatrix}
$$
is self-adjoint and $\geq 0$, and so are its components. Furthermore, the assumption $\sup_{\lambda\in \IR}|\phi(\lambda^2)|<\infty$ implies
$\phi(\underline{T}^2)\underline{T}\subset \underline{T} \phi(\underline{T}^2)$, in view of Theorem \ref{also} (iv). (The boundedness of $\lambda\mapsto\phi(\lambda^2)$ is really used here for an equality of operators!) Since the spectral calculus commutes with direct sums of Hilbert spaces, we now get
\begin{align*}
\begin{pmatrix} 0  & \phi(T^*T)T^*\\  \phi(TT^*)T & 0\end{pmatrix}=\begin{pmatrix} \phi(T^*T)  &0\\ 0 & \phi(TT^*)\end{pmatrix}\begin{pmatrix} 0  &T^*\\  T & 0\end{pmatrix}\\
\subset \begin{pmatrix} 0  &T^*\\  T & 0\end{pmatrix}\begin{pmatrix} \phi(T^*T)  &0\\ 0 & \phi(TT^*)\end{pmatrix}=\begin{pmatrix} 0  &T^* \phi(TT^*)\\  T\phi(T^*T) & 0\end{pmatrix}.
\end{align*}
Finally, to see the claims concerning the boundedness, it follows from the above inclusions and identities that it is sufficient to prove that $\underline{T} \phi(\underline{T}^2)$ is bounded. To prove this boundedness, note that with $W:=\underline{T}^2=\underline{T}^*\underline{T}$ we have the polar decomposition (cf. Satz 8.22b in \cite{weidmann2}) $\underline{T}=U\sqrt{W}$, with some isometry
$$
U:\overline{\mathrm{Ran}(\sqrt{W})}\longrightarrow \overline{\mathrm{Ran}(T)},
$$
so that 
$$
\left\|\underline{T} \phi(\underline{T}^2)\right\|\leq \left\|U\right\|\left\|\sqrt{W}\phi(W)\right\|,
$$
and $\left\|\sqrt{W}\phi(W)\right\|<\infty$ is implied by the spectral calculus and the spectral theorem, since $\sigma(W)\subset [0,\infty)$, so that $P_W$ is concentrated on $[0,\infty)$, and $\sup_{\lambda\in[0,\infty)}|\sqrt{\lambda} \phi(\lambda)|<\infty$ by assumption.
\end{proof}

\section{Sesquilinear forms in Hilbert spaces}\label{sopq}

In this section, we collect some basic facts about possibly unbounded sesquilinear forms on Hilbert spaces. Unless otherwise stated, all statements below can be found in section VI of T. Kato\rq{}s book \cite{kato1}.\vspace{1.2mm}

Let again $\IHH$ be a complex separable Hilbert space. A \emph{sesquilinear form} $Q$ on $\IHH$ is understood to be a map
$$
Q:\dom(Q)\times \dom(Q) \longrightarrow \IC,
$$
where $\dom(Q)\subset \IHH$ is a linear subspace called the \emph{domain of definition of $Q$}, such that $Q$ is antilinear\footnote{We warn the reader, however, that in \cite{kato1} the forms are assumed to be antilinear in their second slot; thus, if $Q(f_1,f_2)$ is a form in our sense, the theory from \cite{kato1} has to be applied to the complex conjugate form $Q(f_1,f_2)^*$.} in its first slot, and linear in its second slot.\vspace{1.2mm}

Let $Q$ and $Q'$ be sesquilinear forms on $\IHH$ in this section.\vspace{1.2mm}

The sum $Q+Q'$ of $Q$ and $Q'$ is the sesquilinear form which is defined in the obvious way, with its domain of definition given by $\dom(Q+Q')=\dom(Q)\cap \dom(Q')$.

\vspace{1.2mm}

$Q'$ is called an \emph{extension of $Q$}, symbolically $Q\subset Q'$, if $\dom(Q)\subset \dom(Q')$ and if both forms coincide on $\dom(Q)$.\vspace{1.2mm}

$Q$ is called \emph{symmetric}, if $Q(f_1,f_2)=Q(f_2,f_1)^*$, and \emph{semibounded (from below)}, if there exists a constant $C\geq 0$ such that 
\begin{align}\label{uap}
Q(f,f)\geq -C\left\|f\right\|^2\quad\text{ for all $f\in\dom(Q)$,}
\end{align}
symbolically $Q\geq -C$. Again by complex polarization, every semibounded form is automatically symmetric.\vspace{1.2mm}

Following Kato, given a sequence $(f_n)\subset \dom(Q)$ and $f\in \dom(Q)$ we write  $f_n\underset{Q}{\longrightarrow}f$ as $n\to\infty$, if one has $f_n\to f$ in $\IHH$ and in addition
$$
Q(f_n-f_m,f_n-f_m)\to 0\quad\text{ as $n,m\to\infty$.}
$$
Then $Q$ is called \emph{closed}, if $f_n\underset{Q}{\longrightarrow}f$ implies that $f\in\dom(Q)$. A semibounded $Q$ is closed, if and only if for some/every $C\geq 0$ with $Q\geq -C$ the scalar product on $\dom(Q)$ given by 
\begin{align}\label{qap}
\left\langle f_1,f_2\right\rangle_{Q,C}=(1+C)\left\langle f_1,f_2\right\rangle+ Q( f_1,f_2)
\end{align}
turns $\dom(Q)$ into a Hilbert space. Futhermore, for a semibounded $Q\geq -C$ its closedness is equivalent to the lower-semicontinuity of the function
$$
\IHH \longrightarrow [-C,\infty],\>\> f\longmapsto \begin{cases}&Q(f,f),\quad\text{ if $f\in \dom(Q)$}\\&\infty\quad\text{else.}\end{cases} 
$$
\vspace{1.2mm}

The form $Q$ is called \emph{closable}, if it has a closed extension. If $Q$ is semibounded and closable, then it has a smallest semibounded and closed extension $\overline{Q}$, which is (well-)defined as follows: $\dom(\overline{Q})$ is given by all $f\in \IHH$ that admit a sequence $(f_n)\subset \dom(Q)$ with $f_n\underset{Q}{\longrightarrow} f$; then one has 
$$
\overline{Q}(f,h)= \lim_n Q(f_n,h_n),\quad\text{ where $f_n\underset{Q}{\longrightarrow} f$, $h_n\underset{Q}{\longrightarrow} h$.}
$$

If $Q$ is closed, then a linear subspace $D\subset \dom(Q)$ is called a \emph{core} of $Q$, if $\overline{Q|_D}=Q$.

\begin{Propositione} If $Q$ and $Q'$ are semibounded and closed, then $Q+Q'$ is semibounded and closed. 
\end{Propositione}

The following notions will be convenient:

\begin{Definitione}\label{bounded} Let $Q$ be symmetric. If $\dom(Q)\subset \dom(Q')$, then $Q'$ is called
\begin{itemize}
	\item \emph{$Q$-bounded with bound $<1$}, if there exist constants $\delta\in [0,1)$, $A\in [0,\infty)$ such that
	\begin{align}\label{aopq}
	|Q'(f,f)|\leq A\left\|f\right\|^2+ \delta Q(f,f)\quad\text{ for every $f\in\dom(Q)$},
	\end{align}
	\item \emph{infinitesimally $Q$-bounded}, if for every $\delta\in [0,\infty)$ there exists a constant $A=A_{\delta}\in [0,\infty)$ with (\ref{aopq}).
\end{itemize}
\end{Definitione}

The next result from perturbation theory is the famous KLMN (Kato-Lax-Lions-Milgram-Nelson) theorem (see for example Satz 4.16 and its proof in \cite{weidmann2}):

\begin{Theoreme}\label{klmn} Let $Q$ be semibounded and closed, and let $Q'$ be symmetric and $Q$-bounded with bound $<1$. Then $Q+Q'$ is semibounded and closed on its natural domain $\dom(Q)\cap\dom(Q')=\dom(Q)$. Moreover, every form core of $Q$ is also one of $Q+Q'$, and for every constant $c\geq 0$ with $Q\geq -c$ and every $A,\delta$ as in (\ref{aopq}) one has the explicit lower bound 
$$
Q+Q\rq{}\geq -(1-\delta)c-A.
$$
\end{Theoreme}

\begin{proof} The proof is actually very simple: Based on the  assumption on $Q\rq{}$, one immediately finds that the norms $\left\|\bullet\right\|_{Q,C}$ and $\left\|\bullet\right\|_{Q+Q\rq{},C}$ are equivalent for $C>0$ large enough, which proves that $Q+Q'$ is closed and also has the core property. The lower bound is also seen immediately.
\end{proof}

Using the spectral calculus one defines:

\begin{Definitione} Given a self-adjoint operator $S$ in $\IHH$, the (densely defined and symmetric) sesquilinear form $Q_S$ in $\IHH$ given by $\dom(Q_S):=\dom(\sqrt{|S|})$ and
$$
Q_S(f_1,f_2):=\left\langle\sqrt{|S|}f_1, \sqrt{|S|}f_2\right\rangle
$$
is called the \emph{form associated with $S$}.
\end{Definitione}

The following fundamental result links the world of densely defined, semibounded, closed forms with that of semibounded self-adjoint operators (cf. Theorem VIII.15 in \cite{reed1} for this exact formulation):

\begin{Theoreme}\label{kaq1} For every self-adjoint semibounded operator $S$ in $\IHH$, the form $Q_S$ is densely defined, semibounded and closed. Conversely, for every densely defined, closed and semibounded sesquilinear form $Q$ in $\IHH$, there exists precisely one self-adjoint semibounded operator $S_Q$ in $\IHH$ such that $Q=Q_{S_Q}$. The operator $S_Q$ will be called \emph{the operator associated with $Q$}.
\end{Theoreme}

The correspondence $S\mapsto Q_S$ has the following additional properties:

\begin{Theoreme}\label{kaq2} Let $Q$ be densely defined, closed and semibounded. Then:
\begin{itemize}
\item $S_Q$ is the uniquely determined self-adjoint and semibounded operator in $\IHH$ such that $\dom(S_Q)\subset \dom(Q)$ and
$$
\left\langle S_Qf_1,f_2\right\rangle=Q(f_1,f_2)\>\text{ for all $f_1\in\dom(S_Q)$, $f_2\in \dom(Q)$.}
$$
\item $\dom(S_Q)$ is a core of $Q$; some $f_1\in\dom(Q)$ is in $\dom(S_Q)$, if and only if there exists $f_2\in \IHH$ and a core $D$ of $Q$ with
		$$
		Q(f_1,f_3)=\left\langle f_2,f_3\right\rangle\quad\text{ for all $f_3\in D$, }
		$$
and then $S_Qf_1=f_2$.
%\item For every $C\geq 0$ with (\ref{uap}) one has %$\dom(Q)=\dom (\sqrt{H+C})=$ with 
%\begin{align*}
%Q(f_1,f_2)=\left\langle \sqrt{H+C} f_1,\sqrt{H+C}f_2  %\right\rangle - C\left\langle f_1,f_1\right\rangle
%\end{align*}
%for all $f_1,f_1\in \dom (\sqrt{H+C})=\dom (Q)$.
\item One has 
\begin{align*}
&\dom(Q)=\left\lbrace  h\in \IHH:\> \lim_{t\to 0+}\left\langle \f{h-\mathrm{e}^{-t S_Q}h}{t},h\right\rangle<\infty \right\rbrace , \\
&Q(h,h)= \lim_{t\to 0+}\left\langle \f{h-\mathrm{e}^{-t S_Q}h}{t},h\right\rangle.
\end{align*}
\item One has
\begin{align*}
\min \sigma(S_Q)&=\inf\{Q(f,f):\>f\in\dom(Q),\>\left\|f\right\|=1\}\\
&=\inf\{ \left\langle S_Qf,f\right\rangle:\>f\in\dom(S_Q),\>\left\|f\right\|=1\}.
\end{align*}
\end{itemize}
\end{Theoreme}

\begin{proof} The first assertion follows from Theorem 2.1 in \cite{kato1}.\\
 For the asserted heat semigroup characterization, just note that 
\begin{align*}
&\left\langle \f{f-\mathrm{e}^{-t S_Q}f}{t},f\right\rangle=\int^{\infty}_{\min\sigma(S_Q)}\f{1-\mathrm{e}^{-t\lambda }}{t} \left\langle P_{S_Q}(\Id \lambda)f,f\right\rangle\>\>\text{ for all $f\in\IHH$}, \\
&\dom(Q)=\dom\Big(\sqrt{|S_Q|}\Big)\\
&\quad\quad \quad\quad=\left\{  h\in \IHH:\>\>\int^{\infty}_{\min\sigma(S_Q)}|\lambda|\left\langle P_{S_Q}(\Id \lambda)h,h\right\rangle<\infty \right\},\\
&Q(h,h)=\int^{\infty}_{\min\sigma(S_Q)}|\lambda|\left\langle P_{S_Q}(\Id \lambda)h,h\right\rangle,\>\>h\in\dom(Q).
\end{align*}
In particular, the limit of $\left\langle \f{f-\mathrm{e}^{-t S_Q}f}{t},f\right\rangle$ as $t\to 0+$ always exists as an element of $[\min\sigma(S_Q),\infty]$, and it is finite if and only if $f\in\dom(Q)$. In this case, the limit is $Q(f,f)$.\\
The formula for $\min \sigma(S_Q)$ also follows easily from the spectral calculus (cf. Satz 8.27 in \cite{weidmann2}).
\end{proof}

\begin{Notatione}\label{notti} If $Q$, $Q\rq{}$ are symmetric, we write $Q\geq Q\rq{}$, if and only if $\dom(Q)\subset \dom(Q\rq{})$ and $Q(f,f)\geq Q\rq{}(f,f)$ for all $f\in \dom(Q)$.
\end{Notatione}

The Friedrichs extension of a semibounded operator can be defined as follows:

\begin{Examplee}\label{friedrichs} Let $S\geq -C$ be a symmetric (in particular, a densely defined) and semibounded operator in $\IHH$. Then the form $(f_1,f_2)\mapsto\left\langle Sf_1,f_2\right\rangle$ with domain of definition $\dom(S)$ is closable, and of course the closure $\tilde{Q}_S$ of that form is densely defined and semibounded. The operator $S_F$ associated with $\tilde{Q}_S$ is called the \emph{Friedrichs realization of $S$}. The operator $S_F$ can also be characterized as follows: $S_F$ is the uniquely determined self-adjoint semibounded extension of $S$ with domain of definition $\subset \dom(\tilde{Q}_S)$. Let $\IMM_C(S)$ denote the class of all self-adjoint extensions of $S$ which are $\geq -C$. Thus we have $S_F\in \IMM_C(S)$, and in addition the following maximality property holds: 
$$
T\in \IMM_C(S)\quad\Rightarrow \quad Q_T\leq \tilde{Q}_S.
$$
In particular, $S_F$ has the smallest bottom of spectrum $\min \sigma(S_F)$ among all operators in $\IMM_C(S)$. This is Krein\rq{}s famous result on the characterization of semibounded extensions \cite{krein2} \cite{krein1}.
 \end{Examplee}

\section{Strong convergence results for semigroups}

\subsection{Semigroup convergence from convergence on a core}

The following result is probably the most elementary convergence result for infinite-dimensional self-adjoint semigroups (cf. Theorem VIII.25 and Theorem VIII.20 in \cite{reed1}): 

\begin{Theoreme}\label{core} Let $S$ and $S_n$, $n\in\IN$, be self-adjoint semibounded operators in a complex separable Hilbert space $\IHH$, and assume that there exists a subspace $D\subset \IHH$ which is a common core for $S$ and $S_n$ for all $n$, such that $S_nf \to S f$ as $n\to\infty$ for all $f\in D$. Then for all $t\geq 0$ and all $f\in\IHH$ one has $\mathrm{e}^{-tS_n}f \to \mathrm{e}^{-tS}f $ as $n\to\infty$.
\end{Theoreme}

\subsection{Monotone convergence of sesquilinear forms}

We record the following two classical results concerning the monotone convergence of sequences of sesquilinear forms here (cf. Theorem 3.1 and Theorem 4.2 in \cite{sim}). First for increasing sequences:

\begin{Theoreme}\label{monn} Let $ Q_1\leq Q_2\leq\dots$ be a sequence of densely defined, closed and semibounded sesquilinear forms on a common complex separable Hilbert space $\IHH$. Assume that
\begin{align}\label{domi}
\left\{f\in\bigcap_{n}\dom(Q_n):\sup_n Q_n(f,f)<\infty\right\}\subset \IHH
\end{align}
is dense. Then $Q(f_1,f_2):=\lim_n Q_n(f_1,f_2)$ (polarization!) with the domain as given by (\ref{domi}) is a closed, semibounded sesquilinear form $Q$ in $\IHH$, and with $S_n$ the operator corresponding to $Q_n$ and $S$ the operator corresponding to $Q$, one has
$$
\mathrm{e}^{-t S_n}f\to \mathrm{e}^{-t S}f\>\>\text{  as $n\to\infty$, for all $t\geq 0$ and all $f\in\IHH$.}
$$ 
\end{Theoreme}

The situation for decreasing sequences is a little more subtle, since one has to \emph{assume} the closability of the limit form. (In fact, one can drop closability, in which case one gets convergence to the operator corresponding to the closure of the \emph{regular part} of $Q$, with $Q$ as below; we will not need this subtle generalization.)

\begin{Theoreme}\label{monn2} Let $Q_1\geq Q_2\geq\dots$ be a sequence of densely defined, closed and semibounded sesquilinear forms on a common complex Hilbert space $\IHH$. Assume that the sesquilinear form in $\IHH$ given by
$$
Q(f_1,f_2):=\lim_n Q_n(f_1,f_2),\quad\dom(Q):=\bigcup_n \dom(Q_n) 
$$
is closable. Then $Q$ is automatically closed (it is obviously densely defined and semibounded), and with $S_n$ the operator corresponding to $Q_n$ and $S$ the operator corresponding to $Q$, one has
$$
\mathrm{e}^{-t S_n}f\to \mathrm{e}^{-t S}f\text{ in $\IHH$ as $n\to\infty$, for all $t\geq 0$, and all $f\in\IHH$.}
$$ 
\end{Theoreme}

This follows from Theorem 3.2 and Theorem 4.2 in \cite{sim}.

\section{Abstract Kato-Simon inequalities}

Let $E\to M$ be a smooth metric $\IC$-vector bundle over a smooth Riemannian manifold $M$. We recall our previous convention that $\left\langle\bullet,\bullet\right\rangle$ (respectively $\left\|\bullet\right\|$) denotes the various $\IL^2$-scalar products (respectively norms), whereas $\left(\bullet,\bullet\right)$ (respectively $\left|\bullet\right|$) denotes the various fiberwise taken finite-dimensional scalar products (respectively norms). The Riemannian volume measure is denoted by $\mu$. The following result has been shown by Hess/Schrader/Uhlenbrock in \cite{hess}:

\begin{Theoreme}\label{hsi} Let $S$ be a self-adjoint nonnegative operator in $\Gamma_{\IL^2}(M,E)$, and let $T$ be a self-adjoint nonnegative operator in $\IL^2(M)$. Then the following statements are equivalent:\\
(i) For all $f\in\Gamma_{\IL^2}(M,E)$ and all $t\geq  0$, one has
$$
|\mathrm{e}^{-t S} f|\leq \mathrm{e}^{-t T}|f|\quad\text{ $\mu$-a.e.} 
$$
(ii) There exists an operator core $D$ for $S$ such that for all $\lambda>0$, $f_1 \in D$,  $h\in \IL^2_{\geq 0}(M)$, there exists an $f_2 \in\Gamma_{\IL^2}(M,E)$ with the following properties:
\begin{itemize}
\item $|f_2|=(T+\lambda)^{-1}h$ $\mu$-a.e.
\item $\left\langle f_1,f_2\right\rangle=\left\langle|f_1|,|f_2|\right\rangle$
\item $\Re  \left\langle Sf_1,f_2\right\rangle\geq \left\langle|f_1|,T|f_2|\right\rangle.$
\end{itemize}
\end{Theoreme}

The following simple observation, which can also be found in \cite{hess}, is sometimes useful:

\begin{Remarke}\label{dap} Let $A$ be a bounded operator in $\Gamma_{\IL^2}(M,E)$, and let $B$ be a bounded operator in $\IL^2(M)$. Then the following statements are equivalent:
\begin{itemize}
\item For all $f\in\Gamma_{\IL^2}(M,E)$, one has
$$
|A f|\leq B|f|\quad\text{ $\mu$-a.e.} 
$$
\item For all $f_1,f_2\in\Gamma_{\IL^2}(M,E)$, one has
$$
\left\langle A f_1,f_2\right\rangle\leq \left\langle B|f_1|,B|f_2|\right\rangle\quad\text{ $\mu$-a.e.} .
$$
\end{itemize}
\end{Remarke}

Note that Theorem \ref{hsi} is in fact entirely measure theoretic in the sense that one could replace $M$ by an arbitrary sigma-finite measure space. In this case, $E\to M$ could be taken to be any \lq\lq{}measurable metric $\IC$-vector bundle\rq\rq{} (with the canonically induced $\IL^2$-spaces). In the trivial bundle case $E=M\times \IC\to\IC$, the implication (ii) $\Rightarrow$ (i) had already been shown by B. Simon in \cite{simon1}, who also conjectured the implication (i) $\Rightarrow$ (ii) therein (and independently gave a proof of (i) $\Rightarrow$ (ii) for this special case in \cite{simon2}). For further abstract results that are in the spirit of Theorem \ref{hsi}, we refer the reader to I. Shigekawa's paper \cite{shig}, and the results by E.M. Ouhabaz \cite{ouhab1,ouhab2}. For results of this type for operators on Banach lattices, we refer the reader to W. Arendt's paper \cite{arendt}.

\section{Weak-to-strong differentiability theorem}

\begin{Theoreme}\label{weaki} Let $M$ be a smooth $m$-manifold, let $U\subset M$ be open, let $\IHH$ be a Hilbert space, and let $k\in\IN_{\geq 1}$. Then every map
$$
\Psi:U\longrightarrow \IHH
$$  
that is $C^k$ in the weak sense is automatically $C^{k-1}$ in the norm sense.
\end{Theoreme}

Clearly, one can assume that $U\subset \IR^m$ for the proof. By an induction argument we can even assume $U\subset \IR^1$. This situation is covered in Section 1.5 of \cite{davies}. As one might expect, the key to this result is the uniform boundedness principle.

\section{Trotter's product formula}

The following is T. Kato\rq{}s version of Trotter\rq{}s product formula \cite{trotter}: 

\begin{Theoreme}\label{trotter} Let $H_1$, $H_2$ be self-adjoint semibounded operators in a common Hilbert space, with the corresponding sesquilinear forms denoted by $Q_1$ and $Q_2$, respectively. Assume that $Q:=Q_1+Q_2$ is densely defined on its natural domain $\dom(Q_1)\cap\dom(Q_2)$ (it is automatically closed and nonnegative), and let $H$ denote the operator corresponding to $Q$. Then one has
$$
\mathrm{e}^{-t H}=\lim_{n\to\infty} \left(\mathrm{e}^{-(t/n) H_1}\mathrm{e}^{-(t/n) H_2}\right)^n\quad\text{ strongly as $n\to\infty$.}
$$ 
\end{Theoreme}

\chapter{Some measure theoretic results}\label{appendix_c}

The following well-known result on measurable spaces is often useful:

\begin{Propositione}\label{el} Let $X\equiv (X,\IAA)$ be a measurable space, let $(Y,\varrho)$ be a complete metric space (equipped with its Borel sigma-algebra), and let $f_n: X\to Y$, $n\in\IN$, be a sequence of measurable maps. Then the set
$$
X\rq{}:=\left\{x: \lim_{n\to\infty} f_n(x)\>\text{\emph{exists}}\right\}\subset X
$$
is measurable.
\end{Propositione}

\begin{proof} We know this fact from \cite{elworthy}. The proof is actually simple: One just has to note that
$$
X\rq{}=\left\{x  : \limsup_{n\to\infty}\sup_{k,l\in\IN, l\geq k} \varrho(f_k(x),f_l(x))=0\right\},
$$
which is obviously a measurable subset.
\end{proof}

In the sequel, given a measure\footnote{A measure will always be understood to be nonnegative. Also, it is not assumed that measures are complete or sigma-finite, unless otherwise stated.} space $(X,\IAA,\mu)$, whenever there is no danger of confusion, we will ommit the sigma-algebra $\IAA$ in the notation and simply write $(X,\mu)\equiv (X,\IAA,\mu)$, and 
$$
\IL^{q}_{\mu}(X)\equiv \IL^{q}(X,\IAA,\mu)
$$
for the corresponding \emph{complex} Banach space. The corresponding $\IL^q$-norms are denoted by $\left\|\bullet\right\|_{\IL^{q}_{\mu}}$, and the operator norms for linear operators from $\IL^{q_1}_\mu(X)\to \IL^{q_2}_\rho(Y)$ by $\left\|\bullet\right\|_{\IL^{q_1}_{\mu},L^{q_2}_{\rho}}$. Let us also abbreviate that given a \emph{real-valued} measurable function $f$ on $(X,\mu)$,
 its integral $\int f \Id \mu$ is also allowed to take the value $+\infty$ in case 
$$
\int\max(f,0)\Id\mu=\infty\quad\text{and}\quad
\int (- \min(f,0))\Id\mu<\infty,
$$
and $\int f \Id \mu$ is also allowed to take the value $-\infty$ in case 
$$
\int\max(f,0)\Id\mu<\infty\quad\text{and}\quad\int (- \min(f,0))\Id\mu=\infty.
  $$
This is implicitely understood in the formulation of Theorem \ref{domin} below.

\section{Monotone class theorem}

Let $X$ be a set and let $\IMM$ be a system of subsets of $X$. Then $\IMM$ is called
\begin{itemize}
\item a \emph{$\pi$-system in $X$}, if $\IMM$ is nonempty and stable under taking finitely many intersections
 %\item a \emph{set algebra on $X$}, if $\IMM$ is nonempty and stable under taking finitely many intersections and under taking complements in $X$
\item a \emph{monotone Dynkin-system in $X$}, if one has $X\in\IMM$ together with the following two properties:
\begin{align*}
&A,B\in \IMM, A\subset B\>\Rightarrow\> B\setminus A\in\IMM\\
&(A_j)_{j\in\IN}\subset\IMM,\>A_1\subset A_2\subset\dots \>\Rightarrow\>\bigcup_{j\in\IN}A_j\in\IMM.
\end{align*}
\end{itemize}

As for sigma-algebras, given a system of subsets $\IMM$ of $X$ there always exists a smallest monotone Dynkin-system which contains $\IMM$, and every sigma-algebra clearly is a monotone Dynkin-system. One has the following elementary but nevertheless useful measure theoretic monotone class theorem (cf. Satz 1.4 in \cite{hack}):

\begin{Theoreme}\label{monocl} Let $X$ be a nonempty set, and let $\IMM$ be a $\pi$-system in $X$. Then the smallest monotone Dynkin-system which contains $\IMM$ is equal to the smallest sigma-algebra which contains $\IMM$.
\end{Theoreme}

A typical application of this fact is the following basic uniqueness result for sigma-finite measures: 

\begin{Corollarye}\label{cara} Let $X$ be a nonempty set, let $\IMM$ be a $\pi$-system in $X$, and let $\left\langle\IMM\right\rangle$ denote the smallest sigma-algebra which contains $\IMM$. Then for every pair of sigma-finite measures $\mu_1$ and $\mu_2$ on $\left\langle\IMM\right\rangle$, one has the implication
$$
\mu_1|_{\IMM}=\mu_2|_{\IMM}\quad \Rightarrow\quad  \mu_1=\mu_2.
$$
\end{Corollarye}

Indeed, one first considers the case that both measures are finite. Then the collection 
$$
\mathscr{C}:=\{A\in \left\langle\IMM\right\rangle: \mu_1(A)=\mu_2(A)\}
$$
is a monotone Dynkin-system which contains $\IMM$. Thus $\mathscr{C}=\left\langle\IMM\right\rangle$. The extension to the sigma-finite case is straightforward.

\section{A generalized convergence result for integrals}

\begin{Theoreme}\label{domin} Let $(X,\mu)$ be a measure space, and assume that $f_n$, $n\in\IN$, $f$, $h$ are real-valued measurable functions on $X$, which satisfy
$$
h\in\IL^1_\mu(X),\> f_n\leq h ,\> f_n\geq f\>\text{ for all $n$ ,} \>\lim_{n\to\infty} f_n= f\>\text{ $\mu$-a.e.}
$$ 
Then one has 
$$
\lim_{n\to\infty}\int f_n\Id\mu=\int f \Id\mu\in \IR \cup\{-\infty\}.
$$
\end{Theoreme}

\begin{Remarke} Likewise, one can also assume
$$
h\in\IL^1_\mu(X),\> f_n\geq h ,\> f_n\leq f\>\text{ for all $n$ ,} \>\lim_{n\to\infty} f_n = f\>\text{ $\mu$-a.e.,}
$$ 
to deduce
$$
\lim_{n\to\infty}\int f_n\Id\mu=\int f\Id\mu \in \IR \cup\{+\infty\}.
$$
We refer the reader to Theorem 12.2.6 in \cite{lapidus}.
\end{Remarke}

\section{Regular conditional expectations}\label{condi}

The following simple definition will be convenient:

\begin{Definitione} A measurable space $(\Omega,\IFF)$ is called a \emph{standard measurable space}, if $(\Omega,\IFF)\cong (\IR,\IBB(\IR))$ ($\IR$ with its Borel sigma-algebra), in the sense that there exists a measurable bijection
$$
J: (\Omega,\IFF)\longrightarrow (\IR,\IBB(\IR))
$$
such that $J^{-1}$ is measurable, too.
\end{Definitione}

Uncountable Polish spaces with their Borel sigma-algebras are $\IR$-standard, and so are uncountable measurable subsets of Polish spaces with their induced (trace-) sigma-algebras \cite{ikeda}.

\begin{Theoreme}\label{port} Let $(\Omega,\IFF)$ and $(\Omega\rq{},\IFF\rq{})$ be both standard measurable spaces, with $\IFF\rq{}$ containing all singletons. Assume further that $\IP$ is a finite measure on $(\Omega,\IFF)$ and that 
$$
F: (\Omega,\IFF)\longrightarrow (\Omega\rq{},\IFF\rq{})
$$
is measurable. Then there exists a $\IP^F:=F_* \IP$-uniquely determined map\footnote{This means that, if 
\begin{align*}
&\Omega\rq{}\longrightarrow \big\{\text{probability measures on $(\Omega,\IFF)$}\big\},\\
&\omega\rq{}\longmapsto \widetilde{P(\bullet|F=\omega\rq{})}
\end{align*} 
is another map with the stated property, then for $\IP^F$-a.e. $\omega\rq{}$ and for all $N\in\IFF$, one has $\widetilde{\IP(N|F=\omega\rq{})}=\IP(N|F=\omega\rq{})$.} 
\begin{align*}
&\Omega\rq{}\longrightarrow \big\{\text{\emph{probability measures on $(\Omega,\IFF)$}}\big\},\\
&\omega\rq{}\longmapsto \Big(\>\IFF\ni F \longmapsto \IP(N|F=\omega\rq{}) \in [0,1]\>\Big)
\end{align*}
with the following property: For all Borel functions
$$
\Psi:(\Omega,\IFF)\longrightarrow [0,\infty),\>\>\Psi\rq{}:(\Omega\rq{},\IFF\rq{})\longrightarrow [0,\infty),
$$
the function 
$$
(\Omega\rq{},\IFF\rq{})\ni \omega\rq{}\longmapsto\int_{\Omega}\Psi(\omega)\Id\IP(\omega|F=\omega\rq{})\in [0,\infty]
$$
is Borel with
\begin{align}\label{gqq}
\int_{\Omega} \Psi\rq{}(F(\omega)) \Psi(\omega)\Id\IP(\omega)=\int_{\Omega\rq{}}\Psi\rq{}(\omega\rq{})\left(\int_{\Omega}\Psi(\omega)\Id\IP(\omega|F=\omega\rq{})  \right) \Id\IP^F(   \omega\rq{}).
\end{align}
The map $\omega\rq{}\mapsto\IP(\bullet|F=\omega\rq{})$ is called the \emph{regular conditional expectation of $\IP$ with respect to $F$}. It has the following additional property: For all $\omega\rq{}\in \Omega\rq{}$, $B\rq{}\in\IFF\rq{}$, one has 
$$
\IP (\{F\in B\rq{}\}|F=\omega\rq{}) =1_{B\rq{}}(\omega\rq{}),\>\text{ in particular, }\>
\IP( \{F=\omega\rq{}\}|F=\omega\rq{}) =1.
$$
\end{Theoreme}

\begin{proof} First of all, by replacing $\IP$ with $\IP/\IP(\Omega)$ if necessary, we can and we will assume that $\IP$ is a probability measure. In this case, all statements follow from Theorem 3.3 in \cite{ikeda} and its corollary, except that there one finds the statement
$$
\int_{\{F\in B\rq{}\}} \Psi(\omega)\Id\IP(\omega)=\int_{B\rq{}}\left(\int_{\Omega}\Psi(\omega)\Id\IP(\omega|F=\omega\rq{})  \right) \Id\IP_F(   \omega\rq{}),
$$
for all $B\rq{}\in\IFF\rq{}$, instead of (\ref{gqq}). This is, however, easily seen to be equivalent to the latter (by approximating $\Psi\rq{}$ with a monotonely increasing sequence of simple functions and using monotone convergence).  
\end{proof}

We also stress the fact that even if $\IP$ is not assumed to be a probability measure, in any case $\IP(\bullet|F=\omega\rq{})$ is a probability measure for all $\omega\rq{}\in\Omega\rq{}$.

\section{Riesz-Thorin interpolation}

We will make use of the following interpolation theorem for complex-valued $\IL^{q}$-spaces. The proof uses Hadamard\rq{}s three line theorem. As one might guess, the result turns out to be \lq\lq{}quantitatively\rq\rq{} wrong for real-valued $\IL^{q}$-spaces, if one allows all values of $a_j$, $b_j$ in the below result (although of course the result remains \lq\lq{}qualitatively\rq\rq{} true in the real case, meaning that one simply gets worse constants then). We refer the interested reader to \cite{vogt} for a detailed discussion of these subtleties, and to Satz 2.65 in \cite{weidmann2} for a detailed proof of Riesz-Thorin\rq{}s interpolation theorem.

\begin{Theoreme}[Riesz-Thorin\rq{}s interpolation theorem]\label{riesz} Let $(X,\mu)$ and $(Y,\rho)$ be sigma-finite measure spaces, let $a_0,a_1,b_0,b_1\in [1,\infty]$, and assume that
$$
T:\IL^{a_0}_\mu(X)\cap \IL^{a_1}_\mu(X)\longrightarrow  \IL^{b_0}_\rho(Y)\cap \IL^{b_1}_\rho(Y)
$$ 
is a complex linear map. Assume further that there are numbers $C_0,C_1>0$ such that for all $f\in \IL^{a_0}_\mu(X)\cap \IL^{a_1}_\mu(X)$ one has
$$
\left\|Tf\right\|_{L^{b_0}_{\rho}}\leq C_0\left\|f\right\|_{L^{a_0}_{\mu}},\>\>\left\|Tf\right\|_{L^{b_1}_{\rho}}\leq C_1 \left\|f\right\|_{L^{a_1}_{\mu}}.
$$ 
Then for any $r\in [0,1]$, there exists a bounded extension 
$$
T_{a_r,b_r}\in \ILL\big(\IL^{a_r}_\mu(X),\IL^{b_r}_\rho(Y)\big)
$$
of $T$, which satisfies
$$
\left\|T_{a_r,b_r}\right\|_{L^{a_r}_{\mu},L^{b_r}_{\rho}}\leq C_0^{1-r} C_1^{r},\>\text{ where }\>\f{1}{a_r}:=\f{1-r}{a_0}+\f{r}{a_1},\>\f{1}{b_r}:=\f{1-r}{b_0}+\f{r}{b_1},
$$
with the usual conventions $1/\infty:=0$, $1/0:=\infty$.
\end{Theoreme}

\section{Pitt\rq{}s Theorem}

Finally, let us record the following result by L.D. Pitt \cite{pitt} on the stability of the compactness of linear operators (whose proof is surpringly complicated!):

\begin{Theoreme}\label{pittts} Let $(X,\mu)$, $(Y,\rho)$ be measure spaces, let $1<p_1<\infty$, $1\leq p_2\leq \infty$ and let 
$$
S,T\in \ILL\big(\IL^{p_1}_\mu(X), \IL^{p_2}_\rho(Y)\big)
$$ 
be bounded operators such that $S$ is positivity-preserving and such that for any $f\in\IL^{p_1}_\mu(X)$ one has $|Tf|\leq S|f|$ $\rho$-a.e. Then, if $S$ is a compact operator, $T$ is also a compact operator.
\end{Theoreme}

\end{document}